\documentclass[12pt]{book} 
\pdfoutput=1

\usepackage{hyperref}
\usepackage{imakeidx} 
\usepackage{endnotes}

\usepackage{graphicx}
\usepackage[utf8]{inputenc}
\usepackage{indentfirst}
\usepackage[mathlines]{lineno}
\usepackage{mlmodern} 
\usepackage{xcolor}
\usepackage{version} 

\usepackage{amsmath, amssymb, amsfonts, amscd, amsthm, mathrsfs}
\usepackage{mathabx} 

\usepackage{titlepic}
\usepackage{pdflscape}
\usepackage{cleveref}
\usepackage{subfig}
\usepackage{caption} 
\usepackage{pifont}
\usepackage{mathpartir}
\usepackage{eso-pic}
\usepackage{microtype}
\usepackage{changepage}
\usepackage{etoolbox} 
\usepackage{stmaryrd} 
\usepackage{enumitem}
\usepackage{placeins} 

\usepackage{setspace}
\usepackage{tikz} 

\makeatletter
\renewcommand{\@biblabel}[1]{[#1]\hfill}
\makeatother
\makeatletter

\newtheorem{theorem}{Theorem}
\newtheorem*{theorem*}{Theorem}

\newtheorem{lemma}[theorem]{Lemma}
\newtheorem{proposition}[theorem]{Proposition}

\newtheorem{corollary}[theorem]{Corollary}
\newtheorem{definition}[theorem]{Definition}
\newtheorem*{notation}{Notation}
\newtheorem*{example}{Example}
\newtheorem{remark}[theorem]{Remark}
\newtheorem{problem}[theorem]{Problem}
\newtheorem*{assumption}{Assumptions}
\newtheorem*{conjecture}{Conjecture}
\numberwithin{theorem}{section}

\numberwithin{equation}{section}
\numberwithin{figure}{section}

\usetikzlibrary{chains,shapes,arrows,trees,matrix,positioning,backgrounds,fit,calc,fadings}

\def\tikzfig#1#2#3{%
\begin{figure}[htb]%
  \centering
\begin{tikzpicture}#3
\end{tikzpicture}
  \caption{#2}
  \label{fig:#1}%
\end{figure}%
}
\def\smalldot#1{\draw[fill=black] (#1) %
node [inner sep=1.3pt,shape=circle,fill=black] {}}

\def\vvert{\|}
\def\aa{\xi}   
\def\uedge{u_{edge}}
\def\R{\mathbb{R}}
\def\C{\mathbb{C}}
\def\Z{\mathbb{Z}}
\def \h{\mathfrak{h}}

\def\hstar{\mathfrak{h}^{\star}}
\def\H{\mathcal{H}}
\def\A{\mathcal{A}}
\def\bstar{+}
\def\Hstar{\H^{\bstar}}
\def\Ssing{\mathcal{S}_{sing}}
\def\truncstar{\mathfrak{h}^{\star\star}} 

\def\SL{\mathrm{SL}_2}
\def\OR{\mathrm{O}_2}

\def\sl{\mathfrak{sl}_2}

\def \sotwo{\mathfrak{so}(2,1)}

\def \SU{\mathrm{SU}(1,1)}
\def\su{\mathfrak{su}(1,1)}
\def\SO{\mathrm{SO}_2(\R)}
\def\op#1{\mathrm{#1}}
\newcommand{\abss}[1]{\ldbrack{#1}\rdbrack}
\def\RX{\R X}
\def\Ad{\mathrm{Ad}}
\def\Jsu{J_{\mathfrak{su}}}

\def\ad{\mathrm{ad}}
\def\packing{\mathcal{P}}
\def\LL{\mathbf{L}}
\def\Dih{\mathbf{Dih}_6} 
\def\ep{\mathbf{e}} 
\def\I{\mathbf{I}}
\def\Kbal{\mathfrak{K}_{bal}}
\def\tr{\mathrm{trace}}
\def\O{\mathcal{O}}
\def\OX{\mathcal{O}_{X}}
\def \delx{\frac{\partial}{\partial x}}
\def \dely{\frac{\partial}{\partial y}}

\def\W{{\mathcal W}}
\def\G{{\mathcal G}}
\def\CC{{\mathbf C}}
\def\DD{{\mathbf D}}
\def\rs{r_{scale}}

\def\leftopen{(}
\def\rightclosed{]}
\def\itf{\iota_{\tau{}F}}
\def\bd{\partial}

\newcommand{\mattwo}[4]{\left(
\begin{array}{cc}
#1 & #2 \\
#3 & #4 \\  
\end{array}
\right)
}
\newcommand{\dettwo}[4]{\left|
\begin{array}{cc}
#1 & #2 \\
#3 & #4 \\  
\end{array}
\right|
}

\newcommand{\bracks}[2]{
\left\langle #1 , #2 \right\rangle
}

\newcommand{\Kccs}{\mathfrak{K}_{ccs}}
\newcommand{\RR}{\Re}   

\newcommand{\mb}[1]{\mathbf{#1}}

\newcommand\ee[1]{e_{#1}^*}
\newcommand{\partials}[2]{\frac{\partial #1}{\partial #2}}

\newcommand\mcite[1]{}  

\makeatletter
\patchcmd{\BR@backref}{\newblock}{\newblock(on page~}{}{}
\patchcmd{\BR@backref}{\par}{)\par}{}{}
\makeatother

\makeindex[name=n,title={Index of Notation}]
\makeindex

\newcommand*\linenomathpatch[1]{%
  \cspreto{#1}{\linenomath}%
  \cspreto{#1*}{\linenomath}%
  \csappto{end#1}{\endlinenomath}%
  \csappto{end#1*}{\endlinenomath}%
}

\linenomathpatch{equation}
\linenomathpatch{gather}
\linenomathpatch{multline}
\linenomathpatch{align}
\linenomathpatch{align*}
\linenomathpatch{alignat}
\linenomathpatch{flalign}

\date{}
\title{Packings of Smoothed Polygons}

\titlepic{\includegraphics{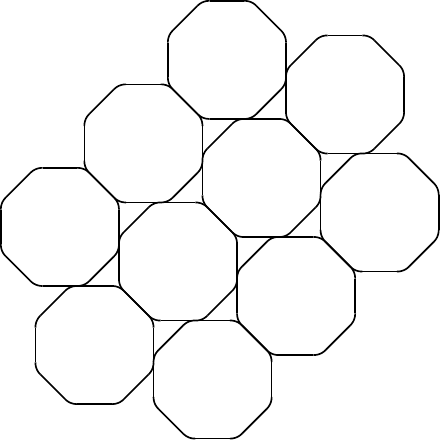}}

\author{Thomas Hales and Koundinya Vajjha}

\begin{document}
\maketitle




\chapter*{\centering Abstract}
{
\parskip=0.5\baselineskip

This book uses optimal control theory to prove that the most
unpackable centrally symmetric convex disk in the plane is a smoothed
polygon.  A smoothed polygon is a polygon whose corners have been
rounded in a special way by arcs of hyperbolas.  To be highly
unpackable means that even densest packing of that disk has low
density.

Motivated by Minkowski's geometry of numbers, which investigates
lattice packings of convex bodies, researchers (notably Blaschke and
Courant) began to search for the most unpackable centrally symmetric
convex disk (in brief, \emph{the most unpackable disk}) starting in
the early 1920s.  In 1934, Reinhardt conjectured that the most
unpackable disk is a smoothed octagon.  Working independently of
Reinhardt, but also motivated by Minkowski's geometry of numbers,
Mahler attempted without success in 1946 to prove that the most
unpackable disk must be a smoothed polygon.  This book proves what
Mahler set out to prove: \emph{Mahler's First} conjecture on smoothed
polygons.  His second conjecture is identical to the Reinhardt
conjecture, which remains open.

This book explores the many remarkable structures of this packing
problem, formulated as a problem in optimal control theory on a Lie
group, with connections to hyperbolic geometry and Hamiltonian
mechanics.  Bang-bang Pontryagin extremals to the optimal control
problem are smoothed polygons.  Extreme difficulties arise in the
proof because of chattering behavior in the optimal control problem,
corresponding to possible smoothed ``polygons with infinitely many
sides'' that need to be ruled out.  To analyze and eliminate the
possibility of chattering solutions, the book introduces a discrete
dynamical system (the Poincar\'e first recurrence map) and gives a
full description of its fixed points, stable and unstable manifolds,
and basin of attraction on a blowup centered at a singular set. Some proofs 
in this book are computer-assisted using a computer algebra system.

}


\tableofcontents

\part{Preliminaries}
\chapter{Introduction}%

\newcommand{\multi}{\mb{s}}

\index{density!packing}
\index{convex body}
\index[n]{zd@$\delta$, density!$\delta(K,\packing)$, packing}
\index[n]{K@$K$, convex disk!body}
\index{disk!convex}

This book shows how the still-unsolved Reinhardt conjecture in
discrete geometry can be formulated as a problem in optimal control
theory.  A proof of Mahler's First conjecture is presented, which is a
weak form of the Reinhardt conjecture, asserting that the most
unpackable centrally symmetric convex disk is a smoothed polygon.

Discrete geometers are interested in the class of problems which
minimize or maximize the \emph{packing density} $\delta(K,\packing)$
of a \emph{convex body} $K \subset \R^n$; that is, a convex compact
set with nonempty interior.%
\footnote{This and other terms are defined
at the beginning of Chapter 2.}
A convex body in $\R^2$ is called a \emph{convex
disk}. The packing density is roughly the fraction
of space taken up by non-overlapping congruent copies of a convex body
$K$ when they are arranged according to the packing
$\packing$\index[n]{P@$\packing$ packing} in Euclidean space
$\R^n$. Since $\delta(K,\packing)$ is function of two variables,
different flavors of this question may be posed: we may restrict the
classes of convex bodies $K$ under consideration, or we may restrict
the type of packings $\packing$.

\index[n]{zd@$\delta$, density!$\delta(K)$, greatest}
\index{sphere packing problem}

For example, the \emph{sphere packing problem} fixes the convex body
$K$ to be $B^n$\index[n]{B@$B^n$, unit ball} (the unit ball in $\R^n$)
and asks us to determine $\delta(K) := \sup_{\mathcal{P}}
\delta(B^n,\packing)$, where the supremum ranges over all possible
packings $\mathcal{P}$. Determining $\delta(K)$ for an arbitrary
convex body $K$ is an extremely hard optimization problem in general,
even in low dimensions. In three dimensions, the sphere packing
problem is the Kepler conjecture, which was asserted over 400 years
ago, but not solved until 1998~\cite{hales2005proof}.  In 2022, Maryna
Viazovska\index{Viazovska, Maryna} received a Fields medal for the
solution of the sphere packing problem in eight
dimensions~\cite{cohn2022work}.
\begin{figure}[ht]
\centering
\includegraphics[scale=0.9]{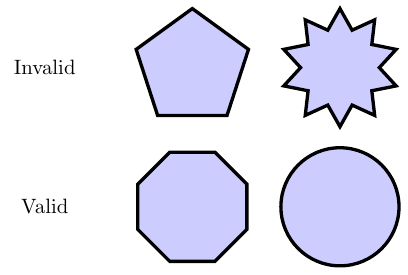}
\caption{Valid and Invalid Convex Centrally Symmetric Disks.}
\label{fig:class-ccs}
\end{figure}


\index[n]{KK@$\mathfrak{K}$, set of convex disks} 
\index[n]{KK@$\mathfrak{K}$, set of convex disks!$\Kccs$, centrally symmetric convex}

\section{Karl Reinhardt's Problem}
\index{Reinhardt, Karl} Let $\mathfrak{K}$ be the set of convex disks
in the plane, and let $\Kccs\subset\mathfrak{K}$ be the set of
centrally symmetric convex disks in the plane $\R^2$. Examples of
compact sets belonging to (and not belonging to) $\Kccs$ are shown in
Figure~\ref{fig:class-ccs}. The Reinhardt problem is to determine the
infimum
\[
\inf_{K \in \Kccs} \delta(K)
= \inf_{K \in \Kccs} \sup_{\mathcal{P}} \delta(K,\packing)
\]
and also that centrally symmetric convex disk which whose greatest
packing density achieves this minimum.  The Reinhardt problem is
structured as a minimax problem: finding the infimum of a supremum (or
to find that disk whose \textit{greatest} packing density is the
\textit{least}).  In our situation, a centrally symmetric convex disk
achieving this minimum exists. We will see below that an affine
transformation does not change the greatest packing density of a
centrally symmetric convex disk.  The minimizer is conjectured to be
unique up to affine transformation.

\begin{figure}[ht]
\centering
\includegraphics{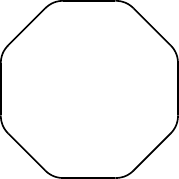}
\caption{The Smoothed Octagon.}
\label{fig:smoothed}
\end{figure}

\index{unpackable}

Although a plausible first guess for the minimizer is the circular
disk in the plane, it turns out that there is a candidate which is
slightly worse. (We say that one convex disk is \emph{worse} than
another if its greatest packing density is smaller.  To be worse is to
be less \emph{packable} and more \emph{unpackable}.) In 1934,
Reinhardt conjectured that the minimum is achieved by the so-called
\textit{smoothed octagon} pictured in
Figure~\ref{fig:smoothed}~\cite{reinhardt1934dichteste}.
Independently, Kurt Mahler arrived at the same conjecture in
1947~\cite{mahler1947minimum}.

\begin{conjecture} [Reinhardt \cite{reinhardt1934dichteste}, Mahler \cite{mahler1947minimum}]
\index{conjecture!Reinhardt}\index{Reinhardt!conjecture}
The smoothed octagon achieves the least greatest packing density among all
other centrally symmetric convex disks in the plane. Its density is
given by
\begin{equation}\label{eqn:density-formula}
\inf_{K \in \Kccs} \delta(K) 
= \frac{8 - \sqrt{32} - \ln 2}{\sqrt{8} - 1} \approx 0.902414.
\end{equation}
\mcite{MCA:7481306}
\end{conjecture}
\index[n]{0=@$\approx$, approximate equality}

\begin{figure}[ht]
  \centering
  \includegraphics{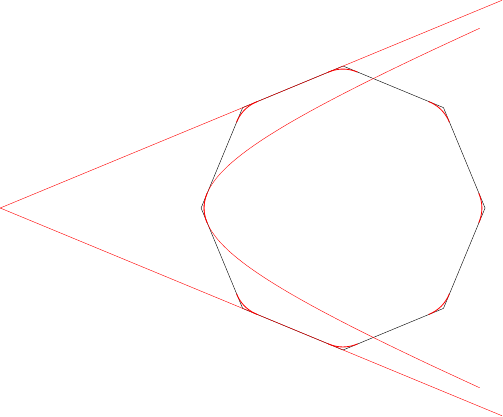}
  \caption{Construction of the Smoothed Octagon starting from a regular octagon. 
  The hyperbolic arcs clipping each vertex are shown in red. }
  \label{fig:smoothed-construction}
\end{figure}
  
The smoothed octagon is constructed by clipping the vertices of a
regular octagon by hyperbolic arcs which are tangent to the two edges 
at the vertex, and whose asymptotes pass through two further edges of the octagon, as shown in Figure~\ref{fig:smoothed-construction}.
The density formula appears in Fejes T\'oth~\cite[p.106]{toth2023lagerungen}.
A calculation of the density formula appears in an example
before Theorem~\ref{thm:6k+2}.

\section{History of the Reinhardt Problem}

The earliest mention of the Reinhardt problem is as Problem 17 in \S27
of a 1923 book by Wilhelm Blaschke,\index{Blaschke, Wilhelm} where it
is called \emph{Courant's
  conjecture}\index{conjecture!Courant}~\cite{blaschke1945differentialgeometrie}, 
  stating that the worst of all centrally symmetric convex
disks is the circular disk (whose greatest packing density in the
plane is ${\pi}/{\sqrt{12}}$).

\index[n]{zp@$\pi=3.14\ldots$}

Less than a decade later, Richard Courant's conjecture was shown to be
false by Reinhardt in his 1934 article, by his construction of the
smoothed octagon.  In his article, Reinhardt also proved fundamental
results on the existence, structure, and regularity of a minimizer.
The title and motivation for Reinhardt's article came from
Minkowski's\index{Minkowski,~Hermann} work from 1904 on the lattice
packings of convex bodies~\cite{minkowski1904dichteste}.  Reinhardt
wrote,
\begin{quote}
\emph{``Bei unseren Bereichen kommt diejenige
Figure in Betracht, welche aus einem regelm\"a\ss igen Achteck ensteht,
wenn man jede Ecke durch diejenigen Hyperbel abschneidet, die die
beiden austo\ss enden Seiten ber\"uhrt, und die beiden wieder an diese
grenzenden Seiten zu Asymptoten hat''}~\cite[p.~230]{reinhardt1934dichteste}. 

\smallskip
Among our regions, that
figure comes into consideration which arises from a regular octagon, if
one cuts off each corner with that hyperbola which
is tangent to the two outgoing sides, and
again has the two sides bordering on these as asymptotes. (Compare
Figure~\ref{fig:smoothed-construction}.)
\end{quote}

More than a decade later, Kurt Mahler\index{Mahler, Kurt} was also led
to the smoothed octagon in a series of articles in 1946--47. Mahler's
first article used the calculus of variations to refute Courant's
conjecture by proving the existence of a convex disk whose packing
density was worse than the circular disk~\cite{mahler1947minimum}. In
this paper, Mahler formulates the packing problem, considers a
parameterized family of convex domains adapted to this problem, and
writes down necessary conditions these domains should satisfy.  Making
the assumption that the boundary is sufficiently smooth, and by
disregarding the convexity constraint, he shows that the only solution
to the Euler-Lagrange equation is a circle, up to affine
transformation.  He then takes a second variation of the circle to
show that it is not second-order optimal.  In this way, he learns that
the convexity condition cannot be disregarded.  Our treatment of the
circle is similar to his~\cite[\S5.1,\S5.2]{hales2011}.  Like Mahler,
we use parameterizations in $\SL(\R)$.

\index{Euler-Lagrange equation}
\index{calculus of variations}

In the same article, Mahler makes the final remark:
\begin{quote}
    \textit{It seems highly probable from the convexity condition, that the
    boundary of an extreme convex domain consists of line segments and
    arcs of hyperbolae. So far, however, I have not succeeded in
    proving this assertion.}
\end{quote}
\index{smoothed octagon}

We refer to this final remark as \emph{Mahler's First conjecture:} the
most unpackable centrally symmetric convex disk is a smoothed polygon.

In a follow-up article, Mahler gives an explicit construction of the
smoothed octagon~\cite{mahler1947area}. The term \emph{smoothed octagon}
appears explicitly in a later article by Mahler and Ledermann in
1949~\cite{ledermann1949lattice}. Although we may tend to cite
Mahler more frequently than Reinhardt, and although they worked from
different perspectives, we wish to make it clear that priority for many 
early results belongs to Reinhardt.  

Further progress was achieved by V. Ennola in 1961 and Paul Tammela in
1969 where they showed that $\inf_{K \in \Kccs} \delta(K)\ge
0.8926...$ ~\cite{ennola1961lattice}~\cite{tammela1969estimate}.
Fedor Nazarov proved that the smoothed octagon is a local minimum in
the space of convex disks equipped with the Hausdorff
metric~\cite{nazarov1988reinhardt}.  Discussions of the Reinhardt
conjecture appear in the books by J\'{a}nos Pach and Pankaj Agarwal
and by L. Fejes T\'oth
~\cite{pach2011combinatorial}~\cite{toth2013lagerungen}.  Hales's
earlier work treats the Reinhardt problem as a problem in the
calculus of variations~\cite{hales2011}.

As of 2024, the full Reinhardt conjecture is still beyond our
immediate reach, having remained open since 1934.  However, we firmly
believe that optimal control theory is the proper framework for the
study of this conjecture.  This book uses optimal control theory to
give a proof of Mahler's First conjecture.

\section{Book Summary}

This book is an extension of a 2017 preprint of Hales in which the
Reinhardt problem is reduced to an optimal control problem on the
tangent bundle of the Lie group $\SL(\R)$~\cite{hales2017reinhardt}.
The book also grows out of the 2022 PhD thesis of Vajjha, which
considerably extends the theoretical
framework~\cite{vajjha-phdthesis}.  We include all the results from
that preprint and thesis, and we carry the program much further still.

As we show, the Reinhardt optimal control problem has a remarkable
amount of structure with deep connections with hyperbolic geometry,
Hamiltonian mechanics and the theory of chattering control.  It is our
belief that the Reinhardt conjecture has now been transformed from an
impossible problem to a difficult, but approachable one.

\index{Mahler's First conjecture}
\index{multi-curve}

In Part I of the book, we recall Reinhardt's and Mahler's results,
which will be essential for the construction of our control problem.
In the formulation of the control problem, properties proved by
Reinhardt himself in 1934 play an essential role.  As an example,
Reinhardt proved that the boundary of the minimizer is described by
six points moving to generate six curves, which close up seamlessly
into a single simple closed curve.  The origin and any three
of these consecutive points form a parallelogram whose area 
remains fixed as the three points move around the
boundary. This is shown in Figure \ref{fig:multi-curve-oct}.  The six
curves (with centrally symmetric pairs colored similarly) 
form a \emph{multi-curve}, in a sense made precise in Definition~\ref{def:multi-pt-multi-curve}.

\begin{figure}[!htbp]
\centering

\subfloat{\includegraphics[width=0.33\textwidth]{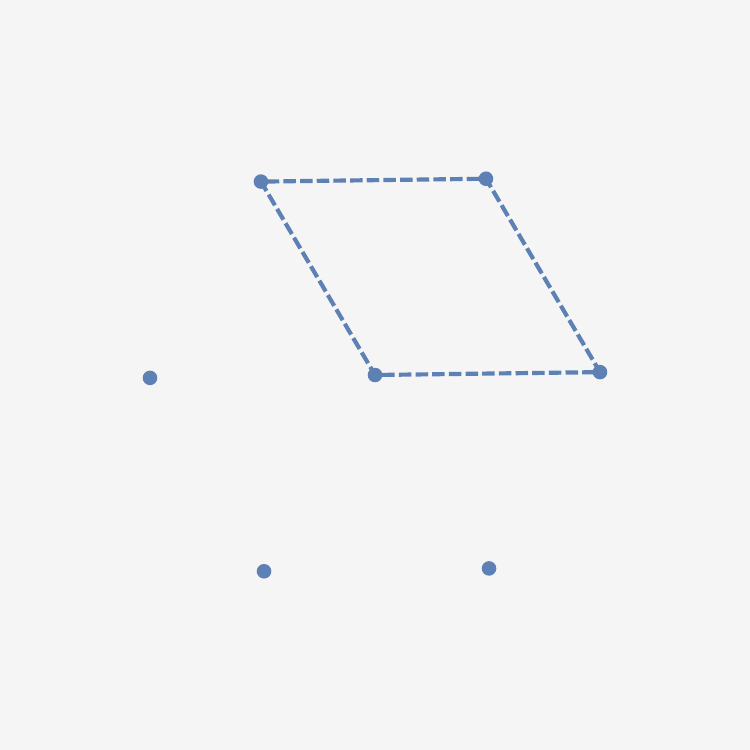}}
\subfloat{\includegraphics[width=0.33\textwidth]{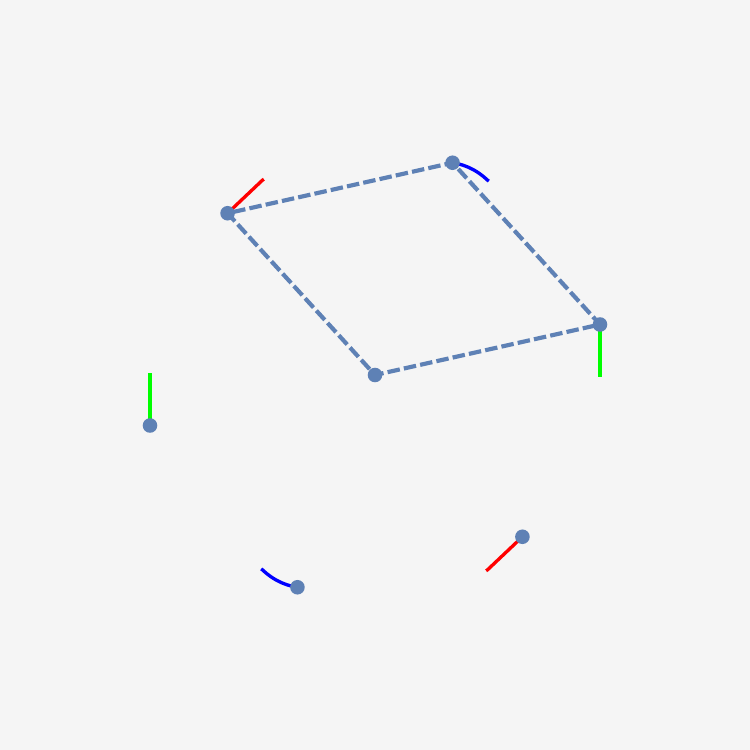}}
\subfloat{\includegraphics[width=0.33\textwidth]{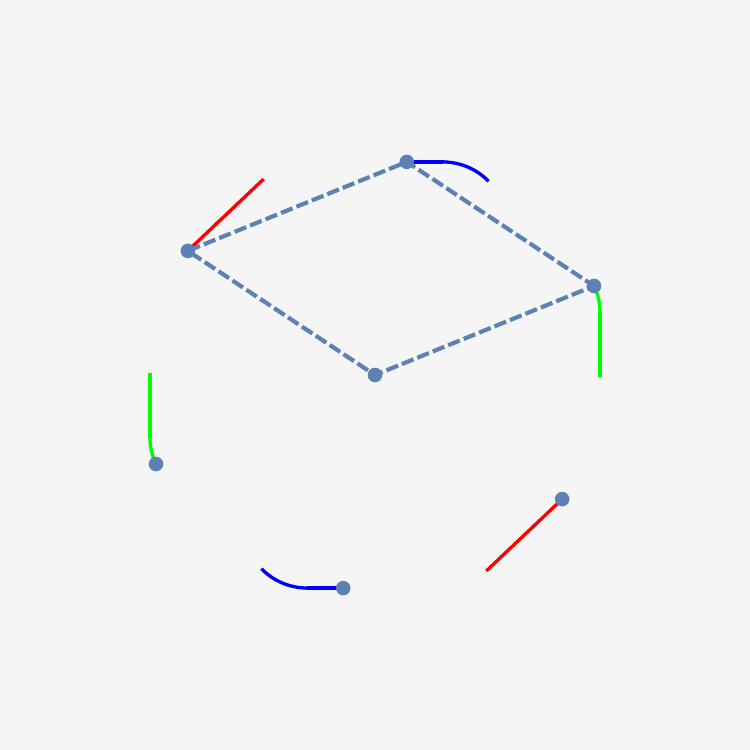}}

\subfloat{\includegraphics[width=0.33\textwidth]{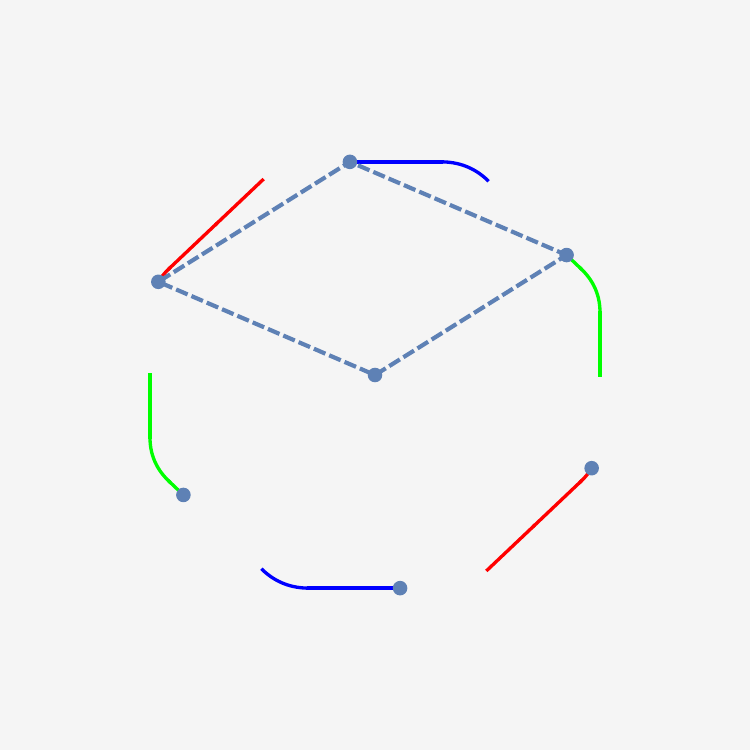}}
\subfloat{\includegraphics[width=0.33\textwidth]{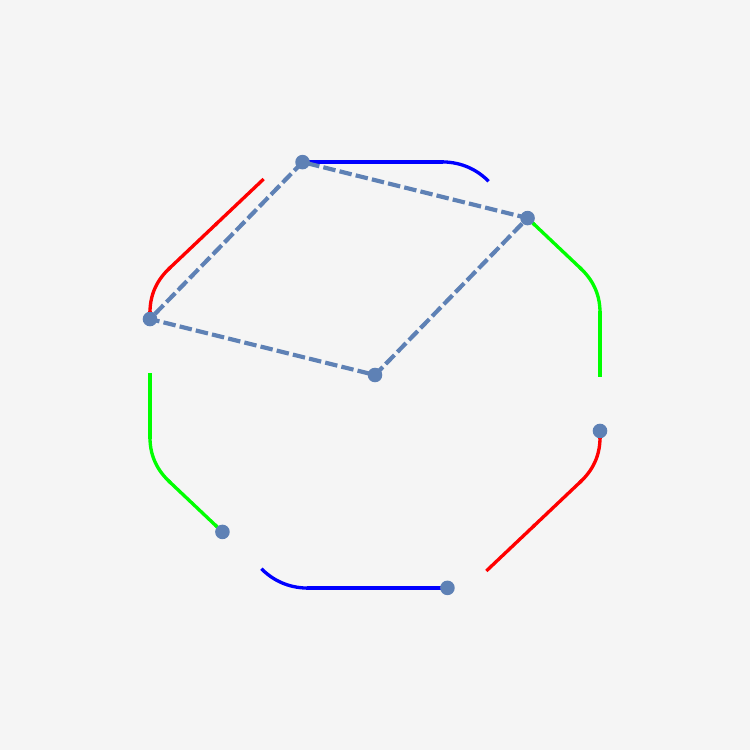}}
\subfloat{\includegraphics[width=0.33\textwidth]{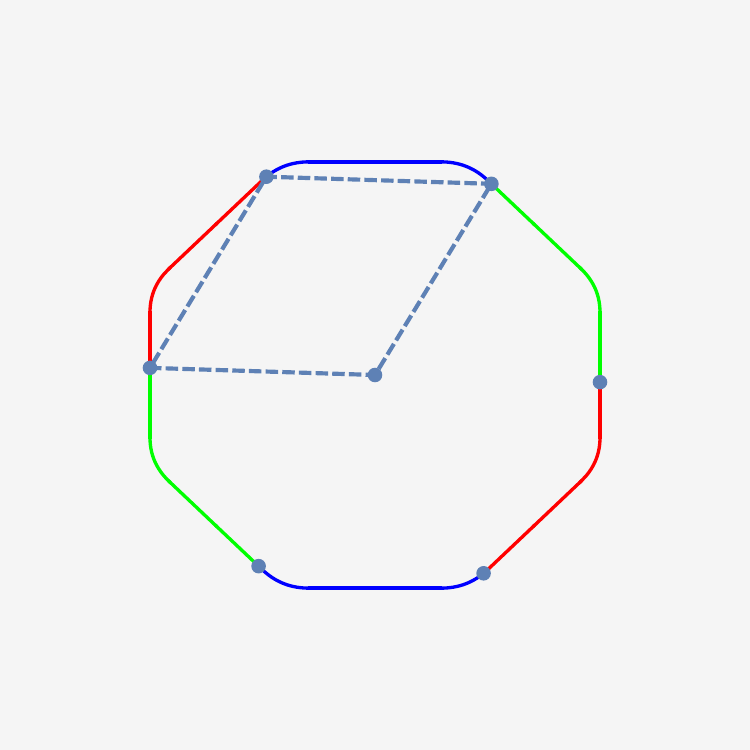}}
\caption{Multi-curves generating the smoothed octagon.  The parallelograms
all have the same area.}\label{fig:multi-curve-oct}

\end{figure}

The points move along the boundary curves in a way that yields the
convexity of the enclosed disk. Convexity is imposed \textit{locally}
via a local curvature nonnegativity condition and \textit{globally} via
conditions on the tangents to the six curves. The curvatures
of these curves play a role in determining the packing density of the
resultant disk in the plane. The control problem reformulation takes
all these conditions into account. 

A noteworthy feature of our control problem is that the set of
controls is the standard two-simplex in $\R^3$. Each point in the
control set can be viewed as a normalized ordered triple of
curvatures, which are used to specify the planar curvatures of the six
curves describing the boundary of a minimizer $K$.  The bounding edges
of the two-simplex constrain the planar curvatures to be nonnegative,
enforcing the convexity of $K$.  We prove an analogue of the
Frenet-Serret formulas, showing that the six curves are determined
from the curvature control function, by solving a second order
ordinary differential equation with initial values. These differential
equations appear in what we call the \emph{state equations} of the
Reinhardt control problem~(equations \eqref{eqn:trisystem-g} and
\eqref{eqn:trisystem-X} in
problem~\ref{pbm:reinhardt-optimal-control}).

\index{state equations}

\index{Frenet-Serret formula}
\index[n]{SL@$\mathrm{SL}_n$, special linear group}

The six curves describing the boundary of a minimizer $K$ can be
generated in a uniform way from a single curve, taking values in the
Lie group $\SL(\R)$ of $2\times2$ matrices with real entries and
determinant one.  The control problem will occur naturally on a
manifold that is closely related to this Lie group (its tangent
bundle).  \index{coadjoint orbit} \index{hyperbolic!geometry}
\index{upper-half plane}

Symmetry is visible throughout this book and plays an important
role. The control problem alluded to above is a left-invariant control
problem on a Lie group, and it is well known that such problems admit
a reduction in dynamics to coadjoint orbits in the corresponding Lie
algebra~\cite{jurdjevic_1996}.  The Poincar\'{e} upper-half plane (a
well-known model of hyperbolic geometry) and its invariant metric also
appear in a natural way -- being symplectomorphic to this coadjoint
orbit. Lemma~\ref{lem:Riemannian-speed} shows that many important
trajectories of the control problem have constant speed with respect
to the Riemannian metric on the upper-half plane.

Another prominent symmetry is the discrete dihedral symmetry of the
equilateral triangle (expanded on in Section~\ref{sec:dihedral})
arising from the standard two-simplex, which is our control set.
Using the isomorphism between $\SL(\R)$ and the special unitary group
$\SU$, we transfer the dynamics to the hyperboloid model of the
hyperbolic plane.  In the hyperboloid model, the symmetries take a
particularly nice form.

Part II of this book explores the \emph{Reinhardt optimal control
problem} \ref{pbm:reinhardt-optimal-control}, and highlights its
remarkable structure and these connections with hyperbolic geometry and
Hamiltonian mechanics.
\index{Reinhardt!optimal control problem}

The state space of the control problem is unbounded.  In
Chapter~\ref{sec:compactification}, the state space is compactified by
cutting down the region of interest to a compact set containing the
global minimizer $K$.  This compactification is achieved by giving a
geometric interpretation to trajectories that stray outside the
compact set.  In a sense that we make precise, such trajectories
correspond to convex disks that are approximately parallelograms.  Any
such approximate parallelogram yields a packing in the plane with high
density.  In particular, it cannot have the worst greatest packing
density, and thus it can be ruled out.  Our motivation for
compactifying the state space has been to make the Reinhardt control
problem more amenable to numerical computer experiments and possibly
also more amenable to computer-assisted proof.

\index{PMP@Pontryagin Maximum Principle}
\index{PMP, Pontryagin Maximum Principle}

First order necessary conditions for optimality of an optimal control
problem are given by the Pontryagin Maximum Principle (PMP) which
states that the optimal trajectory is given by a projection of the
lifted controlled trajectory (living in the cotangent bundle of the
underlying manifold)~\cite{zelikin2004control}.  This means that our
control problem is a \textit{higher-order} variational problem since
it involves the cotangent bundle of the tangent bundle of a Lie
group~\cite{colombo2014higher}.  The lifted controlled trajectory is
the Hamiltonian flow of the \textit{maximized
  Hamiltonian}\index{Hamiltonian!maximized}, which is the pointwise
maximum of a control-dependent Hamiltonian function
\index{Hamiltonian!control-dependent} on the control set.

\index{bang-bang}

A control function is said to be \emph{bang-bang} if its range is
contained in the set of extreme points of the control set, with
discontinuous switching.  In our setting, the extreme points of the
control set are the vertices of the two-simplex. Critical points of
control problems are frequently given by bang-bang controls.  The
smoothed octagon is an example of an explicit solution to the
Reinhardt optimal control problem with bang-bang control.  In earlier
research, we viewed the Reinhardt conjecture as a problem in the
calculus of variations~\cite{hales2011}.  However, one of our primary
reasons to reformulate the conjecture as a problem in optimal control
theory is the bang-bang behavior of the smoothed octagon.  This
behavior can be seen geometrically in the way its boundary switches
suddenly between linear segments and segments that are as curved as
possible at the rounded corners.  This insight explains its shape.
Lemma~\ref{lem:hyp-arc} shows that the rounded corners of the smoothed
octagon are indeed as highly curved as possible, subject to the
constraints of the problem.

In Part III, we construct explicit solutions to the Reinhardt optimal
control problem \ref{pbm:reinhardt-optimal-control}.  A bang-bang
control function with finitely many switches always produces a
smoothed polygon, and a smoothed polygon with the corner-rounding of
sort considered by Reinhardt and Mahler has a bang-bang control
function with finitely many switches.  In particular, in Section
\ref{sec:6k+2-gons}, the smoothed octagon is shown to be a critical
point of the optimization problem (a Pontryagin extremal)\index{Pontryagin extremal} given by a
bang-bang control.  More generally, for each $k=1,2,\ldots$,
Theorem~\ref{thm:6k+2} constructs a smoothed $6k+2$-gon that is a
Pontryagin extremal of the Reinhardt optimal control problem. The
associated control is bang-bang.  As $k$ tends to infinity, the
smoothed $6k+2$-gon converges to the circle, which is also a
Pontryagin extremal, but its control function is not bang-bang.  Among
this explicit list of extremals, the smoothed octagon has the worst
greatest packing density.

If we examine the initial and terminal conditions of the optimal
control problem, we find that the boundary conditions are periodic
modulo a rotation by angle $\pi/3$.  We can use the rotational
symmetry of the boundary conditions to extend every extremal
trajectory to a periodic orbit with a discrete rotational symmetry.
In this way, the global minimizer of the Reinhardt control problem can
be viewed as a periodic orbit of the dynamical system.  Because of
this, the research focus should be on the periodic Pontryagin
extremals.  For example, if we could classify the periodic extremals,
then the global minimizer could be picked out from among them.

No other extremals have been found, but we have no proof that no other
extremals exist.  In light of our results, a proof that no other
extremals exist would complete a proof of the Reinhardt conjecture.
However, we do not hazard a guess about whether other extremals might exist.

\index{singular extremal}
\index{abnormal extremal}
\index{edge!control problem}

\emph{Singular arcs} in optimal control problems arise when the
maximization condition in the Pontryagin maximum principle fails to
determine a unique control over an interval of time. In such a case, a
face of the control two-simplex can be found such that the entire face
satisfies the maximization condition (by Lemma
\ref{lem:control-face-lemma}).  When that face is an edge of the
two-simplex, an anomalous situation occurs. Abnormal Pontryagin
extremals exist, but those abnormal extremals are obviously spurious
solutions.  In this case, we modify the control problem slightly to
form what we call the \emph{edge control problem}.  In the modified
control problem, the spurious extremals disappear, and every
Pontryagin extremal of the edge control problem has a bang-bang
control with finitely many switches.  This is Theorem~\ref{thm:edge}.

\index{face}

As we have just mentioned, associated to each
singular arc is a maximizing face of the control two-simplex.  That
face can be the entire control simplex.  
In this case, 
Theorem~\ref{thm:circle-iff-singular} gives a characterization of
singular extremals. The result states that up to an affine
transformation, the arc of a circle (generating the circular disk $K$
in the plane) is the unique such singular extremal of the Reinhardt optimal
control problem.  Although the arc of a circle is an extremal in the
sense of Pontryagin, it does not satisfy the necessary second order
conditions to be a global minimizer.
Theorem~\ref{thm:no-singular-arcs} proves that the global minimizer
contains no singular arcs.

The region of state space where singular behavior might appear is
called the \emph{singular locus}.  As mentioned, the global minimizer
does not remain in the singular locus during any positive time
interval.  Nevertheless, it is possible for trajectories to approach
the singular locus, without remaining in the locus during a positive
time interval.  Extremals that completely avoid the singular locus
have a simple form, as follows.

\begin{theorem*}[\ref{thm:finite-bang-bang}]
Every Pontryagin extremal of the Reinhardt control problem which does
not meet the singular locus is given by a bang-bang control
function with finitely many switches.
\end{theorem*}

In terms of the centrally symmetric convex domain $K$, this theorem
implies that any such extremal is a smoothed polygon whose corners are
rounded by hyperbolic arcs, according to Reinhardt's corner-smoothing
procedure.  Thus, if the global minimizer avoids the singular locus,
then the global minimizer is a smoothed polygon.  We eventually show
that the global minimizer does indeed avoid the singular locus, and
this yields a proof of Mahler's First conjecture,
which is the main result of the book.

\begin{theorem*}[Mahler's First~\ref{thm:mahler}]
The global minimizer of the Reinhardt optimal control problem is a
bang-bang solution with finitely many switches.  In particular, the
minimizer $K_{\min}$ of the Reinhardt problem is a finite-sided
smoothed polygon with rounded hyperbolic arcs at each corner of the
sort described by Reinhardt and Mahler.
\end{theorem*}

One of the most intriguing aspects of the Reinhardt problem is the
behavior of trajectories near the singular locus.  The only way for a
Pontryagin extremal to approach the singular locus is through
\textit{chattering}, which is the term used to denote the phenomenon
when the control function performs discontinuous and increasingly
rapid transitions between extreme points of the control set in order
to approach the singularity~\cite{zelikin2012theory}. They were first studied 
in a problem of A.~T.~Fuller in 1963 \cite{fuller1963study} and were considered
pathological for a time, but were eventually proven to be
\textit{ubiquitous} in a very precise sense by Ivan A. Kupka in the
1990s~\cite{kupka2017ubiquity}. One of the main results of this book
is the recovery of the Fuller optimal control system in a neighborhood
of the singular locus.  To explain this precisely, we first discuss
\emph{circular control} sets.

\index{disk!circular}

An important strategy for us is to change the shape of the control set
from the two-simplex to a circular disk.  Part IV of this book is
devoted to the study of the control problem for a circular control
set. By changing the control set, the control problem changes, and we
are no longer studying the Reinhardt problem in discrete geometry.
However, there are good reasons to investigate the optimal control
problem with a circular control set.

There are a few different ways to imagine the relationship between the
original Reinhardt problem and the modified control problem with
circular control.  First, if we take the circular control set to be
the circumscribing circle of the triangular control, then the modified
optimal control problem is a \emph{relaxation} of the original
Reinhardt optimal control problem.  This means that a lower bound on
the cost in the modified problem should be a lower bound on the cost
of the original problem.%
\footnote{We do not claim this relaxation result as a theorem
  because of technicalities related to the fact that we have not
  extended the compactification result of
  Chapter~\ref{sec:compactification} to the relaxed problem.  The
  issue is that the relaxed problem is defined on a \emph{smaller
    domain} than the original problem, and a global minimizer of
  a relaxed problem on the smaller domain is not \emph{a priori} a
  lower bound to the unrelaxed problem on the full domain.  }

Second, the triangular control is a discretization of the circular
control.  We know that the optimal control function for trajectories that avoid
the singular locus takes values in the set of extreme points of the
control set.  For circular control, each point on the circle is an
extreme point, and the control function is continuous.  For triangular
control, the control function is discontinuous, taking values at the
vertices of the two-simplex.  Hence, the Reinhardt problem can be viewed
as a three-point discretization of the continuous control.

Third, we can view the triangle as a continuous deformation of the
circle.  We can study the properties of the dynamical system for the
circular control set.  We can ask to what extent these properties are
preserved as the circular control is deformed back into a triangular
control set.

Finally, the triangular control is a symmetry breaking of the circular
control problem.  All data used to specify the Reinhardt control
problem have a rotational symmetry except for the control set.  The
modification of the control set to make it circular allows us to
construct a conserved quantity.  This we do by appealing to a
control-theoretic version of the classical Noether
theorem\index{Noether's theorem}, proved by Hector Sussmann
\cite{sussman}.

\begin{theorem*}[\ref{thm:angular-momentum}]
  In the control problem with the circular disk control set, each
  Pontryagin extremal satisfies a conservation law (which we call the
angular momentum).
\end{theorem*}
The symmetry is broken by the triangular control set, where the
angular momentum is approximately but not exactly preserved.  The
circular control problem with its extra symmetry is a toy model for
the original Reinhardt problem.  We study the modified problem in the
hope that it will lead to useful insights into the Reinhardt
conjecture.

With this background about circular control sets, we return to examine
the trajectories near the singular locus in greater detail.  The
conservation of angular momentum gives us valuable information about
the optimal control. 

Using this conservation law, we perform a truncation of the Pontryagin
control system by estimating the magnitude of terms in the system of
equations and then discarding all higher-order terms.  We make the
remarkable discovery that the truncation of our optimal control
problem is precisely the Fuller optimal control problem for a chain of odd
length.  We find inward and outward logarithmic spiral solutions to
the Fuller system, centered at the singular locus.  (In a similar
fashion, we construct triangular inward and outward spirals, when the
control set is the two-simplex.)

We make a complete analysis of the global dynamics of this Fuller
system.  The dynamical system maps onto a simpler dynamical system in
the plane.  In the planar system, there are only two critical points.
One is asymptotically stable and the other is unstable.  Every point
in the plane, except for the unstable equilibrium point, is in the
basin of attraction of the stable critical point.  Going from the
planar system back to the full Fuller system, we find that the only
trajectory that converges to the singular locus is an inward
logarithmic spiral centered at the singular locus.  The only
trajectory that escapes from the singular locus is an outward spiral,
which is unique up to rotational symmetry.  The inward spiral is
unstable, and the outward spiral is stable, so that a trajectory that
is not exactly an inward spiral must necessarily swerve away from the
singular locus, then reapproach an outward spiral.

We plot some solutions numerically and observe that the solutions
appear to behave chaotically. We conjecture that for certain parameter
values, the trajectories are indeed chaotic.  For this and other
research problems, we refer the reader to Appendix~\ref{sec:problems}.

In Part~\ref{part:mahler}, we return to the Reinhardt problem with
triangular control set.  Several further ideas are introduced to give
a proof of Mahler's conjecture.  Blowing up at the singular locus (in
the sense of algebraic geometry) creates an exceptional divisor, which
becomes the focus of attention.  We make a detailed study of the
Fuller system with a triangular control set on the exceptional
divisor.  By restricting the dynamical system to switching times, the
Fuller system becomes a discrete dynamical system whose dynamics are
given by a Poincar\'e first recurrence map.

We find that the discrete dynamical system has several features that
are remarkably similar to features that were found in the toy system
with circular control.  There are exactly two fixed points. One is
stable and the other is unstable.  The two fixed points are related by
a time-reversing symmetry.  The stable fixed point has a global basin
of attraction.  The fixed points can be interpreted as self-similar
spirals in a larger dynamical system that does not factor out
by symmetries.

To prove that the stable fixed point has a global basin of attraction
we introduce an explicit geometric partition of the exceptional
divisor into finitely many compact pieces.  On each piece, the
discrete dynamical system is continuous.  The dynamical system
acquires a block upper triangular form with respect to the geometric
partition.  The strictly upper triangular blocks represent transient
behavior of the dynamical system, and the diagonal terms are localized
around the stable and unstable fixed points.  In this way, the claim
of global stability can be reduced to a statement about local
stability. From a slightly different perspective, the upper triangular
structure can be interpreted as a discrete Lyapunov function with
respect to the geometric partition.

The stable and unstable fixed points for the discrete Fuller dynamical
system are hyperbolic fixed points for the discrete Reinhardt
dynamical system.  On the blowup, the discrete Reinhardt dynamical
system extends by analytic continuation to a neighborhood of the
hyperbolic fixed points.  We study the local stable and unstable
manifolds near the fixed points.  The global stability result (for the
discrete Fuller system) is used to show that a chattering solution to
the Reinhardt dynamical system must approach and depart the blown up
singular locus through the stable and unstable manifolds of the
hyperbolic fixed points.  Explicit calculations show that trajectories
on these stable and unstable manifolds cannot be periodic. However,
the solution to the Reinhardt problem is necessarily periodic. We
conclude that the solution to the Reinhardt problem is not a
chattering solution and does not meet the singular locus.  From this
conclusion, it follows that the solution to the Reinhardt problem is a
smoothed polygon, affirming Mahler's First.

\bigskip

In many ways, the Reinhardt problem is a textbook control problem,
because of the way it employs significant parts of the general theory
in a single problem.  Among other structures, we encounter Lax equations,
control problems on Lie groups, the symplectic structure on coadjoint
orbits, Poisson brackets, Lie-Poisson dynamics, Euler-Arnold
equations, Lyapunov functions, a conserved quantity via the
Noether-Sussmann theorem, singular arcs, chattering, the Fuller system,
bang-bang solutions, and even an ODE without a Lipschitz condition.

Although this book does not succeed in resolving the Reinhardt
conjecture, it is our firm belief that optimal control theory is the
proper framework for understanding this problem.  In particular, the
Reinhardt conjecture is formulated as an entirely explicit control
problem.  This book brings us one step closer to a complete solution.

\section*{Acknowledgements}
We would like to thank Velimir Jurdjevic and Greg Kuperberg for helpful 
discussions. 

Vajjha would like to dedicate this book to his grandmother Saraswati Mokkapati. 

\newpage
\chapter{Historical Results}

\section{A Statement of the Reinhardt Conjecture}
In this section, we state the Reinhardt conjecture and introduce 
terminology used throughout this book. 

\index[n]{0@$(-,-)$, ordered pair}
\index[n]{KK@$\mathfrak{K}$, set of convex disks!$\Kccs$, centrally symmetric convex}
\index[n]{K@$K$, convex disk}
\index{body, convex}
\index{disk!convex}
\index{symmetric, centrally}
\index[n]{O@$\mathbf{0}$, origin}

We will call a compact, convex set in $\R^n$ with nonempty interior
a \emph{convex body}, and a convex body in $\R^2$
will be called a \emph{convex disk}.  By
a \emph{centrally symmetric} convex disk
in the Euclidean plane, we mean a convex disk $K$ in $\R^2$
such that if $\mb{v} \in K$ then $-\mb{v} \in K$. Here, and throughout this
chapter, we assume the center of symmetry is the origin $\mathbf{0}:=(0,0)$.
We denote by $\Kccs$ the set of all centrally symmetric convex disks in
the plane $\R^2$, which have the origin as the center of symmetry.

A family of convex disks in $\R^2$ is called
a \emph{packing}\index{packing} if any two distinct convex disks in
the family \emph{do not overlap}; that is, they have disjoint
interiors. We can now define the \emph{packing density}\index{density!packing} and \emph{greatest packing density}\index{density!greatest
packing} of a packing in $\R^2$. Intuitively the packing density
corresponds to the proportion of the plane taken up by the packing.

\index[n]{area@$\mathrm{area}$!Lebesgue measure on $\R^2$} 
\index[n]{r@$r$, real number} 

\begin{definition}[Greatest packing density]
Let $B_r$ be a ball of radius $r$ in $\R^2$ centered at the origin, and
let $\op{area}$
be the Lebesgue measure on $\R^2$. The \emph{upper}
and \emph{lower density} of a packing $\packing$ are defined to
be 
\[ 
\limsup_{r \rightarrow \infty} \frac{1}{\op{area}(B_r)}\sum_{K' \in \packing} \op{area}(K' \cap
B_r) 
\quad \mathrm{and}\quad 
\liminf_{r \rightarrow \infty} \frac{1}{\op{area}(B_r)}\sum_{K' \in \packing} \op{area}(K' \cap
B_r) 
\] 
respectively. If they both exist and coincide, the common number is
called the \emph{density of the packing}\index{density!packing}
$\packing$ and is denoted $\delta(K,\packing)$. Given a convex body
$K$ we define the \emph{greatest packing density} as the packing
density formed with congruent copies of $K$:
\[
\delta(K) := \sup \left\{ \delta(K,\packing)\mid 
\delta(K,\packing)\text { exists, and } 
\mathcal{\packing} \text{ is a packing with congruent copies of }K \right\}.
\]
\end{definition}

\index[n]{zd@$\delta$, density!$\delta(K,\packing)$, packing}
\index[n]{L@$\LL$, lattice}

It can be proved that for a convex disk $K$, one can always find a
packing $\mathcal{P}$ such that $\delta(K,\mathcal{P})$ exists and is
equal to $\delta(K)$~\cite[Exercise~3.2]{pach2011combinatorial}.

A \emph{lattice}\index{lattice} is a discrete additive subgroup of
$\R^2$ of full rank.  An important class of packings are \emph{lattice
  packings}\index{packing!lattice}, which consist of lattice
translates of of a convex disk $K$: If $\LL$ is a lattice in $\R^2$
and $K$ is a fixed convex disk, then we consider the \emph{packings of
  translates of} $K$ under $\LL$, provided the translates of $K$ do
not overlap (called the \emph{lattice packing} of $K$).  We write
$K+\LL$ for the packing and write $K+l$, for the lattice translate of
the convex disk $K$, for $l\in\LL$. We can now similarly define the
lattice packing density and greatest lattice packing density.

\begin{definition}[Greatest lattice packing density]
We define the \emph{upper} and \emph{lower densities} of a lattice
packing $K+\LL$ of congruent copies of a convex disk $K$ to be respectively
\[ 
\limsup_{r \rightarrow +\infty} 
\frac{\sum_{l \in \LL}\op{area}(B_r \cap (K + l))}{\op{area}(B_r)} 
\qquad \mathrm{and} \qquad 
\liminf_{r \rightarrow +\infty} \frac{\sum_{l \in \LL}\op{area}(B_r \cap (K + l))}{\op{area}(B_r)}.
\]

It can be proved that given a convex disk $K$ and a lattice $\LL$,
the upper and lower lattice packing densities of the packing $K+\LL$
(provided that the $\LL$-translates of $K$ have disjoint interiors)
coincide and are both equal to 
\[
\delta(K,\LL)= \frac{\op{area}(K)}{\det(\LL)}
\index[n]{zd@$\delta$, density!$\delta(K,\LL)$, of lattice packing}
\] 
where $\det(\LL)$ is the determinant of the lattice $\LL$ (see
Definition \ref{def:lattice-determinant})~\cite[Corollary~30.1]{gruber2007convex}.

The common value $\delta(K,\LL)$ of the upper and lower densities is
called the \emph{density of the lattice packing}.  The \emph{greatest
lattice packing density}\index{density!greatest lattice packing} is defined as
\[
\delta_L(K) := \sup \{ \delta(K,\LL) \ | \ \LL \text{ a lattice in } \R^2 
\text{ such that }
K + \LL\text{ is a packing }\}.
\index[n]{zd@$\delta$, density!$\delta_L(K)$, greatest lattice}
\]
\end{definition}

\index[n]{K@$K$, convex disk!$K_{sym}$, symmetrization}

If $K$ is a convex disk, let $K_{sym} := \{(\mb{v}-\mb{w})/2 \mid \mb{v},\mb{w}\in K\}$ be its
symmetrization\index{symmetrization}.  Then $K_{sym}$ is a centrally
symmetric convex disk. For a centrally symmetric convex disk $K$, we
have $K =K_{sym}$.  Minkowski made the simple observation that $K+\LL$
is a packing if and only if $K_{sym}+\LL$ is a packing.  In this way,
questions about lattice packings for $K$ can usually be reduced to
corresponding questions for $K_{sym}$.  This led early researchers to
focus on centrally symmetric convex disks.


A key point is the following theorem of L. Fejes T\'{o}th\index{Fejes
T\'oth, Laszlo}, which states that the greatest packing density and
greatest lattice packing densities are actually equal for the class
$\Kccs$~\cite{toth1950some},\cite{toth2013lagerungen},\cite{toth2023lagerungen}.
\begin{theorem}[Fejes T\'{o}th]\label{thm:fejes-toth-packing}
If $K \subset \R^2$ is a \emph{centrally symmetric} convex disk, then  
\begin{equation}\label{eqn:fejes-toth}
\delta(K) = \delta_L(K).
\end{equation}
\end{theorem} 



Many of the early research articles on the Reinhardt conjecture were
restricted in scope to lattice packings.  However, in view of Fejes
T\'oth's theorem, results about greatest lattice packing density
actually imply results about greatest packing density (for the set $\Kccs$).
In lattice form, packings of convex bodies were studied by
multiple authors, beginning with Minkowski. 

Now consider the infimum of densities:
\[ 
\delta_{\min} := \inf_{K \in \Kccs} \delta(K).
\index[n]{zd@$\delta$, density!$\delta_{\min}$, minimax}
\]
So $\delta_{\min}$ is defined as a \emph{minimax}: the least (or
worst) greatest packing density among all centrally symmetric convex
disks in $\R^2$.

In 1904, Minkowski established a lower
bound on this infimum~\cite{minkowski1907diophantische}. In 1923,
Blaschke called \emph{Courant's conjecture}\index{conjecture!Courant}
the statement that the ellipse minimizes the greatest packing
density~\cite{blaschke1945differentialgeometrie}.  Later, Reinhardt
proved that a minimizer exists, and proved several properties about
it, including the fact that the ellipse is not the minimizer, refuting
Courant's conjecture~\cite{reinhardt1934dichteste}.

\index[n]{K@$K$, convex disk!$K_{\min}$, minimax optimizer}
\index{unpackable}

Reinhardt's problem now is to explicitly describe a $K_{\min}\in\Kccs$
for which $\delta(K_{\min}) = \delta_{\min}$, and also determine this
worst greatest packing density.  It is the problem of finding the
\emph{most unpackable shape} in $\Kccs$.  As mentioned in the previous
chapter, Reinhardt suggested a specific candidate, the \emph{smoothed
  octagon}\index{octagon!smoothed}, to be the most unpackable;
that is, to be the minimax optimizer.  The smoothed octagon is a
regular octagon whose vertices have been clipped by hyperbolic arcs
(shown in Figure \ref{fig:smoothed}).

If $\LL$ gives a lattice packing of $K$, and if $g$ is an affine
transformation, then $g\LL$ gives a lattice packing of $gK$ of the same
density.  Because of this affine invariance, the set of worst disks is
stable under the group of affine transformations.

\index[n]{g@$g$, group element!affine transformation}
\index[n]{0=@$\approx$, approximate equality}

\begin{conjecture} [Reinhardt \cite{reinhardt1934dichteste}, Mahler \cite{mahler1947minimum}]
The smoothed octagon achieves the least greatest packing density among
all centrally symmetric convex disks in the plane. Its density is
given by
\[ 
\frac{8 - \sqrt{32} - \ln 2}{\sqrt{8} - 1} \approx 0.902414.
\index[n]{0@$0.902414\ldots$, smoothed octagon density}
\mcite{MCA7481306}
\]
The smoothed octagon is uniquely the worst, up to affine transformation.
\end{conjecture}

This book investigates the Reinhardt conjecture by restating it as a
problem in control theory. To do this, we rely on numerous geometric
properties of the worst convex disk $K_{\min}$, which we collect in
the following sections.
These sections review the results contained in Reinhardt's article of
1934 and Mahler's articles written in the 1940s.

\section{Reinhardt's Approach}\label{sec:reinhardt-approach}

In this section, we briefly review Reinhardt's article of 1934.
The proofs that we give will be sketches, because the full details
are available in Reinhardt's article.
Let $K$ be a centrally symmetric convex disk.  Let $K+\LL$ be a lattice
packing of $K$.

\index[n]{HK@$H_K$, critical hexagon}

\begin{lemma}\label{lem:reinhardt-tiling}  
The packing $K+\LL$ is realized by placing $K$ inside an appropriate
parallelogram or centrally symmetric convex hexagon $H_K$, tiling the
plane with translates of $H_K$, then placing a copy of $K$ inside each
translate of $H_K$.
\end{lemma}

\index[n]{0@$-^{cr}$, crop}
\index[n]{r@$r$, real number!homothety} 
\index[n]{L@$\LL$, lattice!$l$, lattice element}

\begin{proof} Homothetically expand $K$ (to $r K$)
and its translates by $\LL$ until two translates $rK$ and
$rK + l$ come in contact.  Draw a
separating line between these two translates (by the separating
hyperplane theorem).  Similarly separate other translates of $rK$ using
translates of the separating line.  Continue to homothetically expand
$rK$, but now cropping $rK$ to $(rK)^{cr}$ so as to lie between
its bounding separating lines, so that cropped translated regions
$(rK)^{cr} + l$ do not overlap.  Continue to expand until a new
point of contact is formed.  Repeat the process, adding new
separating lines, cropping, and then continuing with cropped homothetic
expansion.  Eventually, after repeating the process a finite number of
times, the plane is tiled by the translates of the cropped homothetic
expansion $H_K$ of $K$.

By construction $H_K$ is a convex polygon, because it is bounded by the
finitely many separating lines that were introduced.  It is centrally
symmetric by central symmetry of $K$ and the symmetric placement of
the separating lines.  By considerations of Euler characteristic of a
polygon tiling, the number of edges is at most six.  Thus $H_K$ is a
parallelogram or centrally symmetric hexagon. 
\end{proof}

Every centrally symmetric hexagon tiles the plane.  A parallelogram
$H_K$ never gives smaller area than that of the smallest centrally
symmetric hexagon containing $K$, because its corners can be clipped
to give a smaller hexagon containing $K$, except when $K$ itself is a
parallelogram.  In this exceptional case, $K$ itself tiles and has
greatest packing density $1$.  We exclude this exception from our
further discussions.

\begin{theorem}[Reinhardt-Fejes-T\'oth]\label{thm:frac-hexagon}
    If $K \subset \R^2$ is a \emph{centrally symmetric} convex disk
    that is not a  parallelogram, then its packing density is
\begin{equation}\label{eqn:reinhardt-fejes-toth}
\delta(K) = \frac{\op{area}(K)}{\op{area}(H_K)},
\end{equation}
    where $H_K$ is a centrally
    symmetric hexagon of least area circumscribing $K$.
\end{theorem}

We remark that the hexagon of smallest area circumscribing
a centrally symmetric disk $K$ can be
realized as a centrally symmetric hexagon
\cite[Theorem~2.5]{pach2011combinatorial}.



\begin{proof}
By Fejes-T\'oth (Equation~\ref{eqn:fejes-toth}), we have
$\delta_L(K)=\delta(K)$.  Let $\LL$ be a lattice that realizes this
equality.  By the previous lemma, the lattice packing $K+\LL$ is
obtained by tiling a centrally symmetric hexagon $H_K$.  The density of
this packing is given by \eqref{eqn:reinhardt-fejes-toth}. This
hexagon has least area among centrally symmetric ones, because
every centrally symmetric hexagon tiles, and one of smaller area
circumscribing $K$ would lead to a packing of greater density.
\end{proof}

From now on, $H_K$ will denote a centrally symmetric hexagon of
smallest area circumscribing $K$.  We call such a $H_K$ a
\emph{critical hexagon}\index{hexagon, critical}.  This terminology is
further explained in Definition~\ref{def:critical-hexagon} and
Theorem~\ref{thm:min-critical}.

The midpoints of the edges of $H_K$ lie on the boundary of $K$.  For
otherwise, an edge of $H_K$ can be rotated about some point on that
edge to create a hexagon of smaller area.
\begin{figure}
    \centering
    \includegraphics{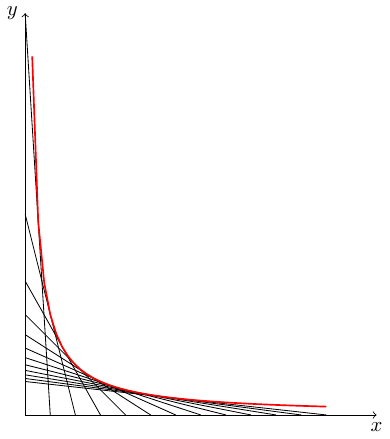}
     \caption{The hyperbola is the envelope of a pencil of 
     lines the product of whose $x$ and $y$-intercepts is constant.}
     \label{fig:hyperbolic-envelope}
\end{figure}
If we slide one edge of a convex polygon,
where the movement 
is constrained to keep the area of the polygon fixed, then the envelope 
of the moving edge is a hyperbola whose asymptotes are the lines through 
the two adjacent edges. This observation follows from the fact that 
the hyperbola is the envelope of a pencil of lines the product of 
whose $x$ and $y$-intercepts is constant. Additionally,  every 
point on the hyperbolic envelope is the midpoint of a unique line segment 
formed by a line in the pencil and the $x$ and $y$ axes. See 
Figure~\ref{fig:hyperbolic-envelope}.
\index{line, support}

\begin{definition}[Support line]\label{def:support}
\index{line, support}\index{support}
For a convex disk $K$, a support line is a line containing at least
one point of $K$ but does not separate any points of $K$.
\end{definition}

From this observation about the hyperbola, it follows that the six
edge midpoints of $H_K$ on the boundary of $K$ are not corners of $K$:
the unique support line at each midpoint is the tangent to the
hyperbolic envelope at that point.  Otherwise the area of $H_K$ is not
minimal.   

\begin{remark}\label{rem:arc-hyperbola}
  We also observe that locally around each midpoint, $K$
contains an arc of the hyperbola. It is here that we first see the
significance of the hyperbola for the Reinhardt conjecture.
\end{remark}

The ratio $\op{area}(K)/\op{area}(H_K)$ in 
Theorem~\ref{thm:frac-hexagon} is scale invariant and so there is no
loss of generality in fixing the denominator $\op{area}(H_K)$.  We
choose the normalization $\op{area}(H_K)=\sqrt{12}$, which is the area
of a regular hexagon with inradius $1$.  This choice
has the advantage of making the unit circular disk $K$ satisfy
the normalization condition.

\index[n]{s@$\multi_i$, multi-point} 
\index{multi-point}
\index[n]{det@$\det(\mb{v}_1,\mb{v}_2)$, $2\times2$ determinant, columns $\mb{v}_i$}

Let $\det(\mb{v}_1,\mb{v}_2)$ denote the
determinant of the $2\times 2$ matrix with with columns
$\mb{v}_1,\mb{v}_2\in\R^2$.
If $H_K$ is any critical hexagon of $K\in\Kccs$, then the edge
midpoints $\multi_0,\multi_1,\ldots,\multi_5$ give six
points, ordered counterclockwise, around the boundary of $K$.  
The six points, by virtue of being the edge midpoints of a centrally
symmetric convex hexagon, satisfy the following \emph{multi-point} conditions
\begin{equation}\label{eqn:multi}
\multi_0 + \multi_2 + \multi_4 = 0,\quad \multi_{j+3} = -\multi_j,\quad \det(\multi_j,\multi_{j+1}) = \text{constant},
\end{equation}
for all $j\in \Z/6\Z$, the constant being independent of $j$.
Moreover, $\det(\multi_j,\multi_{j+1})$ is independent of the critical hexagon
$H_K$, because it is a fixed fraction of the area of $H_K$.
By fixing the area of $H_K$ at $\sqrt{12}$, we have
\begin{equation}\label{eqn:sqrt12}
\det(\mb{s}_j,\mb{s}_{j+1}) = \sqrt{3}/2.
\end{equation}


\index[n]{hK@$h_K$, inscribed hexagon}
\index[n]{T@$T$, triangle}
\index[n]{s@$\multi_i$, multi-point!$\mb{s}_j^*$,  sixth roots of unity} 
\index[n]{ij@$i,j$, integers} 

Replacing $K$ by its image under an affine transformation, we may
assume that $\multi_0,\multi_1,\ldots,\multi_5$ are the sixth roots of
unity $\mb{s}_j^*$ in the plane, with complex coordinates
$\mb{s}_j^*=\exp(2\pi i j/6)$, where
$i=\sqrt{-1}$\index[n]{i@$i=\sqrt{-1}$}.  The convex hull of the six
points $\mb{s}_j^*$ is a regular hexagon $h_K$ contained in $K$.  It
follows from the convexity of $K$ that the boundary of $K$ is
contained in the union of six equilateral triangles $T_j$, where
triangle $T_j$ has vertices $\mb{s}_j^*$, $\mb{s}_{j+1}^*$,
$\mb{s}_j^*+\mb{s}_{j+1}^*$, for $j=0,\ldots,5$. See Figure
\ref*{fig:starofdavid}.
\begin{figure}
    \centering
    \includegraphics{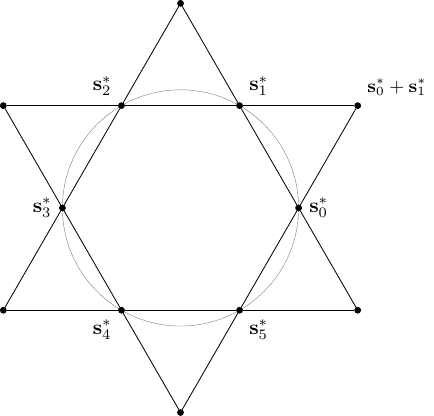}
    \caption{Hexagon formed by the sixth roots of unity along with 
    triangles $T_j$ formed by the vertices
     $\mb{s}_j^*$, $\mb{s}_{j+1}^*$,
    $\mb{s}_j^*+\mb{s}_{j+1}^*$. 
}
    \label{fig:starofdavid}
\end{figure}

\begin{lemma}\label{lem:no-corner} 
A disk $K_{\min}$ with the worst greatest density lattice packing exists.
$K_{\min}$ has no corners.  That is, there is a unique support line to
$K_{\min}$ at each boundary point.  Moreover, the support line of each
boundary point of $K_{\min}$ contains an edge of some critical
hexagon.
\end{lemma}

\index[n]{u@$\mb{u}$, vector!in $\R^2$}

\begin{proof}  
Reinhardt uses the Blaschke selection theorem\index{Blaschke selection
  theorem} to prove the existence of a centrally symmetric $K$ with a
worst greatest density lattice packing.  We fix one such $K=K_{\min}$.
The set of critical hexagons of $K$ is closed: a convergent limit of
critical hexagons is again critical.  Moreover, every point on the
boundary of $K=K_{\min}$ lies on some edge of a critical hexagon.
Otherwise, if the point $\mb{u}$ is not on the edge of any critical
hexagon, then by closedness, the same holds for all nearby boundary
points,
and a small area can be shaved in a centrally symmetric
manner from $K$ at $\pm \mb{u}$ to decrease the area of $K$ without
changing the minimal area of the critical hexagons.  This contradicts
the density minimax property of $K_{\min}$.  If a boundary point
$\mb{u}$ is a midpoint of a critical hexagon $H_K$, then as seen above
with the hyperbolic envelope, $\mb{u}$ is not a corner.

If a boundary point $\mb{u}$ lies on some edge of a critical hexagon
without being its midpoint, then an entire segment containing $\mb{u}$
of the edge lies along the boundary of $K$.  The segment determines
the unique support line for points in the relative interior of the
segment.  Each endpoint of the segment is the midpoint of an edge of a
critical hexagons and hence not a corner, for otherwise it can be
shaved as above.
\end{proof}

\begin{lemma}\label{lem:tri-interior}
  Assume that the boundary of $K_{\min}$ has critical hexagon with
  edge midpoints $\mb{s}_j^*$.
Consider a second critical hexagon of $K_{\min}$ with edge midpoints
$\multi_0,\ldots,\multi_5$, indexed so that $\multi_0\in
T_0\setminus\{\mb{s}_0^*,\mb{s}_1^*\}$.  Then for all $j$, we have that
$\multi_j$ lies in the interior of $T_j$.
\end{lemma}

Reinhardt calls this property \emph{monotonicity}.  As the midpoint
$\multi_0$ advances counterclockwise beyond $\mb{s}_0^*$, the other
midpoints $\multi_j$ advance counterclockwise beyond $\mb{s}_j^*$ into
the interior of $T_j$. See Figure \ref*{fig:monotonicity}. Note that
if $\{\multi_j\mid j\}\ne \{\mb{s}_j^*\mid j\}$, we can always assume
without generality that the subscripts are indexed such that
\[
\multi_0\in T_0\setminus\{\mb{s}_j^*\mid j\} = 
T_0\setminus\{\mb{s}_0^*,\mb{s}_1^*\},
\]
to satisfy the assumption of the lemma.
\begin{figure}
    \centering
    \includegraphics{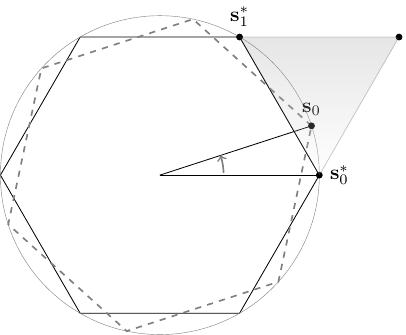}
    \caption{This figure shows \emph{monotonicity}. 
    The points $\multi_j$ of the rotated hexagons lie in triangles $T_j$.}
    \label{fig:monotonicity}
\end{figure}
\index{Minkowski sum}

\begin{proof}
Assume $\mb{s}_0\in T_0$.  For some $j$ and $k$, we have
$\multi_1\in\,T_j$ and $\multi_2\in T_k$.  As above, we have $\multi_1
= \multi_0+\multi_2$, which gives a system of constraints
\[
\multi_1\in T_j\cap (T_0+T_k),
\]
expressed using the \emph{Minkowski sum}
\[
T_0+T_k := 
\{\mb{u}_0 + \mb{u}_k\mid \mb{u}_0\in T_0,\quad \mb{u}_k\in T_k\}.
\]
Also, for fixed $\multi_0$, the inequality $\det(\multi_0,\multi_1)>0$
places a half-plane constraint on $\multi_1$.  These constraints imply
that $\multi_1\in\,T_1$ and $\multi_2\in T_2$.  Using
$\mb{s}_{j+3}=-\mb{s}_j$ and rotational symmetry, we have a weak form
of monotonicity: if for some $j$, we have
$\mb{s}_j\in{}T_j\setminus\{\mb{s}^*_j,\mb{s}^*_{j+1}\}$, then for all
$k$, we have $\mb{s}_k\in T_k$.

\index[n]{0@$\partial$, boundary}
\index[n]{0@$(-,-)$, edge between endpoints}
\index[n]{t@$t$, real number!$t_0,t_1$, scalars}
\index[n]{u@$\mb{u}$, arbitrary multi-point}

Now suppose for a contradiction that $\multi_0$ lies on the relative
interior of the edge $(\mb{s}_0^*,\mb{s}_1^*)$ of $T_0$.  By the
convexity of $K$, the entire edge $[\mb{s}_0^*,\mb{s}_1^*]$ lies on
the boundary of $K$.  Let $\mb{u}_1\in\partial K$ be any midpoint of a
critical hexagon in the interior of $T_1$ such that its support line
is different from the support lines at $\mb{s}_1^*$ and $\mb{s}_2^*$.
The point $\mb{u}_1$ exists because there are no corners, by
Lemma~\ref{lem:no-corner}.  Let $\mb{u}_j$ be the other midpoints. By
the weak form of monotonicity in the previous paragraph, the point
$\mb{u}_0$ lies in $T_0$, hence along the edge of $T_0$.  The
multi-point conditions \eqref{eqn:multi} now imply that for some
$t_0,t_1\in\R$, we have

\[
\mb{u}_0 = \mb{s}_0^* +t_0 \mb{s}_2^*,\quad 
\mb{u}_1 = \mb{s}_2^* + t_1 \mb{u}_0\quad 
\mb{u}_2 = \mb{s}_2^* + (t_1-1)\mb{u}_0.
\]
The condition that $\mb{u}_1$ is an interior point of $T_1$ gives $0<t_1<1$.
Thus, $\mb{u}_1,\mb{s}_2^*,\mb{u}_2$ are distinct collinear points on $\partial K$ so
that the entire segment from $\mb{u}_1$ to $\mb{u}_2$ lies on the boundary of
$K$.  This shows that $\mb{s}_2^*$ and $\mb{u}_1$ have the same support line,
which is contrary to the construction of $\mb{u}_1$.  Thus $\mb{u}_0$ does not
lie on the relative interior of the edge.
Repeating this argument for each $j$, we find that no $\mb{s}_j$ lies
on the relative interior of the edge $(\mb{s}_j^*,\mb{s}_{j+1}^*)$ of $T_j$.

Now assume that $\multi_0\in T_0\setminus[\multi_0^*,\multi_1^*]$.
We claim that $\multi_1\ne \mb{s}_1^*$.  
Otherwise, the multi-point condition $\multi_5=\multi_0 -\multi_1^*$ 
forces $\multi_5$ to lie along the relative interior of the edge
$(\mb{s}_5^*,\mb{s}_0^*)$, which we have shown to be impossible.
Similarly, we claim that $\multi_1\ne \mb{s}_2^*$.  Otherwise, the
multi-point condition gives $\multi_0 = \mb{s}_2^* - \multi_2$, which
forces $\multi_0$ to lie on the forbidden edge
$[\mb{s}_0^*,\mb{s}_1^*]$.

Thus, $\multi_1$ satisfies the hypotheses of the previous paragraph
(shifting indices by one): $\multi_1 \in T_1\setminus
[\mb{s}_1^*,\mb{s}_2^*]$.  Iterating the argument of the previous
paragraph for consecutive
$j$, we find that for all $j$, we have
$\multi_j\in{T_j}\setminus[\mb{s}_j^*,\mb{s}_{j+1}^*]$.  Furthermore,
no $\multi_j$ is on the boundary of $T_j$.  Otherwise, the convexity
of $K$ forces $\multi_{j+1}$ or $\multi_{j-1}$ to lie on the forbidden
edge.  This completes the proof.
\end{proof}

\begin{lemma}\label{lem:multi-cont}
  Let $K=K_{\min}$ have worst greatest packing density.
Then every point of the boundary of $K$ is the midpoint of an edge of a unique
critical hexagon.   As $\multi_0$ advances around the boundary of $K$ in a 
counterclockwise direction, the five other midpoints $\multi_1,\ldots,\multi_5$ advance
strictly monotonically and continuously in a counterclockwise direction.
\end{lemma}

\begin{proof}
Strict monotonicity is established in the previous lemma.  In
particular, each point on the boundary is the edge midpoint of at most
one critical hexagon.  Continuity follows from monotonicity if we show
that there are no jumps.  We show that every open interval along the
boundary of $K$ contains the edge midpoint of a critical hexagon.

Suppose for a contradiction that an open interval exists without a
such a midpoint.  Then picking the interval to be as large as
possible, there exist critical hexagons and edge midpoints
$\multi_0,\mb{u}_0$ marking the endpoints.  Let
$\multi_0,\ldots,\multi_5$ and $\mb{u}_0,\ldots,\mb{u}_5$ be the
corresponding midpoints of the six edges of these two critical
hexagons.  
For each $j$, we claim that
no point on $\partial K$ between $\multi_j$
and $\mb{u}_j$ is an edge midpoint of a critical hexagon.
For otherwise, by
monotonicity the same critical hexagon has an edge midpoint between
$\multi_0$ and $\mb{u}_0$.  By Lemma~\ref{lem:no-corner}, the boundary
segments of $\partial K$ between $\multi_j$ and $\mb{u}_j$ are
straight lines, included in edges of critical hexagons.  This forces
the critical hexagon for $\multi_j$ and $\mb{u}_j$ to be equal (both
hexagons having their edges along these straight lines), and since
these are the edge midpoints $\multi_j=\mb{u}_j$ for all $j$.  Thus,
no such open interval exists.
\end{proof}


Thus, in summary, excluding the degenerate case when $K$ is a parallelogram,
Reinhardt constructed the hexagon $H_K$ as the centrally symmetric hexagon of
least area containing $K$. He showed that the midpoints of the edges of $H_K$ 
lie on the boundary of $K$. He constructed $h_K$ as the centrally symmetric polygon joining
these midpoints and showed that to achieve the densest lattice
packing of $K$, the plane is tiled by copies of $H_K$. Reinhardt also proved the 
existence of a disk $K_{\min}$ which has the worst greatest lattice packing density, and 
proved properties about its boundary. 


\section{Mahler's Approach}\label{sec:mahler}

In the previous chapter, we have mentioned Minkowski's discovery of
the connection between centrally symmetric convex disks and lattices,
which resulted in the famous Minkowski theorem on lattice
points~\cite{minkowski1910geometrie}.  These results initiated
the \emph{geometry of numbers}.  Kurt Mahler rediscovered the Reinhardt
conjecture while attempting to extend Minkowski's results. These
results were published in a series of papers in the 1940s.  We review
his approach in this section.

\index[n]{O@$\mathbf{0}$, origin} 

\begin{definition}[Admissible lattice]\label{def:admissible}
\index{lattice!admissible}\index{admissible}
For a $K \in \Kccs$ centered at the origin, a lattice $\LL$ is called
$K$-admissible if no point of $\LL$ other than $\mathbf{0}=(0,0)$
lies in the interior of $K$.
\end{definition}

The lattice $\LL$ of a centrally symmetric convex disk $K$ is
$K$-admissible if and only if $K/2 + \LL$ is a lattice packing of
$K/2$.  Thus, results about admissible lattices translate readily into
results about lattice packings.

\index[n]{det@$\det(\LL)$, determinant of a lattice}
\index{covolume}
\index[n]{v@$\mb{v}$, vector!$\mb{v}_i\in\R^2$}
\index[n]{u@$\mb{u}$, vector!$\mb{u}_0,\mb{u}_1\in\R^2$}


\begin{definition}[Determinant of a Lattice]\label{def:lattice-determinant}
  For any lattice $\LL$
  with
  basis $\mb{u}_0,\mb{u}_1 \in \R^2$, the \emph{determinant} $\det(\LL)$ is equal
 to the absolute value $|\det(\mb{u}_0,\mb{u}_1)|$. This is also sometimes
 called the \emph{covolume} of the lattice $\LL$.
\end{definition}

\begin{definition}[Minimal determinant]
For $K\in\Kccs$, the minimizer 
\[ 
\Delta(K) := \inf_{K-\mathrm{admissible}} \det(\LL),
\index[n]{zD@$\Delta(K)$, minimal determinant}
\] 
where the infimum runs over all $K$-admissible lattices, is called
the \emph{minimal determinant}\index{determinant, minimal} of the
convex disk $K$.
\end{definition}

\begin{definition}[Critical lattice]
A lattice is called \emph{critical}\index{critical
lattice}\index{lattice!critical} for a convex disk $K$ if its
determinant is equal to the minimal determinant of $K$.
\end{definition}

\begin{definition}[Irreducible disk]\label{def:irreducible}
 A convex disk $K \in \Kccs$ is called \emph{irreducible}\index{irreducible disk} if every
 boundary point of $K$ lies on a critical lattice of $K$.
\end{definition}

We remark that this is not the original definition of irreducibility
of a convex disk. We choose our definition based on of
\cite[Lemma~3]{mahler1947minimum}, which shows that this definition is
equivalent to the original definition.  We reiterate that the most
significant results of this section were known to Reinhardt in 1934,
without using the language of critical lattices, minimal determinants,
and irreducibility.

Minkowski proved the following theorem which gives conditions under
which points on $K$ give rise to critical lattices. 

\begin{theorem}[Minkowski~\cite{minkowski1907diophantische}, Mahler~\cite{mahler1946lattice}]
\label{thm:minkowski-inscribe}
Let $\LL$ be a critical lattice of a convex disk $K \in \Kccs$. Then
$\LL$ contains three points $\multi_0,\multi_1,\multi_2$ on the boundary of $K$ such
that (i) $\multi_0,\multi_1$ is a basis of the lattice $\LL$, and (ii)
$\mb{0}\multi_0\multi_1\multi_2$ is a parallelogram of area $\det(\LL) = |\det(\multi_0,\multi_1)|
= \Delta(K)$, the minimal determinant of $K$. Conversely, if
$\multi_0,\multi_1,\multi_2$ are three points on the boundary of $K$ such that
$\mb{0}\multi_0\multi_1\multi_2$ is a parallelogram, then the area of this parallelogram
is not less than $\Delta(K)$ and is equal to $\Delta(K)$ if and only
if the lattice with basis $\multi_0,\multi_1$ is critical.
\end{theorem}

\index{hexagon!inscribed}
\index[n]{hK@$h_K$, inscribed hexagon}
\index[n]{hK@$h_K$, inscribed hexagon!$h$, hexagon}

The parallelogram of the theorem above is shown in Figure~\ref{fig:multi-curve-oct}. 

Since centrally symmetric hexagons can be decomposed into three
parallelograms, the above result shows that a critical lattice of a
convex disk $K \in \Kccs$ gives rise to an \emph{inscribed centrally
symmetric hexagon} $h_K$ within our convex disk $K$ so that
$\Delta(K)=\op{area}(h_K)/3$ which is \emph{minimal} in the sense
that
\[ 
\op{area}(h_K) = \inf_{h} \op{area}(h),
\]
where the infimum is taken over all hexagons $h$ with vertices
$\multi_0,\multi_1,\multi_2$ (and their reflections about $\mb{0}$) on the 
boundary of $K$ and with $\multi_0 - \multi_1 + \multi_2 = 0$. 
In 1947, Mahler~\cite{mahler1947minimum} proved an analogous
result for \emph{circumscribed hexagons} of $K$:

\index[n]{L@$\ell_i$, support line} 

\begin{theorem}[Mahler~\cite{mahler1947minimum}]\label{thm:mahler-circumscribe}
Let $K \in \Kccs$ be a convex disk which is not a parallelogram. Let
$\LL$ be a critical lattice of $K$ with lattice point
$\multi_0,\ldots,\multi_5$ on the boundary of $K$ satisfying $\multi_0
- \multi_1 + \multi_2= 0$ and $\multi_{j+3}=-\multi_j$.  Then there
are unique symmetric support lines $\ell_j$ of $K$ at
these points, such that
\begin{enumerate}
\item 
no two of these lines coincide;
\item 
the area of the centrally symmetric hexagon $H_K$ bounded by the
support lines is given by $\op{area}(H_K) = 4\Delta(K)$;
\item 
each side of $H_K$ is bisected at the lattice point $\multi_j$ where
it touches the boundary of $K$;
\item 
the hexagon $H_K$ is \emph{minimal} in the sense that 
\[ 
\op{area}(H_K) = \inf_{H} \op{area}(H),
\]
where the infimum is taken over the set of all hexagons $H$ 
bounded by symmetric support lines 
of the convex disk $K$.
\end{enumerate}
\end{theorem} 

\begin{figure}
\centering
\includegraphics{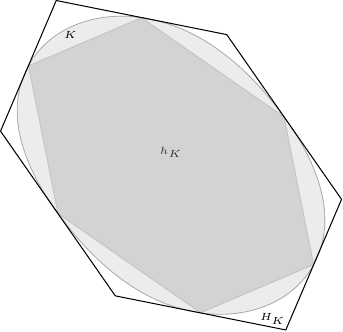}
\caption{Critical hexagons for an ellipse.}
\label{fig:hexagons}
\end{figure}

\noindent
By construction, $h_K,H_K \in \Kccs$. 
The hexagons $H_K$ and $h_K$ for an ellipse are shown in Figure
\ref{fig:hexagons}.  
The critical lattice of any $K \in \Kccs$ gives rise to hexagons
$h_K$ and $H_K$ whose areas are related as
\begin{equation}\label{eqn:outer-inner-area}
\Delta(K) = \frac{1}{3}\op{area}(h_K) = \frac{1}{4}\op{area}(H_K).
\end{equation}

\begin{definition}[Critical Hexagon]\label{def:critical-hexagon}
\index[n]{HK@$H_K$, critical hexagon} For a convex disk $K \in \Kccs$,
a hexagon $H_K$ given by the construction of
Theorem~\ref{thm:mahler-circumscribe} is called a \emph{critical
  hexagon}\index{hexagon!critical}.
\end{definition}

In the section on Reinhardt's approach, we defined 
critical hexagons differently.  The following theorem shows that the
definitions are compatible.

\begin{theorem}\label{thm:min-critical} Let $K\in\Kccs$.  Let $H_K$ be a centrally
symmetric hexagon of least area circumscribing $K$.  Then $H_K$ is a
critical hexagon in the sense of
Definition~\ref{def:critical-hexagon}.
\end{theorem}

\begin{proof} 
Let $\mb{s}_j$ be the edge midpoints of $H_K$.  These points lie on
the boundary of $K$.  Since $H_K$ is centrally symmetric, the points
$\mb{s}_j$ satisfy the multi-point conditions \eqref{eqn:multi}.

Tile the plane with translates of $H_K$.  Let $\LL$ be the lattice
generated by $\mb{s}_0$ and $\mb{s}_1$.  The centers of the tiles form
the sublattice $2\LL$ generated by $2\mb{s}_0$ and $2\mb{s}_1$.  A
lattice point with one or two odd coordinates is the midpoint of an
edge of a translate of $H_K$ centered at an adjacent even lattice
point. Thus, none of the nonzero lattice points of $\LL$ lie in the
interior of $K$.  That is, $\LL$ is $K$-admissible.  Note that $H_K$
is a fundamental domain for the lattice $2\LL$.  Thus,
$4\Delta(K)=\det(2\LL)=\mathrm{area}(H_K)$.

We claim that $\LL$ is critical.  Let $\tilde\LL$ be a critical
lattice with critical hexagon $\tilde H_K$.  Since $H_K$ and the
critical hexagon $\tilde H_K$ both have smallest area among centrally
symmetric circumscribing hexagons, their areas must be equal.  They
are both fundamental domains for the corresponding even sublattices
$2\LL$ and $2\tilde\LL$. Hence the determinants of the lattices are
equal.  This implies that $\LL$ is a critical lattice.

It is now easy to see that the properties of
Theorem~\ref{thm:mahler-circumscribe} all hold for $H_K$, if we take
$\ell_j$ to be the line through the $j$th edge of $H_K$.
\end{proof}

\begin{corollary} Let $K=K_{\min}$ have worst greatest packing density.
  Then $K$ is irreducible.
\end{corollary}

\begin{proof} By Reinhardt, every boundary point of $K$ is
  a midpoint of an edge of an area-minimizing centrally symmetric
  circumscribing hexagon $H_K$.  By Theorem \ref{thm:min-critical} and
  its proof, the boundary point is a lattice point of a critical
  lattice $\LL$.
\end{proof}

\section{Boundary Parameterization of Minimizer}\label{sec:kmin}\label{sec:boundary-parameterization}

We return to Reinhardt's setup from
Section~\ref{sec:reinhardt-approach}.  Let $K=K_{\min}$ be a centrally
symmetric convex disk that gives the worst greatest density.

We have seen that each point $\mb{s}_0$ on the boundary of $K$ is associated
with other points $\mb{s}_1,\ldots,\mb{s}_5$ that satisfy the multi-point
conditions \eqref{eqn:multi}, with area normalization \eqref{eqn:sqrt12}.

\index{multi-curve}
\index[n]{zs@$\sigma_i$, multi-curve}

\index[n]{t@$t\in\R$, time}

Since the boundary of $K$ does not contain any corners, we can
parameterize the boundary by a regular $C^1$ curve
$t \mapsto \sigma_0(t)$, traversing the boundary in
the \emph{counter-clockwise direction}. We will call this the \emph{positive
orientation}. Then by the above discussion,
the point $\sigma_0(t)$ gives rise to other points
$\sigma_1(t),\ldots,\sigma_5(t)$ which
are subject to the multi-point conditions at each time $t$:
\begin{equation}\label{eqn:multi-point-properties}
\sigma_j(t) + \sigma_{j+2}(t) + \sigma_{j+4}(t)= 0,\quad
\sigma_{j+3}(t) = -\sigma_j(t),\quad
\det(\sigma_j(t),\sigma_{j+1}(t)) = \frac{\sqrt{3}}{2}.
\end{equation}
\index{positive orientation}
\index{multi-point}
\index[n]{C@$C^k$, differentiability class}

\begin{definition}[Multi-point and multi-curve]\label{def:multi-pt-multi-curve}
  A function $\multi : \Z/6\Z \to \R^2$ such that
  is called a \emph{multi-point} if it satisfies the multi-point
  conditions~\eqref{eqn:multi} with normalization \eqref{eqn:sqrt12}.
  An indexed set of
  $C^k$ curves $\sigma:\Z/6\Z\times[0,t_f]\to\R^2$ is
  a $C^k$ \emph{multi-curve}\index{multi-curve} if for all $t\in[0,t_f]$,
  $j\mapsto\sigma_j(t)$ is a multi-point.  
\end{definition}
If the differentiability
  class $C^k$ is not specified, then $C^1$ is assumed.
The regularity of the curves will be established in the next section.

\begin{example}
\normalfont\leavevmode
\begin{itemize}
\item The collection of sixth roots of unity $\mb{s}_j^* = \exp\left(\frac{2\pi
    ij}{6}\right) \in \C$, viewed as points in $\R^2$, is an
    example of a multi-point.
\item 
If $\theta:[0,t_f]\to\R$ is a $C^1$ curve, then the rotation
\[
\sigma_j(t) = 
\begin{pmatrix}
\cos(\theta(t))&-\sin(\theta(t))\\
\sin(\theta(t))&\cos(\theta(t))
\end{pmatrix} \multi_j^*
\] 
of a multi-point
is an example of a multi-curve.
\item Section \ref{sec:hypotrochoids} in the Appendix gives an example
  of a hypotrochoid curve\index{hypotrochoid} in $\R^2$ which
  satisfies the multi-curve properties in Equation
  \eqref{eqn:multi-point-properties}.
\end{itemize}

\end{example}

\subsection{Regularity Properties of Multi-Curves}


\begin{lemma}[Hales~\cite{hales2011}]\label{lem:sigma2:c1}
  Let $K=K_{\min}\in\Kccs$ be a centrally symmetric convex disk that has
  the worst greatest packing density.  Consider a $C^0$ multi-curve
  $\sigma$ parameterization of its boundary.  If $t\mapsto\sigma_0(t)$
  is a positively oriented $C^1$ regular curve parameterizing the
  boundary of $K$, then so is $\sigma_j(t)$ for $j=1,\dots,5$.
\end{lemma}

Positively-oriented regular $C^1$ parameterizations $\sigma_0$ exist
for the boundary of $K$. For example, the arclength parameterization
has this property.  By Lemma~\ref{lem:multi-cont}, the curves
$\sigma_j(t)$ are continuous.

\begin{proof}

\index[n]{s@$s$, real parameter!speed}
\index[n]{0@$-'$, derivative} 
\index[n]{u@$\mb{u}$, vector!unit tangent}

Given that $\sigma_0'(t)$ is
continuous, we show that $\sigma_2'(t)$ exists and is
continuous~\cite[Lemma~11]{hales2011}.  Since $K$ has no corners, the
unit tangent $\mb{u}(t)$ to $\sigma_2(t)$, with the orientation given
by $\sigma_0$, is a continuous function of $t$.  It is enough to check
that the speed $s_2$ of $\sigma_2$ is continuous in $t$.  We know that
$\det(\sigma_0(t),\sigma_2(t))$ does not depend on $t$.

We claim that $\det(\sigma_0(t),\mb{u}(t))\ne 0$.  Let $h_K(t)$ be the
hexagon given by the convex hull of $\{\sigma_j(t)\}$.  If
$\det(\sigma_0(t),\mb{u}(t))=0$, then the tangent line to $\sigma_2$
at $t$ contains the edge of $h_K(t)$ through $\sigma_2(t)$ and
$\sigma_1(t)$.  This is contrary to Lemma~\ref{lem:tri-interior}.
In fact, $\det(\sigma_0(t),\mb{u}(t))<0$.
Similarly, we claim that $\det(\sigma_0'(t),\sigma_2(t))\ne0$.
Otherwise the tangent line to $\sigma_0$ at $t$ lies along another
edge of $h_K(t)$, which is contrary to Lemma~\ref{lem:tri-interior}.
In fact, $\det(\sigma_0'(t),\sigma_2(t))<0$.

\index[n]{s@$s$, real parameter!speed}

Define a positive continuous function $s_2:\R\to(0,\infty)$ by the equation
\begin{equation}
\label{eqn:sB}
\det(\sigma_0'(t),\sigma_2(t)) + \det(\sigma_0(t),\mb{u}(t))s_2(t) = 0.
\end{equation}
The curve
\[
\tilde\sigma_2(t):= \int_{t_0}^t \mb{u}(t)s_2(t) dt + \sigma_2(t_0),
\]
has the same initial value at $t=t_0$ as $\sigma_2$, 
the same tangent direction for all $t$, and
satisfies the same area relation
\[
\det(\sigma_0(t),\tilde\sigma_2(t))=\det(\sigma_0(t),\sigma_2(t))
=\sqrt{3}/2
\]
by \eqref{eqn:sB}.
We conclude that $\sigma_2=\tilde\sigma_2$ and that
$\sigma_2'(t)=\mb{u}(t)s_2(t)$.  The regularity condition is
$s_2(t)\ne0$.  (Compare the proof of Lemma~\ref{lem:star-conditions}.)

Similar statements for other curves $\sigma_j$ follow by iteration over $j$. 
\end{proof}

\begin{lemma}\label{lem:tangents-lipschitz}
    Let $\sigma(t)$ denote a $C^1$ multi-curve parameterization of the
    boundary of $K=K_{\min}\in\Kccs$, giving worst greatest packing density.
    Assume that the curve $\sigma_0$ is parameterized according to
    arclength. Then, the tangents $\sigma_j'$ are Lipschitz
    continuous\index{Lipschitz!continuity} for all $j$.
\end{lemma}
\begin{proof}
This is Hales~\cite[Lemmas~17,18]{hales2011}.  We recall the proof.
We start by establishing the Lipschitz continuity of $\sigma_0'$.  We
parameterize the curve $\sigma_0$ according to arclength $s$.  Then
$\sigma_0'$ is a unit tangent vector.  The vector $\sigma_0'$ is
continuous, because the convex region $K$ has no corners, and the
support lines are unique.

\index[n]{zk@$\kappa$, curvature!planar}
\index[n]{s@$s$, real parameter!arclength}
\index[n]{zc@$\gamma$, planar curve!$\gamma_s$, hyperbolic arc}

For each value $s$ of arclength, let $\gamma_s$ be the hyperbola
through $\sigma_0(s)$ tangent to the curve $\sigma_0$ at $s$, whose
asymptotes are the lines in direction $\sigma'_j(s)$ through
$\sigma_j(s)$, for $j=\pm 1$.  By Remark~\ref{rem:arc-hyperbola},
locally near $\sigma_0(s)$, the arc $\gamma_s$ lies inside $K$.  As
$s$ varies, by continuity over the compact boundary, the curvatures of
the hyperbolas $\gamma_s$ at $\sigma_0(s)$ are bounded above by some
$\kappa\in \R$.  (The curvature of the hyperbola depends analytically
on the parameters defining the asymptotes and tangent lines.  These
parameters vary continuously along the boundary of $K$. Thus, the
curvature of the hyperbola varies continuously along the boundary of
$K$, even when the second derivative of $\sigma_0$ and the curvature
of $\sigma_0$ do not exist.)  This means that an osculating circle of
fixed curvature $\kappa$ can be placed locally in $K$ at each point
$\sigma_0(s)$ so that $\sigma_0'(s)$ is tangent to the circle.  By
convexity, the curve $\sigma_0$ near $s$ is constrained to pass
between the tangent line at $\sigma_0(s)$ and the osculating circle of
curvature $\kappa$.  If we parameterize the curve by arclength, then
$\sigma_0'(s)$ has unit length.  Lipschitz continuity now follows from
this bound $\kappa$ on the curvature.

Now consider the Lipschitz continuity for $j\ne0$.  By evident
symmetries, it is enough to consider $j=2$.  Let $t$ be the arclength
parameter for the curve $\sigma_0$ and let $s$ be the arclength
parameter for the curve $\sigma_2$. Write $s=s(t)$ and $t=t(s)$
for the reparameterizations.  
By Lemma~\ref{lem:sigma2:c1}, the functions $s(t),t(s)$ are $C^1$.
Set $\tilde\sigma_2(s)=\sigma_2(t(s))$.
The derivative of $\det(\sigma_0(t),\sigma_2(t))=\sqrt3/2$ gives
\[
\det(\sigma_0'(t),\sigma_2(t))+
\det(\sigma_0(t),\frac{d\tilde\sigma_2(s(t))}{ds})\frac{ds}{dt}=0.
\]
The Lipschitz continuity of $ds/dt$ (and of $dt/ds$) follows from the Lipschitz
continuity of the other functions $\sigma_0'$, $\sigma_2$, $\sigma_0$,
and $d\tilde\sigma_2/ds$ in that equation.
Then we also have the Lipschitz continuity of
\[
\sigma_2'(t) = \frac{d\tilde\sigma_2}{d s} \frac{ds}{dt}.
\]
\end{proof}

\begin{corollary}\label{lem:multi-curve-c2}
The functions $\sigma_j'(t)$ are differentiable almost-everywhere.
\end{corollary}
\begin{proof}
This follows by Rademacher's theorem\index{Rademacher's theorem} and
Lemma \ref{lem:tangents-lipschitz}.
\end{proof}
Until further notice, we assume that the curve $t \mapsto \sigma_0(t)$
is parameterized according to arclength.  See
Proposition~\ref{prop:reparam-lipschitz}.

\subsection{Convexity of Multi-Curves}

This subsection investigates the convexity conditions on the curves $\sigma$.

\begin{lemma}[Star conditions]\label{lem:star-conditions}\index{star!condition}
    Let $K\in\Kccs$ be a convex disk with boundary parameterized by
    the $C^1$ regular multi-curve $\sigma$. At a given time $t$, let
    $j\mapsto\multi_j=\sigma_j(t)$ be a multi-point on the boundary of the
    convex disk $K$. Then for each $j$ and time $t$, the tangent
    $\sigma_j'(t)$ points into the open cone with apex $\multi_j$ and
    bounding rays through $\multi_{j+1}$ and $\multi_j +
    \multi_{j+1}$.
\end{lemma}
\begin{proof}
This situation is depicted in Figure~\ref{fig:ellipse-global-convex}.
This is asserted in Hales~\cite{hales2011,hales2017reinhardt} and is
called the \emph{star condition}.  By convexity of the disk $K$, at
any time $t$ the hexagon $h_K(t)$ is a subset of $K$ (as $h_K$ is the
convex hull of the points $\{\multi_j\}$). Now, the vector
$\sigma_j'(t)$ cannot point into the hexagon, because continuity would
then create a non-convex piece of the curve $\sigma_j$. Dually, it
cannot point beyond the ray from $\multi_j$ through
$\multi_j+\multi_{j+1}$, as that would force $\sigma_{j+2}'(t)$ to
point into $h_K$.  Thus, the tangent vector points into the closed
cone.

If the vector $\sigma_j'(t)$ points along the edge $\multi_j\multi_{j+1}$ of the
triangle, then it would have to remain pointing in that same direction
until reaching $\multi_{j+1}$, as it cannot point inward (by the above
argument) or outward (as then it would not be convex).  This implies
that $\sigma_j'(t)$, $\sigma_{j+1}'(t)$, $\sigma_{j+3}'(t)$, and
$\sigma_{j+4}'(t)$ are all parallel.  The relation
$\sigma_0'(t)+\sigma_2'(t)+\sigma_4'(t)=0$ implies that
$\sigma_{j+2}'(t)$ and $\sigma_{j+5}'(t)$ are parallel as well.
However, the star domain of $\sigma_{j+2}'(t)$ contains no vectors in
that direction, forcing $\sigma_{j+2}'(t)=0$.  This contradicts the
regularity of the curve $\sigma_{j+2}$.

Finally, if $\sigma_j'(t)$ points along the edge $\multi_j(\multi_j+\multi_{j+1})$, then
$\sigma_{j-1}'(t)$ points along $\multi_{j-1}\multi_j$, and the argument can be
repeated with $j-1$ in place of $j$.
\end{proof}

\begin{figure}
\centering
\includegraphics{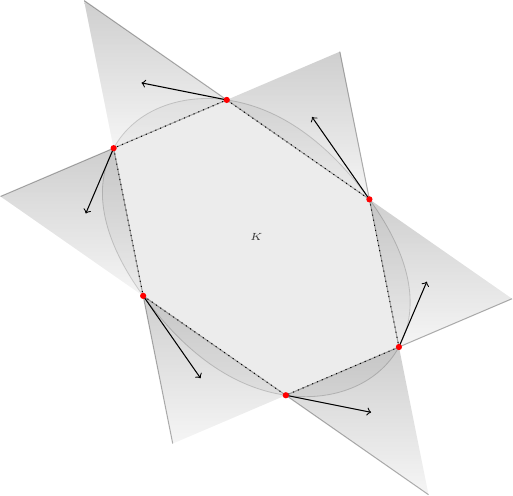}
\caption{Global convexity condition for the ellipse.}
\label{fig:ellipse-global-convex}
\end{figure}

Apart from the star conditions, there is another condition on the
curvature of the boundary curve which needs to be imposed:

\begin{lemma}[Curvature constraint]\label{lem:local-curvature-constrt}
Consider a convex disk $K \in\Kccs$ with boundary parameterized by a 
 $C^1$ regular multi-curve $\sigma$ with Lipschitz continuous derivative.
 Then we have the following condition almost everywhere:
\begin{equation}\label{eqn:local-convexity}
\det(\sigma_j'(t),\sigma_{j}''(t)) \ge 0 \quad j \in \Z/6\Z.
\end{equation}
\end{lemma}
\begin{proof}
The derivatives exist almost everywhere 
by the Lipschitz assumption.
A well-known theorem (see
Proposition 3.8 of Shifrin~\cite{shifrin2015differential}) states that
a simple closed regular plane curve is convex if and only if its
orientation can be chosen in such a way so that its signed planar curvature
is everywhere nonnegative. The left-hand side of
\eqref{eqn:local-convexity} is the planar curvature 
\[
\frac{\det(\sigma_j'(t),\sigma_{j}''(t))}{|\sigma'(t)|^3}.
\]  
up to a positive
factor depending on the parameterization.
The assertion follows.
\end{proof}

\subsection{A Characterization of Balanced Disks}

\index[n]{KK@$\mathfrak{K}$, set of convex disks!$\Kbal$, balanced}
\index{balanced disk}

Summarizing, we define a class $\Kbal\subset\Kccs$ (the class of
balanced disks) of centrally symmetric disks $K$ as those satisfying
the properties that we have established.

\begin{definition}\label{def:balanced}
A centrally symmetric convex disk $K\in\Kccs$ is \emph{balanced} if
the following conditions hold.
\begin{enumerate}
\item The boundary of $K$ is parameterized counterclockwise 
    by six regular $C^1$
    curves $\sigma_j:\R\to\R^2$.  
\item The derivatives $\sigma_j'$ are Lipschitz.
\item For each $t$,
    $j\mapsto\sigma_j(t)$ is a multi-point with normalization convention 
   \eqref{eqn:sqrt12}.  
\item For almost all $t$, we
    have the curvature constraint
    $\det(\sigma_j'(t),\sigma_{j}''(t)) \ge 0$.  
\item For each $t$,
    the vector $\sigma_j'(t)$ points into the open cone with apex $\sigma_j(t)$
    and rays passing through  $\sigma_{j+1}(t)$ and
    $\sigma_j(t) + \sigma_{j+1}(t)$.  
\item 
The image of each curve $\sigma_j:\R/(6t_f\Z)\to\R^2$ is the same
simple closed curve in $\R^2$, where $6t_f$ is the common period of
the functions $\sigma_j$.
%
%
\end{enumerate}
Let $\Kbal$ denote the set of all balanced disks.  A \emph{balanced
  pair} $(K,\sigma)$ consists of a balanced disk $K$ and a boundary
parameterization by a multi-curve $\sigma$ satisfying the foregoing
enumerated properties.
\end{definition}

\index{balanced pair}

By convention, the area of the hexagon $H_K$ has area normalized to
$\sqrt{12}$.  By the results of this section, every minimizer
$K_{\min}$ following this convention belongs to $\Kbal$.

By affine invariance, there is no loss of generality in assuming that
the sixth roots of unity $\{\mb{s}_j^*\mid j\}$ lie on the boundary of
the convex disk $K$. We make this multi-point the \emph{initial}
position of the multi-curve $\sigma$.  
Following
Hales~\cite{hales2011}, we call this representation the \emph{circle
  representation of the convex disk}\index{circle representation} $K$.

\begin{problem}[Balanced Reinhardt Problem in Circle Representation]
\label{pbm:reinhardt-reduction-2}
Describe those $K_{\min} \in \Kbal$ in circle representation for which
\[
\op{area}(K_{\min}) = \inf_{\{K \in \Kbal\mid K\text{ in circle representation}\}} \op{area}(K),
\]
where $\Kbal$ is the class of balanced disks.
\end{problem}

\newpage
\part{Optimal Control}

\chapter{A Control Problem}

\index{cost functional}
\index{state space}
\index{control!parameter}
\index[n]{SL@$\mathrm{SL}_n$, special linear group}
\index{Lie group}

Now that we have a description of the set $\Kbal$ of balanced disks,
we can use this to restate the Reinhardt conjecture as an optimal
control problem. Optimal control problems solve for a policy that drives
an agent in an environment over a period of time such that a cost
function is optimized. We pin down our reformulation in three steps:

\begin{enumerate}
\item Recast the system as a dynamical system on a \emph{state space} (which, in our case, will be a manifold).
\item Introduce a well-defined \emph{cost functional}.
\item Determine a well-defined \emph{control parameter}.
\end{enumerate}This will be our focus in this section.

\section{State Dynamics in the Lie Group}\label{sec:lie-group-dynamics}
Let $\SL(\R)$ be the Lie group consisting of all $2\times 2$ matrices
with real entries and determinant $1$.  The multi-curve conditions
give rise to a curve in $\SL(\R)$ in the following sense.

\index[n]{k@$k$, integer}
\index[n]{g@$g$, group element!$g(t)$, curve in $\SL$}

\begin{theorem}[Mahler~\cite{mahler1947area}, Hales~\cite{hales2011}]\label{thm:defn-g}
  Let $\sigma$ be a $C^k$ multi-curve, and let $\mb{s}$ be any
  multi-point.  Then $\sigma$ determines a $C^k$ curve
  $g:[0,t_f]\to\SL(\R)$, by the conditions 
\begin{equation}\label{eqn:g-sigma}
\sigma_j(t) =  g(t)\mb{s}_j
\end{equation}
for all $j$ and all $t\in[0,t_f]$.
\end{theorem}
\begin{proof}
Given a multi-curve $\sigma$ and $t\in[0,t_f]$,
the value $\sigma(t)$ is a multi-point. 
It is enough to construct a 
$2 \times 2$ real matrix $g(t)$ so that
\[
\sigma_j(t) =  g(t) \mb{s}_j, \quad j=0,2,
\]
because the multi-point conditions then imply by linearity that
$\sigma_j(t) = g(t)\mb{s}_j$ for all $j \in \Z/6\Z$.  A unique such
matrix $g(t)$ can always be found by linear independence:
\[
\det(\mb{s}_0,\mb{s}_1)=\det(\mb{s}_0,\mb{s}_2) \ne 0.
\]
The identity
\[
\det(\sigma_0(t),\sigma_2(t))=
\det(g(t)\mb{s}_0, g(t)\mb{s}_2) =\det(g(t))\det(\mb{s}_0,\mb{s}_2)
\]
gives $\det(g(t)) = 1$, and $g(t)\in\SL(\R)$. Thus, if we have a $C^k$
curve of multi-points $\sigma(t)$, we obtain a unique induced $C^k$
curve $g:[0,t_f]\to\SL(\R)$.
\end{proof}

\begin{remark}
\normalfont \leavevmode
\begin{itemize}
\item 
If $t_0,t_1,t_2$ are three time instants, and
$\sigma_k(t_i)=g(t_i,t_j)\sigma_k(t_j)$, then
$g(t_2,t_0)=g(t_2,t_1)g(t_1,t_0)$.
\item 
Later in the book, the multi-point is $\mb{s}_j^*$, and the
multi-curves $\sigma_j$ are given by $\sigma_j(t) = g(t)\mb{s}_j^*$,
where $g(t) \in \SL(\R)$.
\end{itemize}
\end{remark}

\index[n]{SL@$\mathrm{SL}_n$, special linear group!$\mathfrak{sl}_n$, Lie algebra}

Let $\sl(\R)$ be the Lie algebra\index{Lie algebra} of $2\times 2$
matrices with real entries and trace $0$.  Let
$\Ad_{g}(X)=gXg^{-1}$ be the \emph{adjoint representation} of
$\SL(\R)$ on its Lie algebra.

\index[n]{Ad@$\Ad$, adjoint representation of Lie group}

Similarly, the tangents $\{\sigma_j'(t)\}$
give rise to a corresponding curve in the Lie algebra $\sl(\R)$ as
follows.
\begin{definition}\label{def:X-defn}
For a $C^1$ curve $g:[0,t_f]\to \SL(\R)$ as above, define
$X:[0,t_f]\to\mathfrak{gl}_2(\R)$ by $g'(t)=g(t)X(t)$.
\end{definition}

\index[n]{X@$X$, Lie algebra element!$X=g^{-1}g'$, curve in Lie algebra}
\index[n]{X@$X$, Lie algebra element!$X\in\sl$}

\begin{theorem}\label{thm:defn-X}
Assume that we have a convex disk $K$ 
boundary parameterized by a regular $C^1$ multi-curve $\sigma$. Let
$g$ be the induced curve in $\SL(\R)$ for some multi-point $\mb{s}$
given by Equation \eqref{eqn:g-sigma}, and define $X$ by $g'=gX$.  Then
\begin{enumerate}
\item $X(t) \in \sl(\R)$.
\item $\sigma_j'(t) = \Ad_{g(t)}(X(t))\sigma_j(t)=g(t)X(t)g(t)^{-1}\sigma_j(t)$ for all $j \in \Z/6\Z$.
\item $X:[0,t_f]\to\sl(\R)$ is Lipschitz continuous. 
\end{enumerate}
\end{theorem}
\begin{proof}
First of all, the matrix $X(t)$ belongs to $\sl(\R)$ because if
$g:[0,t_f]\to\SL(\R)$ is any differentiable curve, then $X =
g^{-1}g'\in\sl(\R)$.  Let $\sigma_j(t)=g(t)\mb{s}_j$, for some
multi-point $\mb{s}_j$.  We then have
\[
\sigma_j'(t) = g'(t)\mb{s}_{j} = 
g(t)X(t)\mb{s}_{j}=g(t)X(t)g(t)^{-1}g(t)\mb{s}_{j}=
\Ad_{g(t)}(X(t))\sigma_j(t).
\]
To see that $X(t)$ is Lipschitz, it is enough to show that $X(t)
\mb{s}_{j} = g(t)^{-1} \sigma_j'(t)$ is Lipschitz for each $j$, which
follows because $\sigma_j'$ is Lipschitz by
Lemma \ref{lem:tangents-lipschitz}, and $g$ is a $C^1$ curve on a
compact interval.
\end{proof}

\index[n]{J@$J$, infinitesimal generator of rotations}
\index[n]{SO@$\mathrm{SO}$, special orthogonal group!$\mathrm{SO}_n$, compact orthogonal}

\begin{definition}
Let 
\[
J := \mattwo{0}{-1}{1}{0}\in\sl(\R).
\]
$J$ is the infinitesimal generator of the rotation group $\SO$.  This
infinitesimal generator gives rotations $\exp(J t) \in \SO$.  In
particular, 
\begin{equation}\label{eqn:R}
R := \exp(J\pi/3)\index[n]{R@$R$, rotation by $\pi/3$}
\end{equation}
is counterclockwise rotation by angle
$\pi/3$.  The matrices $J$ and $R$ are global notations throughout.
\end{definition}

\index[n]{zr@$\rho_j$, star function}

\begin{corollary}\label{cor:X-props}
For a balanced disk $K\in\Kbal$, with $C^1$ boundary parameterization
$\sigma_j(t)=g(t)\mb{s}_j^*$,
we have the following
properties for $X=X(t)$ defined in Definition \ref{def:X-defn} for all
times $t$.
\begin{itemize}
\index[n]{abcd@$a,b,c,d$, matrix entries!of $X$}
\item We have $\rho_i(X)>0$, for $i=0,1,2$, where
\begin{align}\label{eqn:rho}
\rho_0(X) &:= \frac{c + \sqrt{3}a}{\sqrt3},\\
\rho_1(X) &:= \frac{c-\sqrt{3} a}{\sqrt3},\\
\rho_2(X) &:= \frac{-(3b + c)}{2\sqrt{3}}
\end{align}
and
\[
X = \mattwo{a}{b}{c}{-a}.
\] 
In particular, $\sqrt{3}|a| < c$, $3b+c<0$, 
$0<c$, and 
$\tr(JX)=b-c<0$.\index{trace} 
\item
    $\det(X(t))>0$ for all $t$.
\end{itemize}
\end{corollary}
We call the inequalities $\rho_i(X)>0$ the \emph{star
  inequalities}\index{star!inequality} for $X\in\sl(\R)$.
\begin{proof}
At time $t$, the multi-point is given by  $g(t)\mb{s}_j^*$.
The star conditions in Lemma \ref{lem:star-conditions} imply that the
tangent vector $\sigma_j'(t) = g(t) X(t)\mb{s}_{j}^*$ lies in
the open cone with apex at the origin
bounded by the rays 
through the points $g(t)\mb{s}_{j+1}^*$ and $g(t)\mb{s}_{j+2}^*$.
After cancelling a factor of $g(t)$ from both sides, and writing $X$
for $X(t)$, these open cone conditions become
\index[n]{zr@$\rho_j$, star function!$\tilde\rho_j$}

\mcite{MCA:rho-star-inequalities}
\begin{equation}\label{eqn:rho-star-inequalities}
     X \mb{s}_j^* = \rho_j(X) \mb{s}_{j+1}^* + {\tilde\rho}_{j+1}(X) \mb{s}_{j+2}^*,
\end{equation}
for some unknowns $\rho_j(X),{\tilde\rho}_j(X) > 0$, for $j\in\Z/6\Z$
and $X\in\sl(\R)$.  By central symmetry, $\rho_{j+3}=\rho_j$ and
$\tilde\rho_{j+3}=\tilde\rho_j$.  Solving this systems of linear
equations for $\rho_j,{\tilde\rho}_j$ in terms of the matrix entries $a,b,c$
of $X$, we obtain for all $j$ that $\rho_j$ is given by the statement
of the lemma.  Also, ${\tilde\rho}_{j}=\rho_j$.  It follows that
\[
\sqrt{3}|a| < c \quad
3b+c<0.
\] 
In particular, $c>0$. Using these values, we have by direct
calculation that
\begin{equation}\label{eqn:det-rho}
\det(X) = \rho_0 \rho_1 + \rho_1 \rho_2 + \rho_2 \rho_0 > 0.
\end{equation}
\end{proof}

\begin{remark}[Reconstructing the hexagon]\label{rem:hexagon}  
Given $X\in\sl(\R)$ satisfying the star inequalities in the conclusion of
the corollary, we can reconstruct a centrally symmetric hexagon whose
edge midpoints are the points $\mb{s}_j^*$ and such that the hexagon
satisfies the star conditions.  The edge direction at $\mb{s}_j^*$ is
$X\mb{s}_j^*$. Two elements in the Lie algebra give the same hexagon
if one element is a positive multiple of the other.

The vertices of the hexagon are constructed as the solutions to linear
equations; each vertex is the point of intersection between the line
through $\mb{s}_j^*$ in direction $X \mb{s}_j^*$ and the line through
$\mb{s}_{j+1}^*$ in the direction $X \mb{s}_{j+1}^*$. Explicitly, by solving the
equations, the vertex is
\begin{equation}\label{eqn:hex-vertex}
\mb{s}_j^* + \frac{{\rho}_{j+2}(X)}{\det(X)} X \mb{s}_j^* 
= \mb{s}_{j+1}^* - \frac{\rho_{j}(X)}{\det(X)} X \mb{s}_{j+1}^*.
\end{equation}
Each $\mb{s}_j^*$ is manifestly an edge midpoint, given as the midpoint of
vertices $\mb{s}_j^* \pm {\rho}_{j+2}(X) X \mb{s}_j^*/\det(X)$.
\end{remark}

The following fact is an easy corollary of Reinhardt's observations
about the significance of hyperbolic arcs.


\begin{lemma}\label{lem:hyp-arc}
  Suppose that $K\in\Kccs$.
  Assume the boundary is parameterized by a $C^1$ multi-curve $\sigma$.  
If two curves $\sigma_{i-1}$ and
$\sigma_{i+1}$ move along straight lines during some time interval,
then the third curve $\sigma_{i}$ moves along a hyperbolic arc, whose
asymptotes are the lines determined by $\sigma_{i-1}$ and $\sigma_{i+1}$.
\end{lemma}

\index[n]{P@$P_K$, parallelogram}

\begin{proof}
For simplicity and without loss of generality, take $i=0$.  The curves
$\pm\sigma_1$ and $\pm\sigma_{-1}$ trace out lines that form four of
the edges (forming a fixed parallelogram $P_K$) of the time dependent
critical hexagon $H_K(t)$.  The tangent lines to $\pm\sigma_0(t)$ form
the final two edges of the critical hexagon, and $\pm\sigma_0(t)$ are
the midpoints of those edges.  As $t$ varies, the area cut off by
these two tangent lines from the parallelogram $P_K$ is constant,
because the areas of $P_K$ and $H_K(t)$ are both constant.  As
Reinhardt observed, the pencil of lines cutting a constant area from
(two adjacent edges of) a parallelogram has an envelope that is a
hyperbola whose asymptotes are the lines extending the edges of the
parallelogram.  The midpoints $\sigma_0(t)$ must lie on that
envelope, a hyperbola with the required properties.
\end{proof}

If $K\in\Kbal$ has boundary parameterization $\sigma_j(t) =
g(t)\mb{s}_j^*$, then the midpoint hexagon $H_K$ of $K$ at the
multi-point $g(t_0)\mb{s}_j^*$ is the left translate by $g(t_0)$ of
the hexagon constructed in the Remark~\ref{rem:hexagon} using
$X=X(t_0)=g^{-1}g'(t_0)$.


We have one equation for our state space dynamics viz., equation
$g'=gX$. Before deriving dynamics in the Lie algebra, 
we shift to a more convenient choice of parameterization.

\index[n]{0@$-\tilde{\phantom{-}}$, transformed quantity}

\begin{proposition}\label{prop:reparam-lipschitz}
Let $s$ denote the arclength parameter of $\sigma_0$, and let
$\tilde{X}(t)$ denote the reparameterization of the matrix-valued
curve $X$, using the reparameterizion of $g$ to a time variable $t$
such that $\det(\tilde{X}(t))=1$. Then we have that
$t\mapsto\tilde{X}(t)$ is a Lipschitz continuous function.
\end{proposition}
\begin{proof}
Reparameterize $\tilde{g}(t) := g(s(t))$. In this new parameterization,
we define $\tilde X$ by the differential relation
$d\tilde{g}(t)/dt= \tilde{g}(t)\tilde{X}(t)$.  We have by
the chain rule:
\begin{align*}
\frac{d\tilde{g}(t)}{dt} = \frac{d}{dt}g(s(t)) &= 
\frac{d}{ds}g(s(t))\frac{ds}{dt} 
\\ 
    &= g(s)X(s)\frac{ds}{dt}.
\end{align*}
So that $\tilde{X}(t) = X(s(t))\frac{ds}{dt}$. Now we require $\det(\tilde{X}(t))=1$, which forces
\[
\frac{ds}{dt} = \frac{1}{\det(X(s))^{1/2}},
\]
which gives us the reparameterization equation. By
Corollary \ref{cor:X-props} we have $\det(X(s)) > 0$ so that the
right-hand side is real and finite. Recall that we have that $X(s)$ is
Lipschitz by Theorem $\ref{thm:defn-X}$ and $\det(X(s))$ is bounded
away from zero since it a continuous function on a compact
interval. Then $ds/dt$ is Lipschitz.  This proves that $\tilde{X}$ is
Lipschitz as well.
\end{proof}

By abuse of notation, we let $t$ denote the new parameter, so that
$\det(X(t)) = 1$.

\begin{corollary}\label{cor:X-diff-ae}
With respect to the parameterization making $\det(X(t)) = 1$, the
curve $X$ is differentiable almost everywhere.
\end{corollary}
\begin{proof}
This follows from Rademacher's theorem and Proposition \ref{prop:reparam-lipschitz}.
\end{proof}

\index[n]{GL@$\mathrm{GL}_n$, general linear group}
\index{Jacobi's formula}

\begin{corollary}\label{cor:param-equiv}
A parameterization which makes $\det(X(t))$ a constant also makes $X
+ X^{-1}X' \in \sl(\R)$ almost everywhere, and conversely.
\end{corollary}
\begin{proof}
This is immediate from the Jacobi's formula for the
derivative of a determinant (Lemma~\ref{lem:Jacobi}):
\begin{equation}\label{eqn:log-deriv}
\frac{\det(X)'}{\det(X)} = \tr\left(X^{-1}X'\right),
\end{equation}
\index[n]{zl@$\lambda$, eigenvalue}
and from $\tr(X)=0$.  
\end{proof}

\index[n]{L@$L$, matrix}
\index[n]{A@$A$, matrix or linear map}
\index[n]{D@$D$, matrix}
\index[n]{I@$I$, identity matrix}
\index[n]{zl@$\lambda$, eigenvalue}

\begin{lemma}[Jacobi's formula]\label{lem:Jacobi}
Let $A$ be a differentiable function taking values in 
$\mathrm{GL}_n(\C)$.  Then
\[
{\det(A)'} = {\det(A)}\tr\left(A^{-1}A'\right).
\]
(The same formula holds without the assumption that
$A$ is invertible, if $\det(A)A^{-1}$ is replaced with
the adjugate of $A$.)
\end{lemma}

\begin{proof}
Both sides of the Jacobi formula are polynomials in the matrix
coefficients of $A$ and $A'$.  It is therefore sufficient to verify
the polynomial identity on the dense subset where the eigenvalues of
$A$ are distinct and nonzero.

If $A$ factors differentiably as $A=A_1A_2$, then
\[
\mathrm{tr}(A^{-1}A')=
\mathrm{tr}(A_1^{-1}A_1')+
\mathrm{tr}(A_2^{-1}A_2').
\]
In particular, if $L$ is invertible, then $I=L^{-1}L$ and
\[
0=\mathrm{tr}(I^{-1}I')=
\mathrm{tr}(L(L^{-1})')+
\mathrm{tr}(L^{-1}L').
\]
Since $A$ has distinct eigenvalues, there exists a differentiable
complex invertible matrix $L$ such that $A = L^{-1}DL$ and $D$ is
diagonal.  Then
\[
\mathrm{tr}(A^{-1}A')=
\mathrm{tr}(L(L^{-1})')+
\mathrm{tr}(D^{-1}D')+
\mathrm{tr}(L^{-1}L')=
\mathrm{tr}(D^{-1}D').
\]
Let $\lambda_i$, $i=1,\ldots,n$ be the eigenvalues of $A$.  Then
\[
\frac{\det(A)'}{\det(A)} = \sum_{i=1}^n \lambda_i'/\lambda_i = 
\mathrm{tr}(D^{-1}D')=
\mathrm{tr}(A^{-1}A'),
\]
which is the Jacobi formula for matrices $A$ with distinct nonzero
eigenvalues.
\end{proof}

\section{The Cost Functional}\label{sec:the-cost-functional}

We now compute the cost functional in terms of the matrix-valued curve
$X$ parameterized as above. From the balanced Reinhardt
problem \ref{pbm:reinhardt-reduction-2} in circle representation, we
see that the quantity to be minimized is the area of a convex disk in
$\Kbal$. Our strategy is to compute this area using Green's
theorem, by using the pullback of the one-form
$xdy-ydx$ on $\R^2$.  We let $\cdot^{tr}$ denote the transpose of a matrix.

\index{Reinhardt problem!balanced}
\index{Green's theorem}
\index[n]{zc@$\gamma$, planar curve}
\index[n]{zh@$\theta$, differential one-form}
\index[n]{0@$-^{tr}$, transpose}
\index[n]{u@$\mb{u}$, vector!in $\R^2$}
\index[n]{w@$\mb{w}$, vector!in $\R^2$}
\index[n]{0@$-^*$, pullback of differential form}

\begin{lemma}\label{lem:area-helper-1}
Let $g : [0,t_f] \to \SL(\R)$ be a path so that $g' = gX$ as above and
let $\mb{u} \in \R^2$. Define $\gamma : [0,t_f] \to \R^2$
by
$\gamma(t):=g(t)\mb{u}$.  Consider the one-form $\theta = xdy - ydx$ on
$\R^2$. Then we have the following formula for the pullback of
$\theta$ to $[0,t_f]$:
\[
\gamma^*\theta = -\mb{u}^{tr} J X \mb{u}\, dt.
\]
\end{lemma}
\begin{proof}
If we write $\gamma(t) = (\gamma_1(t),\gamma_2(t))$, then 
\begin{align*}
\gamma^*\theta &= \theta(\gamma(t))
\\ 
&= (\gamma_1(t)\gamma_2'(t) -\gamma_2(t)\gamma_1'(t))\,dt
\\ 
&= (\det(\gamma, \gamma'))\,dt 
\\ 
&= (\det(g \mb{u} , g X \mb{u}))\, dt 
\\ 
&= (\det(\mb{u} , X \mb{u}))\, dt 
\\ 
&=-\mb{u}^{tr} JX \mb{u}\, dt,
\end{align*}
\index[n]{u@$\mb{u}$, vector!in $\R^2$}
since $\det(g) = 1$ and 
\begin{equation}\label{eqn:wedge-inner}
\det(\mb{u} , X \mb{v}) = -\mb{u}^{tr}JX \mb{v},
\end{equation} 
for all $X \in \sl(\R)$ and $\mb{u},\mb{v} \in \R^2$.
\end{proof}
The lemma above enables us to compute pullbacks of $\theta$ by the
multi-curves $\sigma_j$. Indeed, 
the boundary $\partial K$ of an arbitrary balanced convex disk
$K \in\Kbal$ 
is a simple closed curve parameterized by the curves
$\sigma$, given by $\sigma_j(t) = g(t)\mb{s}_j^* = g(t)(R^{j}\mb{s}_0^*)$. 


\index[n]{Y@$Y\in\mathfrak{g}$, Lie algebra element!in $\sl$}

\begin{lemma}\label{lem:area-helper-2}
Let $Y \in \sl(\R)$. Then we have 
\[
JY + (R^2)^{tr}JYR^2 + (R^4)^{tr}JYR^4 = \frac{3\tr(JY)}{2}I_2,
\]
where $I_2$ is the $2\times 2$ identity matrix.
\index[n]{I@$I_2$, $2\times2$ identity matrix}
\end{lemma}
\begin{proof}
This is a simple computation. 
\end{proof}

We now derive a formula for the area of $K\in\Kbal$.  
By Green's theorem and the lemmas from above,
we have
\begin{align*}
\op{area}(K)
&= \frac{1}{2}\oint_{\partial K}\theta 
\\ 
    &= \frac{1}{2}\int_{0}^{t_f}\gamma^*\theta dt 
\\ 
    &= \int_{0}^{t_f}\sigma_0^*\theta +\sigma_2^*\theta +\sigma_4^*\theta dt \quad 
\text{(by Lemma \ref{lem:area-helper-1})}
\\ 
    &= \int_{0}^{t_f}-\mb{u}^{tr} JX \mb{u} -(R^2\mb{u})^{tr} JX (R^2\mb{u})
    - (R^4\mb{u})^{tr} JX (R^4 \mb{u}) dt 
\\ 
    &= -\int_{0}^{t_f} \mb{u}^{tr} 
\left(JX + (R^2)^{tr}JXR^2 + (R^4)^{tr} JXR^4\right)\mb{u} dt 
\\ 
    &=-\int_{0}^{t_f} |\mb{u}|^2 \frac{3\tr(JX)}{2} dt \quad 
\text{(by Lemma \ref{lem:area-helper-2}).}
\end{align*}
The parameterization of convex disk $K$ is of the form $\sigma_0(t) =
g(t)\mb{s}_0^*$, which means that $\mb{u}=\mb{s}_0^*=(1,0)$.
So $|\mb{u}|=1$ and
\begin{equation}\label{eqn:cost}
\op{area}(K) = -\frac{3}{2}\int_{0}^{t_f} \tr(JX) dt,
\end{equation}
which is the quantity to be minimized.  

\begin{example} 
We show that the area of a unit disk $K$ is $\pi$, as expected, by
using \eqref{eqn:cost}.  In this case,
\[
g(t) = \exp(J t),\quad X(t) = J,\quad t_f = \pi/3\quad \text{(one-sixth the circumference)},
\]
where $\exp(-)$ is the matrix exponential.
\index[n]{exp@$\exp$, exponential and matrix exponential}
Note that $X$ is a constant curve in this case.
Then
\[
\op{area}(K) = -\frac{3}{2} \int_{0}^{t_f} \tr(JX) dt = -\frac{3}{2} \tr(J^2) 
\int_{0}^{\pi/3} dt 
= -\frac{3}{2} (-2) \frac{\pi}{3} = \pi.
\]
\end{example}

\section{Control Sets}\label{sec:control-sets}
We now investigate the \emph{control} parameters which affect this
cost. It is intuitively obvious that the curvature of the curves
making up the boundary of the convex disk $K \in\Kbal$ affect its
area. So, it makes sense to allow the curvatures to play the role
of the controls. Problems in which curvature plays the role of a
control are well-studied in the literature, the Dubins-Delauney
problem being one a prominent such
example~\cite{jurdjevic2014delauney}. Other examples, such as
Kirchoff's problem and the elastic problem are discussed in
\cite{jurdjevic2016optimal}.

To begin, recall that Lemma \ref{lem:local-curvature-constrt} says
that $\det(\sigma_j'(t), \sigma_j''(t)) \ge 0$ almost everywhere in
$t$.  Now since our convex disk is in circle representation, we have
$\sigma_j(t) = g(t)\mb{s}_{j}^*$. Then we have, for $j=0,1,2$,
\begin{align}
\kappa_j(t) := \det(\sigma_{2j}'(t),\sigma_{2j}''(t)) &= \det(g'(t)\mb{s}_{2j}^*, g''(t) \mb{s}_{2j}^*) \nonumber 
\\ 
&= \det(gX\mb{s}_{2j}^*, (gX^2+gX') \mb{s}_{2j}^*) \nonumber 
\\ 
&= \det(X\mb{s}_{2j}^*, (X^2+X') \mb{s}_{2j}^*) \nonumber 
\\ 
&=\det(\mb{s}_{2j}^*,\left(X+X^{-1}X'\right) \mb{s}_{2j}^*) \ge 0 \label{eqn:state-dep-curvature},
\end{align}
\index[n]{zk@$\kappa$, curvature!$\kappa_j$, state-dependent curvature}
almost everywhere in $t$. Here we used
Proposition~\ref{prop:reparam-lipschitz} and
Corollaries \ref{cor:X-props} and \ref{cor:X-diff-ae}. We call
$\kappa_j(t)$ the \emph{state-dependent curvature} as it depends
on where we are in the state space.
Note the indexing conventions $j\leftrightarrow 2j$ relating
$\kappa_j$ and $\sigma_{2j}$.

\begin{lemma}\label{lem:curvature-pos}
Let $g : [0,t_f] \to \SL(\R)$ be $C^1$ with Lipschitz derivative satisfying
the star and curvature conditions \eqref{eqn:state-dep-curvature}. 
Then almost everywhere,
there exists an index $j$ so that $\kappa_j(t) > 0$.
\end{lemma}


\begin{proof}
Take $X = g^{-1}g'$.  Assuming $\kappa_0,\kappa_1=0$,
a short calculation shows almost everywhere that
\[
\mcite{MCA:kappa2-pos}
\kappa_2(t) = \det(\mb{s}_{4}^*, \left(X+X^{-1}X'\right) \mb{s}_{4}^*)
= \frac{\sqrt{3}\det(X)}{\rho_0(X)}.
\]
This is strictly positive by the star inequalities in
Corollary \ref{cor:X-props}. (This gives a second interpretation of
the functions $\rho_i$ in terms of state-dependent curvatures.)

A second proof can be obtained from Lemma~\ref{lem:hyp-arc}: if two of
the state-dependent curvatures are zero, then the third curvature is
such that associated curve is an arc of a hyperbola. Hence the third
curvature is positive.

A third proof appears at the end of the proof of
Theorem \ref{thm:X-dynamics-sl2}, which gives a formula for
$\kappa_0+\kappa_1+\kappa_2$ as a ratio of negative numbers.
\end{proof}

Since the state-dependent curvatures $\kappa_j(t)$ depend on $X$ and $X'$
in general, they are not suitable as control variables for our control
problem. To this end, we introduce normalizations of the state-dependent
curvatures as follows.

\index[n]{u@$u$, control!$u_j$, $j$th component of control}

\begin{definition}[control variables]\label{def:control-variables}\index{control!variable}
For each $j=0,1,2$, define control variables given by
the normalized state-dependent curvatures as
\[
u_j := \frac {\kappa_j}{\kappa_0+\kappa_1+\kappa_2}.
\]
\end{definition}

Note that the denominator is positive 
by Lemma \ref{lem:curvature-pos}. The control
variables $u_i$ evidently satisfy $0\le u_i \le 1$ and $u_0 + u_1 +
u_2 = 1$. Note also that the control variables are functions of time.

\begin{definition}[Triangular control set]\label{def:triang-control-set}
The \emph{triangular} or \emph{simplex} control set is the set 
\[
U_T
:= \left\{(u_0,u_1,u_2)\mid 0\le u_i \le 1, ~u_0+u_1+u_2 = 1\right\},
\index[n]{U@$U$, control set!$U_T$, triangular}
\]
which is just the two-simplex in $\R^3$.
\end{definition}


We map the control set $U_T$ into the Lie algebra
$\sl(\R)$ using the following transformation:
\begin{equation}\label{eqn:Z0}
Z_u=
\left(
\begin{matrix}
\frac{u_1 - u_2}{\sqrt{3}} & \frac{u_0 - 2u_1 - 2u_2}{3} 
\\ 
    u_0 & \frac{u_2 - u_1}{\sqrt{3}}
\end{matrix}
\right)    
\in \sl(\R),\quad u=(u_0,u_1,u_2)\in U_T.
\index[n]{YZ@$Z_u$, control matrix}
\end{equation}

This \emph{control matrix} $Z_u\in\sl(\R)$ is uniquely
determined by the equations
\begin{equation}\label{eqn:def-Z0}
u_j = \det(\mb{s}_{2j}^*, Z_u \mb{s}_{2j}^*)  \qquad j=0,1,2.
\end{equation}

In summary, the optimal control function of the control problem
takes values in the two-simplex $U_T$.  The values in $U_T$ specify
the values of curvature functions, which determine the boundary curves
of a convex disk $K \in\Kbal$.  The optimal control function minimizes
the area of $K$.

\index[n]{zr@$\rho_j$, star function}
\index[n]{0@$\bracks{-}{-}$, bilinear form!trace form on Lie algebra}

Henceforth, we adopt the notation $\bracks{X}{Y}:=\tr(XY)$, for any two
matrices $X,Y \in \sl(\R)$. This form is a nondegenerate invariant
bilinear form on $\sl(\R)$.

We can now prove an equivalent star condition.
\begin{lemma}\label{lem:sl2-star-condition}
The star inequalities on $X$ hold if and only if $\bracks{Z_u}{X}<0$
for all controls $u \in U_T$.
\end{lemma}
\begin{proof}
An easy calculation gives the following identity.
\[
\bracks{Z_u}{X} = -\frac{2}{\sqrt3} (\rho_2(X) u_0 + \rho_1(X) u_1 + \rho_0(X) u_2).
\]
The right-hand side is everywhere negative on $U_T$ if and only if
$\rho_j(X)>0$, for $j=0,1,2$. These are the star inequalities on $X$.
\end{proof}

\section{Lie Algebra Dynamics}\label{sec:lie-algebra-dynamics}

\index[n]{X@$X$, Lie algebra element!$X,Y,Z,W\in\sl$}
\index[n]{0@$[-,-]$, Lie bracket}

Let us first collect a number of results about
matrices in $\sl(\R)$ which we will need. All of these are elementary
and so we admit them without proofs.  Let $[X,Y] = XY - YX$ be the Lie
algebra commutator of two matrices $X,Y \in \sl(\R)$.

\begin{lemma}\label{lem:sl2-lemmas}\mcite{MCA:sl2-lemmas}
We have the following results about matrices in $\sl(\R)$.
\begin{enumerate}
\item If $X \in \sl(\R)$, $\bracks{X}{X} = -2\det(X)$.
\item If $X,Y$ are any two matrices in $\sl(\R)$, and $\multi$ is a multi-point
\[
\det(\multi_{j}, X \multi_{j}) = \det(\multi_{j}, Y \multi_{j}),\quad j=0,1,2,
\]
then $X=Y$.
\item If $X,Y \in \sl(\R)$ then $XY + YX = \bracks{X}{Y} I_2$.
\item For any matrices $X,Y,Z \in \sl(\R)$ we have $\bracks{X}{[Y,Z]} = \bracks{[X,Y]}{Z}$.
\end{enumerate}
\end{lemma}

We return to $X$ as the trajectory defined by $g'=gX$.
Let us now derive the control-dependent dynamics for $X$.

\begin{theorem}[Dynamics for $X$]\label{thm:X-dynamics-sl2}
The dynamics for $X$ (which is control-dependent) is given by
\[
X' = \frac{\left[Z_u,X\right]}{\bracks{Z_u}{X}}.
\]
\end{theorem}
It is shown in Lemma~\ref{lem:sl2-star-condition} that the
star conditions imply $\bracks{Z_u}{X}<0$, for all controls $u\in
U_T$.  The denominator
$\bracks{Z_u}{X}$ appearing in the theorem is therefore nonzero.
\begin{proof}
By Corollary \ref{cor:param-equiv}, we find $X + X^{-1}X' \in \sl(\R)$. 
From Equations \eqref{eqn:state-dep-curvature} and \eqref{eqn:def-Z0} we find 
\index[n]{zk@$\kappa$, curvature!$\kappa=\kappa_1+\kappa_2+\kappa_3$}
\begin{align*}
u_j &= \det( \mb{s}_{2j}^*, Z_u \mb{s}_{2j}^*) 
\\ 
\kappa_j &=  \det(\mb{s}_{2j}^* , \left(X+X^{-1}X'\right)\mb{s}_{2j}^*),
\end{align*}
for each $j$. Let $\kappa=\kappa_1+\kappa_2+\kappa_3$.
Since, by Definition \ref{def:control-variables}, we
have $\kappa u_j = \kappa_j$, by Lemma \ref{lem:sl2-lemmas},(2) we obtain
that
\[
X+X^{-1}X' = \kappa Z_u,
\]
from which we obtain 
\begin{equation}\label{eqn:X-orig}
X' = X(\kappa Z_u - X).
\end{equation}
Taking traces and using $\tr(X') = 0$, we obtain $\kappa =
\bracks{X}{X}/\bracks{X}{Z_u} = -2/\bracks{X}{Z_u}$, where the last
equality uses Lemma~\ref{lem:sl2-lemmas},(1).
Let $P = Z_u/\bracks{X}{Z_u}$.  We have $\bracks{P}{X}=1$ and
  \begin{align*}
    [P,X] &= PX- XP = -2XP + (PX + XP) \\
    &= -2XP +\bracks{P}{X}I_2\quad 
\mathrm{from~Lemma~\ref{lem:sl2-lemmas},(3)}\\
    &= -2XP + I_2\\
    &= \kappa X Z_u - X^2\quad
\mathrm{from~Lemma~\ref{lem:sl2-lemmas},(1,3)}\\
    &= X',
  \end{align*}
  which proves the claimed differential equation.
\end{proof}

\index{Lax equation}
\index{Lax equation}
\index[n]{P@$P$, normalized control matrix!Lax equation}

\begin{remark}\label{rmk:X-det-one}
\normalfont \leavevmode
\begin{enumerate}
\item 
The equation $X' = [P,X]$ where $P,X$ are time-dependent
matrices is called the \emph{Lax equation} and $X,P$ so related are
called a \emph{Lax equation}. Lax equations are well-studied in the theory of
integrable systems (see Perelomov~\cite{perelomov1990integrable},
Jurdjevic~\cite{jurdjevic2016optimal}, Babelon et
al.~\cite{babelon2003introduction}). Lax representations of integrable
systems are quite desirable since the evolution of a Lax equation
is \emph{isospectral},\index{isospectral} meaning that the spectrum of the matrix $X$
is an invariant of motion.
\item 
The dynamics for $X$ is Hamiltonian for a particular Hamiltonian
defined on $\sl(\R)$, with respect to a Poisson structure on $\sl(\R)$
called the \emph{Lie-Poisson structure}.\index{Lie-Poisson structure} See
Appendix \ref{sec:X-lie-poisson} for more details.
\item 
As explained in Perelomov~\cite[p.~52]{perelomov1990integrable}, the
spectral invariants guaranteed by the dynamics for $X$ are \emph{trivial}
integrals, and so it is more accurate to consider the dynamics for $X$
as giving a control-dependent infinitesimal generator for the
(co)adjoint action of $\SL(\R)$ on $\sl(\R)$, rather than to regard it
as describing the dynamics of an integrable system.
\item 
The equation \eqref{eqn:X-orig} appears in~\cite{hales2017reinhardt}, 
where its Lax equation reformulation was not explicitly recognized.
\end{enumerate}
\end{remark}

\section{Initial and Terminal Conditions}\label{sec:X-init-term-conds}
We now have dynamics for $g$ and $X$ in the Lie group and Lie algebra
respectively. We also have an associated cost objective. The only
thing remaining is to specify initial and terminal conditions.
Since our convex disk is in circle representation, this means that
$\sigma_j(0) = \mb{s}_j^*$ so that we start out at the sixth roots of
unity, and so we set $g(0) = I_2$.  The initial condition $X(0)=X_0$
\index[n]{X@$X$, Lie algebra element!$X_0$, initial condition for $X(t)$}
may be an arbitrary matrix in $\sl(\R)$ of determinant $1$,
provided it satisfies the star conditions in
Corollary \ref{cor:X-props}.

The terminal conditions $g(t_f)$ should be such that the curves
$\sigma_j$ close up seamlessly to form a simple closed curve:
\begin{equation}
g(t_f)\mb{s}_{j}^* = g(0)\mb{s}_{j+1}^* \Leftrightarrow g(t_f) = R,
\end{equation}
where $R$ is the usual rotation matrix. For terminal
conditions on $X$, note that we have the following conditions on $g$
which we obtain by the remark following Theorem \ref{thm:defn-g} (and
setting $t_0=0$ there):
\[
g(t+t_f)\mb{s}_j^* = g(s(t))\mb{s}_{j+1}^* = g(s(t))R\mb{s}_j^*; \qquad g(t+t_f)=g(s(t))R,
\]
for some orientation-preserving reparameterization $s(t)$ such that $s(0)=0$.
Differentiating, we obtain
\[
X(t+t_f) = R^{-1}X(s(t))R\frac{ds}{dt},
\]
which gives us $X(t_f) = R^{-1}X_0R ds/dt$.  Using
$\det(X)=1$ and $ds/dt>0$, we get $ds/dt=1$, $s(t)=t$,
and 
\[
X(t_f)=R^{-1}X_0R.
\]

\section{Reinhardt Optimal Control Problem}\label{sec:ROC}

Summarizing the discussion so far, we are finally ready to state the
Reinhardt conjecture as an optimal control problem. Let us begin with
a well-known proposition on the trivialization of the tangent and
cotangent bundles of a Lie group.

\index[n]{0@$-^*$, dual!linear}

\begin{proposition}\label{prop:trivialization}
Let $G$ be any real Lie group and let $\mathfrak{g}$ be its Lie
algebra. Then we have $T^*G \cong G \times \mathfrak{g}^*$ and
$TG \cong G \times \mathfrak{g}$,
where $\mathfrak{g}^*$ is the linear dual of $\mathfrak{g}$.
\end{proposition}
\index[n]{G@$G$, Lie group}
\index{Lie group}
\index[n]{G@$G$, Lie group!$\mathfrak{g}$, Lie algebra}

\begin{proof}
The Lie algebra $\mathfrak{g}$\index{Lie
algebra} is the tangent space
of $G$ at the neutral element\index{element!neutral}
$e\in{G}$. \index[n]{e@$e\in G$, neutral element} Let $L_g:G\to G$\index[n]{L@$L_g$, left multiplication by $g$} be
left-multiplication by $g$, given by $L_g(h):= g h$, The tangent map
\[
TL_g:T_hG\to T_{gh} G,
\index[n]{T@$Tf:T_qM\to T_{f(q)}N$, tangent map of $f$}
\]
at $h=e$ is an isomorphism from $T_eG$ to $T_{g}G$, the tangent spaces
at $e$ and $g$.  This isomorphism gives the trivialization of the
tangent bundle $TG\cong G\times \mathfrak{g}$.  Dually, each fiber
$T^*_gG$\index[n]{T@$T^*_qM$, cotangent space at $q$} of the cotangent
bundle is canonically isomorphic to the dual $T^*_eG
= \mathfrak{g}^*$.  This trivializes the cotangent bundle.
\end{proof}

The above proposition and remark apply to $\SL(\R)$. Using this, we
group together the state equations, controls and cost functional to give 
a well-defined control problem.
\index{state equations}

\begin{problem}[Reinhardt Control Problem]\label{pbm:reinhardt-optimal-control}
The convex disks in $\Kbal$ in circle representation
arise via the following optimal control problem. On the manifold
$\SL(\R) \times \sl(\R) \cong T\SL(\R)$, consider the following
optimal control problem with free-terminal time.
\begin{align}
g' &= gX, \qquad \qquad \qquad \qquad g:[0,t_f] \to \SL(\R) \label{eqn:trisystem-g};
\\ 
X' &=  \frac{\left[Z_u, X\right]}{\langle Z_u,X \rangle},  \qquad\qquad\quad~~ 
X:[0,t_f] \to \sl(\R) \label{eqn:trisystem-X}; 
\\ 
-\frac{3}{2}\int_0^{t_f}&\langle J,X\rangle dt \to \mathrm{min}, \qquad \quad \quad 
J = \left(\begin{matrix}  0 & -1 \\ 1 & 0 \end{matrix}\right), 
\label{eqn:trisystem-cost}
\end{align}
where the set of controls for this problem is the image of the
two-simplex $U_T$ in $\R^3$ inside the Lie algebra $\sl(\R)$ via the
affine map $Z_u$.
\begin{align}
    Z : U_T &= \left\{(u_0,u_1,u_2) \ | \ \sum_i u_i = 1,  \ u_i \ge 0 \right\} 
\to \sl(\R) \label{eqn:trisystem-U} 
\\ 
    Z_u &= \left(\begin{matrix}
\frac{u_1 - u_2}{\sqrt{3}} & \frac{u_0 - 2u_1 - 2u_2}{3} 
\\ 
    u_0 & \frac{u_2 - u_1}{\sqrt{3}}
\end{matrix}\right). \label{eqn:trisystem-Z0}
\end{align}
The initial conditions are $g(0) = I_2 \in \SL(\R)$ and $X(0) =
X_0 \in \sl(\R)$ satisfying the star conditions. Also, the terminal
conditions are $g(t_f) = R$ and $X(t_f)=RX_0R^{-1}$ where $R$ is
the usual rotation matrix~\eqref{eqn:R}.
\end{problem}

\newpage

\chapter{The Upper Half-Plane}

Now that we have the optimal control problem fully stated, a natural
next step would be to write down the necessary conditions for
optimality of trajectories. But before we do that, we will first 
cut down the state space of the problem.

\section{The Adjoint Orbit}

Recall that the star conditions (Corollary \ref{cor:X-props}) on the
matrix $X$ imply that $\bracks{J}{X}$  
is negative. We have also imposed the condition
$\det(X)=1$. We begin with a characterization of such matrices.

\begin{lemma}\label{lem:X-hyperbolic}
The set of matrices $X \in \sl(\R)$ with $\det(X)=1$ and
$\bracks{J}{X}<0$ is the adjoint orbit\index{adjoint orbit} 
\(
\O_J:=\{\Ad_gJ\mid g\in \SL(\R) \}
\)
\index[n]{O@$\O_-$, adjoint or coadjoint orbit}
in $\sl(\R)$ of the infinitesimal generator $J$.
\end{lemma}
\begin{proof}
The adjoint orbit $\O_J$
of $J$  in $\sl(\R)$ consists of elements 
$gJg^{-1}$ for $g \in \SL(\R)$. The Iwasawa decomposition\index{Iwasawa decomposition} 
of $\SL(\R)$
implies that $g$ belongs to a left coset $h\,\SO$ where
\index[n]{g@$g$, group element!in $\SL(\R)$}
\index[n]{h@$h\in G$, element of Lie group!component of Iwasawa decomposition}
\index[n]{x@$x$, real part of $z$}
\index[n]{y@$y$, imaginary part of $z$}
\begin{align*}
     h  &= \mattwo 1 {x} 0 1\mattwo {\sqrt{y}} 0 0 {1/\sqrt{y}}, \ y > 0.
\end{align*}
Since $\SO$ centralizes $J$, the orbit consists of elements
\begin{equation}
\index[n]{zV@$\Phi:\h\to\sl(\R)$}
gJg^{-1} = hJh^{-1}= \left(
\begin{array}{cc}
      x/y & -(x^2+y^2)/y 
\\
      1/y & -x/y 
\\
\end{array}
\right)
=:\Phi(z),\quad z = x+iy\in\h.
\end{equation}
This is precisely the form of a general element $X$ satisfying the
conditions of the lemma.  Hence, the result follows.
\end{proof}
\begin{remark}\normalfont
Since $\sl(\R)$ admits a nondegenerate invariant symmetric bilinear form
$\bracks{\cdot}{\cdot}$, the Lie algebra can be identified with its
linear dual, by identifying $X\in\sl(\R)$ with the linear functional
\[
Y\mapsto\bracks{Y}{X}
\]
on $\sl(\R)$.
Under this identification, the coadjoint orbits and
adjoint orbits become identified. See Appendix \ref{sec:kirillov} and
also Chapter 5 of Jurdjevic~\cite{jurdjevic_1996}.
\end{remark}

\section{Transfer of Dynamics to the Upper Half-Plane}\label{sec:transfer-dynamics}

\index{M\"obius transformation}
\index{linear fractional transformation}
\index[n]{abcd@$a,b,c,d$, matrix entries!of linear fractional transformation}
\index[n]{h@$\h$, upper-half plane}

The group $\SL(\R)$ acts on the upper-half plane\index{upper-half plane}
\[
\h = \left\{x+iy\mid y>0\right\}
\]
by linear fractional transformations (or M\"obius transformations).
\begin{equation}
\index[n]{0@$\cdot$, action!linear fractional on $\h$}
\SL(\R)\times\h\to \h,\quad \begin{pmatrix}a & b\\c & d\end{pmatrix}\cdot z = 
\frac{az+b}{cz+d}.
\end{equation}
We denote the action by $(\cdot)$.

By the orbit-stabilizer theorem, we
have $\O_J \cong \SL(\R)/\SO$, since the stabilizer of $J$ in
$\SL(\R)$ under the conjugation action (the centralizer) is
$\SO$. Viewing this in a different way, the group $\SL(\R)$ acts on
the upper half-plane $\h$ by linear fractional transformations, with
stabilizer of $i \in \h$ being given by $\SO$. Thus, the quotient is
isomorphic to the Poincar\'{e} upper half-plane $\h$. Putting all
of this together, we have

\index{orbit-stabilizer theorem}

\index[n]{zV@$\Phi:\h\to\sl(\R)$}
\index[n]{yz@$z\in\h$!$z=x+iy$}

\begin{lemma}\label{lem:def-phi}
The following map $\Phi$ is a isomorphism.
\begin{align*}
\Phi : \h &\to \mathcal{O}_J 
\\
     z = x+iy &\mapsto \Phi(z):=\mattwo {x/y} {-(x^2 + y^2)/y} {1/y} {-x/y}.
\end{align*} 
\end{lemma}
\begin{remark}\leavevmode\normalfont
\begin{itemize}
\item 
Note that $\mathcal{O}_J = \mathcal{O}_{\Phi(z)}$ as
$\Phi(z) \in \mathcal{O}_J$.
\item  
We write $X$ in place of $\Phi(z)$ for simplicity, bearing in mind that
$\Phi$ is surjective onto $\O_J$.
\item  
Note that $\Phi(z)$ is a regular semisimple element of the Lie algebra
$\sl(\R)$ because the element $J$ is.
\item\mcite{MCA:Phi-equivariant}
The map $\Phi$ is $\SL(\R)$-equivariant for the action by linear
fractional transformations on $\h$. That is, for every $g\in\SL(\R)$,
\[
g \Phi(z) g^{-1} = \Phi(g\cdot z).
\]
\end{itemize}
\end{remark}

This map $\Phi$ allows us to move back and forth between the upper
half-planes and the adjoint orbit in the Lie algebra $\sl(\R)$. Also,
the map $\Phi$ is more than just a bijection --- we show later that
this map is actually an
anti-symplectomorphism\index{symplectomorphism} and use this to
transfer the state and costate dynamics from the Lie algebra to the
upper half-plane. But first, we compute the tangent map $T \Phi$ at a
$z \in \h$.

\begin{lemma}\label{lem:half-plane-lie-algebra-iso}
For any $X \in \sl(\R)$, we have 
\[ 
\index[n]{R@$\RX$, span of $X$}
T_X \OX = \{[Y,X] \ | \
 Y \in \sl(\R) \}\cong \sl(\R)/\RX,
\index[n]{T@$T_qM$, tangent space at $q\in{M}$}
\] 
where $\RX$ is the
span of the element $X$ and 
 $T_X \OX$ denotes the tangent space to $\OX$ at $X$.  
\end{lemma}

Explicitly, the first equality views the tangent space at $X$ of the
 manifold $\OX$ as a subspace of $\sl(\R)\cong T_X\sl(\R)$, 
the tangent space at $X$ of the 
 ambient space $\sl(\R)$.  The isomorphism on the right is given by
 $[Y,X]\mapsto Y + \RX$.

\begin{proof}
Note that $\OX = \{\Ad_g X \ | \ g \in \SL(\R) \}$. We have to
describe tangent vectors to $\OX$. For any $Y \in \sl(\R)$,
$\Ad_{\exp(tY)} X$ is a curve in $\OX$. Thus, the tangent vector to
this curve is computed as
\[ 
\frac{d}{dt} \Ad_{\exp(tY)}X\Bigr|_{t=0} = \ad_Y X := [Y,X] \in T_X\OX.
\]
\index[n]{ad@$\ad$, adjoint representation of the Lie algebra}
This calculation is actually finding infinitesimal generators of the
adjoint action. 
There is an isomorphism
\[
\sl(\R)/\sl(\R)_X \cong
\{[Y,X] \ | \
 Y \in \sl(\R) \},\quad Y \mapsto [Y,X],
\]
where $\sl(\R)_X$ is the isotropy algebra\index{isotropy algebra} (the centralizer)
\index{centralizer}
of the element $X$. 
The element $X$ is regular in the rank one algebra $\sl(\R)$, so
that its centralizer $\sl(\R)_X$\index[n]{0@$-_X$, centralizer of $X$} 
is the span $\RX$ of $X$.
\end{proof}


We write $[Y]$ for the coset $Y+\RX$ in $\sl(\R)/\RX$.


\begin{lemma}\label{lem:tangent-maps}\mcite{MCA:6479593}
  We have the following expression for the tangent map $T \Phi$.
\begin{align}
\index[n]{r@$r$, real number!$r_i$, scalar}
     T \Phi : T_z \h &\to T_{X}\OX \cong \sl(\R)/\RX 
\\
     (r_1 \delx + r_2 \dely) &\mapsto \mattwo {r_2 /2y}{(yr_1 - r_2 x)/y}{0}{-r_2/2y} \mod \RX.
\end{align}
\end{lemma}
\begin{proof}
\index[n]{Y@$Y\in\mathfrak{g}$, Lie algebra element!in $\sl$}
We have, at $z = x+iy$ and $X = \Phi(z)$:
\begin{align}
     T_z \Phi (r_1,r_2) &= \frac{d}{dt}{\Phi}(x + t r_1, y + t r_2)\Bigr|_{t=0} 
\\
     &= r_1 \frac{\partial {\Phi}}{\partial x} + r_2 \frac{\partial {\Phi}}{\partial y} \in T_{X}\OX.
\end{align}
We know by the previous lemma that there exists a matrix $Y_z$  such
that 
\[
T_z \Phi(r_1,r_2) = r_1 \frac{\partial {\Phi}}{\partial x} +
r_2 \frac{\partial {\Phi}}{\partial y} = [Y_z(r_1,r_2),{X}].
\] 
Using this
equation to solve for this matrix $Y_z$ gives us the following.
%
\begin{align}
Y_z 
&\equiv \mattwo {r_2 /2y}{(yr_1 - r_2 x)/y}{0}{-r_2/2y} \mod \RX.
\end{align}
So, for any arbitrary vector $(r_1,r_2) \in T_z \h$ we obtain its image
inside the quotient space $\sl(\R)/\RX$.
\end{proof}

\section{The Cost Functional in Half-Plane Coordinates}
We can also compute the cost functional that we derived in
Section~\ref{sec:the-cost-functional} in half-plane coordinates. From
equation \eqref{eqn:cost}, we have
\begin{align}
-\frac{3}{2}\int_0^{t_f}\bracks{J}{X}dt &=    
\frac{3}{2}\int_0^{t_f}\frac{x^2+y^2+1}{y}dt 
\rightarrow \text{min}. \label{eqn:cost-upper-half-plane}
\end{align}
\mcite{MCA:6654903}

\index[n]{A@$A$, matrix or linear map!rotation}

The cost functional is $\SO$-invariant, because if $A$
is any rotation
matrix, then
\[
\bracks{J}{\Ad_A X}=\bracks{\Ad_{A^{-1}} J}{X}=\bracks{J}{X}.
\]
The circular symmetry is also apparent in this reinterpretation. The
level sets of $(x^2+y^2+1)/y$ are concentric circles (with respect to the
hyperbolic metric)\index{hyperbolic!metric} centered at the
point $i$ in the upper half-plane. Thus, the cost is $\SO$-invariant.

\begin{lemma}\label{lem:tf-pi3}
The global minimizer of the Reinhardt control problem has terminal
time $t_f<\pi/3$.
\end{lemma}

\begin{proof}
The global minimizer has area less than the area $\pi$ of
the unit circle $K$.  We have $(x^2 + y^2 + 1)/y \ge (y + 1/y)\ge 2$,
so that
\[
\pi> \op{area}(K_{\min}) =\frac{3}{2}\int_0^{t_f} \frac{x^2 + y^2 + 1}{y} dt
\ge \frac{3}{2}\int_0^{t_f} 2\,dt = 3t_f.
\]
\end{proof}

\begin{remark}\normalfont
The Poincar\'{e} upper half-plane is conformally
equivalent\index{conformal equivalence} to other models of hyperbolic
geometry such as the Poincar\'{e} disk and the hyperboloid model. The
cost functional derived above can also be derived in these models. In
the disk model, $\mathbb{D}=\{w \in \C\mid|w| <
1\}$,\index[n]{D@$\mathbb{D}$, disk model of hyperbolic geometry} for
example, the cost of a path $w:[0,t_f]\to\mathbb{D}$
\index[n]{w@$w$ path in the hyperbolic disk}
becomes
\[
3\int_0^{t_f}\frac{1+|w|^2}{1-|w|^2}dt \rightarrow \text{min}.
\]
In the hyperboloid model of hyperbolic geometry, the model is
the upper sheet of the two-sheeted hyperboloid.  In that model,
the cost functional becomes the integral of the height function on the
hyperboloid sheet.  See~\cite{hales2017reinhardt}.
 
\end{remark}
\section{The Star Domain in the Upper Half-Plane}
We can now prove our first state space reduction result.

\index[n]{0@$-^\star$, star domain}
\index[n]{h@$\h$, upper-half plane!$\hstar$, star domain}
\index{star!domain}

\begin{theorem}\label{prop:star-domain}\mcite{MCA:2104115}
The dynamics of the Reinhardt control problem is constrained to an
\emph{ideal triangle}\index{ideal triangle} in the upper half-plane.
\begin{equation}\label{eqn:star}
\hstar := 
\left\{x+iy \in \h\mid-\frac{1}{\sqrt{3}}<x<\frac{1}{\sqrt{3}},
\quad\frac{1}{3}<x^2 + y^2\right\}.
\end{equation}
 Thus, the new
state-space of the control problem is $\SL(\R)\times \hstar$.
\end{theorem} 
\begin{proof}
The star conditions on $X$ in Corollary \ref{cor:X-props} applied to
$X = \Phi(z) = \Phi(x+iy)$ give us the conditions on $x$ and $y$ which an
\emph{admissible trajectory}\label{admissible trajectory} should satisfy.

This region, called the \emph{star domain},\index{star!domain} is the
interior of an ideal triangle in the upper half-plane. The vertices of
this triangle are the points $z=\pm\frac{1}{\sqrt{3}}$ and $z
= \infty$. A picture of the star domain is shown in
Figure \ref{fig:star-dom}.
\end{proof}

\begin{figure}[htbp]
\centering
\includegraphics{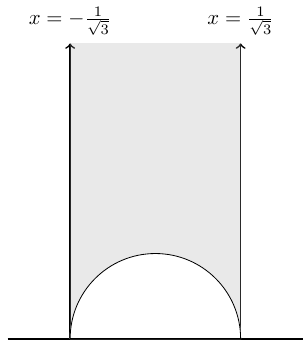}
\caption{The star domain in the upper half-plane.}
\label{fig:star-dom}
\end{figure}

Summarizing the results so far, we have parameterized the boundary of
convex disks in $\Kbal$ as $U_T$-controlled paths
$(g(t),z(t)) \in \SL(\R)\times \hstar$ subject to the terminal
conditions. Our task is to find a \emph{control
function}\index{control!function} $u(t) \in U_T$ which minimizes the
area enclosed by the resulting curve, given in hyperbolic coordinates
by equation \eqref{eqn:cost-upper-half-plane}.

\section{Control Problem in the Half-Plane}

Evolution of $X(t)=h(t)Jh(t)^{-1}$ by adjoint action in
$\mathcal{O}_J$ corresponds to evolution by linear fractional
transformations of the corresponding element $z_0$ in the upper
half-plane picture.  If $\Phi(z_0)=X_0$ and $\Phi(z(t))=X(t)$, then by
the $\SL(\R)$-equivariance of $\Phi$,
the initial and terminal conditions derived in
Section \ref{sec:X-init-term-conds} are transformed as
\begin{align*}
X(0) = X_0 &\Longleftrightarrow z(0) = z_0, 
\\
X(t_f) = R^{-1}X_0 R &\Longleftrightarrow z(t_f) = R^{-1}\cdot z_0,
\end{align*}
where $R$ is the usual rotation matrix \eqref{eqn:R}. 

Thus, we obtain the following reformulation of
the Reinhardt conjecture from the coadjoint orbit of the Lie algebra
to the Poincar\'{e} upper half-plane.

\begin{problem}[Half-Plane Control Problem]\label{pbm:plane-optimal-control-problem}
On the set $\SL(\R)\times\hstar \subset T\SL(\R)$, consider the
following free-terminal time optimal control problem.
\begin{align*}
\mcite{MCA:4861125}
g' &= gX, \quad X = \mattwo {x/y} {-(x^2 + y^2)/y} {1/y} {-x/y}=\Phi(x+iy),  \\[5mm]
x' &= f_1(x,y;u) := 
\frac{y \left(2 a x+b-c x^2+c y^2\right)}{2 a x+b-c x^2-c y^2}, \\[5mm]
y' &=  f_2(x,y;u):= \frac{2 y^2 (a-c x)}{2 a x+b-c x^2-c y^2},         \\[5mm]
\frac{3}{2}&\int_0^{t_f}\frac{x^2+y^2+1}{y}dt 
\rightarrow \text{\normalfont min},  \\[5mm]
 & g:[0,t_f] \to \SL(\R),\quad x,y:[0,t_f] \to \hstar,
\end{align*}
where the coefficients $a,b,c$ are the following 
affine functions of the control \eqref{eqn:trisystem-Z0}.
\begin{equation*}
    a = a(u)=\frac{u_2 - u_1}{\sqrt{3}}, \quad
    b = b(u)=\frac{u_0 - 2u_1 - 2u_2}{3}, \quad
    c = c(u)=u_0, 
\end{equation*}
with $u=(u_0,u_1,u_2) \in U_T$, which is the two-simplex in $\R^3$. This
problem has intial conditions $g(0) = I_2 \in \SL(\R)$ and $z(0) =
z_0 \in \hstar$ and terminal conditions $g(t_f) = R$ and
$z(t_f)=R^{-1}\cdot z_0$ where $R$ is the usual rotation. 
\end{problem}

\begin{lemma}\label{lem:X-dynamics}
The ODE \eqref{eqn:trisystem-X} for $X$
implies the system of ODEs for $x',y'$ in the half-plane optimal
control problem \ref{pbm:plane-optimal-control-problem}.
\end{lemma}

\begin{proof}\mcite{MCA:4861125}
We compute
\[
\begin{pmatrix}
(y x' - x y')/{y^2} & * \\ -y'/{y^2} & *
\end{pmatrix}
= \Phi(z(t))' = \frac{[Z_u,\Phi(z)]}{\langle Z_u,\Phi(z)\rangle} =
\begin{pmatrix}
(y f_1 - x f_2)/y^2 & * \\ - f_2/y^2 & *
\end{pmatrix}.
\]
Comparing the left and right-hand sides of this equation,
find that $x'=f_1$ and $y' = f_2$.
This also shows that
\begin{equation}\label{eqn:tangent-map-f1f2}
     T\Phi(f_1,f_2) = 
\left[\frac{Z_u}{\langle Z_u,X \rangle}\right] \in T_X\OX.
\end{equation}
\end{proof}


Thus, we have transferred the Lie algebra dynamics to the upper
half-plane. In Appendix~\ref{sec:kirillov}, we also prove that the map
$\Phi$ is actually an anti-symplectomorphism onto the upper half-plane.
Thus, it is entirely equivalent to study the control problem in the Lie algebra
picture or the half-plane picture. We may also transfer the dynamics
to other models of hyperbolic geometry: for example, the Poincar\'{e}
disk model\index{disk model of hyperbolic geometry} or the hyperboloid
model.\index{hyperboloid!model of hyperbolic geometry} Each picture
has its advantages, with some simplifying equations while others are
better since the symmetries are more apparent.

We have finally reached the end of the reduction chain and have
transformed a problem in discrete geometry to an optimal control problem on
$T\SL(\R)$. Already, we see that this problem is remarkably rich, with
connections to Hamiltonian mechanics and hyperbolic geometry.

\section{Dihedral Symmetry}\label{sec:dihedral}

\index[n]{D@$\Dih$, dihedral group of order $12$} 

The dihedral group\index{dihedral group} $\Dih$
of order $12$ 
of the hexagon acts on the sixth roots of unity
through orthogonal transformations.  This action of the dihedral group
extends to many of the constructions throughout this book.

\index[n]{O@$\OR$, orthogonal group}
\index[n]{SO@$\mathrm{SO}$, special orthogonal group!$\mathrm{O}_n$, orthogonal group}
\index[n]{A@$A$, matrix or linear map!in dihedral group}
\index[n]{ze@$\epsilon\in\{-1,0,1\}$, sign!$\epsilon_A$, determinant of $A$}

Let $\sigma$ be a multi-curve parameterizing the boundary of
$K\in\Kbal$.  We assume that $K$ is in the circle representation, and
that $\sigma_j(0)=\mb{s}_j^*$.  Let $A$ be an element of the dihedral group of the hexagon, considered
as an element of the orthogonal group $\OR(\R)$.
Let $\epsilon_A=\det(A)\in\{\pm1\}$ be the
determinant.  Let $\sigma_j(t)=g(t)\mb{s}_j^*$ as usual, with $g(0)=I_2$.
Then $\tilde{g}(t)=Ag(\epsilon_At)A^{-1}$ determines a multi-curve
\index[n]{0@$-\tilde{\phantom{-}}$, transformed quantity}
\[
\tilde\sigma_j(t) = A g(\epsilon_A t)A^{-1}\mb{s}_j^*,
\]
parameterizing the boundary of a convex disk $AK\in\Kbal$ with the
same area as $K$.  The sign $\epsilon_A$ is chosen to make the
multi-curve parameterize the boundary of $AK$ in a counterclockwise
direction.  Then the action extends to the Lie algebra
\[
(\tilde{g}^{-1} \tilde g')(t) 
= \tilde X(t) = \epsilon_A A X(\epsilon_A t) A^{-1}.
\]
Generators of the dihedral group are the rotation $R$
and the reflection across the vertical axis:
\[
S=\begin{pmatrix}-1 & 0 \\ 0 & 1\end{pmatrix},
\index[n]{S@$S\in\sl(\R)$, reflection matrix}
\]
with $\epsilon_S=\det(S)=-1$.
Writing $X = \Phi(z)$, the reflection acts by $z\mapsto -\bar z$, where $\bar z$
is complex conjugation:
\index[n]{yz@$z\in\h$!$\bar{z}$, complex conjugate}
\index[n]{0@$-\bar{\phantom -}$, complex conjugate}
\[
\epsilon_S S \Phi(z) S^{-1} = \Phi(-\bar z).
\]
This preserves the upper-half plane, but is not orientation preserving.
The rotation $R$ (with $\epsilon_R=1$) acts by linear fractional transformation
\[
R \Phi(z) R^{-1} = \Phi(R\cdot z).
\]
\index[n]{zh@$\theta$, angle} 
%
%
The actions of the dihedral group $\Dih$ on the multi-point
$\mb{s}_j$, on multi-curves, on the control set, on the star
conditions, on the ideal triangle, and on the upper half plane through
M\"obius transformation are all compatible.  That is, many of our maps
are equivariant with respect to the dihedral group.  Recall that the
linear fractional action of $\exp(J\theta)$ on $\h$ acts on the
tangent space $T_i\h$ at $i=\sqrt{-1}$ by a \emph{clockwise} rotation
by angle $2\theta$.  $R^3=-I_2$ acts trivially, so that the action of
the dihedral group of the hexagon factors through the dihedral group
of an equilateral triangle -- the symmetric group on three letters.

The action of the dihedral group permutes the star inequalities.
In terms of the linear functions $\rho_j:\sl(\R)\to\R$ defined in
Corollary~\ref{cor:X-props}, we have
\begin{align*}
\rho_j(\epsilon_S S X S^{-1}) &= \rho_{1-j}(X),
\\
\rho_j(R X R^{-1}) &= \rho_{j-1}(X),\quad j\in \Z/6\Z.
\end{align*}
It follows that the dihedral group acts on the star domain $\hstar$.  
The group permutes the ideal vertices of $\hstar$.  
The action on the ideal vertices $\pm1/\sqrt3$ and $+\infty$ is
by linear fractional transformations on 
$\mathbb{RP}^1$. Here we are viewing the boundary of the upper-half
plane, consisting of the real axis and the point at infinity, as a real
projective line.  We have
\begin{align*}
R\cdot(+\infty)&=1/\sqrt3,\quad R\cdot(1/\sqrt3)=-1/\sqrt3,\quad R\cdot(-1/\sqrt3)=+\infty,\\
S\cdot(\pm1/\sqrt3)&=\mp1/\sqrt3.
\end{align*}
\mcite{MCA:rho-rotate}
\index[n]{RP@$\mathbb{RP}^1$, real projective line}

We describe a fundamental domain for the action of the dihedral group
(the symmetric group on three letters) on $\hstar$.  The positive 
imaginary axis is a geodesic in the upper-half plane.  Under the action,
the orbit of this geodesic is a set of three geodesics.  The other
two geodesics are the circles of radius $2/\sqrt3$ centered at 
the two cusps $(0,\pm1/\sqrt3)$ on the real axis.  These three geodesics
meet at $z = 0+i\in\hstar$ and partition $\hstar$ into six sectors.
Each of these sectors is a fundamental domain for the action.
See Figure~\ref{fig:dihedral-fund-domain}.
Specifically, one such fundamental domain is given by
\[
\{z=x+iy \in \hstar \mid x \ge 0,\quad (x-1/\sqrt3)^2 + y^2 \le 4/3\}.
\]

\begin{figure}[htbp]
  \centering
  \includegraphics{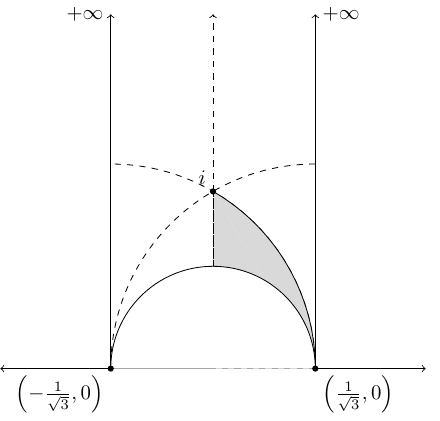}
  \caption{A fundamental domain for the dihedral action on $\hstar$ is shaded 
  in gray. The generators of the dihedral group $R$ and $S$ take the shaded
  domain to the other unshaded ones.}
  \label{fig:dihedral-fund-domain}
  \end{figure}

The dihedral group acts on \emph{everything in sight}, such as the
control set $U_T$, and so forth.  The dihedral group acts on 
the control by the rule
\[
\index[n]{0@$\cdot$, action!of dihedral group on control}
\epsilon_A AZ_u A^{-1} = Z_{A\cdot u},\quad A\in \Dih,\quad u\in U_T.
\]
Explicitly,
\begin{align*}
R\cdot(0,0,1)&=(0,1,0)\quad
R\cdot(0,1,0)=(1,0,0),\quad
R\cdot(1,0,0)=(0,0,1)\in{}U_T.
\\
RZ_{(0,0,1)}R^{-1}&=Z_{(0,1,0)},\quad RZ_{(0,1,0)}R^{-1} = Z_{(1,0,0)},\quad
RZ_{(1,0,0)}R^{-1}=Z_{(0,0,1)}.\\
-SZ_{(0,1,0)}S^{-1}&=Z_{(0,0,1)},\quad -SZ_{(1,0,0)}S^{-1}=Z_{(1,0,0)}.\\
\end{align*}

The action on $U_T$ is such that if $X$ is a solution to the Lie
algebra state equation with constant control $u$, then the transform
$\tilde X$ by $A$ is a solution to the state equation with constant
control $\tilde u = A\cdot u$, as can be checked directly from the
ODE $X'=[Z_u,X]/\langle Z_u,X \rangle$.

For example, the trajectory with control $(0,0,1)$ has state-dependent
curvatures $\kappa_0=0$ and $\kappa_1=0$, so that $g(t)\mb{s}_0^*$ and
$g(t)\mb{s}_2^*$ are straight lines, while $g(t)\mb{s}_4^*$ moves in a
hyperbolic arc.  Taking $A=R$, we see that $\tilde\sigma_2(t) =
Rg(t)R^{-1}\mb{s}_2^*=-Rg(t)\mb{s}_3^*$ also moves in a hyperbolic arc, and
its control is $\tilde u=(0,1,0) = R\cdot (0,0,1)$. 

\chapter{Compactification of the Star Domain}\label{sec:compactification}



While the star domain is a reduction of the state space, it is an
ideal triangle with one vertex at infinity. Thus, it is open and
unbounded in the upper half-plane. Our task in this section will be to
explore a further reduction of this admissible region.

Empirical observations show that if $z \in \hstar$ is close to the
boundary curves of the star domain, then the corresponding critical hexagon 
(constructed in Lemma \ref{lem:pt-to-hexagon}) is close to a
parallelogram. Classical results of Mahler and Reinhardt state that
the only convex disks in $\Kccs$ with a parallelogram for a
(degenerate) critical hexagon are parallelograms themselves. This
suggests that there is a neighborhood of the boundary of the star
domain which gives rise to convex disks in $\Kccs$ whose packing
density is close to one and so these convex disks can be
excluded from consideration, since they are never optimal for our
control problem. We make this intuition precise presently.

Our hope is that if can cut down the state space to a compact region,
eventually computer numerical solutions of the dynamics will become
feasible.  We have obtained the following compactification.
In this chapter, by \emph{compactification}, we mean an explicit
compact subset of the star domain $\hstar$, such that all trajectories
of interest must lie inside that compact set.
We have not optimized parameters to obtain the smallest possible compact
region.  We leave that for future work.


\index{compactification}
\index[n]{0@$-^{\star\star}$, truncated star domain}
\index[n]{h@$\h$, upper-half plane!$\h^{\star\star}$, star domain compactification}
\index{horocycle}
\index{horoball}

We define a \emph{horocycle} in the upper-half plane to be a
horizontal line, or a Euclidean circle in the upper half-plane that is
tangent to the real axis.  We define a \emph{horoball} to be the
region in the upper-half plane that is bounded by the horocycle:
either the region above the horizontal line or the interior of the
circle. In general
the image of the horocycle $y=y_0 \in \hstar$ under a linear fractional
transformation
\index[n]{B@$B$, horoball at a cusp}
\index{horoball}
\index[n]{abcd@$a,b,c,d$, matrix entries!of linear fractional transformation}
\index[n]{r@$r$, real number!radius}
\[
A=\begin{pmatrix}a& b\\ c&d\end{pmatrix}
\]
is a Euclidean circle tangent to the real axis, with center
$(A\cdot\infty) + i r=a/c+ir$ and radius $r = \det\!A/(2y_0c^2)$.  

\begin{definition}[compactification]
Let $\h^{\star\star}\subset\hstar$ be the compact set defined by the
following inequalities.  The inequality $y>4.5$ defines an open
horoball $B(i\infty)$ around the cusp of $\hstar$ at $z=+i\infty$. By
linear fractional transformations $R,R^2\in\SL(\R)$ acting on $\h$, we
obtain open horoballs $B(1/\sqrt3)$ and $B(-1/\sqrt3)$ at the other
cusps (that is, at the ideal vertices) $z=\pm1/\sqrt3$ of $\h$.  If
the linear fractional transformation is $A=R^{\pm1}$ and $y_0=4.5$,
the radius $r$ is $4/27$, and the center is
$(A\cdot\infty)=(R^{\pm}\cdot\infty)=\pm1/\sqrt3$.

\index[n]{zP@$\Pi^+_i$, open half-plane}

The open half-plane $\Pi^+_0$ defined by $y > 15 (1/\sqrt3 - x)$ 
includes the boundary curve $x=1/\sqrt{3}$, $y>0$ of
$\hstar$.  By linear fractional transformations $R,R^2$, we obtain
transformed regions $\Pi^+_1$ and $\Pi^+_2$ around the other boundary
curves. Set
\[
\h^{\star\star} = \hstar \setminus (B(i\infty) \cup B(1/\sqrt3) \cup B(-1/\sqrt3) 
\cup \Pi^+_0 \cup \Pi^+_1 \cup \Pi^+_2).
\]
\end{definition}

\tikzfig{hstarstar}{The central region away from the boundary of the star
domain is the compactification $\h^{\star\star}$ of the star
domain.}  {
\def\rt{1.732}
\begin{scope}[xshift=0in,yshift=0in,scale=2]
\shade[top color=white,bottom color=gray] (-1/\rt,0) rectangle (1/\rt,3cm);
\draw (-1/\rt,0)--(-1/\rt,5);
\draw (1/\rt,0)--(1/\rt,5);
\draw (0,3) node [anchor=north] {$\h^{**}$};
\begin{scope}
\clip (-1/\rt,0) rectangle (1/\rt,3cm);
\draw[fill=white] (0,0) circle (0.577cm);
\draw[fill=white] (-1/\rt,0) circle (0.148cm);
\draw[fill=white] (1/\rt,0) circle (0.148cm);
\draw (0,0.038447) circle (0.578cm);
\end{scope}
\draw (-2,0) -- (2,0);
\draw (-1/\rt,4.5) -- (1/\rt,4.5);
\draw (-1/\rt,0) -- (-0.277,4.5);
\draw (1/\rt,0) -- (0.277,4.5);
\end{scope}
}

The set $\h^{\star\star}\subset\hstar$ is compact.
See Figure~\ref{fig:hstarstar}. 
The shape of the compactification
$\h^{\star\star}$ has been chosen to be invariant under the action of the
dihedral group.
The entire chapter is devoted to the proof of the following theorem.

\begin{theorem}\label{th:compactification}
Let $K$ be a convex disk in $\Kbal$, with corresponding boundary
trajectory $(g,X)$.  Define a trajectory $z$ in $\hstar$ by
$\Phi\circ z=X$.  If any point of the trajectory $z$ is not in $\h^{\star\star}$,
then the cost of the trajectory is strictly greater than the area of
the smoothed octagon.  Hence $K$ is not a global minimizer.
\end{theorem}
\index[n]{0@$\circ$, function composition}

Given a convex disk $K$ in $\Kbal$ in the circle representation, we
parameterize the boundary multi-curve $\sigma_j(t) = g(t)\mb{s}_j^*$, with
$g(0)=I_2$.  At $t=0$, we obtain an element $z\in\hstar$ such that 
$\Phi(z)=X(0)=g^{-1}(0)g'(0)$.  Also associated with $K$ is a critical
hexagon $H_K$ with midpoints at the points $\{\mb{s}_j^*\}$.
Conversely, an element $z\in\hstar$ can be used to reconstruct a
centrally symmetric hexagon $H_K$ with midpoints $\{\mb{s}_j^*\}$ as
follows.

\begin{lemma}\label{lem:pt-to-hexagon}
Every $z\in\hstar$ gives rise to a centrally symmetric hexagon $H_K(z)$
whose midpoints are at $\{\mb{s}_j^*\}$ and whose oriented directions along
the edges point into the star domain.  If $z$ is constructed from the
boundary parameterization of a convex disk $K\in\Kbal$ in the
circle representation as described
above at multi-point $\{\mb{s}_j^*\}$, then $H_K(z)$ is the critical hexagon
at the multi-point $\{\mb{s}_j^*\}$ of $K$.
\end{lemma}

\begin{proof}
The element $z$ in the upper half-plane determines a matrix
$\Phi(z)$ in the adjoint orbit of $J$ in the Lie algebra
$\mathfrak{sl}_2$.  The centrally symmetric hexagon $H_K(z)$ is then
reconstructed from $\Phi(z)$ according to Remark~\ref{rem:hexagon}.
\end{proof}

\begin{figure}[htbp]
\centering
\includegraphics{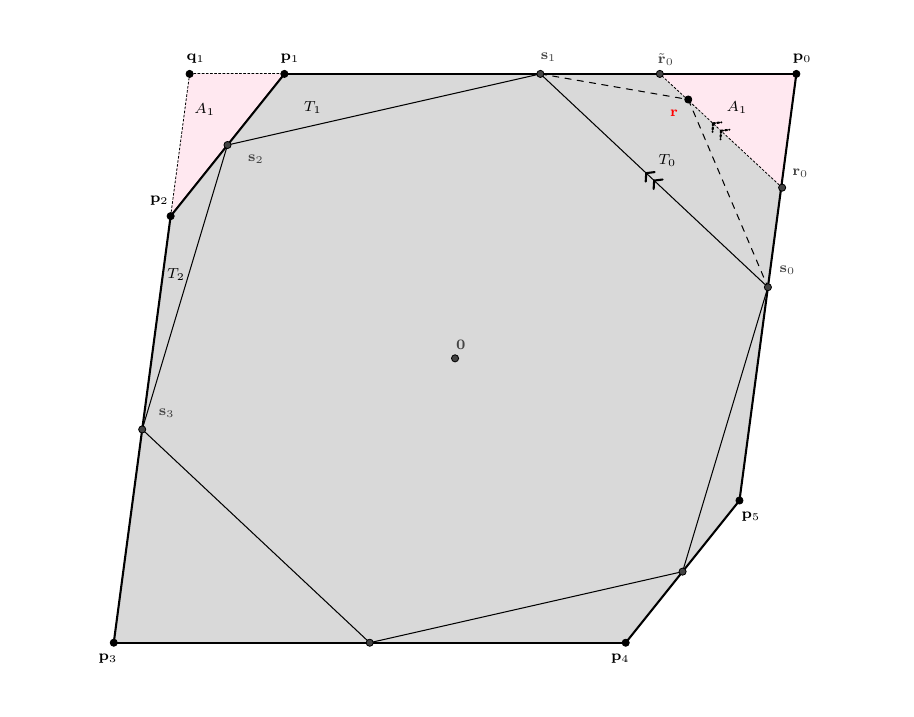}
\caption{Critical Hexagon. (The figure has been rotated to make
    $\mb{p}_0\mb{p}_1$ horizontal, so that the
    sixth roots of unity $\mb{s}_i$ are also in a rotated position.) The triangles shown in pink are a result of the construction in Lemma~\ref{lem:geom-thales}.}
\label{fig:compactification-result}
\end{figure}

Lower case bold letters will denote
points $\mb{p}_i,\mb{q}_i,\mb{r}_i,\mb{s}_i$. Upper case
will denote triangles $T_i,T^{ext}_i$ in the plane and
convex regions $H,K$.   We often use the same upper case letter for a
triangle and its area with respect to Lebesgue measure.  The correct
interpretation can be inferred by context.  
\index[n]{s@$\multi_i$, multi-point}
\index[n]{p@$\mb{p}_i,\mb{q}_i,\mb{r}_i$, points in the plane}
\index[n]{T@$T$, triangle}
\index[n]{T@$T$, triangle!$T^{ext}$, exterior triangle}
\index[n]{H@$H$, Euclidean region!$H,K$, convex regions in the plane}
\index[n]{r@$\mb{r}$, point in the plane}
\index[n]{O@$\mathbf{0}$, origin}



We consider subscripts modulo $6$, as we do elsewhere in the book,
with the understanding that when it comes to area computations, the
area is preserved under central reflection, so that area calculations
have a smaller period of $3$.

\index[n]{1@$\triangle$, triangle}
The situation is depicted in
Figure \ref{fig:compactification-result}. Recall that our convex disk
is in circle representation, having the sixth roots of unity on its
boundary.  Let $\mb{p}_0,\mb{p}_1,\mb{p}_2,\mb{p}_3,\mb{p}_4,\mb{p}_5$
be the vertices of the critical hexagon with midpoints at
$\mb{s}_j=\mb{s}_j^*$.  We define (interior) triangles
$T_0=\triangle \mb{s}_0 \mb{p}_0 \mb{s}_1$,
$T_1=\triangle \mb{s}_1 \mb{p}_1 \mb{s}_2$, $T_2
= \triangle \mb{s}_2 \mb{p}_2 \mb{s}_3$. Let $\mb{q}_1$ denote the
point of intersection of the lines ${\mb{p}_3\mb{p}_2}$ and
${\mb{p}_0\mb{p}_1}$ through nonadjacent edges of the hexagon.
Similarly for $\mb{q}_0$ and $\mb{q}_2$. This now determines the
exterior triangles $T^{ext}_0=\triangle{\mb{p}_1 \mb{q}_1 \mb{p}_2}$,
$T^{ext}_1 = \triangle{\mb{p}_2\mb{q}_2\mb{p}_3}$, and
$T^{ext}_2=\triangle{\mb{p}_0\mb{q}_0\mb{p}_1}$. (The latter two
triangles are not depicted in the figure.) For our compactification
result, we will need the areas of these triangles in terms of $z = x +
iy$.

The functions $\rho_0,\rho_1,\rho_2$ of Equation \eqref{eqn:rho} are
linear functions of $X\in\sl(\R)$.  Considering them as a function of
$z\in\h$ through the map $X=\Phi(z)$, we abuse notation slightly by
writing $\rho_j(z)$ for $\rho_j(\Phi(z))$, where now $\rho_j:\h\to\R$.
The star domain $\h^*$ is defined in $\h$ by the star inequalities
$\rho_j(z)>0$ for $j=0,1,2$.

\index[n]{zr@$\rho_j$, star function!$\rho_j(z):=\rho_j(\Phi(z))$}

\begin{lemma}\label{lem:triangle-areas}
We have
\[
\op{area}(T_0) = \frac{\sqrt3}{4} \rho_0\rho_2,
\qquad 
\op{area}(T_1) = \frac{\sqrt{3}}{4}\rho_0\rho_1,
\qquad
\op{area}(T_2) =  \frac{\sqrt{3}}{4}\rho_1\rho_2,
\]
and 
\[
\op{area}(T^{ext}_0) = \sqrt{3}\rho_1^2, 
\qquad
\op{area}(T^{ext}_1) = \sqrt{3}\rho_2^2, 
\qquad 
\op{area}(T^{ext}_2) = \sqrt{3}\rho_0^2,
\]
where $\op{area}$ is the Lebesgue measure on $\R^2$.  
Furthermore, we have 
\[
\op{area}(T_0) + \op{area}(T_1) +\op{area}(T_2) = \frac{\sqrt{3}}{4}.
\]
\end{lemma}

\begin{proof}
Lemma \ref{lem:pt-to-hexagon} leads to an explicit construction of the
coordinates of the points $\mb{p}_i$ and $\mb{q}_i$ for
$i=0,1,2$. Once we have coordinates of all these points, finding the
areas of the associated triangles is straightforward from Equations
\eqref{eqn:rho}, \eqref{eqn:rho-star-inequalities} and
\eqref{eqn:hex-vertex}, with $\det(X)=1$.  For example, 
\[
\mb{p}_0=\mb{s}_0+\rho_2Xs_0 = 
\mb{s}_0+\rho_2(\rho_0\mb{s}_1+\rho_1\mb{s}_2)
\]
gives
\begin{align*}
\op{area}(T_0) &= \frac{1}{2} \det((\mb{p}_0 - \mb{s}_0), \mb{s}_2)
= \frac{1}{2} {\rho}_2 \det(X \mb{s}_0 , \mb{s}_2)
= \frac{1}{2} {\rho}_2 \rho_0 \det(\mb{s}_1, \mb{s}_2)
= \frac{\sqrt{3}}{4} \rho_0 \rho_2.
\end{align*}
\mcite{MCA:7511280}
We have $\mb{q}_1=\mb{p}_1 + 2\rho_1 X\mb{s}_1=\mb{p}_2-2\rho_1 X\mb{s}_3$, and
\begin{align*}
\op{area}(T^{ext}_0) &= \frac{1}{2}\det((\mb{p}_2-\mb{q}_1), (\mb{p}_1-\mb{q}_1) )
\\
&= -2 \rho_1^2 \det(X\mb{s}_3, X\mb{s}_1)
= -2 \rho_1^2 \det(\mb{s}_3, \mb{s}_1)
= \sqrt3\rho_1^2. 
\end{align*}
The other cases are similar, by shift of indices.

The sum of the areas $T_i$ is obtained by Equation \eqref{eqn:det-rho}.
\[
\frac{\sqrt{3}}{4}(\rho_0 \rho_2 + \rho_1 \rho_2 + \rho_0 \rho_1) = \frac{\sqrt{3}}{4}\det(X)
=\frac{\sqrt3}{4}.
\]
To give a second proof
that the sum of the areas of $T_i$ is $\sqrt{3}/4$, an equilateral
triangle of edge length $1$ can be dissected into three
triangles congruent to $T_1,T_2,T_3$.
\end{proof}

\begin{lemma}\label{lem:geom-thales}
\begin{figure}[t]
\centering
\includegraphics{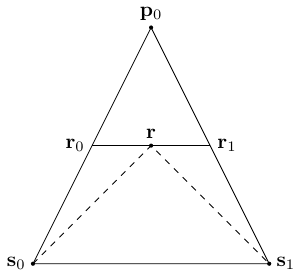}
\caption{Area of cutoff triangles.}
\label{fig:triangle-lemma}
\end{figure}
As shown in Figure~\ref{fig:triangle-lemma}, in triangle $\triangle
\mb{s}_0\mb{s}_1\mb{p}_0$, let $\tilde{\mb{r}}_0$ and $\mb{r}_0$ be points on $\mb{s}_1\mb{p}_0$ and $\mb{s}_0\mb{p}_0$
respectively such that $\mb{r}_0\tilde{\mb{r}}_0$ is parallel to $\mb{s}_0\mb{s}_1$. 
If $\mb{r}$ is any
point on $\tilde{\mb{r}}_0\mb{r}_0$, then 
\[
\triangle \mb{s}_0\mb{s}_1\mb{r} = \triangle \mb{s}_0\mb{p}_0\mb{s}_1
- \sqrt{\triangle \mb{r}_0\mb{p}_0\tilde{\mb{r}}_0\,\triangle \mb{s}_0\mb{p}_0\mb{s}_1}.
\]
\end{lemma}
\begin{proof}
\index[n]{r@$r$, real number}
By an affine transformation, we may assume the angle at $\mb{s}_1$ is a
right angle, $\mb{s}_0\mb{s}_1=1,~\mb{s}_1\mb{p}_0=1,~\mb{s}_1\tilde{\mb{r}}_0=r$ where $r\in(0,1)$. Then
the identity to be proved is
\[
\frac{r}{2} =  \frac{1-\sqrt{(1-r)^2}}{2},
\]
which is immediate. 
\end{proof}

We now prove a lower bound on  area. Let us denote by 
\[
\index[n]{h@$\h$, upper-half plane!$\h_i$, subset of $\hstar$}
\h_i
:= \left\{z \in \hstar \ | \ T_i(z) \ge T^{ext}_{i}(z)\right\},\quad
i=0,1,2.
\] 
Here and below, we consider the indices modulo $6$, but $\h_i$ has
period three: $\h_i = \h_{i+3}$.  
We derive a lower bound
for all convex disks $K$ having $H_K(z)$ as a minimal midpoint
hexagon, where $z$ belongs to the regions $\h_0, \h_1$ or $\h_2$. For
$i=0,1,2$ define
\index[n]{area@$\mathrm{area}$!$\mathrm{area}_i$, area of a triangle}
\[ 
\mathrm{area}_i(z) :=  T_i(z) - \sqrt{T^{ext}_{i}(z) \, T_i(z)}.
\]
\index[n]{I@$\I_i$, indicator function}
Let $\I_{i}:\h\to\{0,1\}$ be the indicator function of the set $\h_i$. 

\begin{theorem}
If $K \in\Kbal$ is in circle representation and has $H_K(z)$ as a
critical hexagon, where $z\in\hstar$,
then we have that
\[
\op{area}(K) \ge \frac{3\sqrt{3}}{2} + 2 \ \sum_{i=0}^2\I_i(z)\mathrm{area}_i(z).
\]
\end{theorem}
\begin{proof}
As in the Figure \ref{fig:compactification-result} above, let
$\mb{p}_0\mb{p}_1\mb{p}_2\mb{p}_3\mb{p}_4\mb{p}_5$ be the critical hexagon of an undepicted convex
disk $K$. The convex disk $K$ is inscribed in this hexagon and
passes through the points $\mb{s}_j^*=\mb{s}_j$ which are midpoints of its sides.

We may assume that $z\in\hstar$ lies in the set
\[
\h_0 \cup \h_1 \cup \h_2 = \{z \in \hstar \mid \
T_0(z) \ge T^{ext}_0(z) \ \mathrm{ or } \ T_1(z) \ge T^{ext}_1(z) \ \mathrm{ or
} \ T_2(z) \ge T^{ext}_2(z) \}.
\]
Otherwise, the inequality to be shown reduces to $\op{area}(K)\ge \op{area}(h_K) =
3\sqrt{3}/2$, where $h_K$ is the convex hull of the points $\mb{s}_i$.
This area inequality holds because $K\supset h_K$.  (This area inequality
appears in Reinhardt's 1934 article and was used in his proof of the
existence of a minimizer.)

We show that the set $\h_0 \cap \h_1 \cap \h_2$ is
empty.  By the Cauchy-Schwarz inequality and 
the area formulas of Lemma~\ref{lem:triangle-areas},
we have
\[
\sum_{i=1}^3 T_i = 
\frac{\sqrt3}{4} (\rho_0\rho_2 + \rho_0\rho_1 + \rho_1\rho_2)
\le
\frac{\sqrt3}{4} (\rho_0^2+\rho_1^2+\rho_2^2)
=\frac{1}{4} \sum_{i=1}^3 T^{ext}_i < \sum_{i=1}^3 T^{ext}_i.
\]
This shows we cannot have $T_i(z) \ge T^{ext}_i(z)$ for all
$i=0,1,2$. 
So the inequalities defining $\h_i$ must hold individually or
pairwise. The regions $\h_i$ and the other data have a three-fold
symmetry given by shifting indices $i$ modulo $3$.  This gives us two
cases.

\textbf{Case 1}: 
Without loss of generality, by symmetry, assume that $z \in \h_0$ and
$z \notin \h_1 \cup \h_2$. That means $T_0(z) \ge T^{ext}_0(z)$. 
In Figure \ref{fig:compactification-result} above, this gives an
inequality between areas 
\( 
\triangle\mb{s}_0\mb{p}_0\mb{s}_1\ge
\triangle\mb{p}_1\mb{q}_1\mb{p}_2 
\).  
Construct a triangle
$\triangle\mb{r}_0\mb{p}_0\tilde{\mb{r}}_0$ such that 
\(
\op{area}(\triangle\mb{r}_0\mb{p}_0\tilde{\mb{r}}_0)=
\op{area}(\triangle\mb{p}_1\mb{q}_1\mb{p}_2) 
\)
and such that the line segment $\mb{r}_0\tilde{\mb{r}}_0$ is parallel to
$\mb{s}_0\mb{s}_1$. The triangles with equal area are shown in pink.

We claim that there is at least one point $\mb{r}$ on the line segment
$\mb{r}_0\tilde{\mb{r}}_0$ which also lies on the boundary of the
undepicted convex disk $K$. Otherwise, if there were no such point,
then the convex disk $K$ would be contained in the centrally symmetric
hexagon with vertices $\mb{r}_0,\tilde{\mb{r}}_0,\mb{q}_1$, and their
reflections.  This hexagon has the same area as the hexagon of
$\mb{p}_0\mb{p}_1\mb{p}_2\mb{p}_3\mb{p}_4\mb{p}_5$, which has minimal
area.  We reach a contradiction by constructing a centrally symmetric
hexagon containing $K$ of even smaller area: make an inward parallel shift
of the line through the edge $\mb{r}_0\tilde{\mb{r}}_0$ (and its
reflection) until it meets $K$.

The above argument exhibits the point $\mb{r}$ on
$\mb{r}_0\tilde{\mb{r}}_0$ and its reflection $-\mb{r}$ on the
reflected edge respectively, which are also on the boundary of the
convex disk $K$. Since $K$ is convex, it contains the convex hull
$H$\index[n]{H@$H$, Euclidean region!convex hull} of the points
$\mb{s}_0,\mb{s}_1,\mb{s}_2,\mb{r}$ and their reflections.  Thus, we
have
\begin{align*}
\op{area}(K) \ge& \op{area}(H) 
\\
=&  \op{area}(\mb{s}_0\mb{s}_1\mb{s}_2\mb{s}_3\mb{s}_4\mb{s}_5) + 2\triangle \mb{s}_0\mb{r}\mb{s}_1
\\
=& \frac{3\sqrt{3}}{2} + 2\triangle \mb{s}_0\mb{r}\mb{s}_1 
\\
=& \frac{3\sqrt{3}}{2} + 2\left(T_0 - \sqrt{T^{ext}_0 T_0}\right)\qquad 
\text{(using~Lemma~\ref{lem:geom-thales}).}
\end{align*}

\textbf{Case 2}: Assume without loss of generality that
$z\in\h_0\cap\h_1$. We have $T_0(z)\ge T^{ext}_0(z)$ and
$T_1(z)\ge{}T^{ext}_1(z)$. The above argument can also be adapted here
again to exhibit four points (two new points inside the triangles
$\triangle \mb{s}_0\mb{p}_0\mb{s}_1$ and
$\triangle\mb{s}_1\mb{p}_1\mb{s}_2$ respectively, along with their
reflections) also on the boundary of the convex disk $K$, so that we
have
\[
\op{area}(K) \ge \frac{3\sqrt{3}}{2} + 
2\left(T_0 - \sqrt{T^{ext}_0\, T_0}\right) +  2\left(T_1 - \sqrt{T^{ext}_1\, T_1}\right).
\]
Accounting for all cases, we have
\[ 
\op{area}(K) \ge \frac{3\sqrt{3}}{2} + 2 \ \sum_{z \in \h_i} \I_i(z)\mathrm{area}_i(z).
\]
\end{proof}

The regions $\h_0,\h_1,\h_2$ are shown in 
Figure \ref{fig:union-contains-boundary}.
The boundaries of $\h_0$ and $\h_2$ meet along the
imaginary axis at $(x,y)=(0,\sqrt3)$.
\begin{figure}[htbp]
\centering \includegraphics[scale=0.5]{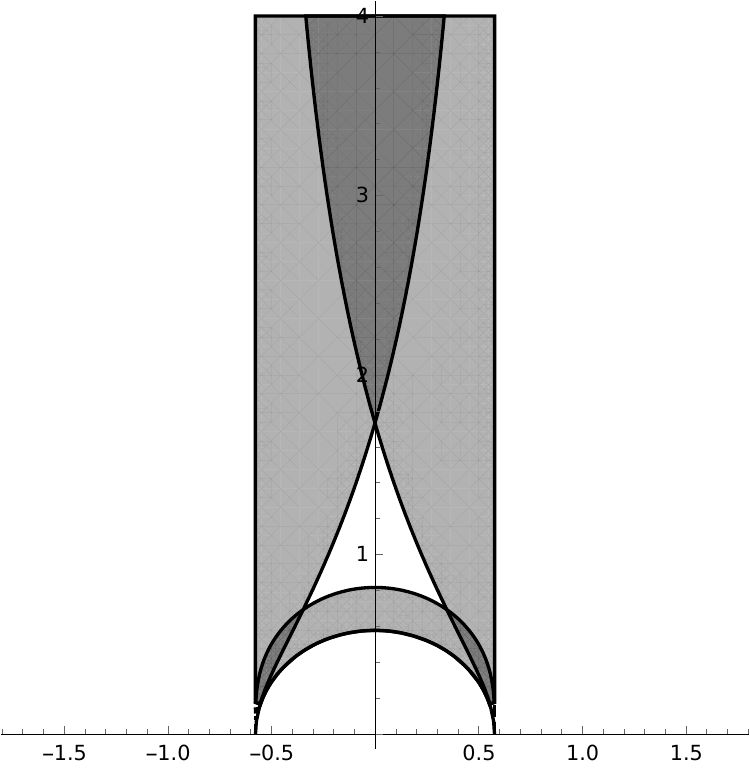} 
\caption{Covered
    boundary of the star domain.  The shaded region along the right edge is
    $\h_0$.  The shaded region along the left edge is $\h_2$, and the shaded
    region along the lower bounding circle is $\h_1$.}
\label{fig:union-contains-boundary}
\end{figure}

Recall that $\op{area}(H_K(z))=\sqrt{12}$.  
Let us write
\index[n]{zd@$\delta$, density!$\delta(z)$ density bound}
\index[n]{zd@$\delta$, density!$\delta_{oct}$, of smoothed octagon}
\[
\delta(z)
:= \frac{1}{\sqrt{12}}\left(\frac{3\sqrt{3}}{2} +
2 \ \sum_{z \in \h_i} \I_i(z)\mathrm{area}_i(z) \right) = \frac{3}{4}
+ \sum_{z \in \h_i} \frac{\I_i(z)\mathrm{area}_i(z)}{\sqrt{3}},
\] 
for the lower bound of the packing density for convex disks $K \in
\Kbal$ having $H_K(z)$ as a critical hexagon.  Also, let
$\delta_{oct}$ denote the packing density of the smoothed octagon.  If
$z \in \hstar$ is such that $\delta(z) > \delta_{oct}$, then by the
above theorem, the density of every convex disk $K$ having $H_K(z)$ as
a critical hexagon is greater than $\delta_{oct}$ and is not a global
minimizer.  Such $K$ can be dropped from consideration. The next
result shows that all but a compact subset of $\hstar$ can be excluded
in this way.  In the next lemma, we write $\I_2,\mathrm{area}_i$, and
$\delta$ as functions of $(x,y)$ instead of $z=x+iy$.

\begin{lemma}\label{lem:compactification-lemma}\leavevmode
\begin{enumerate}
\item 
For all $y\ge 1$, the function $\mathrm{area}_0(x,y)$ is
monotonically increasing in $x$ on $0<x<1/\sqrt{3}$.
\item 
For all $y\ge \sqrt3$, the function
$\mathrm{area}_0(x,y)+\mathrm{area}_2(x,y)$ is monotonically
increasing in $x$ on $0<x<1/\sqrt{3}$.
\item 
For all fixed $y\ge\sqrt3$, the function
$\I_0(x,y)\mathrm{area}_0(x,y)+\I_2(x,y)\mathrm{area}_2(x,y)$ is
minimized at $x=0$ on the domain $x\in(-1/\sqrt{3},1/\sqrt{3})$.
\item For all $x\in(-1/\sqrt3,1/\sqrt3)$ and all $y\ge 4.5$,
we have
\[
\delta(x,y)>\delta_{oct}.
\]
\end{enumerate}
\end{lemma}
\begin{proof}

\index[n]{0@$-_x$, partial derivative}
The functions $\rho_j$ are positive on $\hstar$.
Write $(\cdot)_x$ for the partial derivative with respect to $x$.
Then, $\rho_{0,x} >0$ and $\rho_{1,x}<0$ on $\hstar$.
The sign of $\rho_{2,x}=\sqrt3x/y$ is the same as the sign of $x$.

(1) 
We compute the partial derivative with respect to $x$ of
\[
\mathrm{area}_0(x,y) = \frac{\sqrt3}{4}
\left(\rho_0\rho_2 - 2\sqrt{\rho_0\rho_1^2\rho_2}\right).
\]
The partial derivative is positive on the given domain $y\ge 1$ and
$0<x<1/\sqrt{3}$ because $\rho_{0,x}\rho_{2} >0$, $\rho_0\rho_{2,x}
>0$, and $-(\rho_0\rho_1^2\rho_2)_x>0$, the final inequality being a
polynomial inequality in $x$ and $y$, which is easily checked.

%

(2) 
We consider the domain $\hstar\cap\{x>0,\  y\ge\sqrt3\}$.
We compute the partial derivative
with respect to $x$ of 
\[
\mathrm{area}_0(x,y) + \mathrm{area}_2(x,y) = 
\frac{\sqrt3}{4}
\left(
(\rho_0+\rho_1)\rho_2 - 
2\sqrt{\rho_0\rho_1^2\rho_2}-2\sqrt{\rho_0^2\rho_1\rho_2}
\right)
\] 
and show that this partial derivative is positive.
The first term $(\rho_0+\rho_1)\rho_2$ in the numerator 
has positive partial derivative $2x/y^2>0$
on the given domain.  It remains to show that
\[
0 > \frac{(\rho_0\rho_1^2\rho_2)_x}{\sqrt{\rho_1}}+
\frac{(\rho_0^2\rho_1\rho_2)_x}{\sqrt{\rho_0}}.
\]
The two polynomial numerators on the right are separately negative
when $x\ge1/4$.  When $0\le x\le 1/4$, the two terms on the right are
separately decreasing functions in $x$, and their sum is $0$ at $x=0$.
(This too is a polynomial verification.)  These routine checks prove
the result.

(3) 
Along the boundary of $\h_i$, where $\I_i$ jumps,
we have $\mathrm{area}_i(x,y)=0$.  Thus, 
\begin{equation}\label{eqn:area0-area2}
\sum_i \I_i (x,y) \mathrm{area}_i(x,y)
\end{equation}
is continuous on $\hstar$.  
Since it is monotonic increasing in
$x\in(0,1/\sqrt3)$ on each of $\h_0\setminus \h_2$ and $\h_0\cap\h_2$,
it is also monotonic increasing in $x$ on $\h_0\cap \{x>0,\ y\ge\sqrt3\}$.

Note that $\mathrm{area}_2(x,y)= \mathrm{area}_0(-x,y)$,
because $\rho_0(-x,y)=\rho_1(x,y)$.
Since the function \eqref{eqn:area0-area2} is even, it must
be monotonic decreasing in $x$ on $\h_0\cap\{x<0,\ y\ge\sqrt3\}$.
Thus, the critical point at $x=0$ is a minimum.  

(4) Along the imaginary axis, for $y\ge\sqrt{3}$, we have
\[
\mathrm{area}_0(0,y)=\mathrm{area}_2(0,y) = 
\frac{\sqrt3}{4}
\left(
\sqrt{\rho_0\rho_2} \left({\sqrt{\rho_0\rho_2} - 2\rho_1}\right).
\right)
\]
Both $\rho_0\rho_2 =(y^2 - 1/3)/(2y^2)$ and
$\sqrt{\rho_0\rho_2} - 2\rho_1$ are increasing functions of $y$.
Hence $\mathrm{area}_0(0,y)$ is increasing.
If $x\in (-1/\sqrt3,1/\sqrt3)$ and $y\ge 4.5$,
\[
\delta(x,y)\ge \delta(0,y)\ge
\delta(0,4.5) = 0.9059\ldots > \delta_{oct}.
\]
\end{proof}

\begin{proof}[Proof of Theorem~\ref{th:compactification}]
Everything in this subsection up until this point has been
equivariant with respect to the action of the dihedral group on
$\hstar$.   The region $\h^{\star\star}$ described in the statement of the
theorem is likewise stable under the action of the dihedral group.
Thus, it is enough to prove the theorem for all points $z$ in a
fundamental domain for the action of the dihedral group on $\hstar$.
One such fundamental domain is given in subsection~\ref{sec:dihedral}
as
\[
\{z = x + i y \in \hstar \mid x\ge 0,\ (x-1/\sqrt3)^2 + y^2 = 4/3\}.
\]
We work on the slightly larger subset of $\hstar$ defined by the
inequalities $x\ge0$ and $y\ge1$.  The only horoball (among the three)
meeting this set is $B(i\infty)$, and the only half-plane
meeting this set is $\Pi^+_0$.

If any point of the trajectory $(g,X)$ passes through a point $X =
\Phi(z)$ with $z = x + i y$, where $y>4.5$, then the previous lemma shows
that $\delta(x,y) > \delta_{oct}$.  This shows that we may
exclude all trajectories that enter the horoball $B(i\infty)$.  

Next we show that we can exclude all trajectories that meet the
half-plane $\Pi^+_0$ defined by $y > 15(1/\sqrt3-x)$.  
We show that
\index[n]{zd@$\delta$, density!$\delta_0(z)$, density bound}
\[
\delta(z)\ge \delta_0(z):=
\frac{3}{4} + \frac{\mathrm{area}_0(z)}{\sqrt3} > \delta_{oct},
\]
for all $z\in\hstar\cap \Pi^+_0\cap \{1\le y\le 4.5\}\subset \h_0$.
Proving this inequality will complete the proof of the theorem, because
the regions $\Pi^+_1$ and $\Pi^+_2$ do not meet the fundamental domain.

By the first part of the lemma, we have monotonicity of $\mathrm{area}_0$ in
$x$ along each horizontal slice $y=y_0$ such that $y\ge 1$.
Thus, it is enough to prove the inequality along the graph of the affine map
$y(x)=15(1/\sqrt3-x)$.
Furthermore, we may assume that $y\le 4.5$.
The preimage of $[1,4.5]$ under the map $y$ is contained in $[0.277,0.511]$.
On this domain, we have $\delta_0(x,y(x))>\delta_{oct}$.
\end{proof}

\newpage
\chapter{Hamiltonian and Maximum Principle}

\section{Existence of Optimal Control}\label{sec:existence-opt-control}
The existence of optimal solutions to problems such as the Reinhardt
control problem (Problem \ref{pbm:reinhardt-optimal-control}) is based on
Filippov's theorem which gives conditions under which the
corresponding \emph{attainable set} of the control system in question
is compact. In informal terms, the attainable set corresponds to all
the points in the manifold which are reachable provided one is only
allowed to move according to the control. The compactness of attainable
sets implies the existence of optimal control. The optimal control
function is a \emph{measurable function}.

\index{measurable}
\index{Filippov's theorem}
\index[n]{q@$q$, point on manifold}
\index[n]{zv@$\phi$, cost integrand}
\index[n]{m@$m$, integer dimension}
\index[n]{f@$f$, vector field}
\index{attainable set}
\begin{theorem}[Filippov]
    On a smooth manifold $M$, let $q' = f(q,u)$ be an optimal control
    system with an associated cost objective $\int_0^t \phi(q,u)
    dt \rightarrow \text{\emph{min}}$. Here $u \in U \subset \R^m$
    which is a compact set (the control set).  Assume that the velocity set
    $f(q,U) = \{ f(q,u)\mid u\in U\}$ is convex for each $q \in M$ and
    that the support of $f$ is a compact subset of $M\times U$.  Then,
    the \emph{attainable sets} are compact and an optimal control exists.
\end{theorem}

For the Reinhardt control problem, all assumptions of the Filippov
theorem are shown to hold as follows.
\begin{itemize}
\item 
The control set is the simplex $U_T$ defined in
Definition \ref{def:triang-control-set}, which is obviously compact.
\item 
The velocity sets are convex; in fact, each velocity set is the convex
hull of the velocities at the three vertices the control set, as shown
in Lemma~\ref{lem:control-convex} for the simplex $U_T$.
\item 
Recall that Reinhardt has proved the existence of a optimal solution
to the Reinhardt problem.  We claim that the velocity sets can be
assumed to be compactly supported.  By Definition~\ref{def:X-defn} and
Lemma~\ref{lem:X-dynamics}, the relevant vectors fields are 
\begin{equation}\label{eqn:vf}
[Z_u,\Phi(z)]/\bracks{Z_u}{\Phi(z)}
\quad\text{and}\quad\ g\Phi(z),\quad z\in\hstar.
\end{equation}
By Theorem~\ref{th:compactification}, we may assume that $z$ lies in
the compact set $\truncstar$.  Then $\Phi(z)$ is also confined to a
compact set, as well as the first vector field in \eqref{eqn:vf}.  The
second vector field $g\Phi(z)$ lies in a compact set if $g\in\SL(\R)$
can be shown to be bounded. We have an upper bound $t_f\le\pi/3$ on
the terminal time by Lemma~\ref{lem:tf-pi3}.  Gronwall's inequality
applied to the ODE $g'=gX$ gives a bound on $g$
(Appendix~\ref{sec:gronwall}).  This shows that the vector fields are
confined to a compact subset.

\index{Gronwall inequality}

To make the vector field smooth,
we can multiply the vector fields in the star domain in
the Reinhardt control problem \ref{pbm:plane-optimal-control-problem}
by a smooth cutoff function of compact support $\psi : \hstar \to \R$
with $\psi|_{\truncstar} = 1$, to obtain smooth vector fields of compact
support.  
\end{itemize}

We can apply Filippov's theorem to these smoothed vector fields of
compact support.  Thus, the optimal control exists for the optimal
control system.  The proof used the following lemma.

\index[n]{zy@$\psi$, local auxiliary function or integral!$\psi:\hstar\to\R$, cutoff function}

\begin{lemma}\label{lem:control-convex}
Let $f(x,y;u)=(f_1(x,y;u),f_2(x,y;u))$ be the control-dependent vector
field on $\hstar$ defined by the half-plane control problem
\ref{pbm:plane-optimal-control-problem}.  For each $z=x+iy\in\hstar$,
the image $f(x,y;U_T)$ of $U_T$ in $\R^2$ is a convex set.  Moreover,
the affine
hull of the image $f(x,y;U_T)$ is all of $\R^2$.
\end{lemma}

\begin{proof}
Let $\rho_j$ be the positive functions of $z\in\hstar$ from
Lemma~\ref{lem:triangle-areas}.  Fix $z=x+i y\in \hstar$ and let
$\mb{e}_1,\mb{e}_2,\mb{e}_3$ be the standard basis of $\R^3$, giving
the vertices of $U_T$.  We prove that the image $f(x,y;U_T)$ is in
fact the convex hull of
\index[n]{e@$\mb{e}_i$, standard basis!extreme point of the control set}
\[
\{ f(x,y;\mb{e}_j)\mid j=1,2,3\}.
\]
By explicit calculation, the basis vectors map to distinct points in
the velocity set.  For example, \index[n]{L@$L$, affine function}
\[
\mcite{MCA:4087162}
f(x,y,\mb{e}_3)-f(x,y,\mb{e}_2) = \left(0,\frac{2}{\sqrt{3}\rho_0\rho_1}\right)
\ne (0,0),
\]
which is finite and nonzero by the star inequalities.  Let
$L:\R^2\to \R$ be the nonzero affine function that vanishes at
$f(x,y;\ep_i)$ and $f(x,y;\ep_j)\in\R^2$ and takes value $1$ at
$f(x,y;\ep_k)$, where $(i,j,k)$ is any chosen permutation of
$(1,2,3)$.  Computing, we find that $L(f(x,y;u))$ is the ratio of two
affine functions of $U_T$ (depending on the parameters $x,y$), where
the denominator is positive on $U_T$.  For example, if
$u=(u_0,u_1,u_2)$ and $(i,j,k)=(1,2,3)$, we compute that
\[
\mcite{MCA:8165998}
L(f(x,y;u)) =
\frac{u_2 \rho_0(z)}{u_0\rho_2(z) + u_1\rho_1(z) + u_2\rho_0(z)}.
\]
We observe that the numerator of $L$
vanishes along the segment $[\ep_1,\ep_2]\subset U_T$ and that the
numerator is nonnegative on $U_T$ (in fact strictly positive, except
on the edge segment).  Similarly, for each permutation $(i,j,k)$ of
the vertices of $U_T$, the corresponding line $L=0$ defines a boundary
segment of $f(x,y;U_T)$.  We conclude that the image $f(x,y;U_T)$ is the
convex hull of three points as claimed.  Since the image is a triangle,
its affine hull is all of $\R^2$.
\end{proof}


\section{The Pontryagin Maximum Principle}\label{sec:PMP}


\index{PMP@Pontryagin Maximum Principle}
\index{PMP, Pontryagin Maximum Principle}
The Pontryagin Maximum Principle (PMP) is a powerful first-order
necessary condition for optimality of solutions to an optimal control
problem on a smooth manifold $M$ with \emph{closed} control set $U
\subseteq \R^m$ and free-terminal time. We summarize the basic ideas
below.  For full details and a proof of the PMP, we refer to
\cite{zelikin2004control}.

Given a control system
\begin{align}
q' &= f(q,u) \in T_qM 
\\
q(0) &= q_0 \in M
\end{align}
on a manifold $M$ with an associated cost objective
\[
\min_{u \in U}\int_0^{t_f} \phi(q,u) dt,
\]
we assume that the control-dependent
vector field $f(q,u)$ satisfies
\begin{enumerate}
\item 
$q \mapsto f(q,u)$ is a smooth vector field on $M$ for any fixed $u\in
U$.
\item 
$(q,u) \mapsto f(q,u)$ is a continuous mapping for $q \in M,
u \in U$,
\end{enumerate}
\index{Hamiltonian!control-dependent}
\index{Hamiltonian!cost-extended}
\index[n]{H@$\H$, Hamiltonian}
\index[n]{0@$\bracks{-}{-}_*$, canonical pairing}
\index[n]{zL@$\Lambda$, costate!$\lambda_{cost}$, cost multiplier}
\index[n]{p@$p$, point in bundle!vector in cotangent space}
This optimal control system is denoted by the tuple $(M,U,f,\phi)$. We
sometimes denote the vector field $f(u,q)$ for $u\in U$ by
$f_u(q)$. The PMP relies on the following \emph{control-dependent
Hamiltonian} which is \emph{cost-extended} on $T^*M$:
\[
\H(q,p,u) = \langle p,f(q,u) \rangle_{*} + \lambda_{cost} \phi(q,u) \qquad 
p \in T^*_qM \quad \lambda_{cost} \in \R_{\le 0},
\]
%
where $\bracks{\cdot}{\cdot}_{*}$ is the natural pairing between a
vector space and its dual and
$\lambda_{cost}$, Pontryagin multiplier
is a constant nonpositive scalar. Note that $\H$ is linear in $p$.

\index[n]{u@$u$, control!$u^*$, optimal}
\index[n]{H@$\H$, Hamiltonian!$\Hstar$, maximized}
\index{singular subarc}
\index{Hamiltonian!maximized}

Let $u^*=u^*(p,q)$ denote a function defined implicitly as a function
of $(q,p) \in T^*M$ by
\begin{equation}\label{eqn:maximized-hamiltonian}
     \Hstar(q,p):=
     \H(q,p,u^*) = \max_{u\in U}\H(q,p,u).
\end{equation}
$\Hstar$ is called the \emph{maximized Hamiltonian}.  Note that it is
quite possible that $u^*$ might not be uniquely determined by the
maximization condition. This is related to \emph{singular} subarcs,
discussed later. See
Definition~\ref{def:normal-abonormal-singular-extremals}.  Regardless,
the value of the maximized Hamiltonian $\Hstar$ is independent of the
choice of $u^*$.

The PMP says that the
extremals of the optimal control problem are projections (from $T^*M$
to $M$) of the flow-trajectory of $\vec{\H}^\bstar$, which is
the Hamiltonian vector field corresponding to $\Hstar$ with respect to
the canonical symplectic structure on the cotangent bundle
$T^*M$. Integral curves of the vector field $\vec{\H}^\bstar$
satisfy
\begin{equation}\label{eqn:hamiltons-equations}
q' = \frac{\partial \Hstar}{\partial p} = f(x,u^*), \qquad 
p' = -\frac{\partial \Hstar}{\partial q}.
\end{equation}
This trajectory $(u^*(t),q(t),p(t))$ in $T^*M$ is called
the \emph{lifted controlled trajectory}.

The PMP (for free terminal time periodic problems) also guarantees the
following of lifted trajectories.
\begin{enumerate}
\item Transversality conditions (endpoint conditions for the
    co-state trajectories) hold.  
\item The Hamiltonian $\H(q,p,u^*)$
    vanishes identically along the lifted controlled
    trajectory.  
\item 
The cotangent vector $(\lambda_{cost},p(t))\in\R_{\le0}\times{T^*_{q(t)}M}$ is
    nonzero for all $t \in [0,t_f]$.
\item The scalar
    $\lambda_{cost}$ is constant, and when it is non-zero, may be
    taken to be $\lambda_{cost} = -1$ by rescaling the covector, using
    the linearity of the ODE for $p$.
\end{enumerate} 

\index{Pontryagin extremal}
\index{extremal}
The lifted curves which satisfy the conditions of the PMP are
called \emph{Pontryagin extremals} or simply \emph{extremals}.

\index{abnormal extremal}
\index{normal extremal}
\index{singular subarc}
\begin{definition}[Normal, Abnormal and Singular extremals]
\label{def:normal-abonormal-singular-extremals}\leavevmode
\begin{enumerate}
\item 
An extremal for which $\lambda_{cost} = 0$ is called an \emph{abnormal
    extremal}.  
\item 
An extremal for which $\lambda_{cost} \ne 0$ is
    called a \emph{normal extremal}. Recall that in the normal case a
    renormalization allows us to take $\lambda_{cost} = -1$.
\item 
If there is an open time interval on which
equation \eqref{eqn:maximized-hamiltonian} fails to uniquely determine
the function $u^*(t)$, the trajectory during that interval is called a \emph{singular
subarc}.
\end{enumerate}
\end{definition}

Our strategy is to apply the maximum principle to our problem with the
hope that these necessary conditions will provide us more information
about the structure of the extremals. Recall the framework of the
Reinhardt control problem \ref{pbm:reinhardt-optimal-control}, where
we have dynamics occurring in the Lie group $\SL(\R)$ and the Lie
algebra $\sl(\R)$. If we apply the PMP to this problem, the lifted
trajectories live in
\[T^*(\SL(\R) \times \sl(\R)) 
\cong \left(\SL(\R) \times \sl(\R)\right)\times 
\left (\sl(\R) \times \sl(\R)\right),
\]
where we have used the cotangent bundle trivialization
of Proposition \ref{prop:trivialization} making
$T^*\SL(\R) \cong \SL(\R)\times \sl(\R)^*$ and the identification
$\sl(\R)^* \cong \sl(\R)$ via the nondegenerate trace form as in
Appendix \ref{sec:X-lie-poisson}. So, for the state variables $q =
(g,X) \in \SL(\R) \times \sl(\R)$
the PMP gives corresponding costate variables $p=
(\Lambda_1,\Lambda_2) \in \sl(\R) \times \sl(\R)$.

\index[n]{g@$g$, group element!$(g,X)$, state variables}
\index[n]{zL@$\Lambda$, costate!$\Lambda_i\in\sl$}

Note that the PMP system is a Hamiltonian system on $T^*T\SL(\R)$ and is 
an instance of a \textit{higher-order variational system on a Lie group}. 
Similar problems and the background theory is described in Gay-Balmaz~et.~al.~\cite{gay2012invariant} 
and Colombo~and~de Deigo~\cite{colombo2014higher}. 

We now derive expressions for the Hamiltonian and the costate
equations in both the Lie algebra coordinates and and the upper
half-plane coordinates via the isomorphism described in
Lemma \ref{lem:def-phi}.

\section{Left-invariance}


\index[n]{h@$h\in G$, element of Lie group}

\begin{definition}[Jurdjevic~\cite{jurdjevic_1996}]\label{def:left-invariant-control-system}
An arbitrary optimal problem with control system $dg/dt = f(g,u)$
defined on a real Lie group $G$ with control functions $u(t) \in
U \subseteq \R^m$ is said to be \emph{left-invariant} if
$TL_hf(g,u)=f(hg,u)$ for each $g,h \in G$. Here $L_h(g) = hg$ is the
left-multiplication map, and $TL_h:T_gG \to T_{hg}G$ is its tangent map.

Also, we require the associated cost function $\phi(g,u)$ to be
left-invariant: $\phi(g,u) = \phi(e,u)$ for all $g$ and $u$, where
$e\in G$ is the neutral element.
\end{definition}

The dynamical system breaks into the 
ordinary differential equation (ODE) for the group \eqref{eqn:trisystem-g}
and the ODE for the Lie algebra \eqref{eqn:trisystem-X}. 
We refer to these two subsystems as the dynamics \emph{at the Lie group level}
and the dynamics \emph{at the Lie algebra level}.  The dynamics are
coupled through $X$, which appears in both levels.
\index{level of ODE subsystem}

The dynamics of the Reinhardt optimal problem at the Lie group level
is clearly left-invariant in the sense that the cost function depends
on $X$ but not $g$ and in the sense that the ODE for $g$ can be
left-multiplied by a constant $h\in\SL(\R)$.
Left-invariance of a dynamical system on a Lie group implies that we can
reduce its dynamics to co-adjoint orbits of the associated Lie
algebra.

Since the cost and dynamics for $X$ are independent of $g$, we note
that the only purpose served by the Lie group dynamics for $g$ is to
describe an endpoint (transversality) condition $g(t_f)=R$.  Because
of the minor purpose served by the Lie group dynamics, we can often
focus on the control problem exclusively at the Lie algebra level, and
postpone the endpoint condition on $g$ until the very last step. In
later sections, we will drop the group dynamics and exclusively focus
on state/costate dynamics in the Lie algebra.

\section{Hamiltonian in the Lie Algebra}\label{sec:max-ham}

Following the Reinhardt optimal control
problem \ref{pbm:reinhardt-optimal-control}, the Hamiltonian 
is the sum of the Hamiltonians for the Lie group part and the Lie
algebra part.

The costate variable corresponding to the Lie group element
$g \in \SL(\R)$ is denoted $\Lambda_1 \in \sl(\R)$. Ignoring the Lie
algebra dynamics for a moment, we derive the Hamiltonian corresponding
to the group element $g$. As pointed out earlier, the control problem
is left-invariant (see
Definition \ref{def:left-invariant-control-system}) and Hamiltonians
of left-invariant systems are functions of $\sl(\R)^*$ only.

\begin{proposition}[Jurdjevic~\cite{jurdjevic_1996}]\label{prop:left-invariant-hams}
Consider an arbitrary left-invariant control system $dg/dt = f(g,u)$
on a real Lie group $G$, with control $u \in U \subseteq \R^m$.
Let us also assume that the Lie algebra
$\mathfrak{g}$ of $G$ is equipped with a nondegenerate invariant symmetric
bilinear form
$\bracks{\cdot}{\cdot}$. Then the Hamiltonian function corresponding
to this system is
\[
\H(g,p) = \bracks{p}{f(e,u)}, \quad p \in \mathfrak{g},
\]
where $e \in G$ is the group identity.
\end{proposition}
\begin{proof}
\index[n]{p@$p$, point in bundle!${\tilde p}\in T^*_gG$, cotangent vector}
Let 
$\bracks{\cdot}{\cdot}_*$ be the canonical pairing between a vector
space and its dual.  Let ${\tilde p}\in T^*_gG$.
Using the trivialization
$T^*G \cong G \times \mathfrak{g}^*$ of 
Proposition \ref{prop:trivialization}, 
we write ${\tilde p} =
T^*L_{g^{-1}}(p)$ for some 
$p \in \mathfrak{g}^*$. Then we have by the definition of the Hamiltonian
for a control system,
\begin{align*}
\H(g,{\tilde p}) &=\bracks{{\tilde p}}{f(g,u)}_*\\
&= \bracks{p}{TL_{g^{-1}}\left(f(g,u)\right)}_* 
= \bracks{p}{TL_{g^{-1}}\left(TL_gf(e,u)\right)}_* \\
&= \bracks{p}{f(e,u)}_*,
\end{align*}
since the control system is left-invariant. Since $\mathfrak{g}$ is
equipped with a nondegenerate invariant form $\bracks{\cdot}{\cdot}$,
we have that $\mathfrak{g} \cong \mathfrak{g}^*$. Using this
identification, we have $\H(g,p) = \bracks{p}{f(e,u)}$, where $p$ is
now identified with an element of $\mathfrak{g}$.
\end{proof}
We need a slight extension of this result, where the control
system has the form 
\[
dg/dt = f(g,TL_g X,u),\quad u\in U,\quad X\in \mathfrak{g}.
\]
Left-invariance is expressed as $TL_hf(g,TL_gX,u)=f(hg,TL_{hg}X,u)$,
and the corresponding Hamiltonian in the proposition
becomes $\bracks{p}{f(e,X,u)}$.

In our case, taking $p=\Lambda_1$ and $f(e,X,u)=X$,
this means that the term the Hamiltonian
corresponding to the group is
$\bracks{\Lambda_1}{X}$.
The cost extended term of the Hamiltonian is
\index[n]{H@$\H$, Hamiltonian!$\H_1$, Lie group level term}
\[
\H_1(\Lambda_1,X) := \bracks{\Lambda_1}{X} - \frac{3}{2}\lambda_{cost}\bracks{J}{X}.
\]

\index[n]{H@$\H$, Hamiltonian!$\H_2$, Lie algebra level term}
The Lie algebra part of the Hamiltonian, corresponding to the costate
variable $\Lambda_2$, is
\[
\H_2(\Lambda_2,X;Z_u) := 
\frac{\bracks{\Lambda_2}{[Z_u,X]}}{\bracks{X}{Z_u}} 
= -\frac{\bracks{[\Lambda_2,X]}{Z_u}}{\bracks{X}{Z_u}}.
\]
\index[n]{zL@$\Lambda$, costate!$\Lambda_R\in\sl$, reduced}
The form of the Hamiltonian suggests introducing a new variable
$\Lambda_R := [\Lambda_2,X]$.
The full Hamiltonian of the problem is now
\begin{align}
\H(\Lambda_1, \Lambda_R, X ; Z_u) 
&:= \H_1(\Lambda_1,X) + \H_2(\Lambda_2,X;Z_u) \nonumber 
\\
&= \left\langle \Lambda_1 - \frac{3}{2}\lambda_{cost} J, X \right\rangle 
- \frac{\bracks{[\Lambda_2,X]}{Z_u}}{\bracks{X}{Z_u}} \nonumber 
\\
&= \left\langle \Lambda_1 - \frac{3}{2}\lambda_{cost} J, X \right\rangle  
- \frac{\langle \Lambda_R,Z_u\rangle }{\langle X, Z_u \rangle}.
\label{eqn:full-hamiltonian}
\end{align}

The maximum principle states that the extremals of the control problem
are integral curves of the maximized Hamiltonian, which is
the pointwise maximum of the control-dependent Hamiltonian over the
control set. For our immediate application, we take the control set to
be the simplex $U_T$ (see Definition \ref{def:triang-control-set}).
\begin{align}
\Hstar(\Lambda_1,\Lambda_2, X) 
&:= \max_{u\in U_T}\H(\Lambda_1, \Lambda_R, X ; Z_u) \nonumber 
\\
&=\left\langle \Lambda_1 - \frac{3}{2}\lambda_{cost} J, X \right\rangle 
+ \max_{u\in U_T} \frac{ \langle -\Lambda_R , Z_u \rangle}
{\langle X,Z_u \rangle} \label{eqn:max-ham}.
\end{align}

\section{Costate Variables in Lie Algebra} \label{sec:costate-variables}

\index{costate variables}
\begin{proposition}[Lie algebra costate variables]
The costate variables evolve as
\begin{align}
\Lambda_1' &= [\Lambda_1, X] 
\\
\Lambda_R' &= \left(\left[P,\Lambda_R\right] - 
{\langle\Lambda_R,P\rangle} \left[P,X\right]\right) 
+ \left[ -\Lambda_1 + \frac{3}{2}\lambda_{cost} J ,X\right],
\label{eqn:lamR}
\end{align}
\index[n]{P@$P$, normalized control matrix}
\index[n]{zL@$\Lambda$, costate!$\Lambda_R\in\sl$, reduced}
where $P = Z_u/\bracks X {Z_u}$ and $\Lambda_R := [\Lambda_2,X]$.
\end{proposition}
\begin{proof}
Let $\Hstar$ denote the maximized Hamiltonian in equation \eqref{eqn:max-ham}. 
Let
\[(g,X,\Lambda_1,\Lambda_2) \in \SL(\R) \times \sl(\R) \times 
\sl(\R) \times \sl(\R)\cong T^*(\SL(\R)\times \sl(\R)) 
\] 
denote the state and costate variables in the trivialized bundles.
Since the Hamiltonian is independent of $\SL(\R)$, it is left-invariant. 

By the maximum principle, we have that the state and costate equations
are Hamilton's equations in the cotangent bundle with respect to an
appropriate symplectic form. For the $(g,\Lambda_1)$ pair, the
dynamics is Hamiltonian with respect to the pullback of the canonical
symplectic form on $T^*(\SL(\R))$ to $\SL(\R) \times \sl^*(\R)$. These
equations are called the Euler-Arnold equations\index{Euler-Arnold
equations} for a left-invariant Hamiltonian
(see \cite[pp.~285]{cushman1997global}) and are given by
\index[n]{ad@$\ad^*$, coadjoint representation of the Lie algebra}
\index[n]{zd@$\delta/\delta X$, functional derivative} 
\index{functional derivative}
\begin{align*}
    g'&=g\frac{\delta \Hstar}{\delta \Lambda_1} = gX 
\\
\Lambda_1' &= \ad^*_{\delta \Hstar/\delta \Lambda_1} \Lambda_1 
= -\ad_{X}\Lambda_1 = [\Lambda_1,X],
\end{align*}
where we identify $\sl(\R)^*$ with $\sl(\R)$ as usual, sending the
$\ad^*$-operator to $\ad$-operator, as described in Appendix~\ref{sec:Lie}.
%
%
%
The expression $\delta \Hstar/\delta \Lambda_1$ denotes the functional
derivative of $\Hstar$ with respect to $\Lambda_1$ and is defined in
Appendix \ref{sec:X-lie-poisson} in
equation \eqref{eqn:functional-derivative}.  For the pair
$(X,\Lambda_2) \in \sl(\R) \times \sl(\R)^*$, the dynamics is
Hamiltonian with respect to the canonical symplectic structure on the
trivial cotangent bundle $T^*(\sl(\R))$ which gives us Hamilton's
equations in the usual form.
\begin{align*}
    X' &= \frac{\delta \Hstar}{\delta \Lambda_2} =  [P,X] 
\\ 
\Lambda_2' &= -\frac{\delta \Hstar}{\delta X} 
= -\Lambda_1 + \frac{3}{2}\lambda_{cost}J - [\Lambda_2, P] 
+ \bracks {[\Lambda_2,P]} {X} P.
\end{align*}
Using this, we can derive
\begin{align*}
\Lambda_R' &= [\Lambda_2,X]' = [\Lambda_2',X] + [\Lambda_2,X']
\\ 
    &=\left[-\Lambda_1 + \frac{3}{2}\lambda_{cost}J ,X\right] 
+  \bracks {[\Lambda_2,P]} {X} [P,X] - [[\Lambda_2, P],X] + [\Lambda_2,[P,X]] 
\\ 
    &= \left[-\Lambda_1 + \frac{3}{2}\lambda_{cost}J ,X\right] 
+ \bracks {[\Lambda_2,P]} {X} [P,X] +[P,[\Lambda_2,X]] 
\\ 
    &= \left[-\Lambda_1 + \frac{3}{2}\lambda_{cost}J ,X\right] 
- \bracks {\Lambda_R} {P} [P,X] +[P,\Lambda_R] 
\\ 
    &= \left(\left[P,\Lambda_R\right] 
- \frac{\langle \Lambda_R,Z_u \rangle}{\langle X, Z_u\rangle} \left[P,X\right]\right) 
+ \left[ -\Lambda_1 + \frac{3}{2}\lambda_{cost} J ,X\right].
\end{align*}
\end{proof}

\begin{remark}\normalfont \leavevmode
  \begin{itemize}
  \item Note that the variable $\Lambda_R$ is constrained to lie in the
  two-dimensional subspace $\{A\in\sl(\R)\mid\bracks{A}{X} = 0\}$.
  Thus, we can consider $\Lambda_R$ to be the \emph{reduced}
  costate\index{reduced costate} variable.  (The subscript $R$ stands
  for \emph{reduced}.)
  \item In Appendix~\ref{sec:poisson-bracket}, we show that there is a Poisson bracket 
  with respect to which the Reinhardt control system admits a Poisson bracket representation.   
  \end{itemize}
\end{remark} 

\begin{corollary}[Jurdjevic~\cite{jurdjevic_1996,jurdjevic2016optimal}]\label{cor:lam1-gen-sol}
The costate variable $\Lambda_1$, whose dynamics is given by a Lax equation,
evolves in an adjoint orbit of $\sl(\R)$ through the initial value
$\Lambda_1(0)$ and its general solution is given by
\[
\Lambda_1(t) = \Ad_{g(t)^{-1}}(\Lambda_1(0)) = g(t)^{-1}\Lambda_1(0)g(t).
\]
\index[n]{Ad@$\Ad^*$, coadjoint representation of the Lie group}
Moreover, the determinant $\det(\Lambda_1(t))$ is a constant of motion.
\end{corollary}
\begin{proof}
This can be verified by differentiating. We immediately find the
determinant is constant. If the identification of the Lie algebra with
its dual is not made, the evolution is in a coadjoint orbit through
the representation $\Ad^*$.
\end{proof}


\begin{corollary} \label{cor:control-dep-ad-alternate}
In the ODE for $\Lambda_R$ in equation \eqref{eqn:lamR}, the control
dependent term has the following expression.
\begin{equation}
\left(\left[P,\Lambda_R\right] - \frac{\langle \Lambda_R,Z_u \rangle}{\langle X, Z_u\rangle} 
\left[P,X\right]\right) =  \ad_{Z_u} \frac{\delta}{\delta Z_u} 
\frac{\langle \Lambda_R, Z_u\rangle}{\langle X, Z_u\rangle} 
= \ad_{Z_u} \frac{\delta \H_2}{\delta Z_u},
\end{equation}
\index{functional derivative}
where $\frac{\delta}{\delta Z_u}$ denotes the functional
derivative (defined in Section~\ref{sec:X-lie-poisson}).
\end{corollary}
\begin{proof}
The proof is by computation.
\begin{align*}
\ad_{Z_u} \frac{\delta}{\delta Z_u}
\left( \frac{\langle \Lambda_R, Z_u\rangle}{\langle X, Z_u\rangle}\right) &= 
\left[Z_u,\frac{\delta}{\delta Z_u}
\left( \frac{\langle \Lambda_R, Z_u\rangle}{\langle X, Z_u\rangle}\right) \right] 
\\
&=   \left[Z_u, \frac{\Lambda_R \langle X, Z_u \rangle 
- \langle \Lambda_R, Z_u\rangle X}{ \langle X, Z_u \rangle^2} \right] 
\\
&=  \frac{\left[ Z_u, \Lambda_R \right]}{\langle X, Z_u\rangle} 
 -\frac{\langle\Lambda_R, Z_u \rangle}{\langle X, Z_u\rangle^2}\left[ Z_u, X \right] 
\\
&= [P,\Lambda_R]- \frac{\langle\Lambda_R,Z_u \rangle}{\langle X,Z_u\rangle}[P,X].
\end{align*}
\end{proof}

\section{Transversality Conditions}\label{sec:transversality}
For free terminal time optimal control problems, the Pontryagin Maximum
Principle
specifies \emph{transversality}\index{transversality} conditions which
are endpoint conditions which the extremals need to satisfy.

On a manifold $M$, if $(q(t),u(t))$ is the projection of the lifted
extremal trajectory $p(t) \in T^*_{q(t)}M$ in the cotangent bundle,
then transversality requires that
\index[n]{M@$M$, manifold!optimal control on}
\index{final submanifold}
\index[n]{p@$p$, point in bundle!lifted extremal trajectory}
\[ 
\bracks{p(t_f)}{v}_{*} = 0 \quad v \in T_{q(t_f)}M_f,
\]
where $T_{q(t_f)}M_f$ is the tangent space at $q(t_f)$ of
the \emph{final submanifold} $M_f$ (which is
the submanifold in which the terminal point $q(t_f)$ is allowed to
vary).

More generally, if the initial point $q(0)$ is also allowed to vary in
an \emph{initial submanifold}\index{initial submanifold} $M_0$, then
transversality requires that the lifted extremal trajectory $p(t)$
in the cotangent bundle annihilates the vectors in the tangent spaces
of the initial and terminal manifolds at the initial and terminal
times.
\begin{align*}
\bracks{p(0)}{\mb{v}}_{*} &= 0, \quad \mb{v} \in T_{q(0)}M_0;
\\
\bracks{p(t_f)}{\mb{v}}_{*} &= 0, \quad \mb{v} \in T_{q(t_f)}M_f.
\end{align*} 
\index[n]{v@$\mb{v}$, vector!tangent vector}

Letting $R$ be the usual rotation matrix, we have that our system is periodic up to
a rotation by $R$. In this case, the initial and terminal submanifolds
coincide after rotation by $R$ and so, transversality simply means
that the lifted extremal trajectories are periodic functions modulo
rotation by $R$. Rotations act
through the adjoint action $Y \mapsto \Ad_{R} Y = RYR^{-1}$ 
on the Lie algebra $\sl(\R)$.

For our system, we have $p(t)
= \left(g(t),X(t),\Lambda_1(t),\Lambda_R(t)\right)$. We have already
seen endpoint conditions for $g$ and $X$ in
Section \ref{sec:X-init-term-conds}.  Collecting everything we obtain
\begin{align}\label{eqn:transversality}
\begin{split}
g(t_f) &= R 
\\ 
X(t_f) &= R^{-1} X(0) R 
\\
\Lambda_1(t_f) &= R^{-1}\Lambda_1(0)R 
\\
\Lambda_R(t_f) &= R^{-1}\Lambda_R(0)R.
\end{split}
\end{align}

\begin{remark}\normalfont\leavevmode
\begin{itemize}
\item The transversality conditions in the 
Reinhardt problem require the terminal time to satisfy 
\[
    X(t_f) = R^{-1}X_0R,
\]
where $R$ is the usual rotation matrix. Thus, if we extend time, every optimal
solution can be made into a periodic one.
\[
    X(3t_f) = X_0.
\]
The same requirement also holds of $g, \Lambda_1$ and $\Lambda_R$,
except that the period for $g$ is larger: $g(6t_f)=I_2$.  Thus, every
transversal trajectory determines a periodic solution (with discrete
rotational symmetry) of the lifted trajectories in the cotangent
bundle.
\item 
The transversality condition for
$\Lambda_1$ is a consequence of the transversality condition for $g$.
In fact, we know the general solution for $\Lambda_1$ in terms of $g$.
See Corollary \ref{cor:lam1-gen-sol}.
\end{itemize}
\end{remark}

\section{Summary of State and Costate Equations}
At this point, we have the state and costate equations fully stated in
both the coadjoint orbit picture and the upper half-plane picture. We
also have the transversality conditions stated. We collect them as

\begin{problem}[State-Costate Equations]\label{pbm:state-costate-reinhardt}
The Reinhardt control problem \ref{pbm:reinhardt-optimal-control} is an optimal control
problem on the manifold $M = \SL(\R) \times \sl(\R)$. This problem has
the Hamiltonian
\[
\H(\Lambda_1,\Lambda_R,X;Z_u) = 
\left\langle \Lambda_1 - \frac{3}{2}\lambda_{cost} J, X \right\rangle  
- \frac{\langle \Lambda_R,Z_u\rangle }{\langle X, Z_u \rangle}.
\]
We lift the state trajectories described by $(g,X)$ in the Reinhardt
control problem \ref{pbm:reinhardt-optimal-control} to the following
Hamiltonian system on the cotangent bundle of $M$:
\begin{align*}
g'&=gX 
\\
X' &= [P,X] 
\\
\Lambda_1' &= [\Lambda_1, X] 
\\
\Lambda_R' &= \left(\left[P,\Lambda_R\right] - 
\langle \Lambda_R,P\rangle
\left[P,X\right]\right) 
+ \left[ -\Lambda_1 + \frac{3}{2}\lambda_{cost} J ,X\right],
\end{align*}
where $P = {Z_u}/{\langle Z_u,X\rangle}$. The transversality
conditions described in Section \ref{sec:transversality} hold and the
pair $(X,\Lambda_R)$ have equivalent dynamics in the upper half-plane
described in Lemma \ref{lem:X-dynamics}
and Theorem \ref{thm:lamR-dynamics}. The control dependent part of the
Hamiltonian also has an expression in the upper half-plane picture as
Theorem \ref{thm:plane-hamiltonian}.
\end{problem}

\section{Hamiltonian and Costate Equations in the Upper-half plane}

An earlier preprint expressed the Hamiltonian and the maximum
principle in terms of the cotangent variables
$\nu=\nu_1dx+\nu_2dy\in{T^*_z\h}$~\cite{hales2017reinhardt}.  It turns
out that it is significantly simpler to express the costate
differential equations in terms of the coordinate
$\Lambda_R\in\sl(\R)$ instead of $(\nu_1,\nu_2)$. In this book we use
$\Lambda_R$, rather than $\nu$.  However, for reasons of compatibility
with the earlier preprint, this section briefly reviews the
correspondence between the two coordinate system.  This section will
not be used elsewhere in the book.

With this aim in mind, we now proceed to transport the Hamiltonian from
Lie algebra coordinates to upper half-plane coordinates. Recall that
we have an isomorphism $\Phi : \mathfrak{h} \to \O_X \subset \sl(\R)$
defined in Lemma \ref{lem:def-phi}. This induces the tangent map
$T\Phi : T_z\h \to T_X\O_X$ as described in Lemma
\ref{lem:tangent-maps}. This also induces the dual (cotangent) map:
$T^*\Phi:T^*_X\O_X\to{T^*_z\h}$.

\begin{lemma}\label{lem:cotangent-space-OX}
\index[n]{0@$-^\perp$, orthogonal complement}
\index[n]{W@$W\in\sl(\R)$}
\[ 
     T^*_X \OX \cong X^{\perp} = \{ [W,X] \ | \ W \in \sl(\R)\},
\]
where $X^{\perp} = \{ Y \in \sl(\R) \ | \ \langle Y, X \rangle = 0 \}$.
\end{lemma}
Explicitly, the isomorphism identifies the cotangent
space with the dual of the tangent space, under the identification
of the tangent space with the quotient $\sl(\R)/\RX$ of 
Lemma~\ref{lem:half-plane-lie-algebra-iso}.
\begin{proof}
By a general fact in elementary linear algebra we have
\[ 
\index[n]{0@$-^\circ$, annihilator}
     T^*_X \OX \cong \left(\sl(\R)/\RX\right)^* \cong \RX^\circ = X^\perp,
\]
where 
\[
\RX^\circ = 
\{ Y \in \sl(\R) \ | \ \langle Y, W \rangle 
=
0 \text{ for all } W \in \RX \}
\] 
is the annihilator of the span of $X$.  Here the annihilator,
which is defined as a subspace of the dual vector space, is identified
with a subspace of the Lie algebra itself via the nondegenerate trace
form.

It is clear that any Lie algebra element of the form $[W,X]$ is
orthogonal to $X$ since $\bracks{[W,X]}{X}=\bracks{W}{[X,X]}=0$.  
Dimension counting again gives us that
$T^*_X\OX=\{[W,X]\mid \ W \in \sl(\R)\}$.
\end{proof}

\index[n]{znu@$\nu\in T^*_z\h$, cotangent variable}
\index[n]{znu@$\nu\in T^*_z\h$, cotangent variable!$\nu_1,\nu_2$, components of $\nu$}

We define 
\begin{equation}\label{eqn:nu}
\nu := T^*\Phi(-\Lambda_R) = \nu_1 dx + \nu_2 dy \in T_z^*\h.
\end{equation}
The covector $\nu \in T^*_z\h$ is well-defined since $\Lambda_R \in
X^{\perp}$ by Lemma \ref{lem:cotangent-space-OX}.  

\begin{theorem}[Hamiltonian in Upper Half-Plane]\label{thm:plane-hamiltonian}
    Let $\Phi$ be as in Lemma \ref{lem:def-phi} and let $\Lambda_R$ be
    as above.  Then we have that the Hamiltonian
    $\H(\Lambda_1,\Lambda_R,X;Z_u)$ in
    equation \eqref{eqn:full-hamiltonian}
     in upper half-plane coordinates becomes
\[
     \H(\Lambda_1,x,y,\nu_1,\nu_2;u) =  
\left\langle \Lambda_1 - \frac{3}{2}\lambda_{cost} J, X \right\rangle 
+ \nu_1 f_1 + \nu_2 f_2.
\]
\end{theorem}
\begin{proof}
The only part of
the Hamiltonian in equation \eqref{eqn:full-hamiltonian} which will
change is
$\H_2(\Lambda_2,X;Z_u)$.  
\begin{align*} 
\H_2(\Lambda_R,X;Z_u)
= \frac{\langle -\Lambda_R,Z_u \rangle}{\langle Z_u,X \rangle}
&=\left\langle -\Lambda_R, \frac{Z_u}{\langle
Z_u,X \rangle}\right\rangle \\ &= \left\langle
-\Lambda_R, \frac{Z_u}{\langle Z_u,X \rangle}
+ \RX \right\rangle \\ &= \langle -\Lambda_R,
T \Phi(f_1,f_2)\rangle \qquad \text{by \eqref{eqn:tangent-map-f1f2}} 
\\
&=\langle T^*\Phi(-\Lambda_R), (f_1,f_2)\rangle_* \\ &= \langle \nu,
(f_1,f_2) \rangle_*\\ &= \nu_1 f_1 + \nu_2 f_2,
\end{align*} 
where we
have used the definition of the cotangent map and the fact that
$\Lambda_R \in X^{\perp}$, the annihilator subspace of $X$ in
$\sl(\R)$.
\end{proof}

\begin{theorem}\label{thm:lamR-dynamics}
    The ODE for $\Lambda_R$ is transformed to the ODE
\[
\nu' = -d\H = -\frac{\partial \H}{\partial x} dx - \frac{\partial \H}{\partial y} dy,
\]
    where $\H$ is the Hamiltonian derived in equation \eqref{eqn:full-hamiltonian}. 
\end{theorem}

\begin{proof}
We note that this equation is Hamilton's equation for the
costate.  We omit the long direct calculation that this
equation is equivalent to the ODE for $\Lambda_R$.
Details can be found in~\cite{vajjha-phdthesis}.
\end{proof}

\newpage
\part{Solutions}
\chapter{Bang Bang Solutions}

We have a well-defined control problem on the cotangent bundle, and we
now turn to describing special solutions of this system. We start
with the easiest case, where the control is constant.

\section{Solutions for Constant Control}

\begin{lemma}\label{cor:X-det-const}
The quantity $\bracks{X}{Z_u}$ for a fixed control matrix $Z_u$ is a
constant of motion along $X$.
\end{lemma}
\begin{proof}
The quantity in question is constant since
\begin{equation}\label{eqn:X-Z0-constant}
\langle X, Z_u \rangle' = \langle X', Z_u \rangle 
= \frac{\langle [Z_u, X], Z_u \rangle}{\langle X,Z_u \rangle} = 0,
\end{equation}
where we have used the fact that $\langle [X,Y],Z \rangle = \langle
X,[Y,Z] \rangle$.

\index[n]{X@$X$, Lie algebra element!$X,Y,Z,W\in\sl$}

\end{proof}

\index{Riemannian metric!invariant}

\begin{lemma}\label{lem:Riemannian-speed}
Assume that the control $u\in{}U_T$ is constant.
Then the speed $\bracks{X'}{X'}^{1/2}$ of $X$ is constant.
Moreover, the trajectory $z$ in $\h$, defined by
$X = \Phi\circ z$, has constant speed with respect to the
invariant Riemannian metric on $\h$.
\end{lemma}

\begin{proof}
Let $P=Z_u/\bracks{X}{Z_u}$ with constant control $u\in U_T$.
By the previous lemma, $P$ is constant.
To show that the speed is constant, we differentiate
\begin{align*}
\bracks{X'}{X'}' &= 2\bracks{X''}{X'} 
= 2\bracks{[P,X]'}{X'}=2\bracks{[P,X']}{X'}
\\
&= 2\bracks{P}{[X',X']} = 0.
\end{align*}

The trajectory $z$ also has constant speed because of the
compatibility of the invariant metric on the upper half plane with the
trace form on $\sl(\R)$, by Lemma~\ref{lem:riemannian}.
\end{proof}

In this section, we keep the control matrix $Z_u$ constant and derive
general solutions to the state and costate equations.  This means that
$\bracks{Z_u}{X}$ is a constant of motion (by
equation \eqref{eqn:X-Z0-constant}), and hence
$P={Z_u}/{\bracks{Z_u}{X}}$ is also constant. So, for $g(0) = I_2$ and
$X(0)=X_0$ (or, equivalently $z(0)=z_0$), write
\index[n]{P@$P$, normalized control matrix!$P_0$, constant}
$P_0:={Z_u}/{\bracks{Z_u}{X_0}}$. The general solutions for $(g,X)$
are
\begin{align} 
g(t) &= \exp(t(X_0 + P_0))\exp(-t P_0), \label{eqn:const-control-g} 
\\
z(t) &= \exp\left(tP_0 \right) \cdot z_0,
\\
X(t) &= \exp(tP_0)X_0\exp(-tP_0) = \Ad_{\exp(tP_0)}X_0. 
\label{eqn:const-control-z-X}
\end{align}

As previously noted in Corollary \ref{cor:lam1-gen-sol}, the general
solution for $\Lambda_1$ is
\[
\Lambda_1(t) = \Ad_{g(t)^{-1}}\Lambda_1(0) = g(t)^{-1}\Lambda_1(0)g(t).
\]

We also have a rather complicated (but ultimately elementary) 
expression for the general solution
for $\Lambda_R$.
\[
\Lambda_R(t) = \Ad_{\exp(tP_0)}{\tilde\Lambda}_R(t),
\]
\index[n]{s@$s$, real parameter!dummy integration variable}
\index[n]{zL@$\Lambda$, costate!$\tilde\Lambda_R$, solution specification for $\Lambda_R$}
\index[n]{zy@$\psi$, local auxiliary function or integral!$\psi$, solution specification for $\Lambda_R$}
\index[n]{zY@$\Psi$, auxiliary function in ODE solution of $\Lambda_R$}
where 
\begin{align*}
    {\tilde\Lambda}_R(t) &:= \Lambda_R(0) - [\Psi(t) + \psi(t)P_0,X_0], 
\\
\psi(t) &:= \int_0^t \bracks{P_0}{\Lambda_R(0) - [\Psi(s),X_0]}ds, 
\\
    \Psi(t) &:= \int_0^t \Ad_{\exp(-(X_0+P_0)s)}\Lambda_1(0) 
- \frac{3}{2}\lambda_{cost}\Ad_{\exp(-P_0s)}J ds.
\end{align*}
\index[n]{zl@$\lambda$, eigenvalue}
The two quadratures can be carried out explicitly for any given
matrices $X_0$ and $X_0+P_0$.  The exponentials of these matrices are
expressed in terms of the exponentials $\exp(\lambda s)$ of the
eigenvalues $\lambda$ of these matrices.  The integrands are
exponentials (possibly multiplied by polynomials), and the integrals
are easily computed.  In computing the solution $\Lambda_R$, we first
compute $\Psi$, then $\psi$, then ${\tilde\Lambda}_R$, and finally $\Lambda_R$.

By inspection of the formula for $\psi$, we note that if
${\tilde\Lambda}_{R,0}(t)$ is the specialization of $\tilde\Lambda_R(t)$ to
the initial condition $\Lambda_R(0) = 0$, then the general solution
adds an affine term
\[
{\tilde\Lambda}_R(t) = {\tilde\Lambda}_{R,0}(t) + \Lambda_R(0) - t\bracks{P_0}{\Lambda_R(0)}[P_0,X_0].
\]

\index[n]{YZ@$Z_u$, control matrix!$Z_u$, constant}

\begin{lemma}
The matrices $P_0$ and $X_0+P_0$ have the same characteristic
polynomial (and hence the same eigenvalues).  If $\det(Z_u)<0$, the
eigenvalues are real: $\pm\sqrt{-\det(Z_u)}/\bracks{X_0}{Z_u}$.
\end{lemma}

The most important case occurs when $u$ is a vertex of the control
set $U_T$, where $\det(Z_u)=-1/3<0$.

\begin{proof}
The characteristic polynomial of matrices $P_0,X_0+P_0\in\sl$ is
determined by the determinant.  We have $\det(X_0)=1$, and
$\bracks{X_0}{P_0}=1$.  Then by Lemma~\ref{lem:sl2-lemmas}, we have
\begin{align*}
-2\det(X_0+P_0) &= \bracks{X_0+P_0}{X_0+P_0}
\\ 
&= \bracks{X_0}{X_0}+2\bracks{X_0}{P_0}+\bracks{P_0}{P_0}
\\
&=-2\det(P_0).
\end{align*}
Recall that $\tr(Z_u)=0$ and that $Z_u$ and $P_0$ are scalar
multiples of each other.  When $\det(Z_u)<0$, the control matrix $P_0$
has two real eigenvalues.
\end{proof}

We analyze the solution $z(t) = \exp(t P_0)\cdot z_0$ in greater
detail.  Let $Z_0$ be the value of $Z_u$ at $t=0$.  If $\det(Z_0)<0$,
let $\pm\lambda$ be the real eigenvalues of $P_0$, chosen so that
$\lambda>0$.  The matrix $P_0$ can be diagonalized over $\R$:
\[
\exp(tP_0) = A\, \mathrm{diag}(\exp(t \lambda),\exp(-t\lambda)) A^{-1},
\]
\index[n]{v@$\mb{v}$, vector!$\mb{v}_\pm$, eigenvectors}
\index[n]{A@$A$, matrix or linear map!of eigenvectors}
for some $A\in\SL(\R)$.  The columns of $A=(\mb{v}_+,\mb{v}_-)$ are the column
eigenvectors $\mb{v}_{\pm}$ of $Z_0$, associated to the positive and
negative eigenvectors, respectively.  The matrix $P_0$ has the same
eigenvectors.  The solution is then
\index[n]{yz@$z\in\h$!$\tilde{z}_0$}
\index[n]{0@$\infty$, boundary point of $\h$}
\[
z(t) = A\cdot (\exp(2t\lambda) \tilde{z}_0),\quad \tilde{z}_0 := A^{-1}\cdot z_0.
\]
The image of the trajectory $t\mapsto \exp(2t\lambda) \tilde{z}_0$ is a
Euclidean line through $0+0i$ and $\tilde{z}_0\in\h$.  By adopting the
convention that $\lambda>0$, this linear trajectory tends to $0$ as
time $t$ tends to $-\infty$, reaches $\tilde{z}_0$ at $t=0$, and tends to
infinity as $t$ tends to infinity.  Linear fractional transformations
send generalized circles (that is Euclidean circles or lines) to
generalized circles.  Thus, the image of $z$ is the unique generalized
circle through $A\cdot 0$, $z_0 = A\cdot \tilde{z}_0$, and $A\cdot \infty$.

\index[n]{0@$[-]$, line through $-$}
\index{circle, generalized}
The boundary of $\h$ can be identified with the real projective line
$\R\cup\{\infty\}$.  From the description of $A$ as the column of
eigenvectors, we have $A\cdot 0 = [\mb{v}_-]$ and $A\cdot \infty = [\mb{v}_+]$,
where $[\mb{v}_-]$ and $[\mb{v}_+]$ are the lines through the origin spanned by
the eigenvectors, viewed as points in the real projective line.  Thus,
the trajectories are arcs of generalized circles, from $[\mb{v}_-]$ to
$[\mb{v}_+]$ on the boundary of $\h$.

If $\det(Z_0)>0$, then the eigenvalues are pure imaginary.  The
solutions $z(t)$ in $\h$ are then periodic.  In fact, the solutions
are circles whose center (with respect to the hyperbolic metric) is
the point $z_0\in\h$, defined by the fixed point condition
$Z_0\cdot{z}_0=z_0$.  (Equivalently, $(z_0,1)$ is a complex
eigenvector of $Z_0$, chosen so that $z_0\in\h$.)  Each trajectory
moves at constant speed with respect to the hyperbolic metric on $\h$.
When $u=(1/3,1/3,1/3)$ (the center of $U_T$), $Z_0=J/3$, and the fixed
point is $i\in\h$.
\mcite{MCA:9041208 J/3}

If $\det(Z_0)=0$, then the eigenvalues are $0$ (but $Z_0$ has rank
$1$).  (For example, take control $u=(2/3,1/6,1/6)$.) The solutions $z(t)$ in
$\h$ move along horocycles centered at an ideal point in the real
projective line (viewed as the boundary of $\h$).  That ideal point is
the line formed by the kernel of $Z_0$.

\section{Constant Control at the Vertices}
\label{sec:e2e3}

As we will see in Lemma~\ref{lem:control-face-lemma}, the constant
controls at the vertices of the control triangle have particular
significance, because they often maximize the Hamiltonian.
Assume that the control remains at a
vertex $u$ of the control triangle $U_T$ during some time interval
$t\in[t_1,t_2]$.  By the construction of the control from
state-dependent curvatures, two of the state-dependent curvatures
$\kappa_j,\kappa_{j+1}$ are zero.  Thus the corresponding trajectories
\[
\sigma_{i}(t) = g(t)\mb{s}_{i}^*,\quad i=2j,2j+2,
\]
move along straight lines.  According to Lemma~\ref{lem:hyp-arc}, the
third curve $\sigma_{2j+1}(t)$ moves along an arc of a hyperbola.  The
solution \eqref{eqn:const-control-g} gives explicit parameterizations
of these straight lines and hyperbolic arcs.

If $u$ is a vertex of the control triangle $U_T$, then $\det(Z_u)=
-1/3<0$, and the eigenvalues of $-Z_u$ are $\pm 1/\sqrt3$.  Let
$\mb{v}_\pm$ be eigenvectors for $1/\sqrt3$ and $-1/\sqrt3$
respectively.  
The remarks of the earlier paragraph apply, to show that the
trajectories in $\h$ are generalized circles moving from $[\mb{v}_-]$
toward $[\mb{v}_+]$ on the boundary of $\h$.  The explicit parameterization
is in terms of exponentials, as described above.

If $u=(0,0,1)$, the eigenvectors of $-Z_u$ are computed to be
\[
\mb{v}_- = (-1/\sqrt3,1),\quad \mb{v}_+ = (1,0).
\]
Trajectories $z$ are straight lines moving from the ideal vertex
$[\mb{v}_-] = -1/\sqrt3$ toward the ideal vertex $[\mb{v}_+]=+\infty$.
Explicitly, we have
\index[n]{r@$r$, real number!slope of ODE solution}
\index[n]{C@$C,C_0,C_1,C_2$, local real constant}
\[
x(t) = [\mb{v}_-] + C_0\exp(r t),\quad y(t) = C_0r \exp(r t),
\]
where the constants of integration $r,C_0>0$
are uniquely determined at $t=0$.
\[
x_0=x(0) = [\mb{v}_-] + C_0,\quad y_0= y(0) = C_0r.
\]

If $u=(0,1,0)$, the eigenvectors of $-Z_u$ are computed to be
\[
\mb{v}_- = (1,0),\quad \mb{v}_+=(1/\sqrt3,1).
\]
The trajectories $z$ are straight lines moving from the ideal vertex
$[\mb{v}_-]=+\infty$ toward the ideal vertex $[\mb{v}_+]=+1/\sqrt3$.
The trajectory is
\[
x(t) = [\mb{v}_+] + C_0\exp(r t),\quad y(t) = C_0r \exp(r t),
\]
where now $C_0,r<0$.
Note that when $u=(0,*,*)$, the matrix $Z_u$ is
upper triangular, so that the eigenvectors 
and $\exp(P t)$ are trivial
to compute.

If $u=(1,0,0)$, the eigenvectors of $-Z_u$ are computed to be
\[
\mb{v}_-=(1/\sqrt3,1),\quad \mb{v}_+=(-1/\sqrt3,1).
\]
The trajectories $z$ are Euclidean circles moving from the ideal vertex
$[\mb{v}_-]= 1/\sqrt3$ toward the ideal vertex $[\mb{v}_+]= -1/\sqrt3$.

Note that the solutions at the different vertices are related by
linear fractional transformations $R$, which rotates the star domain
$\hstar$, and permutes the ideal vertices.

\tikzfig{traj-upper}{The trajectories with constant control $u$
  are generalized circles, shown here in the star domain of the upper
  half-plane.}{
\def\rt{1.732}
\begin{scope}[xshift=3in,yshift=0in]
\draw (-1/\rt,0)--(-1/\rt,3);
\draw (1/\rt,0)--(1/\rt,3);
\draw (-1,0) -- (1,0);
\draw (0,0) node [anchor=north] {$u=(0,0,1)$};
\begin{scope}
\clip (-1/\rt,0) rectangle (1/\rt,3cm);
\foreach \t in {10,20,30,40,50,60,70,80} \draw[->,semithick]
(-1/\rt,0) -- ++ (\t:4);
\foreach \t in {10,20,30,40,50,60,70,80} \draw[->,semithick]
(-1/\rt,0) -- ++ (\t:1.5);
\draw[fill=white, fill opacity=0.75] (0,0) circle (0.577cm);
\end{scope}
\end{scope}
\begin{scope}[xshift=1.5in,yshift=0in]
\draw (-1/\rt,0)--(-1/\rt,3);
\draw (1/\rt,0)--(1/\rt,3);
\draw (-1,0) -- (1,0);
\draw (0,0) node [anchor=north] {$u=(0,1,0)$};
\begin{scope}
\clip (-1/\rt,0) rectangle (1/\rt,3cm);
\foreach \t in {10,20,30,40,50,60,70,80} \draw[semithick]
(1/\rt,0) -- ++ (180-\t:4);
\foreach \t in {10,20,30,40,50,60,70,80} \draw[-<,semithick]
(1/\rt,0) -- ++ (180-\t:1.5);
\draw[fill=white, fill opacity=0.75] (0,0) circle (0.577cm);
\end{scope}
\end{scope}
\begin{scope}[xshift=0in,yshift=0in]
\draw (-1/\rt,0)--(-1/\rt,3);
\draw (1/\rt,0)--(1/\rt,3);
\draw (-1,0) -- (1,0);
\draw (0,0) node [anchor=north] {$u=(1,0,0)$};
\clip (-3,0) rectangle (3,3);
\foreach \t/\r in {0/0.577,0.2/0.611,0.4/0.702,0.6/0.832,
0.8/0.986,1/1.154,1.2/1.33,1.4/1.51}
\draw[semithick,gray!30]
(0,\t) circle (\r);
\begin{scope}
\clip (-1/\rt,0) rectangle (1/\rt,3cm);
\foreach \t/\r in {0/0.577,0.2/0.611,0.4/0.702,0.6/0.832,
0.8/0.986,1/1.154,1.2/1.33,1.4/1.51,1.6/1.7}
\draw[semithick]
(0,\t) circle (\r);
\draw[->,semithick] (0,2.33) -- ++ (-180:0.3);
%
\draw[fill=white] (0,0) circle (0.577cm);
\end{scope}
\end{scope}
}

We record the preceding discussion in the form of a lemma.

\begin{lemma}\label{lem:const-control-trajectories}
If the control function is a constant at one of the vertices of $U_T$,
then the trajectories in the star domain are arcs generalized circles.
If the control is $(0,0,1)$ or $(0,1,0)$, then each trajectory moves
along a Euclidean straight line.
\end{lemma}
\begin{proof}  
\end{proof}

\section{Partition of the Cotangent Space}\label{sec:bang}


\index[n]{K@$K$, convex disk!compact convex set in $\R^n$}
\index{face}
\index[n]{F@$F$, face of a convex set}

Let $K$ be a compact convex set in $\R^n$.  A nonempty convex subset
$F$ of $K$ is called a \emph{face} if and only if for all $x,y\in{K}$ and all
$t\in(0,1)$, the membership $t x + (1 - t) y \in F$ implies endpoint
membership: $x,y \in F$.

\begin{lemma}\label{lem:control-face-lemma}
Assume that the control set is $U$ a compact convex set in the affine
plane $\{(u_0,u_1,u_2)\mid\sum u_i =1\}$. For each point in the
cotangent space $(g,X,\Lambda_1,\Lambda_R)$, 
the set of controls $u\in U$ maximizing the
Hamiltonian $\H(\Lambda_1,\Lambda_R,X,Z_u)$ in
equation \eqref{eqn:max-ham} is equal to a face of the control set.
\end{lemma}
\index[n]{u@$u$, control!$u,v\in U$}
\index[n]{t@$t$, real number}
\begin{proof}
Fix $(g,X,\Lambda_1,\Lambda_R)$.  We consider the dependence of the
control-dependent part (denoted $\H_2$) of the Hamiltonian in
equation \eqref{eqn:max-ham}.  As a function of $u\in U$, the
Hamiltonian is a ratio of two affine functions.  Fixing $u,v\in U$,
the dependence on $t$ along the segment $t u + (1-t)v\in U$, for $0\le
t \le 1$, of the control-dependent part of the Hamiltonian takes the
general form
\[
\H_2(t) = \frac{a\,t+b}{c\,t+d}
\]
with nonzero denominator.
The derivative $(ad-bc)/(c\,t+d)^2$ of this expression has fixed
sign. Thus, Hamiltonian is monotonic along the segment. If an internal
point of the segment is a maximizer, then both endpoints are
also maximizers.  According to the definition of face, the set of
maximizers must be a face.
\end{proof}

Thus, if we consider the control set $U_T$, the set of maximizers are
either the entire control set $U_T$, or one of its edges or vertices.

The control-dependent part of the Hamiltonian depends on state and
costate variables through $(X,\Lambda_R)$.  Such pairs can be
identified with the cotangent space of $\O_J$:
\index[n]{I@$I\subseteq\{1,2,3\}$}
\index[n]{U@$U$, control set!$(U_T)_I$, face}
\index[n]{e@$\mb{e}_i$, standard basis!extreme point of the control set}
\index{argmax}
\[
T^*\O_J \cong \{(X,\Lambda_R)\in\O_J\times\sl(\R)\mid \Lambda_R\in X^\perp\}.
\]
For each nonempty subset $\emptyset\ne I\subseteq \{1,2,3\}$, we have a face
$(U_T)_I \subseteq U_T$ defined by the convex hull of $\{\ep_i\mid i\in I\}$,
where $\ep_1=(1,0,0)$, $\ep_2=(0,1,0)$, $\ep_3=(0,0,1)$ is the
standard basis of $\R^3$.
These subsets classify faces of $U_T$.
For each $I$, there is a corresponding region of $T^*\O_J$.
\[
(T^*\O_J)_I :=\{(X,\Lambda_R)\in T^*\O_J\mid
\op{argmax}_{u\in{}U_T}\H_{2}(\cdot,\Lambda_R,X,Z_u) =(U_T)_I\}.
\]
As $I$ runs over nonempty subsets of $\{1,2,3\}$, the sets $(T^*\O_J)_I$
partition $T^*\O_J$ into locally closed subsets.

\index{bang-bang!solution}
The union of the three sets $(T^*\O_J)_I$, for $|I|=1$ is a dense open
subset of $(T^*\O_J)$.  On this dense open subset, the Hamiltonian is
maximized at a uniquely determined vertex.  In general, the control
function $u:[0,t_f]\to{U}_T$ is allowed to be any measurable function.
The solutions of the control system ODEs do not change by modifying
the control $u$ on a set of zero measure in $[0,t_f]$.  If the control
$u$ remains in the dense open subset for all $t\in[0,t_f]$, (that is,
if the image of the control function $u$ is contained in the set of
vertices of $U_T$), we will call the solution a \emph{bang-bang} solution.
Note that the control function is necessarily discontinuous where it
jumps from one vertex to another.  This chapter is concerned with
bang-bang solutions, but later chapters will extend the investigation
to solutions that are not bang-bang.

\index{bang-bang!control}
\begin{definition}[Bang-bang control]
A control function is said to be \emph{bang-bang} if its range is
contained in the set of extreme points of the control set, with
discontinuous switching.
\end{definition}


\index{wall} We call the three sets $(T^*\O_J)_I$, for $|I|=2$ the
\emph{walls} in $T^*\O_J$.  The walls have codimension $1$ in
$T^*\O_J$.  The wall $\{i,j\}$ is contained in boundary between the
open regions with indices $I=\{i\}$ and $\{j\}$.  Finally, there is a
set $(T^*\O_J)_{\{1,2,3\}}$, where the three walls meet.

\section{Constant Control Splines}

By a \emph{spline} we mean a trajectory that has been pieced together
from constant control trajectories, by matching the endpoints of one
trajectory on one subinterval with the initial conditions on the next
subinterval.  In this section, we give explicit constructions of
splines.  In this section, we do not assume that the curves satisfy
the Pontryagin Maximum Principle conditions in the cotangent space.
However, the trajectories are assumed to satisfy the state space ODEs
(for $(g,X)$) and controls at the vertices of $U_T$.

Fix a vertex $u_0=\ep_j\in{}U_T$ in the control simplex.
For $t\ge 0$ and initial position $z=z_0\in\h$,
let $g_0(z,t)\in\SL(\R)$ be the trajectory with solving the state ODE
for $g$ with constant control $u_0$ and initial conditions
\index[n]{g@$g$, group element!$g_i$, trajectory in $\SL(\R)$, constant control}
\index[n]{e@$\mb{e}_i$, standard basis!extreme point of the control set}
\[
g_0(z,0)=I_2,\quad g_0'(z,0) =  \Phi(z).
\]
(As always, prime denotes the derivative with respect to $t$.)  Let
$g_i(z,t)\in \SL(\R)$, for $t\ge 0$, $i\in\Z$, and $z\in \h$ be the
trajectory
\[
g_i (z,t) := R^i g_0(z,t) R^{-i}.
\]
We have 
\[
g_i(z,0) = I_2,\quad 
g_i'(z,0) = \Phi(R^i\cdot z),
\]
with constant control $u_i=R^i\cdot u_0$, using the action of the
cyclic subgroup $\langle R\rangle$ of the dihedral group $\Dih$ on the
control simplex $U_T$.

\index[n]{t@$t\in\R$, time!$\tilde{t}_i$, time parameter}
\index[n]{I@$\mathcal{I}$, tuple of control data}
\index[n]{k@$k$, integer}
\index[n]{g@$g$, group element!$g(\mathcal{I},z,t)\in\SL(\R)$, trajectory with bang-bang control}
\index[n]{X@$X(\mathcal{I},z,t)$, Lie algebra trajectory with bang-bang control}

We define a continuous (shifted) extension of $g_i$ that is
non-constant only for $t\in [\tilde{t}_1,\tilde{t}_2]$:
\[
g_i(z,\tilde{t}_1,\tilde{t}_2,t) := \begin{cases} 
I_2, & \text{if } t \le \tilde{t}_1;
\\
g_i(z,t-\tilde{t}_1), &\text{if } \tilde{t}_1\le t\le \tilde{t}_2;
\\
g_i(z,\tilde{t}_2-\tilde{t}_1), &\text{if } \tilde{t}_2\le t.
\end{cases}
\]
The derivative $g_i'$ has jump discontinuities at $\tilde{t}_1$ and
$\tilde{t}_2$.  Let $z(z_0,t)$ be the solution to the ODE
\eqref{pbm:plane-optimal-control-problem} with constant control $u_0$
and initial condition $z_0$.  For any tuple
\[
\mathcal{I} =
((k_1,t_1),(k_2,t_2),\ldots,(k_{n},t_{n})),
\]
 with $k_i\in\Z$ and
$t_i\ge 0$, and for any $z_0\in \hstar$, let
\begin{align}\label{eqn:gamma-bang-bang}
\begin{split}
\tilde{t}_0 &= 0;
\\
\tilde{t}_{i+1} &= \tilde{t}_i + t_{i+1};
\\
z_{i} &= R^{k_i -k_{i+1}}. z(z_{i-1},t_i);
\\
g(\mathcal{I},z_0,t) &= g_{k_1}(z_0,\tilde{t}_0,\tilde{t}_1,t) 
g_{k_2}(z_1,\tilde{t}_1,\tilde{t}_2,t)\cdots g_{k_n}(z_{n-1},\tilde{t}_{n-1},\tilde{t}_n,t).
\end{split}
\end{align}
Note that on the right-hand side of the last equation, only one factor
at a time is non-constant.  Then $g(\mathcal{I},z,t)$ is continuous in
$t$ and has unit speed parametrization.  Set $X(\mathcal{I},z,t):=
g(\mathcal{I},z,t)^{-1} g'(\mathcal{I},z,t)$.  Note that for $t\in
[\tilde{t}_{i-1},\tilde{t}_{i}]$, when the $i$th factor is active, we have
\begin{align*}
X(\mathcal{I},z_0,t) &= g_{k_i}(z_{i-1},\tilde{t}_{i-1},\tilde{t}_{i},t)^{-1} 
g_{k_i}'(z_{i-1},\tilde{t}_{i-1},\tilde{t}_{i},t) 
\\
  &=g_{k_i}(z_{i-1},t-\tilde{t}_{i-1})^{-1}g_{k_i}'(z_{i-1},t-\tilde{t}_{i-1})
\\
  &= R^{k_i} X(z_{i-1},t - \tilde{t}_{i-1}) R^{-k_i},
\end{align*}
where $X(z,t) = g_0(z,t)^{-1} g_0'(z,t)$.  Comparing left
and right limits of $X(\mathcal{I},z_0,t)$ at the boundary value $t=\tilde{t}_i$, we
find that $X(\mathcal{I},z_0)$ is continuous in $t$.
\begin{align*}
X(\mathcal{I},z_0,\tilde{t}_i^-) &= \Phi(R^{k_i}. z(z_{i-1},t_i)) = \Phi(R^{k_{i+1}}\cdot z_i);
\\
X(\mathcal{I},z_0,\tilde{t}_i^+) &= \Phi(R^{k_{i+1}}\cdot z_i).
\end{align*}


From this, it is easy to see that $g(\mathcal{I},z_0)$ is the general
bang-bang trajectory with finitely many switches (at times
$\tilde{t}_0,\ldots, \tilde{t}_n$), as we vary $\mathcal{I}$ and
$z_0$. The control on the interval $[\tilde{t}_{i-1},\tilde{t}_{i}]$
is $u=R^{k_i}\cdot u_0\in U$.

\index[n]{cost@$\mathrm{cost}$, cost function}
The total $\mathrm{cost}(z_0,[0,t])$ of the trajectory
\eqref{eqn:gamma-bang-bang} with initial condition $z_0$ up to time
$t$ is an easy (freshman calculus) integral to compute from Equation
\eqref{eqn:cost-upper-half-plane}, which we do not display here.  The total cost
of $g(\mathcal{I},z_0,t)$ from time $0$ to $\tilde{t}_n$ is the sum of the costs
on each constant control segment.
\begin{equation}\label{eqn:total-cost}
\sum_{i=0}^{n-1} \mathrm{cost}(z_i,[0,t_{i+1}]).
\end{equation}

\section{Smoothed Polygons}\label{sec:6k+2-gons}

\index{smoothed polygon}

In this section, we construct a family of Pontryagin extremals
of the control problem, all given by a bang-bang control. 
Reinhardt conjectured that the smoothed octagon is the solution to his
problem.  The smoothed octagon belongs to a family of
smoothed \((6k+2)\)-gons, which are all given by bang-bang
controls. The $X$-component of the 
trajectory of the smoothed $8$-gon and the $14$-gon are
shown in Figure~\ref{fig:8-14-gons}.  These are periodic solutions in
$\h$, repeatedly retracing the edges of equilateral triangles in $\h$
(with edges formed by generalized circular arcs).  The triangle has
full dihedral group symmetry about the center $i\in\h$ with respect to
the hyperbolic metric.  The triangles shrink
toward the central point $i \in \hstar$ as $k \to \infty$.  In
$\Kbal$, the smoothed polygons are converging to the circle as $k$
grows.

\begin{figure}[htbp]
\centering
\includegraphics[scale=0.40]{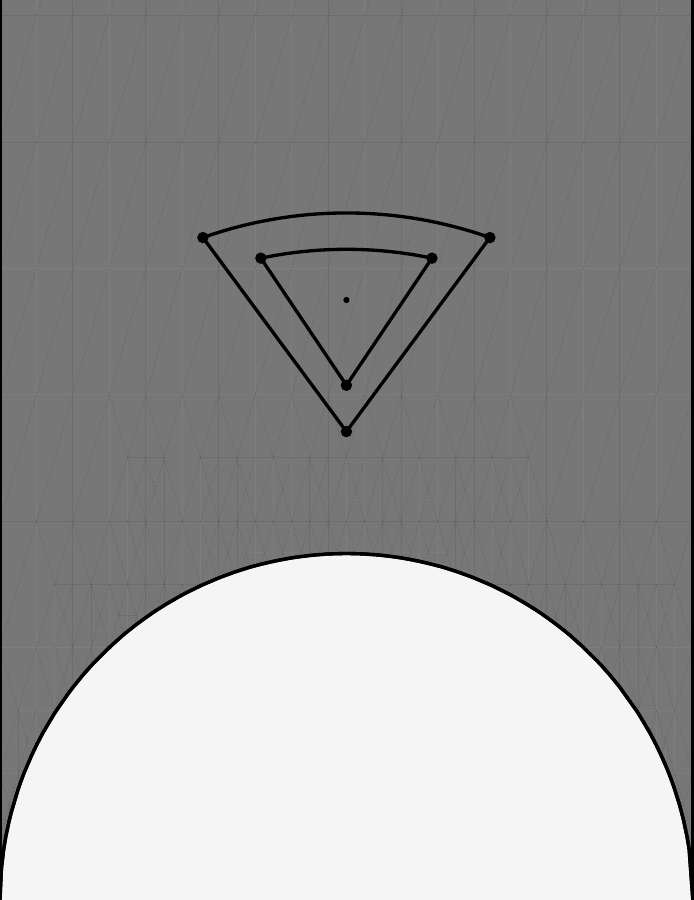}
\caption{Periodic trajectories of the smoothed $8$-gon and $14$-gon in \(\hstar\).
The larger triangle is the orbit of the $8$-gon, 
and the smaller is that of $14$-gon.}
\label{fig:8-14-gons}
\end{figure}
\index{control!mode}

The smoothed octagon comes from a periodic bang-bang control to the
state equations with three modes, corresponding to the three vertices
of the control simplex $U_T$ and the three edges of the triangle in
$\h$.  The trajectory moves in a counterclockwise direction around the
triangle in $\h$, at constant speed in the hyperbolic metric,
completing the four-step mode sequence one vertex
counterclockwise from the starting vertex.  The smoothed octagon
itself can be visualized in $24$ segments: $8$ smoothed
corners and $16$ straight half-edges.  These $24$ segments are arranged into
four groups, each consisting of a multi-curve of $6$ arcs.  The four
groups are congruent to one another, under the rotational symmetry $R$. These 
six arcs are shown in Figure~\ref{fig:multi-curve-oct}.

Now we turn to the rigorous specification of these smoothed polygons,
generalizing the smoothed octagon as follows.  Let $k$ be a positive
integer.  We consider a trajectory $t\mapsto g(\mathcal{I},z_0,t)$ with
a control mode sequence $\mathcal{I}$ of $3k+1$ parts of the same switching time
$t_{sw}$, taking the form
\begin{equation}\label{eqn:3k+1}
\mathcal{I} = ((0,t_{sw}),(-1,t_{sw}),(-2,t_{sw}),\ldots,(-3k,t_{sw})),
\end{equation}
where $t_{sw}>0$ and $z_0\in\h$ are to be determined as functions of
$k\ge1$ in Lemma~\ref{lem:discrete-k}. We use the initial control
$u_0=\ep_3=(0,0,1)\in{}U_T$.

\index[n]{y@$y_0$, $0+iy_0$, smoothed polygon initial condition}
\index[n]{t@$t\in\R$, time!$t_{sw}$, switching time}

We impose the strong boundary condition
\begin{equation}\label{eqn:poly-z}
z(z_0,t_{sw}) = R^{-1}\cdot z_0,\quad 
\text{where } z_0 = 0 + i y_0,\quad y_0\in(1/\sqrt3,1).
\end{equation}
This boundary condition imposes the congruence of the sides of the
triangle in $\h$.  The endpoints $1/\sqrt3$ and $1$ for $y_0$ are
natural; one endpoint is the boundary of the star domain, and the
other endpoint center $i$ of the triangles.  Solving
\eqref{eqn:poly-z} for $t_{sw}$ (the switching time), we
obtain
\begin{equation}\label{eqn:t1}\mcite{MCA:t1}
t_{sw} = \frac{\ln (4/(3y_0^2+1))}{\sqrt{3} y_0}\in(0,\ln2).
\end{equation}

We view $t_{sw}$ as a function of a real variable $y_0\in(1/\sqrt3,1)$.
It is useful to delay imposing the transversality condition $g(t_f)=R$
for as long as possible, which discretizes the problem using the
parameter $k$, and to leave $y_0$ as a continuous variable for now.

The proof that the smoothed $6k+2$-gon are extremals has been broken
into steps.  Theorem~\ref{thm:strong-transversality} constructs one edge of the
triangle in $\h$.  Lemma~\ref{lem:H-max-6k+2} shows the Hamiltonian
maximizing property.  Lemma~\ref{lem:discrete-k} shows how the
terminal condition $g(t_f)=R$ places a discreteness condition on the
size of the triangle in $\h$ to give $6k+2$-gons.  Finally
Theorem~\ref{thm:6k+2} proves extremality of the $6k+2$-gons.

\index[n]{zx@$\chi_{ij}$, switching function}
\index[n]{P@$P$, normalized control matrix!$P_{i,j}$, at mixed controls $\ep_j,\ep_i$}

We define the switching functions $\chi_{ij}$ by
\begin{align}\label{eqn:switch}
\begin{split}
\chi_{ij}(t) &:= 
\bracks{\Lambda_R(t)}{P_{i,j}(t)-P_{i,i}}
\\
P_{i,j}(t) &:= {Z_{\ep_j}}/{\bracks{Z_{\ep_j}}{X(t)}},\quad
(X,\Lambda_R\ \text{with constant control }u=\ep_i)
\end{split}
\end{align}
where $X$ and $\Lambda_R$ are both computed with respect to the
constant control $u=\ep_i$.  By \eqref{eqn:X-Z0-constant}, the
matrix $P_{i,i}$ is a constant $Z_{\ep_i}/\bracks{Z_{\ep_i}}{X_0}$.


\bigskip
\begin{theorem}\leavevmode\label{thm:strong-transversality}
Let $y_0\in(1/\sqrt3,1)$, $z_0=0+iy_0$, and let $t_{sw}>0$ be given by
\eqref{eqn:t1}.  Let $z(z_0,t)$ and $g(t)=g_0(z_0,t)$ be the
solutions to the state equations with constant control $u=\ep_3=(0,0,1)$ on
$[0,t_{sw}]$.  Then these solutions lift uniquely (up to scalar factor)
to a costate trajectory $(\Lambda_1,\Lambda_R,\lambda_{cost})$
satisfying 
\begin{equation}\label{eqn:init-lambdaR}
\Lambda_R(0)\in X(0)^\perp,\quad
\chi_{23}(0)=0,\quad 
\end{equation}
and
\begin{equation}\label{eqn:init-H-chi}
\H_{u=\ep_3}=0,\quad \chi_{31}(t_{sw})=0,
\end{equation}
and strong transversality conditions
\[
\Lambda_R(t_{sw})=R^{-1}\Lambda_R(0)R,\quad
\Lambda_1(t_{sw})=R^{-1}\Lambda_1(0)R.
\]
The trajectory is normal: $\lambda_{cost}\ne0$.
\end{theorem}

\begin{proof}

\index[n]{h@$h\in G$, element of Lie group!in $\SL(\R)$}
\index[n]{r@$r$, real number!trace}
\index{elliptic element of a Lie group}
\index[n]{zL@$\Lambda$, costate!$\Lambda_{10}\in\sl(\R)$}
\index[n]{zL@$\Lambda$, costate!$\lambda_1,\lambda_R$, multiplier}

\mcite{MCA:9670173}
We start with the endpoint condition for
$\Lambda_1$.  Using \eqref{cor:lam1-gen-sol}, we have the condition
\begin{equation}\label{eqn:trans-Lambda}
g(t_{sw})^{-1}\Lambda_1(0)g(t_{sw})=\Lambda_1(t_{sw}) = R^{-1} \Lambda_1(0) R.
\end{equation}
In other words, $\Lambda_1(0)\in\sl(\R)$ centralizes the
element $h:=g(t_{sw})R^{-1}\in\SL(\R)$.  A calculation using the explicit solution
for $g(t_{sw})$ shows that the trace of $h$ is $r:=4/(1+3y_0^2)\in (1,2)$.
This implies that $h$ is a regular elliptic element. Its centralizer in $\sl(\R)$
is $\R{\Lambda_{10}}$, where 
\[
\Lambda_{10}:={h}-r I_2/2
=
\frac{1}{2y_0^2 (1+3 y_0^2)}
\begin{pmatrix}
0 & 4\sqrt3 y_0^4\\
-\sqrt3 (1+y_0^2)(3y_0^2-1)&0
\end{pmatrix}
\in\sl(\R).
\]
Thus $\Lambda_1(0)=\lambda_1\Lambda_{10}$, for some $\lambda_1\in\R$ to be determined.
\mcite{MCA:lambda1}

Next, we turn to the choice of $\Lambda_R(0)\in\sl(\R)$.
The two initial conditions \eqref{eqn:init-lambdaR} force $\Lambda_R(0)$
to have the form
\[
\Lambda_R(0) = \lambda_R 
\begin{pmatrix} 0 & y_0^2\\ 1 & 0\end{pmatrix},
\]
for some $\lambda_R\in\R$.
At this point, the initial state $(\Lambda_1(0),\Lambda_R(0),\lambda_{cost})$
is determined by three scalars: $\lambda_1,\lambda_R,\lambda_{cost}$, where $y_0$
is held fixed.
Equations \eqref{eqn:init-H-chi} place two independent homogenous linear
relations on these three variables, determining them up to a single scalar
multiple.  To avoid the zero solution, we set $\lambda_{cost}=-1$.
A calculation, using the explicit solutions for $\Lambda_R$ gives
\index[n]{l@$\ell(y_0)\in\R$}
\begin{align*}
\lambda_1 &= -((1 + 3y_0^2)(-3 - 
      6y_0^2 + (1 + 3y_0^2)
       \ell(y_0) ))/
   (12\sqrt3y_0^4),\\
\quad
\lambda_R&= -(3 - 12y_0^2 - 
     9y_0^4 + 18y_0^6 + 
     (-1 + 3y_0^2 + 21y_0^4 + 
       9y_0^6)\ell(y_0))/
   (24y_0^6),
\end{align*}
where
$\ell(y_0) = \ln(4/(3y_0^2+1))$.

Finally, we have the transversality conditions at $t_{sw}$.
Remarkably, a calculation shows that the transversality conditions
hold for the calculated values of parameters $\lambda_1,\lambda_R$.

Lemma~\ref{lem:H-max-6k+2} supplements the proof, which shows that the
maximum property is met for the Hamiltonian.
\end{proof}

\begin{remark}
We have a constant of motion
\index[n]{d@$d$, determinant!$d=\det(\Lambda_1)$, a constant of motion}
\begin{align*}
d=d(y_0)&:=\det(\Lambda_1(t))=\lambda_1^2\det(\Lambda_{10})\\
&\phantom{:}=(-1 + 2y_0^2 + 3y_0^4)
  (3 + 6y_0^2 - (1 + 3y_0^2)
     \ell(y_0))^2/
 (144y_0^8).
\end{align*}
The function $d$ is monotonic increasing in $y_0\in(1/\sqrt3,1)$ with
range $(0,9/4)$.  Thus, the determinant uniquely determines the
parameter $y_0$ of the triangle.
\mcite{MCA:9670173}
\end{remark}

\begin{lemma}\label{lem:H-max-6k+2}\mcite{MCA:6384505}
Fix $y_0\in(1/\sqrt3,1)$ and corresponding time $t_{sw}$.  Let
$\chi_{ij}$ be the switching functions, defined for the costate
trajectory constructed in Theorem~\ref{thm:strong-transversality} with
constant control $\ep_3=(0,0,1)$.  Then the PMP conditions hold:
\[
\chi_{31}(t)\ge0,\quad\chi_{32}(t)\ge0,\quad\text{for } t\in[0,t_{sw}].
\]
The functions are zero only when $\chi_{32}(0)=\chi_{31}(t_{sw})=0$.
\end{lemma}

\begin{proof}
An easy substitution gives
\[
\chi_{31}(t_{sw}-t) = \chi_{32}(t),\quad t\in[0,t_{sw}].
\]
Thus, it is enough to show that $\chi_{32}(t)\ge 0$ with equality only
at $t=0$.

\index[n]{y@$y\in\R$, local variable}
\index[n]{r@$r$, real number}
We define new variables $(y,r)$:
\[
y = 1+ 3y_0^2,\quad r = y \exp(\sqrt{3} y_0 t).
\]
The region defined
by $y_0\in(1/\sqrt3,1)$ and $t\in[0,t_{sw}]$ transforms to the triangle 
\index[n]{T@$T$, triangle}
\index[n]{f@$f$, function!$f:\R^2\to\R$}
\[
T=\{(y,r)\in [2,4]^2\mid  y\le r\}.
\]
Note that $t=0$ is transformed to the diagonal $y=r$ of $T$.
Discarding obviously positive multiplicative factors, the inequality
to be proved is $f(y,r)\ge0$ on $T$, where
\begin{align*}
f(y,r) :&= 2r(-1 + y)y\ln(r)
\\ 
&- 
 (r - y)(r(-1 + y) + 
   y(-5 + 2y) + 
   y^2\ln(4/y))
\\ &- 
 2r(-1 + y)y\ln(y).
\end{align*}
\index{Mathematica}

We show that $f$ is nonnegative on the triangle $T$ as follows.
(These calculations appear in the accompanying Mathematica code.
There are several thousand lines of Mathematica code that are used to
support the claims in this book.)  First, an easy substitution gives
$f(y,y)=0$.  (This was already verified above in a different manner,
when we showed that $t=0$ is a switching time.)  Second, the
derivative is negative on the diagonal.
\begin{equation}\label{eqn:f'}
\partials {f}{y}|_{r=y} = y((y-4) + y \ln (4/y)) \le 0.
\end{equation}
Finally, the second derivative is positive on $T$.
\[
\frac{\partial^2 f}{\partial y^2} = 
-10 - 5r + (2r)/y + 7y + 
 4r\ln(r) - 2(r - 3y)
  \ln(4/y) - 4r\ln(y)\ge0.
\]
(We leave this last inequality as a tedious but elementary exercise
for the reader.)  Nonnegativity follows.
\end{proof}

\begin{remark}
Looking more closely at the cases of equality, we see that the only
zero of the switching function on $[0,t_{sw}]$ occurs at $t=0$, and that
the derivative is strictly positive at $t=0$.  (The derivative is zero
in \eqref{eqn:f'} at the corner $r=y=4$ of the disk, but this
corresponds to the unrealizable limiting case $y_0=1$.)
\end{remark}

The following lemma uses transversality conditions to place an
integrality condition $k$ on the size of the triangles in $\h$.

\begin{lemma}\label{lem:discrete-k}
Consider a trajectory $g:[0,t_f]\to\SL(\R)$, with $g(0)=I_2$ following
the dynamical system \eqref{eqn:3k+1}.  
The trajectory reaches $g(t_f)=R$ after a sequence of $3k+1$ control
mode parts $t_f=(3k+1)t_{sw}$ of equal duration $t_{sw}$, provided
\index[n]{zh@$\theta$, angle!$\theta_k=\pi k/(3k+1)$}
\begin{equation}\label{eqn:theta-k}
\frac{4}{3y_0^2+1}= 2\cos\theta_k,\quad\text{where }
\theta_k = \frac{\pi k}{3k+1}.
\end{equation}
\end{lemma}

\index[n]{g@$g$, group element!$g_{sw}\in\SL$, at switching time}

\begin{proof}
Let $g_{sw} = g(z_0,t_{sw})\in \SL(\R)$ be the position at the switching time.
By the spline equations, the condition $g(t_f)=R$ for
\eqref{eqn:3k+1} is
\[
R = g_{sw} (R^{-1} g_{sw} R^1) (R^{-2} g_{sw} R^2) \cdots (R^{-3k} g_{sw} R^{3k}),
\]
or equivalently,
\begin{equation}\label{eqn:eigen}
(R^{-1} g_{sw})^{3k+1} = R^{-3k} = (-I_2)^k.
\end{equation}
\index[n]{zl@$\lambda$, eigenvalue}
Let $\lambda,\lambda^{-1}$ be the eigenvalues of $R^{-1} g_{sw}\in \SL(\R)$.
Comparing eigenvalues on the two sides of \eqref{eqn:eigen}, we obtain
$\lambda^{3k+1} = (-1)^k$, and
\[
\lambda = \exp(\pi i k/(3k+1) + 2\pi i \ell/(3k+1)),\quad \ell \in\Z.
\]
We pick the eigenvalues $\lambda^{\pm1}$ that place $g_{sw}$ in the smallest
neighborhood of $1$; that is, we take $\ell=0,-k$.  (Other pairs of
eigenvalues will produce the right boundary conditions, but the
corresponding multi-curves $\sigma_i$ will have the wrong winding
number around the origin.)  Then
\[
\mathrm{trace}(R^{-1} g_{sw}) = \lambda+\lambda^{-1}=2\cos\theta_k.
\]
The trace $r=4/(1+3y_0^2)$ of $R^{-1}g_{sw}=R^{-1}hR$ is
computed in the proof of Theorem~\ref{thm:strong-transversality}.
\end{proof}

\index[n]{K@$K$, convex disk!$K_{oct}$, smoothed octagon}
\begin{example}[Smoothed Octagon]
For example, $k=1$ for the smoothed octagon $K_{oct}$, and $\theta_k=\pi/4$.
The trace is $\sqrt{2}$, and 
\[
y_0=\sqrt{(\sqrt8-1)/3}\approx0.781,\quad 
t_{sw}=\ln2/\left(2\sqrt{\sqrt8-1}\right)\approx0.256.
\]
If we initialize $X_0=\Phi(0+iy_0)$ according to this value, then we
can compute the density of a packing of smoothed octagons in the plane
using the cost \eqref{eqn:cost}, terminal time $t_f=4t_{sw}$, and
explicit ODE solution with constant control \eqref{eqn:const-control-z-X}.
\begin{align*}
\delta &= \op{area}(K_{oct})/\sqrt{12},
\\
\op{area}(K_{oct}) &= -6\int_0^{t_{sw}}\bracks{J}{X} dt 
\\
&=\sqrt{12} \frac{8 - \sqrt{32} - \ln 2}{\sqrt{8} - 1}\\
&\approx 3.126<\pi.
\end{align*}
This value appears as \eqref{eqn:density-formula}.
\end{example}

\begin{theorem}\label{thm:6k+2}
For each positive integer $k$, the smoothed $6k+2$-gon is a Pontryagin
extremal given by a bang-bang control.
\end{theorem}

\begin{proof}
From \eqref{eqn:poly-z}, it follows that transversality for $z$ from the
half-plane control problem \eqref{pbm:plane-optimal-control-problem} holds with $t_f =
(3k+1) t_{sw}$:
\[
z(z_0,t_f) = R^{-(3k+1)}\cdot z_0 = R^{-1}\cdot z_0.
\]
The strong transversality conditions for $\Lambda_1,\Lambda_R$
imply the transversality conditions at time $t_f=(3k+1)t_{sw}$.
\end{proof}

\section{Supplementary Remarks on Smoothed Polygons}\label{sec:n-gon}


\begin{remark}
A discreteness condition on $y_0$ comes from the transversality
condition $g(t_f)=R$, described in Lemma~\ref{lem:discrete-k}.  We
have solved the nonlinear equations \eqref{eqn:theta-k} and
\eqref{eqn:t1} explicitly for $t_{sw}$ and $y_0$ in the accompanying
code, but we do not display the solution here.  For each positive
integer $k$, the trajectory for the smoothed $6k+2$-gon is now
completely determined by these values of $t_{sw}$ and $y_0$, given as
solutions to nonlinear equations.

We can use Equation \eqref{eqn:theta-k} and the lemma to define
$k$ as a continuous function of $y_0$.  The cost function can then be
interpolated to a function of a real variable $y_0$ (or $k$).
Figure \ref{fig:cost} graphs the area of the smoothed $6k+2$-gon as a
function of $k$.  It appears that the area function is increasing in
$k$ and tends to the area $\pi$ of the circular disk.
\end{remark}

\begin{remark}
A related construction gives a trajectory with $3k-1$ parts in the
control mode sequence -- the smoothed $6k-2$-gon, for $k\ge
2$.  The changes are minor.  We replace equation \eqref{eqn:3k+1} with
\begin{equation}\label{eqn:3k+2}
\mathcal{I} = ((1,t_{sw}),(2,t_{sw}),(3,t_{sw}),\ldots,(3k-1,t_{sw})).
\end{equation}
The parameters are
\[
z_0=0+iy_0,\quad
t_{sw} = -\frac{\ln (4/(3y_0^2+1))}{\sqrt{3} y_0},\quad y_0>1.
\]
The initial control mode is
$u=\ep_2$.
\end{remark}

\begin{remark} 
It seems that the smoothed $6k-2$-gon is {\it not} a Pontryagin
extremal trajectory.  Specifically, all of the conditions seem to
hold, except that the Pontryagin multiplier $\lambda_{cost} >0$ has
the wrong sign.  This suggests that these smoothed polygons are
Pontryagin extremal trajectories for the problem of maximizing the
area.
\end{remark}

\begin{remark}
When $k=1$, the smoothed $6k-2$-gon degenerates to a rectangle with
corners (Figure \ref{fig:rectangle}) and area $\sqrt{12}$.  Allowing
$k$ to be non-integral, for small values of $k>1$, we obtained
smoothed rectangles (that do not quite satisfy the boundary
conditions).
\end{remark}

\tikzfig{rectangle}{ By taking a smoothed $6k-2$-gon and interpolating
  formulas to a fractional number of sides (here $k=1.03$), we see
  that the shape appears to be tending to a rectangle of area
  $\sqrt{12}$ as $k\mapsto 1$.  }{
\def\rt{1.732};
\draw[blue!30] (0,-\rt) -- (1.5,\rt/2) -- (-1.5,\rt/2) -- cycle;
\draw[blue!30] (0,\rt)-- (-1.5,-\rt/2) -- (1.5,-\rt/2) --cycle;
\draw 
(1., 0.)-- (1., 0.0169549)-- (1., 0.048467)-- (1., 0.107035)-- (1., 
  0.215889)-- (1., 0.418202)-- (0.991058, 0.620692)-- (0.960685, 
  0.730146)-- (0.896834, 0.789974)-- (0.77418, 0.823909)-- (0.544074, 
  0.845408);
\draw (0.5, 0.866025) --  (0.269022, 0.870587) --  (0.144746, 
  0.873041) --  (0.0778794, 0.874361) --  (0.0419025, 0.875072) --  (0.0225454,
   0.875454) --  (0.00318826, 0.875836) --  (-0.0327886, 
  0.876547) --  (-0.0996547, 0.877867) --  (-0.223931, 
  0.880321) --  (-0.454909, 0.884882);
\draw (-0.5, 0.866025) --  (-0.730978, 0.853632) --  (-0.855254, 
  0.824574) --  (-0.922121, 0.767326) --  (-0.958097, 0.659183) --  (-0.977455,
   0.457252) --  (-0.98787, 0.255144) --  (-0.993473, 0.146401) --  (-0.996488,
   0.0878927) --  (-0.998111, 0.0564126) --  (-0.998983, 0.0394749);
\draw (-1., 0.) --  (-1., -0.0169549) --  (-1., -0.048467) --  (-1., -0.107035) --   
(-1., -0.215889) --  (-1., -0.418202) --  (-0.991058, -0.620692) --   
(-0.960685, -0.730146) --  (-0.896834, -0.789974) --  (-0.77418,  
-0.823909) --  (-0.544074, -0.845408);
\draw (-0.5, -0.866025) --  (-0.269022, -0.870587) --  (-0.144746, -0.873041) --   
(-0.0778794, -0.874361) --  (-0.0419025, -0.875072) --  (-0.0225454,  
-0.875454) --  (-0.00318826, -0.875836) --  (0.0327886, -0.876547) --   
(0.0996547, -0.877867) --  (0.223931, -0.880321) --  (0.454909, -0.884882);
\draw (0.5, -0.866025) --  (0.730978, -0.853632) --  (0.855254, -0.824574) --   
(0.922121, -0.767326) --  (0.958097, -0.659183) --  (0.977455, -0.457252) --   
(0.98787, -0.255144) --  (0.993473, -0.146401) --  (0.996488, -0.0878927) --   
(0.998111, -0.0564126) --  (0.998983, -0.0394749);
}

The trajectory $X$ in $\h$ for the $6k+2$-gon follows a triangle (with edges
following the arcs of Figure \ref{fig:traj-upper}) centered at
$z=i\in\h$.  It moves counterclockwise around $i$, traversing one edge
for each control mode (Figure \ref{fig:tri}).  The trajectory $X$ in $\h$ for
the $6k-2$-gon also follows an inverted triangle
centered at $z=i\in \h$.  It moves clockwise.

\tikzfig{tri}{The trajectory in the upper-half plane of a smoothed
  $6k+2$-gon follows $3k+1$ edges moving counterclockwise on a
  triangular path centered at $i\in\h$ (left).  The trajectory for the
  smoothed $6k-2$-gon follows $3k-1$ edges moving clockwise on an
  inverted triangle centered at $i\in\h$ (right).}  {
\def\rt{1.732}
\begin{scope}
\node (a) at (0,-1) {};
\node (b) at (\rt/2,0.5) {};
\node (c) at (-\rt/2,0.5) {};
\draw[->] (a) to node {} (b);
\draw[->] (b) to node {} (c);
\draw[->] (c) to node {} (a);
\smalldot{0,0};
\end{scope}
\begin{scope}[xshift=1.5in]
\node (a) at (0,1) {};
\node (b) at (-\rt/2,-0.5) {};
\node (c) at (\rt/2,-0.5) {};
\draw[<-] (a) to node {} (b);
\draw[<-] (b) to node {} (c);
\draw[<-] (c) to node {} (a);
\smalldot{0,0};
\end{scope}
}

The cost increases with $k$ for the $6k+2$-gon and decreases with $k$ for
$6k-2$-gon.  In both cases, the limit of the cost is $\pi$ as
$k\mapsto\infty$.  We show a graph of the costs of the smoothed
polygons as a function of the number $n=6k\pm 2$ of sides (Figure
\ref{fig:cost}). 

\tikzfig{cost}{The graph interpolates the cost $c$ of known critical
  points as a function of the number $n=6k\pm 2$ of straight edge
  segments in the corresponding smoothed polygon.  The cost tends to
  $\pi$ as $n$ increases.  The data is consistent with Reinhardt's conjecture.}{
\begin{scope}[xscale=0.5,yscale=40]
\draw[->,gray] (6,3.14159) to (22,3.14159);
\draw[->,gray] (6,3.1) to (6,3.2);
\node (n) at (23,3.14159) {$n$};
\node (pi) at (5,3.14159) {$\pi$};
\node (c) at (6,3.22) {$c$};
\clip (4,3.1) rectangle (22,3.21);
\draw  (4,3.456) to [out=0,in=180] (8,3.126);
\draw (8,3.126) to [out=0,in=180] (10,3.15);
\draw (10,3.15) to [out=0,in=180] (14,3.139);
\draw (14,3.139) to [out=0,in=180] (16,3.14374);
\draw (16,3.14374) to [out=0,in=180] (20,3.14058);
\foreach \x in {8,10,14,16,20} \draw[gray] (\x,3.14159+0.01) -- 
(\x,3.14159 -0.02) node[anchor=north,black] {$\x$};
\foreach \y in {3.12,3.16} \draw[gray] (6+0.5,\y) -- 
(5.5,\y) node[anchor=east,black] {$\y$};
\end{scope}
}

\section{Smoothed Octagon is an Isolated Extremal}


The previous section showed that the smoothed octagon is an extremal
trajectory.  In this section we show that it is an isolated extremal.

\begin{theorem}\label{thm:oct-isolated}
The lifted trajectory of the smoothed octagon is an isolated
extremal.  That is, in some neighborhood of the set of initial conditions
in the cotangent space, the only extremal trajectory satisfying the
transversality conditions at the endpoints is that of the smoothed octagon.
\end{theorem}

\index[n]{d@$d$, determinant!$d_{oct}=\det(\Lambda_1)$, for smoothed octagon}
\begin{proof}[Proof sketch]
Fix $d>0$ close to the value $d_{oct}=\det(\Lambda_1)$ obtained for a
smoothed octagon.  In the proof, we ignore the transversality on the
group element $g\in G$, and the discreteness parameter $k$ it
produces, until the final lines of the proof.  Instead, we let
$d$ run over a small interval containing $d_{oct}$.

\index[n]{M@$M$, manifold!$M_d$, for smoothed octagon}
\index[n]{N@$N_d$, Poincar\'e section}
\index[n]{q@$q$, point on manifold!$q_0$, in submanifold}
\index[n]{d@$d$, determinant!$d\in\R$}
\index{Poincar\'e section}

We consider the five dimensional manifold $M=M_d$ given by
$q=(X,\Lambda_1,\Lambda_R)\in\sl(\R)^3$ in a neighborhood of the
smoothed octagon parameters subject to four constraints:
\[
\det(X)=1,\quad \bracks{X}{\Lambda_R}=0,\quad\det(\Lambda_1)=d,\quad\H=0.
\]
Fix two vertices $i,j$ of the control simplex $U_T$.  We consider the
four-dimensional \emph{Poincar\'e section} ${N}_d$ obtained by
requiring the vanishing of a switching function
$\chi_{ij}=\bracks{\Lambda_R}{P_i-P_j}=0$.  For each $q_0\in {N}_d$,
let $q(t,q_0)$ be the extremal trajectory in $M$, starting at initial
condition $q_0$.

\index[n]{y@$y_0$, $0+iy_0$, smoothed polygon initial condition}
\index[n]{t@$t\in\R$, time!$t_{sw}$, switching time}
\index[n]{q@$q$, point on manifold!$q_i$}

Fixing $d$, we let $y_0(d)$ be the real number constructed in
Section~\ref{sec:6k+2-gons}.  Associated with $y_0$ and $d$, we have
constructed an extremal lifted periodic 
trajectory $q(t,q_{fix}(d))=(X(t),\Lambda_R(t),\Lambda_1(t))$ in $M_d$
starting at the appropriate initial condition $q_{fix}(d)\in {N}_d$.

Let $t_{sw}(q_{fix}(d))>0$ be the first positive switching time of the
trajectory $q(t,q_{fix}(d))$.  By transversality of switching times
(justified in the remark following Lemma~\ref{lem:H-max-6k+2}), there
is a unique first switching time $t_{sw}(q_0)$ near $t_{sw}(q_{fix}(d))$ for the
extremal trajectory $q(t,q_0)$, when $q_0$ is near $q_{fix}(d)$.

Define a Poincar\'e map $f_d:{N}_d\to {N}_d$ by
\[
f_d(q_0)= \Ad(R)q(t_{sw}(q_0),q_0),
\]
where $\Ad$ acts componentwise on ${N}_d\subset \sl(\R)^3$.  By the strong
transversality conditions for the trajectories in
Section~\ref{sec:6k+2-gons}, $f_d$ has a fixed point at $q_{fix}(d)$.  If
$q_0$ is a nearby initial condition that gives an extremal satisfying
transversality, then setting $q_{i+1}=f_d(q_i)$, we have
\begin{equation}\label{eqn:fp}
\Ad(R^{-4})q_4 = \Ad(R^{-1})q_0,\quad\text{or } f^4q_0=q_0.
\end{equation}
The \emph{four} is half the number of edges of the smoothed octagon;
that is, the terminal time is $t_f=4t_{sw}(q_{oct})$.

\index{inverse function theorem}
\index[n]{I@$I$, identity map}

Let $A_d:=Tf_d:T_q{N}_d\to{}T_q{N}_d$ be the tangent map of $f_d$ at
the fixed point $q=q_{fix}(d)$.  Direct calculation shows that the
fixed point $q_{oct}:=q_{fix}(d_{oct})$ is hyperbolic; that is,
$A_{d_{oct}}$ has no eigenvalues of absolute value $1$.  Then the
fixed point $q_{fix}(d)$ is also hyperbolic for the map $f_d^4$ for
sufficiently nearby parameters $d$.  By the implicit function theorem,
for each $d$ near $d_{oct}$, the function $f_d^4-I$ can be inverted in
an open neighborhood of the fixed point $q_{fix}(d)$, forcing the
fixed point to be isolated (where $I$ is the identity map).  Thus, the
only fixed points near $q_{oct}$ are the fixed points $q_{fix}(d)$.
For $d$ near $d_{oct}$, the only fixed point $q_{fix}(d)$ satisfying
the final transversality conditions \eqref{eqn:theta-k} is $q_{oct}$.
\end{proof}

\begin{remark}
A different proof that the smoothed octagon is isolated 
appears in the preprint \cite{hales2017reinhardt}.
\end{remark}

\newpage
\chapter{Singular Locus}
\label{sec:sing}

\index{singular locus}
\index{edge of control simplex}


\section{Edges of the Control Simplex}\label{sec:edges}

\index{abnormal extremal} 

At the end of this section, Remark~\ref{rem:edge-abnormal} shows that
abnormal extremals can be constructed to the Reinhardt control problem
with an arbitrary measurable control function, taking values in a
fixed edge of the control set $U_T$.  These anomalous abnormal
extremals seem to be an artifact of way we have chosen to encode the
convexity conditions of the convex disk $K$ into the control problem.
These abnormal solutions indicate that the Reinhardt control problem
has unnecessarily many extremals.

In this section, we consider a modified control problem
\eqref{eqn:xy-edge}, which we call the \emph{edge control problem}.
As we will see, in this modified control problem, these particular
abnormal extremals disappear. We will then consider trajectories that
are extremal in two respects: with respect to the Reinhardt optimal
control problem, but also with respect to the edge control problem.



\begin{definition}[Edge Control Problem]\label{def:edge-extremal} 
The edge control system on an interval $[t_0,t_1]$ with free terminal
time $t_1$ is the control problem with state equations
\eqref{eqn:xy-edge}, control set $[-1/2,1/2]$, endpoint
conditions \eqref{eqn:edge:endpoint}, cost functional
\eqref{eqn:edge:cost}, and Hamiltonian \eqref{eqn:edge:ham}.  The
state variables $x,y,s:[t_0,t_1]\to\R$ have range restrictions $y>0$
and $-1/\sqrt3<x<1/\sqrt3$.
\end{definition}

Let $(g,X):[t_0,t_1]\to\SL(\R)\times\hstar$ be a trajectory of the
Reinhardt state equations whose control is restricted to the edge
$u=(0,u_1,u_2)$ of the control set $U_T$.  We show how to define state
variables $s,x,y$ and an edge control problem.  We write the control
in this edge as
\[
u(t)\in \left\{\left(0,\frac{1}{2}+\uedge,\frac{1}{2}-\uedge\right)\in U_T\mid 
\quad -1/2\le\uedge\le 1/2\right\}.
\index[n]{u@$u$, control!$\uedge$, on edge}
\]
Thus the control function on an interval $[t_0,t_1]$ is determined by
a measurable function $\uedge:[t_0,t_1]\to[-1/2,1/2]$.

The variables
$x,y$ are the same as in the Reinhardt problem: $z=x+iy\in\hstar$.
The vector field on $\hstar$ controlling the
Reinhardt state equations (Lemma~\ref{lem:X-dynamics}) is
\[
(x',y')=(f_1,f_2) = (y,f_2(x,y,\uedge)) = 
\left(y,\frac{2\sqrt{3} y^2 \uedge}{-1 + 2\sqrt3 x \uedge}
\right).
\]
In particular, $x'=y>0$, so that $x$ is monotonically increasing.
In this setting, we can solve the state ODE $g' = gX$ with initial
condition $g(t_0)=g_0$ explicitly.
\index[n]{s@$s$, real parameter!$s(t)$, reparameterization}
\index[n]{h@$h\in G$, element of Lie group!$h(t)\in\SL(\R)$}
\begin{align}\label{eqn:g-edge}
g(t) &= g_0h(t_0)^{-1}h(t),\quad\text{where}
\\
h(t) &= 
\begin{pmatrix} 
1 & -x(t)\\
s(t) & 1- s(t)x(t)
\end{pmatrix},
\quad\text{and}\quad
s(t) := \int_{t_0}^t \frac{dt}{y(t)}.
\end{align}
The state equations for the edge control problem take the form
\begin{equation}\label{eqn:xy-edge}
s' = 1/y,\quad x' = y,\quad y' = f_2(x,y,\uedge)
\end{equation}
subject to boundary conditions
\begin{equation}\label{eqn:edge:endpoint}
(s(t_0),s(t_1),x(t_0),x(t_1),y(t_0),y(t_1))=(s_0,s_1,x_0,x_1,y_0,y_1)
\end{equation}
that are chosen to agree with the boundary conditions of the Reinhardt
trajectory $(g,X)$.  We can take $s_0=0$.  (If a constant of
integration is added to $s$, the path $g(t)\in\SL(\R)$ is unchanged.)

\index[n]{zv@$\phi$, cost integrand!$\phi_{edge} = x^2/y$, on edge}
\index[n]{f@$f$, vector field!$f_2$, component}

Up to a positive constant, the cost is given by Equation
\eqref{eqn:cost-upper-half-plane}.
\begin{align*}
\int_{t_0}^{t_1}\frac{1+y^2+x^2}{y}dt 
&= \int_{t_0}^{t_1} (s'(t)+x'(t))dt +
\int_{t_0}^{t_1} \phi_{edge}(x,y)dt 
\\
&=
(s_1+x_1)-(s_0 +x_0) +\int\phi_{edge}(x,y)dt.
\end{align*}
where $\phi_{edge}(x,y) = x^2/y$.
Subtracting a constant determined by the boundary conditions, we
can take the cost to be 
\begin{equation}\label{eqn:edge:cost}
\int\phi_{edge}(x,y)dt.
\end{equation}

This data specifies a control problem with fixed initial time $t_0$
and free terminal time $t_1$.  The state variables are $x,y,s$,
satisfying the ODE \eqref{eqn:xy-edge}.  The control is measurable
control $u_{edge}:[t_0,t_1]\to[-1/2,1/2]$.

\index[n]{zL@$\Lambda$, costate!$\lambda_i$, on edge}

The state variables $(x,y,s)$ take values in an open subset of $\R^3$.
The cotangent space is therefore $T^*\R^3=\R^6$ with variables
$x,y,s,\lambda_1,\lambda_2,\lambda_3$.  The Hamiltonian is
\begin{equation}\label{eqn:edge:ham}
\H = \sum_i\lambda_i f_i +\lambda_{cost}\phi_{edge}= 
\lambda_1 y + \lambda_2 f_2(x,y,\uedge) +\lambda_3 /y 
+\lambda_{cost} \phi_{edge}(x,y).
\end{equation}
The term $\lambda_{cost}$ is constant and is nonpositive.  This
completes our description of the edge control system and its relation
to the Reinhardt control problem with edge-constrained control.  By
rotational symmetry, we obtain edge control systems likewise for the
other edges of $U_T$.

\begin{definition}
We say that a trajectory $(g,X)$ is \emph{edge extremal}, if on every
subinterval of the domain on which one of the components
of the control function $u=(u_0,u_1,u_2)$) is zero a.e. (say $u_j=0$
a.e.), the trajectory is extremal with respect to the
corresponding edge control problem.
\end{definition}

\begin{proposition}\label{prop:edge-extremal} 
The global solution of the Reinhardt problem
is an extremal for the Reinhardt control problem, and it is
also edge extremal.
\end{proposition}

\begin{proof}  We have seen that the global minimizer must be
extremal for the Reinhardt optimal control problem for some optimal
control function $u$ taking values in $U_T$.  On any subinterval of
the domain where the optimal control function takes values in an edge
of $U_T$, then the globally minimizing trajectory must minimize cost
among all trajectories with the same endpoint conditions on the
subinterval and that have their control function similarly restricted
to the edge. Thus, the global minimizer is also edge extremal.
\end{proof}

\index{a.e. almost everywhere}
\index{edge extremal}

The purpose of this section is to prove the following bang-bang
behavior on edges with finite switching.

\begin{theorem}\label{thm:edge}
Consider an extremal lifted trajectory in the edge control system on 
a closed interval $[t_0,t_1]$.
Then the optimal control function $\uedge$ is bang-bang with finitely
many switches.
That is, the control function is 
equal a.e. to a piecewise constant function
\[
\uedge(t)=\pm 1/2,\quad\text{for all } t\in[t_0,t_1],
\]
with finitely many switches on $[t_0,t_1]$.
\end{theorem}

\begin{proof}
The costate equations of the Hamiltonian \eqref{eqn:edge:ham} are
\begin{align}\label{eqn:costate-edge}
\begin{split}
\lambda_1' &= -\frac{\partial \H}{\partial x} 
= \frac{\lambda_2 12y^2 \uedge^2 }{(1-2\sqrt3 x \uedge)^2} 
- \frac{2\lambda_{cost} x}{y}.
\\
\lambda_2' &= -\frac{\partial \H}{\partial y} = 
\frac{\lambda_3}{y^2} - \lambda_1 
- \frac{\lambda_2 4 \sqrt3  y \uedge}{-1+2\sqrt3 x \uedge} + \frac{\lambda_{cost} x^2}{y^2},
\\
\lambda_3' &= -\frac{\partial \H}{\partial s} = 0.
\end{split}
\end{align}
We note that $\lambda_3$ is constant.

The only term of the Hamiltonian depending on the control is
$\lambda_2f_2$.  The function $f_2$ is monotonic decreasing in $\uedge$.
Maximizing the Hamiltonian, when $\lambda_2\ne0$, 
the control is $\uedge=\mp 1/2$ depending on the
sign of $\lambda_2$.  Thus, $\lambda_2$ is a switching function for the
control.

The functions $x,y,s,\lambda_1,\lambda_2,\lambda_3$ are continuous
by construction.  The function $\lambda_2'$ is also continuous by the
form of the ODE it satisfies.  (Although the ODE depends a measurable
control function $\uedge$, when $\lambda_2$ has fixed nonzero sign, the
control function is constant and hence continuous.  Also,
$\lambda_2{f_2(x,y,\uedge)}$ tends to zero with $\lambda_2$.  Thus,
$\lambda_2'$ is continuous.)

We claim that at any point $t_2\in[t_0,t_1]$
where $\lambda_2(t_2)=\lambda_2'(t_2)=0$,
the costate is given by
\begin{align}\label{eqn:edge-zero}
\lambda_3=-x(t_2)^2\lambda_{cost},\quad
\lambda_1(t_2)=0,\quad 
\lambda_2(t_2)=0,\quad
\lambda_{cost}\ne0.
\end{align}
(In particular, $-\lambda_3/\lambda_{cost}\ge0$ and
$x(t_2) = \pm \sqrt{-\lambda_3/\lambda_{cost}}$.)
In fact, under these vanishing conditions on $\lambda_2$ and $\lambda_2'$,
we obtain a nonsingular linear system of two equations and two unknowns 
\[
\H = \frac{\partial \H}{\partial y} = 0
\]
for $\lambda_3$ and $\lambda_1(t_2)$.  The unique solution to this
linear system is as given.  If $\lambda_{cost}=0$, then all of the
costate variables are zero at $t_2$, which is contrary to the
Pontryagin extremality conditions.  Thus, $\lambda_{cost}\ne0$ and the
solution is normal.

In \eqref{eqn:edge-zero},
since $x$ is monotonic increasing, along any extremal lifted
trajectory, there are at most two times $t\in[t_0,t_1]$ such that
$x(t)= \pm \sqrt{-\lambda_3/\lambda_{cost}}$.  At any other switching
time $t$ with $\lambda_2(t)=0$, the trajectory passes transversally
through the wall $\lambda_2=0$ (that is, $\lambda_2'(t)\ne0$). In
such a case the zero of the switching function $\lambda_2$ is isolated.

The next lemma shows that even at times when the conditions of
\eqref{eqn:edge-zero} are met, the zeros of the switching function
$\lambda_2$ are isolated.  By translation in time, we may assume
without loss of generality that $t_2=0$ in the lemma.

In conclusion, all the zeros of the switching function $\lambda_2$ are
isolated and there are at most finitely many switches on any finite
time interval.  Adjusting the control $\uedge$ on a set of measure
zero, we may assume that $\uedge$ is piecewise constant, taking values
$\pm1/2$.
\end{proof}

\index[n]{O@$O$, Landau big oh}
\index{Landau big oh}
\index[n]{C@$C,C_0,C_1,C_2$, local real constant}
\index[n]{t@$t\in\R$, time!$t_1$}
\index[n]{f@$f$, function!$f_i$}

In the next lemma, we write $f_1 = O(f_2)$ to mean that there exist
$t_1>0$ and $C_1 > 0$ such that $f_1,f_2$ are defined on $(-t_1,t_1)$ and
\[
|f_1(t)| \le C_1|f_2(t)|,\quad\text{ for all } t\in(-t_1,t_1).
\]
\begin{lemma}\label{lem:lambda2-Landau}
Fix constants $\lambda_{cost}=-1$, $\lambda_3 = -\lambda_{cost}x_0^2$,
where $|x_0|<1/\sqrt3$, and $y_0>0$.  Let $t_1>0$ and choose any
measurable function $\uedge:(-t_1,t_1)\to[-1/2,1/2]$.  Let
$x,y,\lambda_1,\lambda_2$ be solutions to the state \eqref{eqn:xy-edge}
and costate equations \eqref{eqn:costate-edge} on $(-t_1,t_1)$ with
initial conditions $(x(0),y(0))=(x_0,y_0)$,
$(\lambda_1(0),\lambda_2(0))=(0,0)$, and control function $\uedge$.
Then there exist an integer $n\ge 2$ and a real nonzero constant
$C\ne0$ (both $n$ and $C$ depending on the initial data $x_0,y_0$ but
not on the choice of control function $\uedge$) such that
\[
\lambda_2(t) = Ct^n + O(t^{n+1}).
\]
In particular, having this form, the switching function $\lambda_2$
has an isolated zero at $t=0$ with multiplicity $n$.
\end{lemma}

\index[n]{zL@$\Lambda$, costate!$\tilde\lambda$, approximation to costate on edge}
\index[n]{v@$\mb{v}$, vector!components $v_i$}
\index[n]{A@$A$, matrix or linear map}
\index[n]{b@$\mb{b}\in\R^2$, vector}
\index[n]{zt@$\tau$, dummy variable of integration}

\begin{proof}
If $x_0\ne0$, set $n=2$. Otherwise, set $n=3$.  If $x_0\ne0$, set $C =
-x_0(1+y_0)/y_0^2$.  Otherwise set $C=-2/3$.  
We approximate the functions $\lambda_1,\lambda_2$ by
the functions $\tilde\lambda_1,\tilde\lambda_2$, 
where
\[
\tilde\lambda_1(t):=\int_0^t \frac{2x(\tau)}{y(\tau)} d\tau,\quad
\tilde\lambda_2(t) := -\int_0^t \left(\frac{x(\tau)^2-x_0^2}{y(\tau)^2} 
+ \tilde\lambda_1(\tau)\right) d\tau.
\]
Then it is enough to show that $\tilde\lambda_2(t)=Ct^n+O(t^{n+1})$
and $\lambda_2(t)=\tilde\lambda_2(t)+O(t^{n+1})$.  Consider the error
term $\mb{v}=(v_1,v_2)$, where $v_i:=\lambda_i-\tilde\lambda_i$.
The costate equations for $\lambda_1$ and $\lambda_2$, when expressed
in terms of $\mb{v}$, become $\mb{v}'=-A\mb{v}-\tilde\lambda_2\mb{b}$,
where
\[
A =
\begin{pmatrix}
0 & f_{2 x}\\
1 & f_{2 y}
\end{pmatrix},
\quad
\mb{v} = \begin{pmatrix} v_1\\v_2
\end{pmatrix},
\quad
\mb{b} = \begin{pmatrix} f_{2x}\\  f_{2y}
\end{pmatrix}.
\]
Here $f_{2 x}$ and $f_{2 y}$ are the partial derivatives of $f_2$
evaluated at $(x(t),y(t),\uedge(t))$.

If $x_0\ne0$, we compute
\begin{align*}
y &= y_0 + O(t),\\
x &= x_0 + y_0t + O(t^2),\\
\frac{x^2-x_0^2}{y^2} &= \frac{2 x_0 y_0t}{y_0^2} + O(t^2),\\ 
\tilde\lambda_1 &= \frac{2 x_0 t}{y_0} + O(t^2),\\
\tilde\lambda_2 &= -\int_0^t \frac{2x_0(1+y_0)}{y_0^2} \tau + O(\tau^2)d\tau =
C t^n + O(t^{n+1}).
\end{align*}
If $x_0=0$, we compute
\begin{align*}
y &= y_0 + O(t),\\
x &= y_0t + O(t^2),\\ 
\tilde\lambda_1 &= t^2 + O(t^3),\\
\tilde\lambda_2 &= -\frac{2t^3}{3} + O(t^4) = Ct^n + O(t^{n+1}).
\end{align*}

\index[n]{zt@$\tau$, dummy variable of integration}
\index{norm!natural matrix}
\index{norm!Euclidean}
\index[n]{0@$\vvert-\vvert$, norm}
\index[n]{t@$t\in\R$, time!$t_2$}
\index[n]{C@$C,C_0,C_1,C_2$, local real constant}
\index{Gronwall inequality}

We use the Euclidean norm on $\R^2$ and the natural matrix norm on the
vector space of $2\times 2$ matrices.  By the Cauchy-Schwarz
inequality, we have $\|\mb{v}\|'\le \|\mb{v}'\|$.  Pick $0<t_2<t_1$
such that the denominator of $f_2(x(t),y(t),\uedge(t))$ is bounded
away from zero on $(-t_2,t_2)$.  Then there exists $C_0>0$ such that
$\|A(x(t),y(t),\uedge(t))\|\le C_0$ for all $t\in (-t_2,t_2)$.  We
have for some $C_1>0$,
\[
\|\mb{v}\|' \le \|\mb{v}'\|=\|\tilde\lambda_2 \mb{b}+A\mb{v}\|
\le |\tilde\lambda_2|\|\mb{b}\| + C_0\|\mb{v}\|\le C_1 |t|^n + C_0\|\mb{v}\|.
\]
In integral form,
\[
\|\mb{v}\|\le \frac{C_1\, |t|^{n+1}}{n+1} + C_0\int_0^t \|\mb{v}\|dt.
\]
By the Gronwall inequality (Corollary~\ref{thm:Gronwall}), we have
$\|\mb{v}\| = O(t^{n+1})$.  Then
\[
|\lambda_2- C t^n| = |v_2+\tilde\lambda_2-Ct^n|
\le \|\mb{v}\|+|\tilde\lambda_2-Ct^n| = O(t^{n+1}).
\]
This completes the lemma.
\end{proof}

We adjust the control function $\uedge$ along a set of measure zero in
$[t_0,t_1]$ and assume without loss of generality that
$\uedge\in\{-1/2,1/2\}$.

\index[n]{C@$C,C_0,C_1,C_2$, local real constant}
\index[n]{r@$r$, real number!slope of ODE solution}

If the conditions \eqref{eqn:edge-zero} hold, for each constant
control $\uedge=\pm1/2$, we can solve the state and costate ODEs
explicitly for
$x^{\pm},y^{\pm},\lambda_1^{\pm},\lambda_2^{\pm},\lambda_3^{\pm}$ as a
function of $t\in[t_0,t_1]$.  The solutions for $x,y$ agree with the
solutions obtained in Section~\ref{sec:e2e3}.  The costate ODEs can be
solved without difficulty in Mathematica, but we do not record the
(rather unruly) formulas here.

\index{Mathematica}

\begin{remark}\label{rem:edge-abnormal}\mcite{MCA:4647835}
We have an anomalous situation.  In the proof of
Theorem~\ref{thm:edge}, we showed that $\lambda_{cost}\ne0$ (that is,
the trajectory is normal) when a point exists on the trajectory such
that both $\lambda_2(t_2)=\lambda_2'(t_2)=0$.  At such a point, the
trajectory lies on the wall $\lambda_2=0$ and is tangent to the wall.

However, if we return to the full system of state and costate
equations, still restricting the control function to
$\uedge(t)\in[-1/2,1/2]$, we show that an abnormal solution in fact
exists!  (This is similar to \cite[Sec.~10]{hartl1995survey}, where
the existence of an abnormal solution can depend on how state
constraints are encoded.)  Explicitly, there is an abnormal solution
$\lambda_{cost}=0$ such that the state equations are given by $(g,X)$,
where $g$ is given by \eqref{eqn:g-edge}, $X = \Phi(z)$, and $z=(x,y)$
given by ODE \eqref{eqn:xy-edge}.  The costate solutions are
\[
\Lambda_R = \begin{pmatrix}
x & y^2 - x^2\\
1 & -x
\end{pmatrix},\quad
\Lambda_1 = \begin{pmatrix}
-x & x^2 \\
-1 & x
\end{pmatrix}.
\]
It can be checked that $\Lambda_R$ and $\Lambda_1$ satisfy the costate
ODEs given in Section~\ref{sec:costate-variables}.  It can be checked
that the Hamiltonian is identically zero for these choices.

In particular, the Hamiltonian is independent of the control.  Thus
any measurable control function $\uedge(t)\in[-1/2,1/2]$ maximizes the
Hamiltonian, and we obtain a large family of abnormal trajectories.
However, under the alternative encoding that was used in this section,
these abnormal trajectories do not appear.
\end{remark}

\begin{remark}
A result about edges similar to this section is claimed in the
preprint \cite{hales2017reinhardt}, but the proof there is shaky.
\end{remark}
\section{Singular Locus and Singular Subarcs}\label{sec:singular-locus}

\index[n]{YZ@$Z_u$, control matrix!$Z_u^*$, optimal}
\index[n]{I@$I\subseteq\{1,2,3\}$}

We have stated earlier that the optimal control matrix $Z_u^*$ is
implicitly determined by the Hamiltonian maximization condition in
equation \eqref{eqn:max-ham}. Singular subarcs arise when this
maximization condition fails to produce a unique candidate for the
control matrix $Z_u^*$ over an entire time interval. Throughout this section,
we let $J$ denote the infinitesimal generator of the rotation group,
as usual. 

Recall that we have partitioned the cotangent space $T^*\O_J$
according to subsets $I$ of $\{1,2,3\}$, according to the set of
maximizers in $U_T$ of the Hamiltonian.  The set $I=\{1,2,3\}$
corresponds to the part of the cotangent space on which the
Hamiltonian is independent of the control $u\in U_T$.

\begin{lemma}\label{lem:LambdaR}
For all $X\in\sl(\R)$ and all $\Lambda\in{}X^\perp$, if the control
dependent term of the Hamiltonian $\H_2(\Lambda,X,Z_u)$ is independent
of the control $u\in{}U_T$, then $\Lambda=0$.  If $\Lambda\ne0$, then
the set of controls maximizing the Hamiltonian is a vertex or edge of
the control set $U_T$. 
\end{lemma}

\index[n]{znu@$\nu\in T^*_z\h$, cotangent variable}

\begin{proof}
The part of the Hamiltonian that is dependent on the control can be
written 
\[
\bracks{\Lambda}{Z_u/\bracks{X}{Z_u}}.
\]
We fix $X\in\sl(\R)$ and 
$f$ be the vector field on $T_X\O_X$ defined by 
\eqref{eqn:tangent-map-f1f2}. The affine hull of the image of the
vector field under $U_T$ is the entire tangent space, by
Lemma~\ref{lem:control-convex}.  The value of $\H_2$ must then be
zero, and $\Lambda$ must lie in the orthogonal complement $\{0\}$ of
the entire tangent space.  Thus, $\Lambda=0$.

In contrapositive form, if $\Lambda\ne0$, then the Hamiltonian is not
constant as a function of the control.  By Lemma~\ref{lem:control-face-lemma},
the set of maximizers is a vertex or edge.
\end{proof}

\begin{lemma}\label{lem:Lambda1}
The costate trajectory function $\Lambda_R'$ is continuous at every time $t=t_0$ such
that $\Lambda_R(t_0)=0$. 
If $\Lambda_R(t_0)=\Lambda_R'(t_0)=0$, for some $t=t_0$, then
\[
\lambda_{cost}\ne0,\quad\Lambda_1(t_0) = \frac{3}{2}\lambda_{cost} J.
\]
\end{lemma}

\begin{proof}
By inspection of the ODE, the right-hand side of the costate ODE for
$\Lambda_R'(t_0)$ has a point of continuity when $\Lambda_R(t_0)=0$.
The maximum principle states that the Hamiltonian vanishes identically
along the lifted extremal (See Section \ref{sec:PMP}). Thus, we have
\[
\H(\Lambda_1,\Lambda_R,X; Z_u) = \langle \Lambda_1 - \frac{3}{2}\lambda_{cost} J, X \rangle 
- \frac{\langle \Lambda_R,Z_u\rangle }{\langle X, Z_u \rangle} 
=  \langle \Lambda_{1,cost}, X \rangle 
= 0 \quad \text{ at } t=t_0,
\]
\index[n]{zL@$\Lambda$, costate!$\Lambda_{1,cost}=\Lambda_1 - 3\lambda_{cost}J/2$}
where we set 
\begin{equation}\label{eqn:Lambda1cost}
\Lambda_{1,cost} = \Lambda_1 - 3\lambda_{cost} J/2.
\end{equation}
Thus, $\Lambda_{1,cost}(t_0)\in X(t_0)^\perp$.
The ODE for $\Lambda_R$ gives
\[ 
0 = \Lambda_R'(t_0) =  -[ \Lambda_{1,cost},X(t_0)].
\]
Together, these imply that $\Lambda_{1,cost}\in \R X\cap X^\perp = \{0\}$ at $t=t_0$.
Thus, $\Lambda_{1,cost}(t_0)=0$ and 
$\Lambda_1(t_0) = \frac{3}{2}\lambda_{cost} J$. 

We must have $\lambda_{cost} \ne 0$ for
otherwise we will have $\lambda_{cost} = 0$, $\Lambda_R(t_0)= 0$ and
$\Lambda_1(t_0)=0$, contradicting the non-vanishing of the costate
variables in the maximum principle (see Section \ref{sec:PMP}).  
\end{proof}

The following theorem describes the behavior when $\Lambda_R$
vanishes on an interval.

\index[n]{0@$(-,-)\subset\R$, open interval}

\begin{theorem}\label{thm:circle-singular}
If a lifted extremal has $\Lambda_R$ vanishing identically on an
interval $(t_1,t_2)$, then the control function is constant
$u=(1/3,1/3,1/3)$ (the center of the control set $U_T$) for $t \in
(t_1,t_2)$. Also, the optimal control matrix is $Z_u^*(t) =
\frac{1}{3}J$ on this interval, and this determines an arc
$g(t)\mb{s}_i^*$ of the circle as a singular subarc.  Moreover, the
trajectory is normal, and $X,\Lambda_1$ are constant:
\[
\lambda_{cost}\ne0,\quad
X = J,\quad \Lambda_1 = \frac{3}{2}\lambda_{cost} J,\quad g(t)=g_0 \exp(J t).
\]
\end{theorem}

\begin{proof}
Assume that along an extremal curve, for all $t \in [t_1,t_2]$ we
have $\Lambda_R(t) \equiv 0$. 
On this interval we have $\Lambda_R(t)=\Lambda_R'(t) = 0$.
By the lemma, $\lambda_{cost}\ne0$ and $\Lambda_1(t)\equiv 3\lambda_{cost}J/2$.
The costate equation for $\Lambda_1$ gives
\[
0=\Lambda_1' = [\Lambda_1,X] = \frac{3}{2}\lambda_{cost}\left[J,X \right].
\]
Thus,
we have $[J,X] = 0$ and $X\in \R J\cap \O_J = \{J\}$. 
So $X \equiv J$ on $t \in (t_1,t_2)$.

Note that $\Lambda_R=0$ means that
the Hamiltonian \eqref{eqn:maximized-hamiltonian} does not involve the control
matrix $Z_u^*$ and so the maximization fails to uniquely determine the
control matrix in this interval. Thus, the lifted extremal in this
interval is singular (according to
Definition \ref{def:normal-abonormal-singular-extremals}).
The unique control function $u(t)$ which gives $0=X'=[P,X]=[P,J]$ is
$u(t) = (1/3,1/3,1/3)$ (almost everywhere) 
which is the centroid of the control
set $U_T$.

Now, the curve $g(t) = \exp(Jt)$ satisfies $g'=gX=gJ$ and this is a rotation
matrix in $\SL(\R)$ which gives rise to the circle in the packing plane
as a centrally symmetric convex disk, assuming $g(0) = I_2$.
\end{proof}

Thus, in the \emph{singular locus} of the cotangent space, we must
necessarily have
\begin{align*}
g(t) &= g_0\exp(Jt), &&X(t) = J, &&z(t) = i,
\\
\Lambda_1(t) &= \frac{3}{2}\lambda_{cost}J, &&\Lambda_R(t) = 0,
\end{align*}
where $g(0) = g_0$. 

\index{singular locus}

\begin{definition}[Singular Locus]
The region of the extended state space $T^*(\SL(\R) \times \sl(\R))$ given by 
\[
\Ssing := \left\{(g,\Lambda_1,X,\Lambda_R)=
\left(g_0,\frac{3}{2}\lambda_{cost}J,J,0\right)\mid g_0\in\SL(\R),
\quad\lambda_{cost}\ne0\right\} 
\index[n]{S@$\Ssing$, singular locus}
\]
is called the \emph{singular locus}\index{singular locus}. In the star
domain of the upper half-plane, the singular locus lies over the point
$z=i \in \hstar$.  (That is, $\Phi(i)=J$.)
\end{definition}

\begin{remark}\normalfont
Note that $\Ssing$ gives the initial conditions corresponding
to the circle in $\Kccs$ (which has $g_0 = I_2 \in \SL(\R)$ up to a
transformation in $\SL(\R)$).
\end{remark}

We have seen in Theorem~\ref{thm:edge} that every Pontryagin extremal
of the edge optimal control problem has a bang-bang control function
with finitely many switches.  The other possibility is a singular arc
along which the Hamiltonian is independent of the control function.
That is, the Hamiltonian-maximizing face of $U_T$ is the entire
two-simplex $U_T$.  This is the situation considered in the following
theorem.

\begin{theorem}\label{thm:circle-iff-singular}
  Consider a Pontryagin extremal to the Reinhardt problem that
  contains a singular subarc along which the Hamiltonian is
  independent of the control.  Then during that time interval, the
  extremal remains in the singular locus.  Moreover, the unique
  solution to the system of state and costate equations on that
  interval is a multi-curve of circular arcs, up to affine
  transformation.  Conversely, the lifted trajectory attached to a
  multi-curve of circular arcs is a Pontryagin extremal singular
  subarc.
\end{theorem}

\begin{proof} 
  The proof is a summary of results already obtained.

  If the maximum principle fails to determine a unique control over an
  open time interval,
  then the trajectory remains in the set
  \begin{equation}\label{eqn:edge-singular}
\bigcup_{|I|\ge2} (T^*\O_J)_I
  \end{equation}
By the assumptions of the theorem, the Hamiltonian is independent
of the control.
 Thus, a singular subarc must have the form of Lemma~\ref{lem:LambdaR}.
That is, $\Lambda_R=0$ over some time interval.  By
Theorem~\ref{thm:circle-singular}, the singular subarc is contained in
the singular locus and gives an arc of a circle in $\Kccs$.

Conversely, the multi-curve of circular arcs is represented by
$g=\exp(t J)$, and by the ODE $g'=Xg$, where $X=J$, and $X'=0$, and
$g=g_0\exp(Jt)$.  As remarked in the proof of
Lemma~\ref{thm:circle-singular}, $0=X'=[P,X]$ implies that the control
function is constant almost everywhere, taking value $(1/3,1/3,1/3)$
at the center of the control set $U_T$.  Along the trajectory, the
Hamiltonian is then independent of the control.  By
Lemma~\ref{lem:LambdaR}, we have $\Lambda_R=0$.  By
Lemma~\ref{lem:Lambda1}, we have $\Lambda_1=\frac{3}{2}\lambda_{cost}
J$ and $\lambda_{cost}\ne0$.  These costate values 
lie in the singular locus.
\end{proof}


Although the multi-curve of circular arcs is a Pontryagin extremal, we
can invoke second-order conditions to show that it is not a global
minimizer.  By considering a second variation, Mahler proved that the
circle is not a local minimizer of the Reinhardt
problem~\cite{mahler1947area}.


\begin{theorem}\label{thm:no-singular-arcs}
The global minimizer of the Reinhardt problem does not
contain any singular subarcs.
\end{theorem}

\begin{proof}
  We assume for a contradiction that the global minimizer contains
  a singular subarc.
  The previous theorem shows that the singular subarc comes from
  a multi-curve of circular arcs.
  
We use second order conditions to show that the circular arc is not a
local minimizer on any time interval $(t_1,t_2)$ so that the solution
to the Reinhardt problem contains no circular arcs.  We consider a
deformation of a circular arc of the form
\index[n]{s@$s$, real parameter!$s\in\R$, deformation parameter}
\index[n]{g@$g$, group element!$g_{s}(t)$, deformed curve in $\SL$}
\index[n]{zY@$\Psi_{s}(t)=\exp({s}(-))$, deformation of identity matrix}
\index[n]{zy@$\psi$, local auxiliary function or integral!$\psi_i$, compactly supported functions}
\[
g_{s}(t) = \exp\left({s} 
\begin{pmatrix}\psi_{1}(t) & \psi_{2}(t)\\ \psi_{2}(t) & -\psi_{1}(t)\end{pmatrix}
\right) \exp(J t) =: \Psi_{s}(t)\exp(Jt),
\]
for sufficiently small ${s}>0$ and compactly supported $C^\infty$
functions $\psi_{1}$, $\psi_{2}$ to be determined on the interval
$[t_1,t_2]$.  

\index{Cartan-Maurer one-form}

We point out that $t$ is not a unit speed parameter so that $\det(X)$
need not equal $1$.
The cost is
\[
-\frac{3}{2}\int_{t_1}^{t_2} \bracks{J}{X} dt
= -\frac{3}{2}\int_{t_1}^{t_2} \bracks{J}{g^{-1} dg},
\]
where $g^{-1}dg$ is the Cartan-Maurer one-form on $\SL(\R)$.  This
shows that the cost is independent of parameterization.  If $g(t) =
\Psi(t) \exp(J t)$, where $\Psi$ is an invertible $C^\infty$ matrix of
$t$, then the product rule gives
\[
\bracks{J}{g^{-1}dg}= \bracks{J}{J}dt+\bracks{J}{\Psi^{-1}d\Psi}.
\]

\index[n]{cost@$\mathrm{cost}$, cost function}

Computing the cost of $g_{s}$ on $[t_1,t_2]$ by this formula, we find that
\mcite{MCA:1049350}
\[
\mathrm{cost}(g_{s}) = \mathrm{cost}(g_0) 
- 3s^2
\int_{t_1}^{t_2} \left(\psi_{1}'(t) \psi_{2}(t) - \psi_{1}(t) \psi_{2}'(t)\right)\,dt 
+ O({s}^3).
\]
This is a second order variation that is not detected by Pontryagin
first order conditions.  Choose nonnegative $C^\infty$ compactly
supported functions $\psi_{1}(t),\psi_{2}\ge0$ on $(t_1,t_2)$ such
that $\psi_{1}'(t)>0$ and $\psi_{2}'(t)<0$ on their common support to
make the ${s}^2$-contribution negative.
Then for all
sufficiently small ${s}>0$, we have
\[
\mathrm{cost}(g_{s}) < \mathrm{cost}(g_0)=\pi.
\]
The curvatures of the curves
$t\mapsto \sigma_i(t)=g_{s}(t)\ee{i}$ are $C^\infty$ functions of
${s}$ and $t$.  The curvature functions converge uniformly to the
constant positive curvature of the circle as ${s}$ tends to $0$.
We may pick ${s}>0$ sufficiently small so that the curvatures of
the curves are positive.  The corresponding centrally symmetric convex disk
in $\Kccs$ shows that the circle is not a local minimizer of cost.
\end{proof}

\begin{remark} To obtain rough intuition about the perturbation of the circle
considered in the theorem,  we consider piecewise linear continuous functions 
$\psi_{1}$ and $\psi_{2}$ that are periodic modulo $\pi/3$, where
\begin{align*}
\psi_{1}(t) &= 
\begin{cases}
0, & \\
t-\pi/9, & \\
-t + 3\pi/9, & \\
\end{cases}
&&
\psi_{2}(t) = 
\begin{cases}
t, & \qquad\text{if } t\in[0,\pi/9];\\
-t+2\pi/9, & \qquad\text{if } t\in[\pi/9,2\pi/9];\\
0, & \qquad\text{if } t\in[2\pi/9,3\pi/9].
\end{cases}
\end{align*}
Then $\psi_{1}'\psi_{2} - \psi_{1} \psi_{2}' = \pi/9 >0$ on
$(\pi/9,2\pi/9)$ and is zero on $(0,\pi/9)$ and $(2\pi/9,\pi/3)$.  We
plot the multi-curve
$\sigma_j(t)=\exp(Js_0)\Psi_{s}(t)\exp(Jt)\mb{s}_j^*$ for
$t\in[0,\pi/3]$, ${s}=0.12$, and $s_0=-1/8$ (to rotate the entire
figure) in Figure~\ref{fig:deformed-circle}.  The figure is
approximately a smoothed octagon.  The functions $\psi_{i}$ are not
$C^\infty$ and ${s}$ is so large that the simple closed curve is not
convex.  Nevertheless, this crude numerical example suggests that the
deformation in the theorem can be used to smoothly interpolate between
the circle and the smoothed octagon along an interpolation path that
is strictly decreasing in area.
\end{remark}

\begin{figure}[ht]
\centering
\mcite{MCA:9575710}
\includegraphics[scale=0.65]{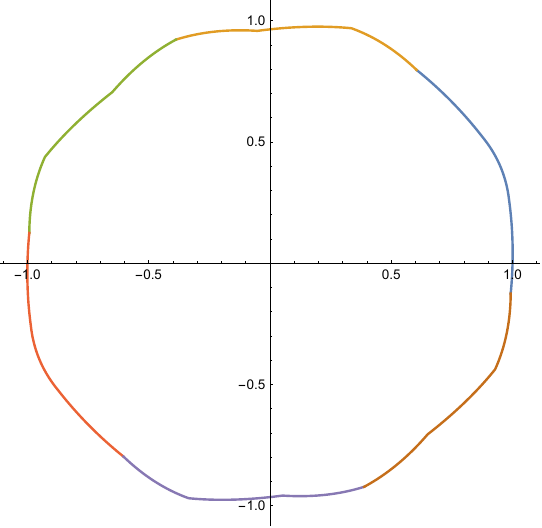}
\caption{The deformation of Theorem~\ref{thm:no-singular-arcs} leads
  to a simple closed (multi) curve $\exp(Js_0)
  \Psi_{s}(t)\exp(Jt)\mb{s}_j^*$ that approximates the smoothed
  octagon.}
\label{fig:deformed-circle}
\end{figure}

\section{Non-chattering away from the Singular Locus}


\begin{theorem}\label{thm:avoid-singular} 
Consider an extremal lifted trajectory on $[t_1,t_2]$ and
$t_0\in[t_1,t_2]$ such that the lifted trajectory does not meet the
singular locus at time $t_0$.  Assume that the trajectory 
is also edge-extremal \eqref{def:edge-extremal}.
Then there exists an edge $\{i,j\}$ and a neighborhood of $t_0$ in
$[t_1,t_2]$ on which
\[
(X(t),\Lambda_R(t))\in \bigcup_{I\subseteq\{i,j\}}(T^*\O_J)_I.
\]
\end{theorem}
That is, near $t_0$, the set of maximizers of the Hamiltonian is
confined to a single edge of $U_T$.

\index[n]{P@$P$, normalized control matrix}
\index[n]{zx@$\chi_{ij}$, switching function}

\begin{proof}
We prove the theorem in contrapositive form.  Assume that in every
neighborhood of $t_0$, the set of maximizers of the Hamiltonian is not
confined to any edge of $U_T$.  Reparameterizing by a time
translation, we may assume that $t_0=0$.  We will prove that the
trajectory meets the singular locus at $t=0$.  We let
$(X,\Lambda_1,\Lambda_R)$ be the controlled trajectory with optimal
control function $u(t)$.  By the continuity of the Hamiltonian, our
assumption gives
\[
(X(0),\Lambda_R(0))\in(T^*\O_J)_{\{1,2,3\}}.
\]
By Lemma~\ref{lem:LambdaR}, we have $\Lambda_R(0)=0$.

\index[n]{Y@$Y\in\mathfrak{g}$, Lie algebra element!in $\sl$}
\index[n]{A@$A$, matrix or linear map!$A:\sl(\R)\to\sl(\R)$}


\index[n]{zL@$\Lambda$, costate!$\Lambda_{1,cost}=\Lambda_1 - 3\lambda_{cost}J/2$}

We let $P(t)=Z_{u(t)}/\bracks{Z_{u(t)}}{X(t)}$ be the normalized
control matrix.  Set $\Lambda_{1,cost}=\Lambda_1 -
3\lambda_{cost}J/2$.  It satisfies (by the costate
equations~\ref{pbm:state-costate-reinhardt})
\[
\Lambda_{1,cost}' = [\Lambda_{1,cost},X] + \frac{3}{2} \lambda_{cost}[J,X].
\]
By the form of the right-hand-side, $\Lambda_{1,cost}$ is continuously
differentiable.  Define $Y(t) := -\int_{0}^t [\Lambda_{1,cost},X]dt$,
and set $\Lambda:=\Lambda_R-Y$.  Let
\[
A(\Lambda,t) = [P(t),\Lambda]-\bracks{\Lambda}{P(t)}[P(t),X(t)],
\]
viewed as a time-dependent linear function $A$ on $\sl(\R)$, along the
state trajectory given by $P=P(t)$ and $X=X(t)$.  The costate ODE for
$\Lambda_R$ takes the form
\[
\Lambda' = A(\Lambda,t) + A(Y,t).
\]
\index{norm!natural matrix}
\index{Gronwall inequality}
\index[n]{C@$C,C_0,C_1,C_2$, local real constant}
We consider the initial value problem for $\Lambda$ with initial
conditions $\Lambda(0)=\Lambda_R(0)={Y}(0)=0$.  Identifying $\sl(\R)$
with $\R^3$, we use the Euclidean norm on the Lie algebra and use the
natural matrix norm for $A$.  By Cauchy-Schwarz, $\|\Lambda\|'\le
\|\Lambda'\|$.  Then
\[
\|\Lambda\|' \le \|\Lambda'\|=\|A[\Lambda]+A[{Y}]\|
\le C_0\|\Lambda\| + C_0\|{Y}\|,
\]
where 
$C_0>0$ is any time-independent bound on the matrix norms $\|A[\cdot,t]\|$ in 
a small neighborhood of $t=0$.
In integral form
\begin{equation}\label{eqn:Lambda-bellman}
\|\Lambda\| \le C_0 \int_0^t \|{Y}\|dt + C_0 \int_0^t \|\Lambda\|dt.
\end{equation}

We claim $\Lambda_{1,cost}(0)=0$.  Note that $\Lambda_{1,cost}(0)\in
X(0)^\perp$ by the vanishing of the Hamiltonian
$\bracks{\Lambda_{1,cost}}{X}-\bracks{\Lambda_R}{P}$ at $t=0$.  Thus,
$\Lambda_{1,cost}(0)=0$ if and only if $[\Lambda_{1,cost}(0),X(0)]=0$.  Suppose
for a contradiction that $[\Lambda_{1,cost}(0),X(0)]\ne0$. Then
$\|{Y}\| = C_2 t + O(t^2)$, for some $C_2\ne0$, and the Gronwall
inequality (Corollary~\ref{thm:Gronwall}) applied to
\eqref{eqn:Lambda-bellman} gives
\[
\|\Lambda_R-{Y}\| = \|\Lambda\| =O(t^2)
\]
and
\[
\Lambda_R = Y + O(t^2) = -[\Lambda_{1,cost}(0),X(0)]t + O(t^2).
\]
The control dependent term of the Hamiltonian is
\begin{align*}
-\bracks{\Lambda_R}{P} &= t\bracks{[\Lambda_{1,cost}(0),X(0)]}
{\frac{Z_{u(t)}}{\bracks{X(0)}{Z_{u(t)}}}} + O(t^2).
\end{align*}
By Lemma~\ref{lem:LambdaR}, for sufficiently small $t>0$ (or for
sufficiently small $t<0)$, there exists a vertex or edge of $U_T$,
that maximizes this Hamiltonian.  If the set of maximizers is an edge
when $t$ has one sign, then the maximizer is the complementary vertex
of $U_T$ when $t$ has the other sign.
If the set of maximizers is an edge (for a given sign of $t$), by the
results of Theorem~\ref{thm:edge}, the control function has finite
bang-bang switching near $t=0$.  This is contrary to the assumption
that the set of maximizers is not confined to any edge of $U_T$.

\index{Gronwall inequality}

We claim that $\Lambda_{1,cost}'(0)=0$. The proof is similar, and
we give a brief sketch.  By the
form of the ODE for $\Lambda_{1,cost}$, we have
$\Lambda_{1,cost}'(0)\in X(0)^\perp$.  It is enough to show
$[\Lambda_{1,cost}'(0),X(0)]=0$.  Suppose for a contradiction that
$[\Lambda_{1,cost}(0),X(0)]=0$, but $[\Lambda_{1,cost}'(0),X(0)]\ne0$.
We again use estimates on the size of $\|{Y}\|$ and the
Gronwall inequality (Corollary~\ref{thm:Gronwall}) to
show that $\Lambda_R$ has an isolated zero at $t=0$.  Then a switching
function has an isolated zero at $t=0$, which again contradicts the
hypothesis of the lemma.

$\Lambda_{1,cost}(0)=0$ implies $\Lambda_1(0)=3\lambda_{cost}J/2$.  
Also $\Lambda_{1,cost}'(0)=0$
implies
\[
0=\Lambda_1'(0)=[\Lambda_1(0),X(0)]=(3/2)\lambda_{cost}[J,X(0)].
\]
The non-vanishing of the costate vector at $t=0$ implies that
$\lambda_{cost}\ne0$.  $X(0)\in\R{J}\cap\O_J=\{J\}$, so $X(0)=J$.
By definition, this is a point in the singular locus.
\end{proof}


\begin{theorem}\label{thm:finite-bang-bang}
  Consider a controlled extremal lifted trajectory that does not meet
  the singular locus $\Ssing$.  Assume that the trajectory $(g,X)$
  is also edge-extremal.
Then the control function is bang-bang with
  finitely many switches.
\end{theorem}

\begin{proof} According to the theorem, for every $t_0\in[t_1,t_2]$, there
exists a neighborhood of $t_0$ on which the control takes values in a
single edge of $U_T$ almost everywhere.  By Theorem~\ref{thm:edge},
the control is bang-bang with finite switching along the edge of the control
simplex  in a
neighborhood of each $t_0$. By the compactness of the interval
$[t_1,t_2]$ the finite switching holds for the entire interval.
\end{proof}

\index{balanced disk}
\index[n]{K@$K$, convex disk!$K(g,X)$, attached to data $(g,X)$}

\begin{remark} In terms of Reinhardt's problem, the theorem implies
  that a trajectory $(g,X)$ that is extremal and edge-extremal that
  does not meet the singular locus $\Ssing$ defines a balanced disk
  $K(g,X)\in \Kbal$ whose boundary is a smoothed polygon, consisting
  of finitely many straight edges and hyperbolic arcs.
\end{remark}

\begin{lemma}\label{lem:constant-not-singular} 
An extremal trajectory
with constant control at a vertex of $U_T$ does not meet the singular
 locus.
\end{lemma}

\index{chattering arc}

(This lemma does not rule out the possibility of a chattering arc
converging to the singular locus, as discussed in the next part of
the book.)

\begin{proof} 
We assume for a contradiction that the trajectory meets the singular
locus at time $t=0$ with constant control $u=\mb{e}_3\in{}U_T$ (say).
Consider the solutions to the state and costate ODEs with constant
control at enumerated the vertex $\mb{e}_3$ of $U_T$ and initial
conditions at the singular locus.  The solutions are analytic.  At the
singular locus, we have $\Lambda_{1,cost}(0)=\Lambda_R(0)=0$ given by
\eqref{eqn:Lambda1cost} and $X(0)=J$.

Let $P_i=Z_{\mb{e}_i}/\bracks{J}{Z_{\mb{e}_i}}$ be the normalized
control matrix, evaluated at $t=0$, $X(0)$, with the controls at the
three vertices $\mb{e}_i\in{}U_T$.  By using the costate ODEs, we also
compute $\Lambda_{1,cost}'(0)=0$.  The following derivatives exist and
have the following values.
\[
\Lambda_R'(0)=\Lambda_R''(0)=0,\quad 
\Lambda_R'''(0)=-6\lambda_{cost}[J,P_3].
\]
The Lie bracket
$[J,P_3]$ is orthogonal to $J$ and $P_3$.
Also $P_1+P_2+P_3 \in \R{J}$. 
We have $\bracks{[J,P_3]}{P_1}=-\bracks{[J,P_3]}{P_2}\ne0$. 
The leading term of the control term of the Hamiltonian 
evaluated at control $u=\mb{e}_i$ is
\begin{align*}
-\bracks{\Lambda_R(t)}{P_i}=
-&\frac{t^3}{3!} \bracks{\Lambda_R'''(0)}{P_i} + O(t^4) 
\\
&t^3 \bracks{[J,P_3]}{P_i} + O(t^4).
\end{align*}
The leading term is zero when $i=3$, but it is nonzero with opposite
signs when $i=1,2$.  Hence, the maximizer of the Hamiltonian is never
$\mb{e}_3$, when $t$ is small.
This is contrary to first-order conditions of extremality.
\end{proof}

\newpage
\part{Circular Control}\label{part:circular}
\chapter{Circular Control Set}\label{sec:circular-control}

We are interested in a modification of the triangular control set to a
circular disk.  This will lead to a modified control problem that we
view as a \emph{toy control problem}.  Insights from the toy problem
will be applied in Part~\ref{part:mahler} to give a proof of Mahler's
First conjecture.  This part of the book is logically independent from
the other parts, and it can be skipped without interrupting the flow
of the text.

\index{toy control problem}
\index[n]{U@$U$, control set!$U_C$, circumscribed}
\index[n]{U@$U$, control set!$U_I$, inscribed}
\index[n]{U@$U$, control set!$U_r$, intermediate}

\begin{definition}[Circular control sets]\label{def:circ-control-set}\leavevmode
\begin{itemize}
\item The \emph{circumscribed} or \emph{disk} control set is the
    set $U_C$ which is the circumscribing disk of the two-simplex in
    $\R^3$:
\[ 
U_C := \left\{(u_0,u_1,u_2)\mid 0\le u_i \le 1, ~\sum_i
    u_i = 1,~\sum_i u_i^2 \le 1\right\}.
\] 
\item The \emph{inscribed}
    control set is the set $U_I$ which is the inscribed disk of the
    two-simplex in $\R^3$: 
\[ 
U_I := \left\{(u_0,u_1,u_2)\mid 0\le
    u_i \le 1, ~\sum_i u_i = 1,~\sum_i
    u_i^2 \le \frac{1}{2}\right\}.
\] 
\item Later (in
Section \ref{sec:su11-control-sets}), we will also be interested in
    control sets which interpolate between $U_I$ and $U_C$. For
    $1/3 \le r^2$ we define
\[ U_r := \left\{(u_0,u_1,u_2)\mid u_0 +
    u_1 + u_2 = 1,~u_0^2+ u_1^2 + u_2^2 \le r^2 \right\}.
\] 
We denote
    the boundary (relative to the affine hull) of these sets by
    $\partial U_C$, $\partial U_I$ and $\partial U_r$ respectively.
\end{itemize}

\end{definition}
\begin{remark}\normalfont
Controls $(u_0,u_1,u_2)$ from $U_C \setminus U_T$ give rise to
boundary curves $\sigma_j$ of regions which fail to be convex. This is
a consequence of the interpretation of the state-dependent
curvatures $u_j\ge0$ as a nonnegative curvature
condition.
\end{remark}

\index{Noether's theorem}

One motivation in considering these control sets is the following: we
can observe that the triangular control set is invariant under
$\Z/3\Z$-rotations while the disk control set is invariant under
rotations by the circle group $S^1$. The latter is important for our
investigations, as it will allow us to employ Noether's theorem to
derive a \emph{first integral} or a conserved quantity of the
dynamics. A conserved quantity is useful because it facilitates a
reduction in dimension of the original problem --- if the group of
symmetries is large enough, reduction by that group may even lead to a
direct solution.  Other motivations are enumerated in the introduction
to the book.

Until now, we have been exclusively working with the simplex
control set $U_T$. Now, we change our control set to be $U_C$, the
circumscribing disk of $U_T$ as described in
Definition \ref{def:circ-control-set}. This change is motivated by the
following theorem.
 
\index[n]{A@$A$, matrix or linear map!rotation}

\begin{theorem}\label{thm:symmetries}
Let $\H$ be the Hamiltonian of the Reinhardt optimal control
problem. Assuming that the control set is closed under rotations
(i.e., if $A \in \SO$ is a rotation and $Z_u$ is in the control
set, then so is $\Ad_A Z_u$), the Hamiltonian $\H$ is invariant
under the action of the subgroup $\SO$ of $\SL(\R)$.
\end{theorem}
\begin{proof}
\index[n]{0@$-\tilde{\phantom{-}}$, transformed quantity}
The Hamiltonian depends on the quantities $X,\Lambda_1,\Lambda_R$ and
the control matrix $Z_u$. Ignoring $Z_u$ for the time being, if
$A \in \SO$ is an arbitrary rotation, then let us see how these
quantities transform. $A$ acts on trajectories in $\h$ by
linear fractional transformations and so, equivalently, we have
$X \mapsto \Ad_A X= A X A^{-1} =: \tilde{X}$ in the adjoint
orbit picture. Now, $\Lambda_1$ transforms as
$\Lambda_1 \mapsto A \Lambda_1 A^{-1} =: \tilde{\Lambda}_1$
since these transformed quantities satisfy the same ODE:
\[
\tilde{\Lambda}_1' = \Ad_A \Lambda_1' 
= \Ad_A [\Lambda_1,X] = [\Ad_A\Lambda_1,\Ad_A X] 
= [\tilde{\Lambda}_1,\tilde{X}].
\]

Similarly, we can also see that $\Lambda_R$ transforms as
$\Lambda_R \mapsto A \Lambda_R A^{-1}$.  Now, the control set
transforms as $Z_u \mapsto A Z_u A^{-1}$, which may, in general,
fall outside the control set given by $U_T$. So, if we modify the
control set so that it does not, a simple computation now using the
expression for the control-dependent Hamiltonian in
equation \eqref{eqn:full-hamiltonian} shows that it is unchanged under
these transformations by $A$.
\end{proof}

\begin{remark}\leavevmode\normalfont
Following the discussion at the end of Section \ref{sec:control-sets},
the control set $U_T$ is only symmetric by discrete
$\Z/3\Z$-rotations and not under general $\SO$ rotations.
\end{remark}

\section{Conserved Quantity for the Circular Control Set}\label{sec:angular-momentum}
Recalling the discussion at the end of Section \ref{sec:control-sets},
we note that we enlarge the control set $U_T$ to $U_C$ to have
continuous $\SO$-symmetry in order to manufacture a conserved quantity
for the dynamics. This is achieved by Noether's theorem.

\index{Noether-Sussmann theorem}

\index[n]{zv@$\phi$, cost integrand}
\index[n]{f@$f$, vector field}
\index[n]{U@$U$, control set}
\index[n]{G@$G$, Lie group}
\index[n]{G@$G$, Lie group!$\mathfrak{g}$, Lie algebra}
\index[n]{Q@$\mathcal{Q}$, optimal control system} 
\index[n]{M@$M$, manifold!for Noether-Sussmann}
\index[n]{V@$V_i\subset M$, open neighborhood of a manifold}
\index[n]{u@$u$, control!$u_i\in U$}
\index[n]{zy@$\psi$, local auxiliary function or integral!$\psi:V_1\to V_2$, diffeomorphism of open neighborhoods}
\index{infinitesimal group of symmetries}

The version of Noether's theorem which we use for optimal control is
the one described by Sussmann~\cite{sussman}. Before recalling the
statement, we begin with a few definitions.  We let $\mathcal{Q} =
(M,U,f,\phi)$ denote a general optimal control system satisfying the
regularity conditions of Section \ref{sec:PMP}.  Also, we assume
that the cost functional $\phi$ is independent of the control as in
our case. In what follows, we will denote the vector field
$f(u,q) \in T_{q}M$ for $u\in U$ by $f_u(q)$. The definitions below
are all from Sussmann~\cite{sussman}.

\begin{definition}[Symmetry of a control system]\label{def:symmetry}
Let $G$ be a Lie group with Lie algebra $\mathfrak{g}$ and let
$\mathcal{Q} = (M,U,f,\phi)$ be an optimal control
system. A \emph{symmetry} of this optimal control system is a
diffeomorphism $\psi : V_1 \to V_2$ where $V_1,V_2 \subset M$ are
open such that for every $u \in U$ there exist $u_1,u_2 \in U$ for
which $d\psi(q)(f_u(q)) = f_{u_1}(\psi(q))$ and
$d\psi(q)(f_{u_2}(q)) = f_{u}(\psi(q))$ for all $q \in V_1$.
\end{definition}

\index[n]{zt@$\tau$, infinitesimal group of symmetries}
\index[n]{zC@$\Gamma^\infty$, smooth sections of a vector bundle}
\index[n]{X@$X$, Lie algebra element!$X,Y,Z\in\mathfrak{g}$}

\begin{definition}[Infinitesimal Group of Symmetries]\label{def:infinitesimal-symmetry}
An \emph{infinitesimal group of symmetries} of a control system
$\mathcal{Q}$ is a smooth action $\tau
: \mathfrak{g} \to \Gamma^{\infty}(TM)$, which assigns to every Lie
algebra element $X \in \mathfrak{g}$ a smooth vector field on $M$,
such that every diffeomorphism $\exp(t\tau(X))$ is a symmetry of
$\mathcal{Q}$.
\end{definition}

\index[n]{exp@$\exp$, exponential and matrix exponential!solution to ODE}
\index[n]{J@$\mathbf{J}^\tau$, momentum map}
\index[n]{0@$-^*$, dual!linear}
\index[n]{q@$q$, point on manifold}

\begin{definition}[Momentum Map]\label{def:momentum-map}
To an infinitesimal group of symmetries
$\tau:\mathfrak{g}\to\Gamma^{\infty}(TM)$ of an optimal control system
$\mathcal{Q}$, we associate the \emph{momentum map}\index{momentum
  map} $\mathbf{J}^\tau:T^*M\to\mathfrak{g}^*$
\index[n]{J@$\mathbf{J}^\tau$, momentum map} given by
\[
\mathbf{J}^\tau(q,p)(X) = \bracks{p}{\tau(X)(q)}_*,
\]
where $X\in\mathfrak{g}$, $q \in M,\ p\in T_q^*M$ and $\tau(X)(q) \in T_q M$. 
\end{definition}

See Abraham~and~Marsden~\cite{abraham2008foundations} for the general
theory of momentum maps in symplectic geometry.

\index[n]{p@$p$, point in bundle!lifted controlled trajectory}
\index{PMP, Pontryagin Maximum Principle}


\begin{theorem}[Noether-Sussmann Theorem]
Assume that $\mathcal{Q} = (M,U,f,\phi)$ is an optimal control system
as above, let $\mathfrak{g}$ be the Lie algebra of the Lie group $G$
and let $\tau : \mathfrak{g} \to \Gamma^{\infty}(TM)$ be an
infinitesimal group of symmetries of $\mathcal{Q}$. Let $p$ be a
lifted controlled trajectory in $T^*M$ satisfying the Pontryagin
Maximum Principle. Then the function
$\mathbf{J}^\tau:T^*M\to\mathfrak{g}^*$ is constant along the
trajectory $p$.
\end{theorem}

\begin{problem}[Circular Control Problem]
We start with problem \ref{pbm:state-costate-reinhardt} and enlarge
the control set $U_T$ to its circumscribing disk $U_C$. This control
problem is called the \emph{circular control problem}.
\end{problem}
Let us now apply the Noether-Sussmann theorem to the circular control
problem.  Define the angular momentum to be $\A :=
\bracks{\Jsu}{\Lambda_1+\Lambda_R}$.

\index{angular momentum}
\index[n]{A@$\A$, angular momentum}

\begin{theorem}[Angular Momentum]\label{thm:angular-momentum}
We have that the angular momentum $\A=\bracks{J}{\Lambda_1+\Lambda_R}$
is conserved along the optimal trajectory of the circular control
problem.
\end{theorem}

\index[n]{zh@$\theta$, angle}

\begin{proof}
The proof analyzes the $\SO$ symmetry.
Our optimal control system consists of data $(g,X,\Lambda_1,\Lambda_2) \in
T^*(T\SL(\R))$.  Note that our Lie group here is $\SO$ and hence its
Lie algebra is one-dimensional $\mathfrak{so}_2(\R) = J\R$, where 
$J$ is the infinitesimal generator of
rotations. This infinitesimal symmetry gives rise to rotations
$\exp(J\theta) \in \SO$, which in turn give rise to the following
action on our manifold:
\begin{align*}
\SO \times T^*(T\SL(\R)) &\to T^*(T\SL(\R)) 
\\
(\exp(J\theta), (g,X,\Lambda_1,\Lambda_2)) 
&\mapsto (\exp(J\theta)g\exp(-J\theta), 
\Ad_{\exp(J\theta)}X, \Ad_{\exp(J\theta)}\Lambda_1, \Ad_{\exp(J\theta)}\Lambda_2),
\end{align*}
where the action on $g$ is by inner automorphisms and the rest are
given by the adjoint action on each copy of $\sl(\R)$. (Note that
throughout, we make the identification $\sl(\R)^* \cong \sl(\R)$ via
the nondegenerate trace form.)  These are \emph{symmetries} by the
proof of Theorem \ref{thm:symmetries} and also since the rotation
action on the control matrix $Z_u$ is
\[
Z_u \mapsto \Ad_{\exp(J\theta)} Z_u \in U_C.
\]
The momentum map is computed
by the canonical pairing between the
costate variables $(\Lambda_1,\Lambda_2) \in T_{(g,X)}^*(T\SL(\R))$
and tangent vectors in $T_{(g,X)}(T\SL(\R))$ giving the infinitesimal
rotation action. The first component of this tangent vector is given
by
\begin{align*}
\frac{d}{d\theta}\exp(J\theta)g\exp(-J\theta)\Bigr|_{\theta=0} 
&= \underbrace{Jg - gJ}_{\in T_g\SL(\R)} 
\\
&=\underbrace{g}_{\in \SL(\R)}\underbrace{(g^{-1}Jg - J)}_{\in \sl(\R)}.
\end{align*}
Thus, we can identify the first component with $g^{-1}Jg - J$ in the
Lie algebra $\sl(\R)$. The second component is given by
\[
\frac{d}{d\theta}\Ad_{\exp(J\theta)}X = \ad_J X = [J,X],
\]
which is already in the Lie algebra $\sl(\R)$. 
So, putting all this together, we obtain the momentum map.
\begin{align*}
\mathbf{J}^\tau\left((g,\Lambda_1,X,\Lambda_2) \right) 
&= \bracks{(\Lambda_1,\Lambda_2)}{(\Ad_{g^{-1}}J - J, [J,X])}
\\
&= \bracks{\Lambda_1}{\Ad_{g^{-1}}J - J} + \bracks{\Lambda_2}{[J,X]} 
\\
&= \bracks{\Lambda_1}{\Ad_{g^{-1}}J} -\bracks{\Lambda_1}{J} - \bracks{[\Lambda_2,X]}{J} 
\\
&= \bracks{\Lambda_1}{\Ad_{g^{-1}}J} - \bracks{J}{\Lambda_1 + \Lambda_R},
\end{align*}
where, as usual, $\bracks{\cdot}{\cdot}_*$ denotes the natural pairing
between a vector space and its dual.  Note however, from
Corollary \ref{cor:lam1-gen-sol}, that $\Lambda_1(t) =
g^{-1}\Lambda_1(0)g(t)$ and so $\bracks{\Lambda_1}{\Ad_{g^{-1}}J}$ is
a constant.

Thus, by the Noether-Sussmann theorem
$\bracks{J}{\Lambda_1+\Lambda_R}$ is a constant of motion along the
optimal trajectory of the circular control problem.
\end{proof}

\begin{remark}\normalfont \leavevmode
\begin{itemize}
\item 
Since it is the conserved quantitiy arising from a rotational
symmetry, $\bracks{J}{\Lambda_1 + \Lambda_R}$ will be called
the \emph{angular momentum}.
\item 
The spurious constant in the expression for the momentum map is a
consequence of the action of $\SO$ on $\SL(\R)$ by inner
automorphisms. This means that we could also modify the action to be
$g \mapsto g \exp(-J\theta)$ to obtain a valid conserved quantity.
\item 
We also obtain the same conserved quantity for a control set $U_r$ which
is a disk of any radius $r$ whenever $r^2 > 1/3$.
    (See Definition \ref{def:circ-control-set}.) 
\end{itemize}
\end{remark}
\index[n]{r@$r$, real number!radius}

Conserved quantities are very useful since they introduce constraints
and cut down the dimension of the problem. But an immediate
application is that they give us a constraint on the optimal control
matrix in terms of the state-costate variables, which is what we aim
to derive. But before we do, we make a quick detour to understand the
structure of the control sets when viewed in the Lie algebra $\su$
of a unitary group.

\index[n]{SU@$\mathrm{SU}(1,1)$, special unitary group!$\su$, Lie algebra}

\section{Control Sets in the Special Unitary Group}\label{sec:su11-control-sets}

\index{Cayley transform}
\index[n]{Cayley@$\mathrm{Cayley}$, Cayley transform}

The special unitary group $\mathrm{SU}(1,1)$ and its
Lie algebra $\su$ are introduced in Appendix~\ref{sec:Lie}.
From that appendix, we know that the Lie algebras
$\su$ and $\sl(\R)$ are isomorphic. This isomorphism is given by the
Cayley transform.
\[
\mathrm{Cayley}(\sl(\R)) = \su.
\]
Under the Cayley transform to $\su$, the image of the control matrix $Z_u$, given in
equation \eqref{eqn:Z0}, is
\index[n]{zf@$\zeta=\exp(2\pi{i}/3)$, cube root of unity}
\begin{equation}\label{eqn:Z0-su11}
\frac{1}{3}\mattwo{-i}{2(u_0 + \zeta u_1 + \zeta^2 u_2)}
{2(u_0 + \zeta^2 u_1 + \zeta u_2)}{i},
\end{equation}
where $\zeta=\exp(2\pi i/3)$ is a primitive cube root of unity.

\index[n]{0@$\partial$, boundary}
\index[n]{U@$U$, control set!$U_r$, intermediate}

Now, let $\partial U_r$ denote the boundary of the disk control set
$U_r$ of radius $r$ in $\R^3$.
\[
U_r := \left\{(u_0,u_1,u_2)~\mid
u_0 + u_1 + u_2 = 1,~u_0^2+ u_1^2 + u_2^2 \le r^2 \right\}.
\]
The following lemma gives a simplification of the control matrix.

\begin{lemma}\label{lem:circle-plane-bijection}
The set $\partial U_r$ and the circle 
\[
\left\{z\in \C\mid|z|^2 
= \left(\frac{3r^2-1}{2}\right)\right\}
\]
are in bijection by
the map $(u_0,u_1,u_2)\mapsto u_0 + \zeta u_1 + \zeta^2 u_2$.
\end{lemma}

\index[n]{zP@$\Pi$, affine plane}
\index[n]{L@$L$, affine function}
\index[n]{yz@$z$, optimal control}

\begin{proof}
Consider the affine plane $\Pi = \left\{(u_0,u_1,u_2)\mid u_0 + u_1 +
u_2 = 1 \right\}$ and consider the map $L : \Pi \to \C $,
defined by $L(u_0,u_1,u_2) := (u_0 + \zeta u_1 + \zeta^2 u_2)$. This
map is the restriction of the linear map
\[
(u_0,u_1,u_2) \mapsto 
\begin{pmatrix}
{1}&{1}&{1}\\{1}&{\zeta}&{\zeta^2}\\{1}&{\zeta^2}&{\zeta}
\end{pmatrix}
\begin{pmatrix}
u_0\\u_1\\u_2
\end{pmatrix}
= 
\begin{pmatrix}1\\z\\\bar{z}\end{pmatrix},
\]
which has non-zero determinant and so $L$ is an isomorphism of affine
planes. This isomorphism restricts to a bijection between the circles
$\partial U_r$ and 
\[
\{z\in \C\mid|z|^2=\left(\frac{3r^2-1}{2}\right)\}
\]
since
\begin{align*}
|(u_0 + \zeta u_1 + \zeta^2 u_2)|^2 &= \frac{3}{2}(u_0^2 + u_1^2 + u_2^2) 
- \frac{(u_0+u_1+u_2)^2}{2}
\\
&= \frac{3r^2 - 1}{2}.
\end{align*}
\end{proof}

\index[n]{za@$\alpha,\beta$, control matrix parameters}

The above lemma shows that if the control set is $U_r$, then we can
take $Z_u$ in general to be
\begin{align}\label{eqn:general-Z0}
\begin{split}
Z_u &=\mattwo{-i \alpha}{\beta z}{\beta \bar{z}}{i \alpha} \in \su,\quad
\text{where} 
\\
|z| &=1,\quad\text{ and } \beta>0,
\end{split}
\end{align}
\index[n]{za@$\alpha,\beta$, control matrix parameters}
where $z\in\C$ and $\alpha,\beta\in \R$.  Also, $\beta z$ gives the
polar coordinate decomposition of upper-right matrix entry of
$Z_u$. With this notation, $\alpha = \frac{1}{3}$ and 
\(
\beta=\frac{2}{3}|u_0+\zeta u_1 + \zeta^2 u_2|
\). 
We obtain
\begin{equation}\label{eqn:det-Z0}
\det(Z_u) = (\alpha^2 - \beta^2) = \frac{1 - 4\left(\frac{3r^2-1}{2}\right)}{9} = \frac{1 - 2r^2}{3}.
\end{equation}
Later chapters will use the parameter $\rho$ instead of radius $r$, where
\[
\rho:=\beta/\alpha,\quad \rho^2 = 6r^2 - 2.
\]

\index[n]{zr@$\rho=\beta/\alpha >0$, control parameter}

Table \ref{tab:control-sets} shows the square of the radii $r^2$ of
the circumscribing disk, inscribed disk and center of the control
simplex $U_T$ and their corresponding radii when viewed as disks in
the complex plane (following
Lemma \ref{lem:circle-plane-bijection}). \newline

\begin{table}[htbp]
\centering
\begin{tabular}{|c|c|c|c|c|c|}
\hline\rule{0pt}{3ex} 
        $\det(Z_u)$ & Relation to $U_T$& $r^2$ 
&Radius $(3r^2-1)/2$ in $\C$& $(\alpha,\beta)$ & $\rho=\beta/\alpha$ 
\\[1ex]
\hline\rule{0pt}{3ex} 
        $-1/3$ & Circumscribing disk $U_C$ & 1 & 1 & (1/3,2/3) & 2
\\[1ex]
        $0$ & Inscribed disk $U_I$ & $1/2$ & 1/4 & (1/3,1/3) & 1
\\[1ex]
        $1/9$ &  Center $(1/3,1/3,1/3)$& $1/3$ & 0 & (1/3,0) & 0
\\[1ex]
\hline
\end{tabular}
\caption{Various control sets and their radii.}
\label{tab:control-sets}
\end{table}

\section{Quadratic Equation for Optimal Control}\label{sec:quadratic-equation}
Henceforth, we let $Z^*$ denote the optimal control matrix for the costate
equations \ref{pbm:state-costate-reinhardt} (with the control
set $U_r$). It depends on a complex control variable $z$.
We derive a constraint on $Z^*$
from angular momentum conservation.

We begin with a few lemmas.

\begin{lemma}\label{lem:boundary-circle-control}
The set of maximizers of the Hamiltonian (considered as a function of
the control) when the control set is the disk $U_r$ is either the
entire disk or just a point on $\partial U_r$.
\end{lemma}
\begin{proof}
This is a direct corollary of Lemma \ref{lem:control-face-lemma}.
\end{proof}

\begin{theorem}
Let $Z^*$ be the optimal control matrix for the costate equations
\eqref{pbm:state-costate-reinhardt} (with the control set $U_r$). We
then have
\[
\left \langle \Jsu, \ad_{Z^*}\frac{\delta}{\delta Z^*}
\frac{\langle \Lambda_R, Z^*\rangle}{\langle X, Z^*\rangle} 
\right\rangle = 0,
\]
\index{functional derivative}
\index[n]{zd@$\delta/\delta X$, functional derivative}
where $\frac{\delta}{\delta Z^*}$ is the functional derivative with
respect to $Z^*$.%
\footnote{Note that the $\Jsu$ in this theorem is
an element of $\su$ and is given by by Cayley transform of $J$.
\[
\Jsu =
\mathrm{Cayley}\mattwo{0}{-1}{1}{0} = \mattwo{-i}{0}{0}{i}.
\]
}
\end{theorem}

\index[n]{P@$P$, normalized control matrix!$P^*$, optimal}

\begin{proof}
We give two proofs of this fact.  For the first proof, we
differentiate the angular momentum (in
Theorem \ref{thm:angular-momentum}) with respect to time, to obtain
the following.
\begin{align}
    0 = \langle \Jsu, \Lambda_1' + \Lambda_R' \rangle 
&= \left\langle \Jsu, [\Lambda_1, X] + [P^*,\Lambda_R] + 
\left[-\Lambda_1 + \frac{3}{2}\lambda_{cost} \Jsu, X \right] - 
\langle P^*,\Lambda_R \rangle [P^*,X] \right\rangle 
\nonumber 
\\
    &=\left\langle \Jsu, [P^*,\Lambda_R] -
 \langle P^*,\Lambda_R\rangle [P^*,X] \right\rangle  
\label{eqn:deriv-ang-momentum}
\\
    &= \left \langle \Jsu, \ad_{Z^*}\frac{\delta}{\delta Z^*}
\frac{\langle \Lambda_R, Z^*\rangle}{\langle X, Z^*\rangle} \right\rangle \qquad 
\text{(by Corollary \ref{cor:control-dep-ad-alternate}),} \nonumber
\end{align}
where $P^* = Z^*/\bracks{Z^*}{X}$.

For the second proof, note that from
Lemma \ref{lem:boundary-circle-control}, the Hamiltonian is maximized
at a point on $\partial U_r$. By the form of the control matrix
$Z^*$, on the boundary of $U_r$ we obtain two constraints.
\[ 
\bracks{\Jsu}{Z^*} = \mathrm{constant}, \qquad \bracks{Z^*}{Z^*} = \mathrm{constant}.
\]
The Hamiltonian maximization of the maximum principle can be
considered as a constrained maximization problem subject to the above
two constraints. The functional derivatives of the above two
constraints are $\Jsu$ and $2Z^*$ respectively. By Lagrange
multipliers, we find the derivative $\delta \H/\delta Z^*$ should lie
in the span of the derivatives $\Jsu$, $Z^*$ of the
constraints. That is, $\ad_{Z^*}(\frac{\delta\H}{\delta Z^*})$ is in
the span of $\ad_{Z^*}\Jsu$. Thus,
\[
\bracks{\ad_{Z^*}\frac{\delta \H}{\delta Z^*}}{\Jsu} = 0,
\]
which gives us the required. 
\end{proof}
We also have the following result, which shows that the angular
momentum and the Hamiltonian are in involution with respect to the
Poisson bracket\index{Poisson!bracket} $\{\cdot,\cdot\}_{ex}$ on the
extended state space $T^*T\SU$.
\index[n]{0@$\{-,-\}$, Poisson bracket!$-_{ex}$, extended}
\index{angular momentum}
\index[n]{A@$\A$, angular momentum}

\index{Poisson!commuting}
\begin{proposition}
The angular momentum $\A$ Poisson
commutes with the Hamiltonian on $T^*(T\SU)$, provided the control
set is rotationally invariant.
\end{proposition}
\index[n]{F@$F,G$, smooth functions}
\begin{proof}
We use the Poisson bracket on the manifold $T^*(T\SU)$ derived in
Section \ref{sec:poisson-bracket}, which we recall here. If $F,G$ are
two left-invariant smooth functions on $T^*T\SU)$, their extended
space Poisson bracket is
\[
\{F,G\}_{ex} := \left\langle \Lambda_1, \left[ \frac{\delta F}{\delta\Lambda_1}, 
\frac{\delta G}{\delta\Lambda_1} \right] \right\rangle
 + \bracks{\frac{\delta F}{\delta X}}{\frac{\delta G}{\delta \Lambda_2}} 
-  \bracks{\frac{\delta F}{\delta \Lambda_2}}{\frac{\delta G}{\delta X}}.
\]
We compute using Theorem~\ref{thm:reinhardt-poisson}
that
\begin{align*}
\{\A,\H\}_{ex} 
&=\{\bracks{\Jsu}{\Lambda_1+\Lambda_R},\H\}_{ex}
=\bracks{\Jsu}{\{\Lambda_1+\Lambda_R,\H\}_{ex}}
\\
&=\bracks{\Jsu}{\Lambda_1'+\Lambda_R'}=0,
\end{align*}
where $\Lambda_1'$ and $\Lambda_R'$ are shorthand for 
the expressions on the right-hand side
of the ODEs for $\Lambda_1$ and $\Lambda_R$.
\end{proof}
\index[n]{P@$P$, normalized control matrix!$P^*$, optimal}

Simplifying equation \eqref{eqn:deriv-ang-momentum} above, we obtain a
more symmetric expression for the optimal control matrix, which is
homogeneous in $Z^*$.
\begin{equation}\label{eqn:symm}
\langle \Jsu,[Z^*,X]\rangle\langle Z^*,\Lambda_R\rangle 
= \langle \Jsu,[Z^*,\Lambda_R]\rangle{\langle Z^*,X\rangle}.
\end{equation}

This is the same as saying
\begin{equation}
\dettwo{ \langle [\Jsu,Z^*],X]\rangle}{\langle [\Jsu,Z^*],
\Lambda_R\rangle}{\langle Z^*,X\rangle}{\langle Z^*,\Lambda_R\rangle} = 0.
\end{equation}

\begin{proposition}
The conservation of angular momentum gives the following constraint on
the optimal control matrix.
\begin{equation}\label{eqn:opt-control-constraint}
\langle Z^*,[\Lambda_R,X] \rangle 
= \frac{\langle Z^*, Z^* \rangle}
{\langle Z^*,\Jsu \rangle}\langle \Jsu, [\Lambda_R,X]\rangle.
\end{equation}
\end{proposition}
\begin{proof}
\index[n]{X@$X$, Lie algebra element!$X,Y,Z,W\in\sl$}
From \eqref{eqn:symm} and Proposition \ref{prop:trace-quotient-sub}
specialized to $[\Jsu,Z^*]$, $Z^*$, $X$, and $\Lambda_R$, we
obtain
\mcite{MCA:5358122 ad-squared}
\begin{equation}\label{eqn:ang-mom-quotient}
0=\langle [\Jsu,Z^*],\Lambda_R \rangle\bracks {Z^*} {X} 
- \langle Z^*,\Lambda_R \rangle\bracks {[\Jsu,Z^*]} {X} 
= -\frac{\left\langle [[\Jsu,Z^*],Z^*], 
[\Lambda_R, X] \right\rangle}{2}.
\end{equation}
For any $Z^*,\Jsu\in\sl(\C)$, we have  that
\mcite{MCA:5358122 ad-squared}
\[
[Z^*,[Z^*,\Jsu]]=
(\ad_{Z^*})^2\Jsu = 2\bracks{Z^*}{Z^*}\Jsu - 2\bracks{Z^*}{\Jsu}Z^*.
\]
If we substitute this equation back into \eqref{eqn:ang-mom-quotient},
we obtain the result.
\end{proof}

\index{weighted determinant}
\index[n]{zL@$\Lambda$, costate!$\Lambda\in\sl(\C)$}
\index[n]{l@$l_{ij}$, matrix coefficients of $\Lambda$}
\index[n]{det@$\det(\Lambda,\alpha,\beta)$, weighted determinant}
\index[n]{za@$\alpha,\beta$, control matrix parameters}


\begin{definition}[Weighted Determinant]
For 
\[
\Lambda = \mattwo{l_{11}}{l_{12}}{l_{21}}{-l_{11}} \in \sl(\C)
\] 
and $\alpha,\beta\in\R$, define the weighted determinant
\[
\det(\Lambda,\alpha,\beta) := -(\alpha^2 l_{11}^2 + \beta^2 l_{12} l_{21}).
\]
\end{definition}
If $\Lambda\in\su$ and $\alpha^2,\beta^2\in\R$, then
$\det(\Lambda,\alpha,\beta)\in\R$.  The ordinary determinant has
weights $\alpha =\beta=1$.
\[
\det(\Lambda,1,1)=\det(\Lambda).
\]



\index[n]{yz@$z$, optimal control}

\begin{proposition}\label{prop:control-quadratic}
The optimal control \eqref{eqn:general-Z0} in the circular control case
$U_r$ is given by the root $z$ of the quadratic equation
\eqref{eqn:quadratic-equation}.
\end{proposition}
\begin{proof}
We know by Lemma \ref{lem:boundary-circle-control} that the optimal
control matrix $Z^*$ takes values in the boundary of the
disk. By equation \eqref{eqn:general-Z0}, we can take
\begin{equation}\label{eqn:general-boundary-Z0}
Z^* = \left(
\begin{matrix}
-i\alpha & \beta z 
\\
{\beta}/{z} & i \alpha
\end{matrix}
\right),
\end{equation}
where $|z|=1$.
With this notation, \eqref{eqn:opt-control-constraint} becomes
\begin{equation}\label{eqn:opt-control-expl}
\langle Z^*,[\Lambda_R,X] \rangle 
= \frac{(\alpha^2 - \beta^2)}{\alpha}\langle \Jsu, [\Lambda_R,X]\rangle,
\end{equation}
where $\alpha, \beta \in \R$ and $\alpha > 0$. 

Simplifying this, we obtain the following quadratic equation in $z$
\begin{equation}\label{eqn:quadratic-equation}
\alpha[\Lambda_R,X]_{21} z^2 - 2i\beta [\Lambda_R,X]_{11}z + \alpha [\Lambda_R,X]_{12} = 0,
\end{equation}
where the subscripts index matrix entries.
We then solve for $z$ to obtain two roots
\begin{equation}\label{eqn:z-quad-roots}
z_\pm = i\left(\frac{\beta[\Lambda_R,X]_{11} 
\pm  \sqrt{-\det([\Lambda_R,X],\alpha,\beta)}}
{\alpha[\Lambda_R,X]_{21}}\right).
\end{equation}


\end{proof}

This determines the optimal control explicitly as a function of the
state-costate variables. We see how modifying the control set to be
more symmetrical has resulted in a conserved quantity which has in
turn given us valuable information about the optimal control.

\newpage

\chapter{Hyperboloid Coordinates}

\index{Cayley transform}
\index[n]{X@$X$, Lie algebra element!$X,Y,Z,W\in\sl$}
\index[n]{yz@$z\in\C$}
\index[n]{t@$t$, real number}
\index[n]{abcd@$a,b,c,d$, matrix entries}

\section{Coordinates}\label{sec:coordinate}
We begin with the following lemma.
\begin{lemma}\label{lem:hyperboloid-coords}
The Cayley transform of
\begin{equation}
\mattwo{a}{b}{c}{-a}\in\sl(\R)\quad\text{is}\quad
\mattwo{it}{z}{\bar{z}}{-it} \in \su,
\end{equation}
where $z = (b+c)/{2} + i a\in\C$ and $t=(b-c)/2\in\R$.
\end{lemma}
\begin{proof} This is an easy calculation.
\end{proof}
The matrices in $\su$ with a given determinant $d=\det(X)$ are in
bijective correspondence with the points on a hyperboloid
\[
\mattwo{i t}{z}{\bar{z}}{-it}\in\su \leftrightarrow \{(t,z)\in\R\times\C\mid t^2 -|z|^2 = d\}.
\]
This observation justifies our nomenclature of \emph{hyperboloid
coordinates}, for the Cayley transform of the $\sl$ coordinate system.
A matrix in $\su$ is determined by its determinant and the complex
number $z$, up to the ambiguity in the sign of $t$.

\index{hyperboloid!coordinate}
\index[n]{0@$\abss{-}_\epsilon = \sqrt{\epsilon + -^2}$}
\index[n]{0@$\abss{-} = \abss{-}_1$}
\index[n]{R@$\RR(-,-)$, sesquilinear form}
\index[n]{ze@$\epsilon\in\{-1,0,1\}$, sign!$\epsilon\in\{-1,0,1\}$, sign $\det(\Lambda_1)$}
\index[n]{yz@$z\in\C$}

\begin{notation} \normalfont \leavevmode
\begin{itemize}
\item For a complex number $z$, and $\epsilon\in\{-1,0,1\}$,
we set
$\abss{z}_\epsilon := \sqrt{\epsilon + |z|^2}$,
and $\abss{z}:=\abss{z}_1 = \sqrt{1+|z|^2}$.
\item  We write
\begin{equation*}
\RR(z_1,z_2) := \Re(\bar{z}_1z_2)\quad\text{(real part)}
\end{equation*}
to denote the $\R$-bilinear form on $\C$, derived from the real part
of a complex number.
\end{itemize}
\end{notation}

\index[n]{J@$\Jsu=\mathrm{diag}(-i,i)$, Cayley transform of $J$}
\index[n]{d@$\op{diag}$, diagonal matrix}
\index[n]{w@$w$, hyperboloid coordinate!$w,b,c$ coordinates}
\index{hyperboloid!coordinate}
\index[n]{d@$d$, determinant!$d_1$,\ $\det(\Lambda_1)=d=\epsilon d_1^2$}

We transform the ODE for $X,\Lambda_1,\Lambda_R\sl(\R)$ into ODEs
given by hyperboloid coordinates.  We now work consistently with the
Cayley transformed version of $X,\Lambda_R$, and $\Lambda_1$ in $\su$.
The ODEs from the costate equations~\ref{pbm:state-costate-reinhardt}
retain exactly the same form, except that $J$ and $Z^*$ must be
replaced with their Cayley transforms $\Jsu=\op{diag}(-i,i)$ and
$Z_{\su}\in\su$.  Using these, our assumptions, and Lemma
\ref{lem:hyperboloid-coords}, we can write
\begin{equation}\label{eqn:st-costate-hyperboloid}
    X := \mattwo{-i\abss{w}}{w}{\bar{w}}{i\abss{w}}, \quad
\Lambda_1 := d_1 \mattwo{i\abss{b}_\epsilon}{b}{\bar{b}}{-i\abss{b}_\epsilon}, \quad
\Lambda_R := \mattwo{-i\frac{\RR(c,w)}{\abss{w}}}{c}{\bar{c}}{i\frac{\RR(c,w)}{\abss{w}}},
\end{equation}
for variable complex numbers $w,b,c \in \C$, with $|b|^2\ge-\epsilon$.  The form of the
expressions ensure that the constraints $\det(X)=1$ and
$\bracks{X}{\Lambda_R} = 0$ are satisfied.  The parameters
$\epsilon\in\{-1,0,1\}$ and $d_1\in\RR$ are constants of motion, and
$\det(\Lambda_1)=\epsilon{d_1^2}$.  The sign of the upper-left matrix
entry $-i\abss{w}$ of $X$ is determined by the sign convention on the orbit
$\O_J$, described in Lemma~\ref{lem:X-hyperbolic}.  Note that
$(w,\abss{w})$ lies on the upper sheet of the hyperboloid
\(
\{(w,t)\in\C\times\R\mid t^2-|w|^2=1\}
\), 
which justifies our nomenclature.

\index{nilpotent cone}
\index{regular!nilpotent class}
\index{conjugacy class}
\index{regular!semisimple element of $\sl$}
\index{Cartan subalgebra}
\index{split!semisimple element of $\sl$}
\index{elliptic semisimple element of $\sl$}

\begin{remark}\label{rem:wbc-exceptions}
The form of element $\Lambda_1$ is general enough to represent a
general element of $\su$.  We take $d_1\in\R^\times$, except when
$\Lambda_1=0$.  If $\det(\Lambda_1)=0$, then $\Lambda_1$ belongs to
the nilpotent cone, which consists of three conjugacy classes: the
vertex of the cone (the zero element) and the positive and negative
cones (the two regular nilpotent classes).  We can take
$d_1=\epsilon=b=0$ (for the zero element) or $(d_1,\epsilon)=(\pm1,0)$
(for the regular nilpotent elements).  In the case of the zero
element, we omit the ODE for $b$.  When $\Lambda_1$ is regular
semisimple ($\det(\Lambda_1)\ne0$), the scalar $d_1$ parameterizes
nonzero elements within a given Cartan subalgebra.  The hyperboloid
has one or two sheets, depending on whether $\Lambda_1$ is split
($\det(\Lambda_1)<0$), or elliptic $\det(\Lambda_1)>0$.  In the split
case, the coordinate system breaks down around the neck $|b|=1$ of the
hyperboloid, and a better coordinate system is introduced in
Remark~\ref{rem:r-theta-hyperboloid}.
\end{remark}
  
\index[n]{zr@$\rho=\beta/\alpha >0$, control parameter}

In the case of circular control, there are two real parameters
$\alpha,\beta$ such that $\det(Z^*)=\alpha^2-\beta^2$.  See Table
\eqref{tab:control-sets} 
and Equation \eqref{eqn:det-Z0}.  We assume in
this chapter that $\alpha,\beta$ are both positive, and we set
$\rho  = \beta/\alpha >0$.  In the case of circular control, we
assume the star condition in the form $\bracks{X}{Z^*}<0$.  For a
control matrix given by equation \eqref{eqn:general-boundary-Z0} (which
is already in $\su$ by the results of
Section \ref{sec:su11-control-sets}) and using
Lemma \ref{lem:sl2-star-condition}, the star condition in these
coordinates can be expressed as
\begin{equation*}
\mu(w,z) := \abss{w}-\rho\,\RR(w,z) 
=  -\frac{\bracks{X}{Z^*}}{2\alpha}  > 0.
\index[n]{zmu@$\mu:\C^2\to\R$, star denominator}
\index[n]{0@$-^*$, special value!optimizer}
\index[n]{yz@$z$, optimal control!$z^*$}
\end{equation*}
We consider the system on a region slightly larger than the star
domain defined by the condition $\mu(w,z^*)>0$ (and $c\ne 0$), where
$z^*$ is the maximizing root of the quadratic equation for the
control.

\section{Hyperboloid ODE}

The ODEs in Lie algebra coordinates for $X,\Lambda_1$ and $\Lambda_R$
were derived in Sections \ref{sec:costate-variables}
and \ref{sec:lie-algebra-dynamics}. We now write them out in the
hyperboloid coordinates we have defined above. This will enable us to
better understand the dynamics near the singular locus.

\index[n]{w@$w$, hyperboloid coordinate!in $\C$}
\index[n]{b@$b,c$, hyperboloid coordinates}
\index[n]{yz@$z$, optimal control!$z^*$}
\index[n]{znxi@$\aa_i$, quadratic control equation coefficients!$\aa_0=2+\vert w\vert^2 - (w\bar{c}/\vert c\vert)^2$}

\begin{theorem}[ODE in Hyperboloid Coordinates]\label{thm:ode-hyperboloid}
In hyperboloid coordinates, the dynamics for $X,\Lambda_1$ and
$\Lambda_R$ take the form
\begin{align}
    w' &= i\frac{w - \rho  \abss{w} z^*}{\mu(w,z^*)}, 
\label{eqn:eqn-w} \mcite{MCA:eqn-w}
\\
    b' &= 2i\left(\abss{b}_\epsilon w +b\abss{w}\right), 
\label{eqn:eqn-b} \mcite{MCA:eqn-b}
\\
c' &= 
\frac{i(1-\rho ^2) \Re(c\aa_0(w,c),z^*)}{2\abss{w}\mu(w,z^*)^2}z^*
\nonumber
\\
&\quad- i\left((2d_1\abss{b}_\epsilon+3\lambda_{cost})w + 2bd_1\abss{w})\right).
\label{eqn:eqn-c} \mcite{MCA:eqn-c}
\end{align}
where $z^*$ is the Hamiltonian maximizing root of the quadratic
equation for the optimal control, and
\(
\aa_0=\aa_0(w,c) := 2+|w|^2- (w\bar{c}/|c|)^2.
\)
Here $w,b,c$ are complex valued functions, satisfying the
restrictions of Section~\ref{sec:coordinate}.
\mcite{MCA:eqn-w}
\mcite{MCA:eqn-b}
\mcite{MCA:ham-wbc} 
\mcite{MCA:eqn-c}
\mcite{MCA:eqn:wc'}
\end{theorem}

\begin{proof}
This is an elementary computation using the hyperboloid form of the
state and costate equations as in
equation \eqref{eqn:st-costate-hyperboloid} and substituting them into
the state and costate ODE derived in
Sections \ref{sec:lie-algebra-dynamics}
and \ref{sec:costate-variables}.
\end{proof}

\index[n]{H@$\H$, Hamiltonian!in hyperboloid coordinates}
\index[n]{A@$\A$, angular momentum! in hyperboloid coordinates}

Similarly, we have expressions for the Hamiltonian and the angular
momentum.
\mcite{MCA:ang-wbc}
\begin{align}
\A &= 2 d_1\abss{b}_\epsilon - 2 \frac{\RR(w,c)}{\abss{w}} 
\label{eqn:ang-wbc}
\\
\H&=(2d_1\,\Re(w,b) + (2 d_1\abss{b}_\epsilon+3\lambda_{cost})\abss{w})
\ -\frac{\Re(w-\rho \abss{w}z^*,c)}{\mu(w,z^*)\abss{w}}
\label{eqn:ham-wbc}
\end{align}
\mcite{MCA:ham-wbc}

\begin{corollary}  The following additional ordinary differential
equations hold.
\begin{align}
\abss{w}' &= \frac{\Re(w',w)}{\abss{w}} 
= \rho  \frac{\Re(iw,z^*)}{\mu(w,z^*)},
\label{eqn:abss{w}'}
\\[2ex]
\abss{b}_\epsilon' &= \frac{\Re(b',b)}{\abss{b}_\epsilon} =  2\Re(iw,b),\quad
\label{eqn:abss{b}'}
\\[2ex]
\mcite{MCA:eqn:wc'}
\left(\frac{\Re(w,c)}{\abss{w}}\right)' &= 2d_1 \Re(i w,b).
\label{eqn:wc'}
\end{align}
\mcite{MCA:wc'}
Moreover, the overall form of the ODEs is pseudo-linear.
\begin{equation}\label{eqn:A}
\begin{pmatrix}w\\b\\c\end{pmatrix}' = 
i A \begin{pmatrix}w\\b\\c\\z^*\end{pmatrix}
\end{equation}
\index[n]{A@$A$, matrix or linear map!$3\times 4$ system of ODEs}
\index[n]{zmu@$\mu:\C^2\to\R$, star denominator!$\mu^*=\mu(w,z^*)$}
where $A$ is a $3\times4$ matrix with real (rotationally invariant)
entries.
\begin{equation}A=
\begin{pmatrix}
{1/\mu^*} & 0 &0 & -{\rho  \abss{w}}/{\mu^*}
\\
2\abss{b}_\epsilon & 2\abss{w} & 0 & 0
\\
-3\lambda_{cost}-2d_1\abss{b}_\epsilon & -2d_1 \abss{w} &0 & 
{\dfrac{(1-\rho ^2)\Re(c\aa_0(w,c),z^*)}{2\abss{w}\mu^{*2}}}
\end{pmatrix},
\end{equation}
where $\mu^*:=\mu(w,z^*)$.
\end{corollary}

\begin{proof} 
These equations are direct consequences of
Theorem~\ref{thm:ode-hyperboloid}.
\end{proof}

\begin{remark}\label{eqn:wbc-reverse} 
The equations have a time-reversal symmetry.  If $(w,b,c,z^*)$ is a
solution, then so is $(\tilde{w},\tilde{b},\tilde{c},\tilde{z}^*)$,
where
\[
\tilde{w}(t)=\bar{w}(-t),\quad
\tilde{b}(t)=\bar{b}(-t),\quad
\tilde{c}(t)=\bar{c}(-t),\quad
\tilde{z}^*(t)=\bar{z}^*(-t).
\]
\end{remark}

\begin{remark}  To double-check answers, we rederive the fact
that the angular momentum and Hamiltonian are constant along
trajectories.  To show the constancy of $\A$, we show its derivative
is zero.  This is a direct consequence of \eqref{eqn:abss{b}'} and
\eqref{eqn:wc'}. The direct verification of the constancy of $\H$ is a
tedious but direct calculation.
\end{remark}

\index[n]{r@$r$, real number!radial component of ODE}
\index[n]{zh@$\theta$, angle}
\index{split!semisimple element of Lie algebra}

\begin{remark}\label{rem:r-theta-hyperboloid}
As mentioned in Remark~\ref{rem:wbc-exceptions}, in the split case
$\epsilon=-1$, the coordinate system for $\Lambda_1$ breaks down
around the neck of the hyperboloid. It is helpful to replace $b$ with
better coordinates $(r,\theta)$.  In the split case, we write
\[
\Lambda_1 = d_1 
\begin{pmatrix} i r & \exp(i\theta) \sqrt{r^2+1}\\
\exp(-i\theta) \sqrt{r^2+1} & - i r
\end{pmatrix},
\quad
r\in\R,\quad d_1>0.
\]
and we replace the ODE for $b$ with the system
\begin{align*}
\mcite{MCA:2384598, neck coordinates}
r' &= 2\sqrt{1+r^2}\,\Re(iw,\exp(i\theta)),
\\
\theta' &= 2\abss{w} + \frac{2r\,\RR(w,\exp(i\theta))}{\sqrt{r^2+1}}.
\end{align*}
\end{remark}

\section{Optimal Control}

Assume
$c\ne0$, and set 
\begin{equation}\label{eqn:tilw}
\tilde{w} := \bar{c}w/|c|,\quad \tilde{z} := \bar{c}z/|c|,
\end{equation}
Note $\abss{w}=\abss{\tilde{w}}$ and $\mu(w,z^*)=\mu(\tilde{w},\tilde{z}^*)$.

\index[n]{w@$w$, hyperboloid coordinate!$\tilde{w}=\bar{c}w/\vert c\vert$}
\index[n]{yz@$z$, optimal control!$\tilde{z}=\bar{c}z/\vert c\vert$}
\index[n]{znxi@$\aa_i$, quadratic control equation coefficients!$\aa_0 = 2 + \vert \tilde{w}\vert^2 - \tilde{w}^2$}
\index[n]{znxi@$\aa_i$, quadratic control equation coefficients!${\aa_1} = 2\rho  (\tilde{w}-\bar{\tilde{w}})\abss{\tilde{w}}$}
\index[n]{Q@$Q$, quadratic polynomial}

\begin{lemma}
If $c\ne0$, the optimal control is $z^* = c \tilde{z}^*/|c|$,
where $\tilde{z}^*$ is a root of the quadratic polynomial
\[
Q(\tilde{z}) := 
\aa_0 + \aa_1 \tilde{z} + \aa_2 \tilde{z}^2,
\]
where
\begin{align}\label{eqn:quadratic-z1}\mcite{MCA:z1-quadratic}
\begin{split}
\aa_0 &= \aa_0(\tilde{w}) := 2 + |\tilde{w}|^2 - \tilde{w}^2
=  2+|w|^2- (w\bar{c}/|c|)^2 =: \aa_0(w,c),
\\
{\aa_1} &= {\aa_1}(\tilde{w}) := 2\rho (\tilde{w}-\bar{\tilde{w}})\abss{\tilde{w}},
\\
\aa_2 &= \aa_2(\tilde w) = -\bar{\aa}_0 = -(2+ |\tilde{w}|^2 - \bar{\tilde{w}}^2).
\end{split}
\end{align}
\end{lemma}

\begin{proof}
The lemma follows from the explicit description of the quadratic
equation for the control in Proposition~\ref{prop:control-quadratic}.  

We can give a second proof as follows.
The dependence of the Hamiltonian on $\tilde{z}$ (that is, on the
circular control set $U_r$), comes through the term
\[
-\frac{\Re(w-\rho \abss{w}z,c)}{\mu(w,z)\abss{w}} = 
-\frac{|c|\Re(\tilde{w}-\rho \abss{\tilde{w}}\tilde{z},1)}
{\mu(\tilde{w},\tilde{z})\abss{\tilde{w}}}
=:\H_0(\tilde{z}).
\]
We compute 
\index[n]{Q@$Q$, quadratic polynomial}
\[
\frac{d\H_0(\tilde{z}\exp(i\theta))}{d\theta}|_{\theta=0}=
\frac{-i Q(\tilde{z})\rho |c|}{4\tilde{z} \mu^2\abss{\tilde{w}}},
\]
where $Q$ is the given quadratic polynomial. 
Thus, the derivative vanishes and 
the Hamiltonian is extremal when $Q(\tilde{z}^*)=0$.
\end{proof}
\noindent
Thus, the root $\tilde{z}^*$ is one of the two roots
\[
\frac{({\aa_1} \pm \sqrt{\Delta})}{2\bar{\aa}_0},\quad
\Delta := {\aa_1}^2 + 4|\aa_0|^2.
\]
\index[n]{zD@$\Delta$, discriminant}
Let us now turn to selecting the Hamiltonian maximizing root.

\index{Landau big oh}

\begin{lemma}\label{lem:maximizing-root}
On a punctured neighborhood of the singular locus, the root $\tilde{z}^* = \bar{c}z^*/|c|$ of the
quadratic equation \eqref{eqn:quadratic-z1} 
giving the optimal control satisfies $\tilde{z}^* = 1 + O(|w|)$ 
and $z^* = c/|c| + O(|w|)$.
\end{lemma}
\begin{proof}
Near the singular locus, the coefficients
are $0\ne \aa_0 = 2 + O(|w|^2)$ and $\aa_1 = O(|w|)$. Using this, we find
the discriminant of the quadratic polynomial \eqref{eqn:quadratic-z1} 
is a positive real number and equals
\[
\Delta = 4|\aa_0|^2+{\aa_1}^2 = 16 + O(|w|^2).
\]
Up to a positive scalar, the Hamiltonian \eqref{eqn:ham-wbc} depends
on the roots $\tilde{z}$ of the quadratic through the term
\begin{equation}\label{eqn:ham-oh}
\frac{\rho \abss{w}\Re(z,c)-\Re(w,c)}{|c|\mu(w,z)\abss{w}} 
=\frac{\rho \abss{\tilde{w}}\Re(\tilde{z})-\Re(\tilde{w})}{\mu(\tilde{w},\tilde{z})\abss{\tilde{w}}}
=\frac{\pm\rho \sqrt{\Delta}}{4} + O(|\tilde{w}|).
\end{equation}
The Hamiltonian is not constant in $\tilde{z}$, and $\Delta\ne0$.
The star domain defines a simply connected set given by the inequality
\[
0<\mu(|\tilde{w}|,1) = \abss{\tilde{w}}-\rho  |\tilde{w}|.
\]
If follows that a coherent Hamiltonian-maximizing choice of the sign
of $\pm\sqrt{\Delta}$ can be made throughout the star domain.  The
big-oh estimate \eqref{eqn:ham-oh} shows that the sign of the square
root should be positive.
So, the maximizing root $z^*$ of the quadratic
in \eqref{eqn:quadratic-equation} is
\begin{equation}
    z^* = \frac{c}{|c|}\tilde{z}^* = \frac{c}{|c|}
\frac{({\aa_1} + \sqrt{\Delta})}{2\bar{\aa}_0} =\frac{c}{|c|} + O(|w|),
\end{equation}
which gives the required. 
\end{proof}

\section{Application to Abnormal Solutions}

\index{abnormal extremal}

In this section, we assume that $\rho =1$ (inscribed circular control
set), $\lambda_{cost}=0$ (abnormal solution),
and give a proof of the following theorem.

\begin{theorem} \mcite{MCA:rho-abnormal}\mcite{MCA:5220458} \mcite{MCA:6051339}
If $\rho =1$ and $\lambda_{cost}=0$, then there does not exist a
periodic solution to the control system of ODEs such that $\Lambda_R$
is nowhere zero.
\end{theorem}

The first lemma gives a simple formula for the Hamiltonian
and system of equations.
\begin{lemma}\label{lem:ode-abnormal}
Assume $\rho =1$, $\lambda_{cost}=0$.
Assume $\Lambda_R(t)\ne0$, for all $t$.
Then
\index[n]{K@$K$, conserved matrix}
\index[n]{K@$K$, conserved matrix!$K_{12}\in\C$, matrix entry}
\index[n]{r@$r$, real number!scalar}

\begin{align*}
X' &= -\frac{\left[\Lambda_R,X\right]}
{\sqrt{2 \langle \Lambda_R, \Lambda_R\rangle}},
\\
\Lambda_R' &= \left[\Lambda_R - K, X \right],
\\
\H &= \langle X,K\rangle 
+ \sqrt{\frac{\langle \Lambda_R, \Lambda_R\rangle}{2}},
\end{align*}
where 
\[
K:= 
\begin{pmatrix}
i\A/2&K_{12}\\
\bar{K}_{12}&-i\A/2
\end{pmatrix},
\] 
is a constant. 
\end{lemma}

\begin{proof}
$\Lambda_R\ne0$ implies $c\ne0$.  Let $\tilde{z}^*$ be the Hamiltonian maximizing root of
the quadratic equation \eqref{eqn:quadratic-z1}.  
\index[n]{lhs@$\op{lhs}$, left-hand side}
\index[n]{rhs@$\op{rhs}$, right-hand side}
We claim
\begin{align}\label{eqn:X'L}
\frac{w - \rho \abss{w} z^*}{\mu(w,z^*)}&= -\frac{\aa_0c}{|c|\eta};
\\[2ex]
-\frac{\RR(w-\rho \abss{w} z^*,c)}{\mu(w,z^*)\abss{w}} &= \frac{|c|\eta}{2\abss{w}},
\label{eqn:ham-sqrt-abnormal}
\end{align}
where $\aa_0=2+|w|^2- (w \bar{c}/|c|)^2$ is given by its usual formula, and
where $\eta=\eta(\aa_0,\bar{\aa}_0)=\sqrt{\aa_0 + \bar{\aa}_0}>0$ (noting that $\aa_0+\bar{\aa}_0$ is a
positive real number).
We prove both identities at the same time.  Let $\op{lhs}$ and
$\op{rhs}$ be the left and right-hand sides of the first claimed
identity \eqref{eqn:X'L}.  Combined as a single fraction, the
numerator of $\op{lhs}^2-\op{rhs}^2$ can be written as a polynomial in
$\tilde{z}^*$ with coefficients that are functions of
$\tilde{w},c,\bar{c}$.  By a symbolic calculation in Mathematica, this
polynomial is zero modulo the quadratic relation
\eqref{eqn:quadratic-z1}.  Thus, $\op{lhs}=\pm \op{rhs}$.  We then
substitute $\pm{\op{rhs}}$ for $\op{lhs}$ into the second claimed
identity \eqref{eqn:ham-sqrt-abnormal} and choose the sign that makes
the entire expression positive (because this term is to be maximized
in the Hamiltonian).  With this choice of sign, both identities hold.

We have by direct calculation that
\begin{align*}
\frac{\bracks{\Lambda_R}{\Lambda_R}}{2}= \frac{|c|^2\eta^2}{4\abss{w}^2}>0,
\qquad
\sqrt{\frac{\bracks{\Lambda_R}{\Lambda_R}}{2}}=
\frac{|c|\eta}{2\abss{w}},
\end{align*}
and the rightmost term is equal to the control-dependent term of the Hamiltonian
\eqref{eqn:ham-sqrt-abnormal}.
We compute
\[
X' = r[\Lambda_R,X],\quad \text{where } w' = irc\xi_0/\abss{w}^2.
\]
Comparing with 
the ODE
\eqref{eqn:eqn-w} for $w$, the ODE for $X$ follows.

\index{Mathematica}

The ODEs for $w,b,c$ take
the following form
\begin{align}
    w' &= -i\aa_0c/(|c|\eta),\label{eqn:w-inner}
\\
    d_1 b' &= -c' = 2id_1\left(\abss{b}_\epsilon w +b\abss{w}\right).\label{eqn:c-inner}
\end{align}
Hence, $K_{12} := bd_1+c\in\R$ is a constant.  Then using
conservation of angular momentum,
$K=\Lambda_1+\Lambda_R$ equals the constant as stated in the lemma.
The Hamiltonian \eqref{eqn:ham-wbc} is
\begin{align*}
\H=2d_1(\RR(w,b) + \abss{b}_\epsilon\abss{w}) + \frac{|c|\eta}{2\abss{w}} =
\langle{X,\Lambda_1}\rangle + 
\sqrt{\frac{\langle \Lambda_R, \Lambda_R\rangle}{2}}.
\end{align*}
Expressing these formulas back in terms of $X,\Lambda_R,K$,
we obtain the result.
The ODE for $\Lambda_R$ is obtained by the ODE for $c$ in \eqref{eqn:eqn-c},
by setting $\rho =1$ and $\lambda_{cost}=0$.
\end{proof}

\begin{remark} If we assume the weaker condition $\rho =1$ and
$\Lambda_R\ne0$ then a similar argument shows that the equations take
a related form
\begin{align*}
    g'&=gX, 
\qquad \qquad \qquad \qquad \quad 
\Lambda_1'= [\Lambda_1,X],
\\
    X'&=- \frac{[\Lambda_R,X]}{\sqrt{2\bracks{\Lambda_R}{\Lambda_R}}},
\qquad \qquad 
\Lambda_R' = \left[-\Lambda_1 + \frac{3}{2}\lambda_{cost} \Jsu, X \right].
\end{align*}
\end{remark}

\bigskip
\index[n]{SU@$\mathrm{SU}(1,1)$, special unitary group}
\index[n]{g@$g$, group element!in $\SU$}

\begin{lemma}\label{lem:su11-invariance}
If $K,X,\Lambda_R$ is a solution to the equations of
Lemma~\ref{lem:ode-abnormal}, then so is $\Ad_g K,\Ad_g
X, \Ad_g \Lambda_R$ (where $g \in \SU$) and they also satisfy the
relations listed at the beginning of the section.
\end{lemma}
\begin{proof}
This is simple to verify once we use the fact that the trace form is a
non-degenerate invariant symmetric bilinear form and using properties
of Lie brackets and bilinear forms.
\end{proof}


\begin{lemma}\mcite{MCA:LR-triple}
Assume $\Lambda_R(t)\ne0$, for all $t$.
Assume $\rho =1$, and $\lambda_{cost}=0$. Then
\[
\langle{\Lambda_R,\Lambda_R}\rangle'''=0.
\]
\end{lemma}

\index[n]{d@$d$, determinant!$d_R=\bracks{\Lambda_R} {\Lambda_R}\in\R$}
\index[n]{r@$r$, real number!$r=\bracks {[\Lambda_R,X]} K$}
\begin{proof} 
Let     $d_R := \bracks {\Lambda_R} {\Lambda_R}$. 
The ODE for $\Lambda_R$ gives $d_R' = {2r}$, where
\begin{align*}
    r &:= \bracks {[\Lambda_R,X]} K 
= (c\bar{w}-\bar{c}w)i\A + 
\frac{i}{\abss{w}}(\bar{K}_{12}\aa_0 c - K_{12} \bar{\aa}_0 \bar{c}).
\end{align*}
Using the conserved quantities $K_{12}$ and $\A$, 
the ODE \eqref{eqn:c-inner} can be written
\begin{equation}\label{eqn:c-inner-A}
c' = \frac{i \aa_0 c}{\abss{w}} -i (w \A +2 K_{12} \abss{w}).
\end{equation}
We compute $r'$ by using \eqref{eqn:c-inner-A} 
and \eqref{eqn:w-inner} in
a tedious Mathematica calculation to obtain
\index{Mathematica}
\[
\left(r' + 2\H\left(\frac{|c|\eta}{\abss{w}}- 
(2\RR(w,K_{12})+\A\abss{w})\right)\right)' = 0.
\]
Since $\H=0$ identically, we find that $r'$ is constant, and
$d_R'''=0$ as claimed. (Adding the multiple of $\H$ before
taking the final derivative significantly simplifies the
calculation.)
\end{proof}

\begin{proof}[Proof of Theorem]
Set
\[
d_R = \langle{\Lambda_R,\Lambda_R}\rangle.
\]
By the preceding lemma, $d_R$ is a polynomial in $t$.  If also periodic,
$d_R$ is constant.
By the formulas
derived in the previous proof, we have orthogonality:
\[
d_R'=2\langle{[\Lambda_R,X],K}\rangle =0.
\]
We write $K$ as a linear combination of a basis of
$\su$:
\index[n]{r@$r$, real number!scalar}
\[
K = r_1 X + r_2 \Lambda_R + r_3 [\Lambda_R,X].
\]
Since $K$ and $[\Lambda_R,X]$ are orthogonal, we have $r_3=0$.  The
vanishing Hamiltonian implies $r_1=\sqrt{d_R/8}$.  In particular,
$r_1$ is constant.  Then
\[
\langle{K,K}\rangle = d_R(-1/4 + r_2^2),
\]
and this implies that $r_2$ is constant.  Then using the ODE for $X$
and $\Lambda_R$, we obtain
\[
0= K' = r_1 X' + r_2 \Lambda_R' = (-1/4 + r_2(1-r_2)) [\Lambda_R,X].
\]
\index{horocycle}
This implies that $r_2 =1/2$ and $\langle{K,K}\rangle=0$.  But
$K\ne0$.  Thus, $K$ is regular nilpotent.  We also have that
$\bracks{X}{K}$ is constant.  This is the locus of a horocycle in
hyperbolic space.  Also $X'$, which is proportional to
$[\Lambda_R,X]$, is never zero. Hence $X$ is not periodic, moving
along the horocycle for all time.
\end{proof}

\begin{remark}
We find that nonperiodic abnormal solutions exist.
Following the proof of the lemma, we impose the
condition $\bracks{\Lambda_R}{\Lambda_R}=8$
and set
\(
K=X + \Lambda_R/2
\).
Then $K$ is a nilpotent constant, and the ODE for $X$ becomes
$X'=-[K,X]/2$, which is Lax's equation with constant $-K/2$,
which is easily solved.
\end{remark}

\newpage
\chapter{The Fuller System}

\index{Fuller system}

In this chapter, we restrict to normal solutions, and take $\lambda_{cost}=-1$. 
The singular locus, as defined in Section \ref{sec:singular-locus} is
the region of the cotangent space $T^*(T\SL(\R))$ given by
\index[n]{S@$\Ssing$, singular locus}
\index{singular locus}
\[
\Ssing = \left\{\left(g_0,-\frac{3}{2}J,J,0\right)\mid g_0\in\SL(\R)\right\} 
\subset \SL(\R) \times \sl(\R) \times \sl(\R)\times \sl(\R).
\]

Recall that Pontryagin extremals which avoid the singular locus that
are also edge-extremal are given by bang-bang controls with finitely
many switches. (See Theorem \ref{thm:finite-bang-bang}.)  The global
optimal trajectory of the Reinhardt control problem with the control
set $U_T$ cannot stay within the singular locus for any positive
interval of time. (See Theorem \ref{thm:no-singular-arcs}.)

By Lemma~\ref{lem:constant-not-singular}, if we are considering the
control set $U_T$, for a Pontryagin extremal to approach the singular
locus, the control must switch infinitely many times around the
boundary of $U_T$ in a finite interval of time. This is the
\emph{chattering} phenomenon (see Fuller~\cite{fuller1963study} and
Zelikin and Borisov~\cite{zelikin2012theory}).
\index{chattering}

If we use a control set $U_r$ which has a smooth boundary, we would
expect the optimal control to perform infinitely many rotations along
the boundary $\partial U_r$ to approach the singular locus in finite
time. A system with exactly this behaviour is described in Manita and
Ronzhina and the associated trajectory is
spiral-like~\cite{manita2022optimal}. The results of this section are
motivated by that paper.

These results warrant a study of the behavior of the system near the
singular locus. To this end, we introduce convenient coordinates and
re-express the state, costate and optimal control equations in these
coordinates. Throughout this section, unless otherwise specified, we
work with the circular control sets $U_r$ and we assume that 
the control matrix parameter $\alpha$ is positive.
(See Section \ref{sec:su11-control-sets}).
\index[n]{za@$\alpha,\beta$, control matrix parameters}

We will find that the system of equations we obtain is a special
case of the Fuller system, which we define in the following way.

\index[n]{F@$\mathbb{F}=\R$ or $\C$, archimedean field}
\index[n]{A@$A$, matrix or linear map!nonsingular}
\index[n]{0@$\vvert-\vvert$, norm}
\index[n]{yz@$z_i$, Fuller system!component}
\index[n]{n@$n$, length of Fuller system}
\index[n]{m@$m$, integer dimension}

\begin{definition}[Length-$n$ Fuller system]
Let $\mathbb{F}=\R$ or $\mathbb{F}=\C$.  Let $z_i:(0,\epsilon)\to \mathbb{F}^m$ be functions, where $z_n$
is nonzero on $(0,\epsilon)$, and let $A$
be a nonsingular $m\times m$ matrix with coefficients in $\mathbb{F}$.  
Let $\|\cdot\|$ be the Euclidean
norm on $\mathbb{F}^m$.  The Fuller system of length $n$, (real or complex)
dimension $m$, and multiplier $A$ is the system of ODEs given by
\begin{equation}\label{eqn:fuller-system-len-n} 
    z_n' = z_{n-1}, \quad 
    z_{n-1}' = z_{n-2}, \quad
\dots \quad
    z_2' = z_1, \quad
    z_1' = \frac{Az_n}{\|z_n\|}.
\end{equation}
\end{definition} 

\index[n]{zc@$\gamma\in\C^\times$, Fuller system multiplier}

\begin{remark}\normalfont
By identifying $\C$ with $\R^2$, a Fuller system of complex dimension
$m$ can be written as a Fuller system of real dimension $2m$.  We are
particularly interested in the Fuller system of complex dimension $1$.
In that case $A=\gamma\in\C^\times$, a nonzero complex scalar.  
By scaling each $z_j \mapsto z_j/\gamma$, the constant $\gamma$ in
equation \eqref{eqn:fuller-system-len-n} scales by
$\gamma \mapsto \gamma/|\gamma|$. Thus, there is no loss of generality
in assuming that $|\gamma|=1$.
\end{remark}

\section{Trajectories near the Singular Locus}

An advantage of switching to hyperboloid coordinates is that at
the singular locus, we have $w=b=c=0$.

\begin{assumption}
In general, the determinant $d$ of $\Lambda_1$ may be positive,
negative or zero. The determinant is constant along trajectories.
Different coordinate systems (other than the one presented below) need
to be used when $d=0$ or $d<0$.  Here, we make the assumption $d>0$,
because this is the case for the singular locus, and set
$d_1=\sqrt{d}>0$.  We also assume that the sign of
$\bracks{J}{\Lambda_1}$ is positive, because that is the sign at the
singular locus: $\bracks{J}{3J\lambda_{cost}/2}=3>0$.
The sign $d_1>0$ in $\Lambda_1$ is chosen (according
to our assumptions) to make $\bracks{\Jsu}{\Lambda_1}>0$.
We also have $\epsilon=1$ and $\abss{b}_\epsilon=\abss{b}$, based on the value of
$\epsilon$ at the singular locus.
\end{assumption}

\begin{assumption}
At the singular locus 
\begin{equation}\label{eqn:sing-constants}
\H=0,\quad \A=3,\quad d_1=3/2,\quad\epsilon=1,\quad\lambda_{cost}=-1.
\end{equation}
These values are constant along extremal trajectories. We assume
these values of the constants throughout this section.
We consider an extremal trajectory with the property that $c(t_0)\ne0$
but as we follow the trajectory back in time there is a most recent time
$t_1<t_0$ when $c(t_1)=0$.  We assume at time $t_1$, 
the trajectory meets the
singular locus: $w(t_1)=b(t_1)=c(t_1)=0$. Reparameterizing by a time
shift, we assume $t_1=0$ and find $t_0>0$ such that
\begin{equation}
w(0)=b(0)=c(0)=0,\qquad\text{and } c(t)\ne0,\text{ for } t\in(0,t_0).
\end{equation}
\end{assumption}

\index[n]{t@$t\in\R$, time!$t_0,t_1$}
\index[n]{O@$O$, Landau big oh}
\index{Landau big oh}
\index{punctured neighborhood of singular locus}
\index[n]{C@$C,C_0,C_1,C_2$, local real constant}
\index[n]{f@$f_i$, function}

\begin{definition}\normalfont
We write
\begin{equation}
f_1 = f_2 + O(f_3)
\end{equation}
to mean that for some $t_1>0$, and some $C>0$,
we have $|f_1(t)-f_2(t)|\le C|f_3(t)|$ for all $t\in (0,t_1)$.
By a \emph{punctured neighborhood} of the singular locus, we mean an
interval $(0,t_1)$ on which $c$ is defined and nonzero.  (The
definition of Landau's O here departs slightly from the
definition before Lemma~\ref{lem:lambda2-Landau}, because the
definition here is one-sided and uses absolute values.)
\end{definition}

The aim of this section is to analyze the asymptotic behavior of
solutions as $t$ tends to zero.  With minor modifications, the same
analysis will apply to trajectories approaching the singular locus
from the left on $(-t_1,0)$.

\begin{theorem}\label{thm:ode-oh}
In the context of the assumptions of this section, 
let $w,b,c$ be a solution to the ODE of Theorem~\ref{thm:ode-hyperboloid}
on $(0,t_1)$, with $c(t)\ne0$ for all $t\in(0,t_1)$.
Assume the solution extends continuously to $w(0)=b(0)=c(0)=0$
at time $t=0$.
The ODEs and solutions $w,b,c$ satisfy the following estimates.
\begin{align*}
    w' &= -i\rho  \frac{c}{|c|} + O(|w|),\\
    b' &= 2iw + O(|b|),\\
    c' &= -3ib +O(|c|+|bw^2|+|b|^2|w|).
\end{align*}
\[
|w|=O(t),\quad |b|=O(t^2),\quad |c|=O(t^3),\quad 
|c|+|bw^2|+|b|^2|w| = O(t^3).
\]
\end{theorem}

\begin{remark}
The proof uses the assumptions $d_1=3/2$ and $\lambda_{cost}=-1$, but
not the assumptions $\H=0$ and $\A=3$.  If we do not assume the
circular control condition of Lemma~\ref{lem:maximizing-root}, but
only that $|z^*|\le1$, then the ODE for $w$ takes the form
\[
w' = -i\rho z^* + O(|w|).
\]
\end{remark}

\begin{proof}
We start with some easy approximations of the sizes of terms in
the ODEs.  The terms $\abss{w}$ and $\mu(w,z^*)$ in the
system of ODEs tend to $1$ as $t\to0$.
The right-hand side of the ODEs are bounded near $t=0$
and the initial conditions are $w(0)=b(0)=c(0)=0$.
Thus, $|w|'\le |w'|\le C$, so that $|w|=O(t)$.  Similarly,
\[
|b|=O(t),\quad |c|=O(t).
\]
Feeding these bounds back into the ODE \eqref{eqn:eqn-b} for $b$,
we obtain $|b|'\le |b'| = O(t)$.  
Thus, $|b|=O(t^2)$. Now
\begin{align*}
\abss{w} &= 1 + O(|w|^2),\quad \abss{b}=1+O(|b|^2),\quad
\mu(w,z^*) = 1 + O(|w|).
\end{align*}
We return to the ODE \eqref{eqn:eqn-w} for $w$
and use Lemma~\ref{lem:maximizing-root} to write it
\[
w' = -i\rho  \abss{w} z^*/\mu + i w/\mu = -i\rho  c/|c| + O(|w|).
\]
We return to the ODE \eqref{eqn:eqn-b} for $b$
and write it
\[
b' = 2i\abss{b} w + 2i b\abss{w} = 2 i w + O(|b|).
\]
The ODE \eqref{eqn:eqn-c} for $c$
takes the form $c'=A(c)+f$, where 
\index[n]{A@$A$, matrix or linear map!as bounded operator}
\index[n]{f@$f$, term of an ODE}
\[
A(c)=\frac{i(1-\rho ^2) \Re(c\aa_0(w,c),z^*)}{2\abss{w}\mu(w,z^*)^2}z^* = O(|c|)
\] 
is a bounded operator, which is linear in the real and imaginary parts
of $c$ through the subterm $c\aa_0(w,c)=2c+|w|^2c-\bar{c}w^2$. The
inhomogeneous term $f$ is
\begin{align*} 
- i(2bd_1\abss{w}&+(2d_1\abss{b}_\epsilon+3\lambda_{cost})w)
\\
&= -3ib\abss{w}  -3i (\abss{b}_\epsilon-1)w
\\
&= -3ib + O(|b||w|^2 + |b|^2|w|).
\end{align*}
So the ODE for $c$ takes the form
\[
c' = -3ib + O(|bw^2|+|b|^2|w|+|c|).
\]
We then have $|c|'\le |c'| \le C_0 |c| + O(t^2)$, for some $C_0>0$.  
So $|c|=O(t^3)$.
This completes the proof.
\end{proof}
\index[n]{C@$C,C_0,C_1,C_2$, local real constant}

\section{Hamiltonian Dynamics of the Truncated System}\label{sec:fuller-system}

\index[n]{0@$-_F$, Fuller truncation}
\index[n]{w@$w$, hyperboloid coordinate!$w_F$, truncated}
\index[n]{b@$b,c$, hyperboloid coordinates!$b_F,c_F$, truncated}
\index[n]{H@$\H$, Hamiltonian!$\H_F$, Fuller system}
\index[n]{A@$\A$, angular momentum!$\A_F$, truncated}
\index[n]{yz@$z_i$, Fuller system!component}

Following Theorem~\ref{thm:ode-oh}, 
we create a \emph{truncated system} of ODEs given by 
\begin{align}
    c_F' &= -3ib_F \label{eqn:eqn-cF} 
\\
    b_F' &= 2iw_F  \label{eqn:eqn-bF} 
\\
    w_F' &= -i\rho  \frac{c_F}{|c_F|},  \label{eqn:eqn-wF}
\end{align}
where the big-oh terms are discarded. This new system governs the
dynamics of the Reinhardt system very close to the singular locus.  We
introduce new coordinates
\begin{align}\label{eqn:wbc-to-z}
z_3 := c_F/(6\rho),\quad
z_2 :=c_F'/(6\rho)  = -ib_F/(2\rho) ,\quad
z_1 := c_F''/(6\rho)  =w_F/\rho 
\end{align}
so that the truncated system in
equations \eqref{eqn:eqn-cF}, \eqref{eqn:eqn-bF}, \eqref{eqn:eqn-wF}
becomes
\begin{equation}\label{eqn:fuller-system} 
    z_3' = z_2, \quad 
    z_2' = z_1, \quad
    z_1' = -i\frac{z_3}{|z_3|}.
\end{equation}
This is the Fuller system of length $3$, complex dimension $1$, and
multiplier $-i$.
We also write $z_0 = -i z_3/|z_3|$, so that $z_1'=z_0$.
In this section, we study this Fuller system.

The truncated Hamiltonian and angular momentum are defined as
\begin{align}\label{eqn:truncated-ham}
\begin{split}
\H_F &:= \frac{i}{2}\left( z_2 \bar{z}_1 
- \bar{z}_2z_1\right) + \sqrt{z_3 \bar{z}_3} 
\\
\A_F &:= {z_2} \bar{z}_2 - \left(z_1 \bar{z}_3 + \bar{z}_1z_3\right).
\end{split}
\end{align}
\begin{remark} 
These definitions come from the leading term of the
Hamiltonian and angular momentum for the Reinhardt system.
Writing the Reinhardt quantities $\H$ and $\A$ as functions of $w,b,c$
and $z^* = c/|c| + O(t)$ and their conjugates, we formally expand
using \eqref{eqn:wbc-to-z}.
\begin{align*}
\H(t w_F,t\bar{w}_F,t^2b_F,t^2\bar{b}_F,t^3c_F,t^3\bar{c}_F,\cdots)&=
6\rho^2 t^3\H_F + O(t^4)\\
\A(t w_F,t\bar{w}_F,t^2 b_F,\cdots) &= 3 + 6\rho^2 t^4\A_F + O(t^5),
\end{align*}
Also, if a Hamiltonian for the Fuller system depends on a control
through a term $\Re(z_3,u)$, where the control $u$ satisfies
$|u|\le1$, then the maximized Hamiltonian is achieved when
$u=z_3/|z_3|$, and the term in the Hamiltonian becomes
$\Re(z_3,u)=\sqrt{z_3\bar{z}_3}$, as we find in the formula for
$\H_F$. Thus, $\H_F$ is to be viewed as the maximized Hamiltonian.
\end{remark}

\begin{theorem}[Fuller Hamiltonian System]
The Fuller system \eqref{eqn:fuller-system} is Hamiltonian with
respect to a non-standard Poisson
bracket~\eqref{eqn:fuller-poisson-brack}. The angular momentum is in
involution with the Hamiltonian with respect to this
bracket~\eqref{eqn:non-standard-involution}.  The Poisson bracket
satisfies the Jacobi identity.
\end{theorem}

\index{involution with respect to the bracket}
\index[n]{F@$F,G$, smooth functions}

\begin{proof}
\index{Poisson!bracket}

We regard $z_1$, $z_2$, and $z_3$ as coordinate functions $\C^3\to\C$
and let $\bar{z}_j$ denote the conjugates of these coordinate
functions.  For smooth functions $F,G : \C^3 \to \C$, 
expressed as functions of $z_j$ and $\bar{z}_j$, we
define their non-standard Poisson bracket as
\begin{align}\label{eqn:fuller-poisson-brack}
\index[n]{0@$\{-,-\}$, Poisson bracket!$-_F$, Fuller}
\{F,G\}_F &:= \sum_{j=1}^3
(-1)^{j} 2i\left(
\frac{\partial F}{\partial z_j}
\frac{\partial G}{\partial \bar{z}_{4-j}} -
\frac{\partial F}{\partial \bar{z}_{4-j}}
\frac{\partial G}{\partial {z}_{j}}
\right)\\
&=
\frac{2}{i} \left(\frac{\partial F}{\partial z_1} 
\frac{\partial G}{\partial \bar{z}_3} 
- \frac{\partial F}{\partial \bar{z}_3} \frac{\partial G}{\partial z_1} \right) 
+ \frac{2}{i}\left( - \frac{\partial F}{\partial z_2} \frac{\partial G}{\partial \bar{z}_2} 
+  \frac{\partial F}{\partial \bar{z}_2} \frac{\partial G}{\partial z_2} \right) 
+ \frac{2}{i}\left(\frac{\partial F}{\partial z_3} \frac{\partial G}{\partial \bar{z}_1} 
-  \frac{\partial F}{\partial \bar{z}_1} 
\frac{\partial G}{\partial z_3} \right).
\end{align}

We can now verify directly that the Fuller
equations \eqref{eqn:fuller-system} become
\begin{align*}
    z_1' &= \{z_1,\H_F\}_F,
\\
    z_2' &= \{z_2,\H_F\}_F,
\\
    z_3' &= \{z_3,\H_F\}_F,
\end{align*}
which are Hamilton's equations for this Poisson bracket. We can also
verify that 
\begin{equation}\label{eqn:non-standard-involution}
\{\H_F,\A_F\}_F = 0
\end{equation}
and that the Jacobi identity is
satisfied.
\end{proof}

\index{virial!action}
\index[n]{0@$\cdot$, action!virial}
\index[n]{G@$\G$, virial group}
\index[n]{zh@$\theta$, angle}

\begin{definition}[virial action]
Define the virial group to be the two-dimensional scaling group
$\G = \SO \times \R_{>0}$.  The virial
group acts on the Fuller system \eqref{eqn:fuller-system} by
the rule
\begin{equation}\label{eqn:virial}
    (\exp(i\theta),r)\cdot(z_1(t),z_2(t),z_3(t)) 
:= (\exp(i\theta)r z_1(t/r),\exp(i\theta)r^2 z_2(t/r),\exp(i\theta)r^3 z_3(t/r)).
\end{equation}
\end{definition}

The name \emph{virial group} comes from a similar group that goes by this
name, which acts on
the Kepler dynamical
system~\cite{cushman1997global}.

\index{Kepler dynamical system}
\index[n]{zt@$\tau$, time reversal involution}
\index{time reversal, $\tau$}

If $z=(z_1,z_2,z_3)$ is a solution, then
$(\exp(i\theta),r)\cdot z$ is also a solution.  We also have
an involution given by time reversal
\begin{equation*}
\tau \cdot (z_1(t),z_2(t),z_3(t)) 
:= (\bar{z}_1(-t),-\bar{z}_2(-t),\bar{z}_3(-t))
\end{equation*}
that carries solutions to solutions.

It is noteworthy that the rotation group $\SO$ is a symmetry of the
system and as a result of the classical Noether theorem, we recover
$\A_F$ as a conserved quantity.  The truncated angular momentum and
the Hamiltonian are \emph{exactly conserved} for the truncated
system. Both are identically zero along trajectories that approach the
singular locus.  The next proposition shows that the Poisson bracket
in \eqref{eqn:fuller-poisson-brack} arises via a symplectic structure
on $\C^3$.

\index[n]{zz@$\omega$, two-form!on $\C^3$}
\index[n]{F@$F,G$, smooth functions}
\index[n]{0@$-\vec{\phantom{-}}$, vector field of function}
\index[n]{0@$\wedge$, wedge product of differential forms}

\begin{proposition}
Let $F,G$ be smooth, real-valued functions on $\C^3$. Consider
the following symplectic form on $\C^3$:
\begin{equation}\label{eqn:fuller-symplectic-form}
\omega_F := \sum_{j=1}^3 \frac{(-1)^{j}}{2i}  dz_j \wedge d\bar{z}_{4-j},
\end{equation}
Let $\vec{F}$ and $\vec{G}$ denote the
Hamiltonian vector fields of smooth functions
$F,G$ with respect to this symplectic form. Then
we have
\[
\{F,G\}_F = \omega_F(\vec{F},\vec{G}).
\]
\end{proposition}


\begin{proof} 
We claim that
\begin{equation}\label{eqn:vecF}
\vec{F} 
= \sum_{j=1}^3(-1)^j2i\left(\frac{\partial F}{\partial \bar{z}_{4-j}}
\frac{\partial}{\partial z_j} 
- \frac{\partial F}{\partial z_{4-j}}
\frac{\partial}{\partial \bar{z}_j} \right).
\end{equation}
To see this, we check that if $\vec{F}_{rhs}$ is the right-hand side
of \eqref{eqn:vecF}, then the defining conditions of $\vec{F}$ all
hold
\begin{align*}
\omega_F(\vec{F}_{rhs},\partial/\partial z_k) &= \bracks{dF}{\partial/\partial z_k}_*=
\partial F/\partial z_k\\
\omega_F(\vec{F}_{rhs},\partial/\partial \bar{z}_k) &= \bracks{dF}{\partial/\partial \bar{z}_k}_* =
\partial F/\partial \bar{z}_k,
\quad k=1,2,3.
\end{align*}
We leave this as a routine exercise for the reader.
Similarly,
\begin{align*}
\vec{G} 
= \sum_j(-1)^j2i\left(\frac{\partial G}{\partial \bar{z}_{4-j}}
\frac{\partial}{\partial z_j} 
- \frac{\partial G}{\partial z_{4-j}}
\frac{\partial}{\partial \bar{z}_j} \right).
\end{align*}
From these explicit formulas for $\vec{F}$ and $\vec{G}$, 
we find that the Poisson bracket of two functions $F$ and
$G$ is
\[
\{F,G\}_F = \bracks{dF}{\vec{G}}_*
= \omega_F(\vec{F},\vec{G})
= 
\sum_j(-1)^j2i\left(
\frac{\partial F}{\partial z_j} 
\frac{\partial G}{\partial \bar{z}_{4-j}}
- 
\frac{\partial F}{\partial \bar{z}_j}
\frac{\partial G}{\partial z_{4-j}} \right).
\]
and thus, we obtain the required. 
\end{proof}

\index[n]{a@$a_j,b_j\in\C$, complex coefficients}

\bigskip

We can also generalize to the length $n$ Fuller system of complex dimension
$1$ and multiplier $\gamma=i^n$.
\begin{equation}\label{eqn:n-Fuller}
z_n' = z_{n-1},\quad z_{n-1}'=z_{n-1},\cdots,\quad z_1' = \gamma z_n/|z_n|=z_0,\quad\gamma=i^n.
\end{equation}

\index[n]{zz@$\omega$, two-form!$\omega_n$, for length $n$ Fuller system}
\index[n]{H@$\H$, Hamiltonian!$\H_n$, length $n$ Fuller system}
\index[n]{A@$\A$, angular momentum!$\A_n$, of length $n$ Fuller}

\begin{definition}[Fuller symplectic form]\label{def:fuller-symplectic-form}
On $\C^n$, we have the following symplectic form.
\begin{equation*}
\omega_n := \bar{\gamma}\sum_{j=1}^n (-1)^{j-1}\,  dz_j \wedge d\bar{z}_{n-j+1}.
\end{equation*}
\end{definition}


\begin{theorem}
The length $n$ Fuller system \eqref{eqn:n-Fuller} is the Hamiltonian
vector field (with respect to the Fuller symplectic form) of the
Hamiltonian.
\begin{equation*}
\H_n := \sum_{j=0}^n (-1)^j \RR(z_j,i^n z_{n-j}).
\end{equation*}
The angular momentum
\begin{equation*}
\A_n := \sum_{j=1}^n (-1)^j\RR(i^{n+1}z_j,z_{n-j+1})
\end{equation*}
is conserved along this system. 
\end{theorem}
\begin{proof}
If $G$ is a smooth function, the above proof generalizes in a
straightforward way to give the following expression for the
Hamiltonian vector field of $G$:
\begin{equation*}
\vec{G} = 
\frac{1}{\bar{\gamma}}\sum_{j=1}^n (-1)^{j-1} 
\left(\frac{\partial{G}}{\partial\bar{z}_{n-j+1}}
\frac{\partial~~}{\partial{z}_j}
-\frac{\partial{G}}{\partial{z}_j}
\frac{\partial~~~~~~}{\partial\bar{z}_{n-j+1}}\right).
\end{equation*}
Using this expression, we can compute the Hamiltonian vector field of
$\H_n$ and we recover exactly the
system \eqref{eqn:fuller-system-len-n}, showing that
\[
z_{j-1} = \{z_j,\H_n\},\quad j=1,\ldots,n.
\]
Differentiating $\H_n$ and
$\A_n$ along the length-$n$ Fuller system shows that they are
conserved.
\end{proof}

\section{Log-Spiral Solutions}\label{sec:log-spiral-solutions}

\index{log spiral}
\index[n]{0@$-^*$, special value!spiral solution}
\index[n]{i@$i=\sqrt{-1}$}
\index{Euler-Manchin identity}

The system \eqref{eqn:fuller-system} admits the following
outward-moving logarithmic spiral solution, for $t>0$.
\begin{align}\label{eqn:log-spiral}
    z_3^*(t) &= \frac{1}{10} t^{3-i}
\\
    z_2^*(t) &= \frac{(3-i)}{10}t^{2-i} \nonumber
\\
    z_1^*(t) &= \frac{(2-i)(3-i)}{10}t^{1-i} \nonumber
\\
z_0^*(t) &=  -i t^{-i}\nonumber
\\
    u^*(t) &= \rho  t^{-i}.\nonumber
\end{align}
Here $i=\sqrt{-1}$ in the formulas.%
\footnote{The well-known
Euler-Manchin identity $(3-i)(2-i)(1-i)=-10i$ is used to verify the
solution.  In a different context, Manchin-like formulas are used to
compute digits of $\pi$.
}
Other log-spiral solutions are obtained by the action of
the viral group $\G$.
Note that the log-spiral is self-similar 
by a one-dimensional subgroup of $\G$:
\[
(r^{-i},r)\cdot z^* = z^*.
\]
We can also verify that
$\H_F(z_3^*,z_2^*,z_1^*)=\A_F(z_3^*,z_2^*,z_1^*)=0$.


Time-reversal $\tau$ 
transforms the outward log-spiral into an inward log-spiral.
The inward spiral is defined for $t<t_1$.
\begin{align*}
z_3^\tau(t)&=\phantom{-}\bar{z}_3^*(t_1-t)
= \phantom{-}\frac{1}{10} (t_1-t)^{3+i},
\\
z_2^\tau(t)&=-\bar{z}_2^*(t_1-t) = -\frac{(3+i)}{10}(t_1-t)^{2+i},
\\
z_1^\tau(t)&=\phantom{-}\bar{z}_1^*(t_1-t) 
= \phantom{-}\frac{(3+i)(2+i)}{10}(t_1-t)^{1+i},
\\
z_0^\tau(t)&= -\bar{z}_0^*(t_1-t) = -i (t_1-t)^i.
\end{align*}
Here $t_1$ is arrival time at the singular locus. The trajectory can
be verified by differentiating.  During approach to the singular
locus, the optimal control for the inward log-spiral performs an
infinite number of rotations along the circle $\partial U_r$ in finite
time.

\section{Literature on Fuller Systems}

\index[n]{x@$x,y$, coordinates of the classical Fuller problem}
\index[n]{u@$u$, control!$u\in[-1,1]$, classical Fuller control}

Fuller systems (over $\R$) were first described in
Fuller~\cite{fuller1963study} and arises as the Pontryagin system of
what is now called the classical Fuller optimal control problem. This problem
can be described as
\[
x'=y, \quad y' = u, \quad \int_0^\infty x^2 dt \to \min,
\]
with initial conditions $x(0) = x_0$, $y(0) = y_0$ and $u \in [-1,1]$
is a control variable taking values in an interval. The optimal
trajectory for this problem consists of an arc whose control switches
infinitely many times at the extremes of the control set in a finite
amount of time.

Generalizations of the Fuller phenomenon are studied in the book of
Zelikin and Borisov~\cite{zelikin2012theory}. Problem 5.1 is Chapter 5
of this book is exactly the length $n$ Fuller system in
equation \eqref{eqn:fuller-system-len-n} specialized to $\R$.  This
system is called the \emph{multi-dimensional Fuller problem} with
1-dimensional control.

\index{Fuller system!multi-dimensional}

Our system in equation \eqref{eqn:fuller-system-len-n} is a mild
generalization of that system.
Problem 7.2 of Zelikin and Borisov studies a Fuller problem with
multidimensional control. In particular, equation (7.11) on page 230
of thier book is exactly our system \eqref{eqn:fuller-system-len-n}
for $n=4$. Just as we did, Zelikin and Borisov construct log-spiral
solutions to the Fuller problem for 2-dimensional control but leave
the exploration of other solutions as a research problem in Chapter 7.

The Fuller systems considered in the literature have even length. Our
system, because of left-invariance, has odd chain length. The same
remark also applies in our derivation of the extended state space
Poisson bracket (see Section \ref{sec:poisson-bracket}). Also, in our
case, the extra dimension for the control and the circular symmetry of
the control set gives us an additional symmetry and thus another
conservation law.

\index{Fuller system!ubiquity}

At first, Fuller's problem was viewed  as an
oddity \cite{fuller1963study}, but was later shown to be \emph{ubiquitous} in a very
precise sense in a paper of Kupka~\cite{kupka2017ubiquity}: so long as
the extended state space of our optimal problem is of sufficiently
high dimension, one can find a Fuller trajectory as an extremal.

\index{Poisson!descending bracket}

Recently, Zelikin, Lokutsievskii and
Hildebrand~\cite{zelikin2017typicality} show that for a
linear-quadratic optimal problem with control variables in a
two-dimensional simplex, the extremals perform infinite switchings in
finite time, and their switches are chaotic in nature. Further, they
prove that this behavior is generic for piecewise smooth Hamiltonian
systems near the junction where three hyper-surfaces meet in a
codimension 2 manifold. The main innovation in Zelikin, Lokutsievskii
and Hildebrand~\cite{zelikin2017typicality} is the so-called
\emph{descending system of Poisson brackets}, which is a clever change
of coordinates of the generic system near the singularity made so that
the results of the model problem are applicable. This method is
illustrated in the very recent paper of Manita, Ronzhina and
Lokutsievskii~\cite{ronzhina2021neighborhood}.

A Fuller system with chattering, which is similar to ours in some
respects, has been analyzed by Zelikin, Lokutsievskii, and
Hildebrand~\cite{zelikin2017typicality} They have found that on a set
of full Lebesgue measure, the dynamics of their system is particularly
simple (page 24, sec 2.8).  However, on a set of measure zero, their
dynamical system exhibits complex behavior: chaotic trajectories and a
``Cantor-like structure as in Smale's Horseshoe.''  
However, in our Fuller system, we prove that no such complexities
appear.  It remains to be seen whether Zelikin-type results can be
derived for the Reinhardt problem.

\clearpage
\newpage

\chapter{Global Dynamics of Fuller System}

\index[n]{M@$M$, manifold!for Fuller system dynamics}

In this chapter, we make a thorough analysis of the global
dynamics of the Fuller system (with circular control). 
Define
\[
M := \{z=(z_1,z_2,z_3)\in(\C^{\times})^3\mid 
\H_F(z)=\A_F(z)=0\}, 
\]
where $\H_F$ and $\A_F$ are the truncated Hamiltonian and angular
momentum, defined in \eqref{eqn:truncated-ham}.  Since $\A_F$ and
Hamiltonian $\H_F$ are constant along trajectories, the Fuller system
\eqref{eqn:fuller-system} can be restricted to $M$.

\begin{lemma} $M$ is a real analytic manifold of real dimension four
in $\C^3$.
\end{lemma}
\begin{proof}
At every point of $M$, the gradients of $\A_F$ and $\H_F$ are
linearly independent.
\end{proof}

\section{A Fiber Bundle}

\index[n]{zZ@$\Omega\subset[0,2]^2$}
\index[n]{0@$-^0$, interior}
\index[n]{zp@$\pi:M\to\R^2$, function}
\index[n]{x@$x_i$, coordinates $(x_2,x_3)$ of $\Omega$}

Define
\[
\Omega := \{(x_2,x_3)\in\R^2\mid x_2>0,\ x_3>0,\ 
x_3\le x_2,\ \frac{1}{2}x_2^2\le x_3\}\subset [0,2]\times[0,2].
\]
Let $\Omega^0$ be the interior of $\Omega$, obtained by making
the inequalities strict.
Define $\pi:M\to\R^2$ by 
\[
\pi(z)=\pi(z_1,z_2,z_3):=
\left(\frac{|z_2|\phantom{^1}}{|z_1|^2},\frac{|z_3|\phantom{^1}}{|z_1|^3}\right)=
(x_2,x_3).
\]
\index[n]{yz@$z_i$, Fuller system!$z=(z_1,z_2,z_3)\in\C^3$}
\begin{lemma}
The image $\pi(M)$ lies in $\Omega$.
\end{lemma}

\begin{proof}
The equality $\A_F=0$, after applying the Cauchy-Schwarz inequality to
the term $\RR(z_1,z_3)$ appearing in $\A_F$, gives that $\pi(z)$
satisfies $x_2^2/2\le x_3$.  The equality $\H_F=0$, after applying the
Cauchy-Schwarz inequality to the term $\RR(z_1,z_2i)$ appearing in
$\H_F$, gives that $\pi(z)$ satisfies $x_3\le x_2$.  By definition,
on $M$ we have $z_j\ne0$.  Thus, the image of $\pi$ is contained in
$\Omega$.
\end{proof}

\index[n]{t@$t$, real number!scalar}

Recall that the virial group acts as symmetries of the Fuller system.
The virial group $\G$ restricts to an action on $M$
because of the homogeneities.
\begin{align*}
\H_F(t z_1,t^2 z_2,t^3 z_3)&=t^3\H_F(z_1,z_2,z_3)\\
\A_F(t z_1,t^2 z_2,t^3 z_3)&=t^4\H_F(z_1,z_2,z_3),
\end{align*}
for $t>0$.  
The morphism $\pi:M\to\Omega$ is equivariant with respect to the trivial
action of the virial group on $\Omega$, and each fiber of $\pi$ is a union
of orbits of the group action.

\index[n]{sin@$\sin_i\in\R$, \emph{sine} coordinate over $\Omega$}
\index[n]{cos@$\cos_i\in\R$, \emph{cosine} coordinate over $\Omega$}
\index[n]{ze@$\epsilon\in\{-1,0,1\}$, sign!$\epsilon_i\in\R$, \emph{sign} coordinate over $\Omega$}

\begin{lemma}\label{lem:fiber}
The fiber of $\pi$ over $(x_2,x_3)\in\Omega$ is given by
\begin{align*}
z_1&\in\C^\times,
\\
z_2 &= x_2 z_1|z_1|(\epsilon_2 \cos_2+i \sin_2),
\\
z_3 &= x_3 z_1|z_1|^2(\cos_3 + i \epsilon_3 \sin_3),
\end{align*}
where 
\begin{align}
\sin_2 &:= x_3/x_2,
&\cos_2 &:= \sqrt{1-\sin_2^2} = \sqrt{1 - (x_3/x_2)^2},\nonumber
\\
\cos_3 &:= x_2^2/(2x_3),
&\sin_3 &:= \sqrt{1-\cos_3^2} = \sqrt{1 - x_2^4/(4x_3^2)},\label{eqn:cos2}
\end{align}
and $\epsilon_2,\epsilon_3\in\{\pm1\}$.
\end{lemma}
\begin{remark}
The fiber satisfies identities:
\begin{align*}
\frac{z_2}{|z_2|} &= \frac{z_1}{|z_1|} (\epsilon_2\cos_2 + i \sin_2),
\\
\frac{z_3}{|z_3|} &= \frac{z_1}{|z_1|} (\cos_3 + i\epsilon_3 \sin_3).
\end{align*}
\end{remark}
\begin{proof}
We analyze the fibers of $\pi:M\to\Omega$.  Note that $x_2,x_3>0$,
and $\sin_2,\cos_3\in[0,1]$,
so that the formulas are well-defined.  Let $(x_2,x_3)\in\Omega$.
Using the virial action on fibers, if the fiber over $(x_2,x_3)$ is
nonempty, then it contains a point with $z_1=1$, which we now assume
without loss of generality.  Then $|z_2|=x_2\ge0$ and $|z_3|=x_3>0$.
Thus, there exist $\cos_2,\sin_2,\cos_3,\sin_3\in\R$,
and signs $\epsilon_2,\epsilon_3$ such that
\begin{align*}
z_2 &= x_2(\epsilon_2 \cos_2+i\sin_2),
&z_3 &= x_3(\cos_3+i\epsilon_3 \sin_3),\quad\text{where}
\\
1&= \cos_2^2+\sin_2^2=\cos_3^2+\sin_3^2,
&\cos_2&\ge0,\sin_3\ge0.
\quad\epsilon_2,\epsilon_3\in\{\pm1\}.
\end{align*}
The condition $\A_F=0$ gives an additional constraint
$\cos_3=x_2^2/(2x_3)>0$, and the condition $\H_F=0$ gives the constraint
$\sin_2 = x_3/x_2>0$.  Thus, every point in the preimage of $(x_2,x_3)$
has the form asserted in the lemma.

Conversely, every $(z_1,z_2,z_3)$ of the given form belongs to $M$ and
maps to $(x_2,x_3)$ in $\Omega$.  In particular, the image of $\pi$
is $\Omega$.
\end{proof}

\index[n]{zZ@$\Omega\subset[0,2]^2$!$\Omega_{\epsilon_i,\epsilon_j}$, copies of $\Omega$}
\index[n]{R@$\R^2_\Omega$, topological plane}
\index[n]{zZ@$\Omega\subset[0,2]^2$!$\partial\Omega^+_{\pm,\pm}=\partial^+_{\pm,\pm}$, upper boundary curve of $\Omega_{\pm,\pm}$}
\index[n]{zZ@$\Omega\subset[0,2]^2$!$\partial\Omega^-_{\pm,\pm}=\partial^-_{\pm,\pm}$, lower boundary curve of $\Omega_{\pm,\pm}$}

We let
$\Omega_{\epsilon_2,\epsilon_3}:= \Omega\times\{\epsilon_2\}\times\{\epsilon_3\}$,
where $\epsilon_2,\epsilon_3\in\{\pm1\}$, be four copies of $\Omega$.
Let $\partial\Omega^+_{\epsilon_2,\epsilon_3}$ (or
$\partial^+_{\epsilon_2,\epsilon_3}$, for short) be the upper boundary curve
of $\Omega_{\epsilon_2,\epsilon_3}$ defined by $x_3\le x_2$.  Let
$\partial\Omega^-_{\epsilon_2,\epsilon_3}$ (or
$\partial^-_{\epsilon_2,\epsilon_3}$, for short) be the lower boundary curve
of $\Omega_{\epsilon_2,\epsilon_3}$ defined by $x_2^2/2\le x_3$.

We glue these four copies of $\Omega$ together along boundaries to
form a topological plane $\R_\Omega^2$ as follows.  Along the boundary
edge $x_2=x_3$, we identify $\Omega_{+1,\epsilon_3}$ with
$\Omega_{-1,\epsilon_3}$ (for $\epsilon_3=\pm1$), and along the
boundary edge $x_3=x_2^2/2$, we identify $\Omega_{\epsilon_2,+1}$
with $\Omega_{\epsilon_2,-1}$ (for $\epsilon_2=\pm1$).  All four
copies of the corner $(2,2)\in\Omega_{\pm,\pm}$ are identified by this
process.  The corner $(0,0)$ is excluded from $\Omega$ and from
$\Omega_{\pm,\pm}$ by definition.
\[
\R_\Omega^2 = \left(\bigcup_{\epsilon_2,\epsilon_3} \Omega_{\epsilon_2,\epsilon_3}\right)
/
\{ \partial^+_{-+} = \partial^+_{++},\ \partial^+_{--}=\partial^+_{+-},\ 
\partial^-_{--}=\partial^-_{-+},\ \partial^-_{+-}=\partial^-_{++}\}.
\]

Visually, it helps to imagine $\R_\Omega^2$ as follows.  We take a
conformal transformation of $\Omega_{\epsilon_2,\epsilon_3}^0$ onto
the open $(\epsilon_2,\epsilon_3)$ quadrant, which sends the point
$(0,0)$ to $\infty$, the point $(2,2)$ to $(0,0)$, and the boundary
$\partial^+_{\epsilon_2,\epsilon_3}$ with equation $x_2=x_3$ to the
vertical axis, and the boundary $\partial^-_{\epsilon_2,\epsilon_3}$ with
equation $x_3 = x_2^2/2$ to the horizontal axis. See
Figure~\ref{fig:R-Omega}.

\index{conformal map}

\tikzfig{R-Omega}{The four regions $\Omega_{\pm,\pm}$ map conformally
to the four quadrants in the plane.  
Identifying boundary edges, they form a topological
plane $\R_\Omega^2$.}{
\begin{scope}[scale=0.8]
\draw[black, line width=0.30mm] (0,0) node[anchor=north]{$(0,0)$}
--(2,2) node[anchor=south] {$(2,2)$};
\draw[black, line width = 0.30mm]   plot[smooth,domain=0:2] (\x, {(\x*\x)/2});
\node (f) at (2,1) {$\partial^-_{++}$};
\node (e) at (0,1) {$\partial^+_{++}$};
\node (O) at (0,2.2)  {$\Omega_{++}:$};
\end{scope}
\begin{scope}[scale=0.8,xshift=4cm]
\draw[black, line width=0.30mm] (0,0) node[anchor=north]{$(0,0)$}
--(2,2) node[anchor=south] {$(2,2)$};
\draw[black, line width = 0.30mm]   plot[smooth,domain=0:2] (\x, {(\x*\x)/2});
\node (f) at (2,1) {$\partial^-_{+-}$};
\node (e) at (0,1) {$\partial^+_{+-}$};
\node (O) at (0,2.2)  {$\Omega_{+-}:$};
\end{scope}
\begin{scope}[scale=0.8,xshift=8cm]
\draw[black, line width=0.30mm] (0,0) node[anchor=north]{$(0,0)$}
--(2,2) node[anchor=south] {$(2,2)$};
\draw[black, line width = 0.30mm]   plot[smooth,domain=0:2] (\x, {(\x*\x)/2});
\node (f) at (2,1) {$\partial^-_{--}$};
\node (e) at (0,1) {$\partial^+_{--}$};
\node (O) at (0,2.2)  {$\Omega_{--}:$};
\end{scope}
\begin{scope}[scale=0.8,xshift=12cm]
\draw[black, line width=0.30mm] (0,0) node[anchor=north]{$(0,0)$}
--(2,2) node[anchor=south] {$(2,2)$};
\draw[black, line width = 0.30mm]   plot[smooth,domain=0:2] (\x, {(\x*\x)/2});
\node (f) at (2,1) {$\partial^-_{-+}$};
\node (e) at (0,1) {$\partial^+_{-+}$};
\node (O) at (0,2.2)  {$\Omega_{-+}:$};
\end{scope}
\begin{scope}[scale=1.4,yshift=-3cm,xshift=4cm]
\draw[black] (-2,0)--node[anchor=south]{$\partial^-_{-+}$} node[anchor=north]{$\partial^-_{--}$} (0,0)
node[anchor=south west]{$(2,2)$} --
node[anchor=south]{$\partial^-_{++}$} node[anchor=north]{$\partial^-_{+-}$} (2,0);
\draw (0,-2)--node[anchor=east]{$\partial^+_{--}$}
node[anchor=west]{$\partial^+_{+-}$} (0,0)--
node[anchor=west]{$\partial^+_{++}$}
node[anchor=east]{$\partial^+_{-+}$} (0,2);
\smalldot{0,0};
\node (mp) at (-1.5,1.5) {$\Omega_{-+}$};
\node (f) at (1.5,1.5) {$\Omega_{++}$};
\node (mm) at (-1.5,-1.5) {$\Omega_{--}$};
\node (pm) at (1.5,-1.5) {$\Omega_{+-}$};
\end{scope}
}

\index{overloading}

Lemma~\ref{lem:fiber} shows that there is a multiplicity of signs
$\epsilon_2,\epsilon_3$ along each fiber.  This suggests that we
should extend $\pi:M\to\Omega$ to a map $\pi:M\to\R_\Omega^2$
(overloading the notation $\pi$) as follows.  Set
\begin{align}
\begin{split}
\epsilon_2&=\text{sign}(\RR(z_2\bar{z}_1))\in\{\pm1\}
\\
\epsilon_3&=\text{sign}(\Im(z_3\bar{z}_1))\in\{\pm1\},
\end{split}
\end{align}
and extend the definition of $\pi$ so that
\begin{align}
\pi(z_1,z_2,z_3) = 
\left(\frac{|z_2|\phantom{^1}}{|z_1|^2},
\frac{|z_3|\phantom{^1}}{|z_1|^3},\epsilon_2,\epsilon_3\right)
\in\Omega_{\epsilon_2,\epsilon_3}.
\end{align}
Along the boundary edges of $\Omega_{\epsilon_2,\epsilon_3}$ that have been
identified, there are ambiguities, but the definition of $\pi$ has been
crafted in such a way that $\pi:M\to\R_\Omega^2$ is well-defined.
For example, the sign $\epsilon_2$ cannot be determined from the given
formula when $\cos_2=0$, but this occurs precisely along the edge $x_2=x_3$,
where $\Omega_{+,\epsilon_3}$ is glued to $\Omega_{-,\epsilon_3}$.

\index{trivial principal topological bundle}

\begin{theorem}
$\pi:M\to\R_\Omega^2$ is a trivial principal topological bundle of the virial group.
\end{theorem}

\begin{proof}  By the preceding constructions, each
fiber $\pi^{-1}(x_2,x_3,\epsilon_2,\epsilon_3)$ is a single orbit of
the virial group, given by the formula of Lemma~\ref{lem:fiber}.  Note
that the virial group acts simply transitively on $z_1\in\C^\times$,
which serves as a coordinate along each fiber.

\index{section of a bundle}
\index{trivialization of a bundle}
\index[n]{zy@$\psi$, local auxiliary function or integral!$\psi$, section of a bundle}

We give a global trivialization of the bundle
$M\simeq \R_\Omega^2\times\G$, by constructing
a global section:
\begin{align*}
&\psi(x_2,x_3,\epsilon_2,\epsilon_3)=(z_1,z_2,z_3),\text{ where}
\\
z_1&=1,\quad z_2=x_2(\epsilon_2\cos_2+i\sin_2),\quad z_3=x_3(\cos_3+i\epsilon_3 \sin_3)
\end{align*}
and $\sin_2,\cos_2,\sin_3,\cos_3$ are given as above.

This section is continuous.  In fact, the jumps in signs
$\epsilon_2,\epsilon_3$ occur exactly where $\cos_2=0$ or $\sin_3=0$,
and this occurs along the identifications of the boundary curves of
$\R_\Omega^2$.  The section gives the global trivialization of the
bundle.
\end{proof}

\section{A Vector Field on the Base Space}

\index[n]{v@$\mb{v}$, vector!$(v_2,v_3)$, vector field}

We define a vector field $(v_2,v_3)$ on $\Omega_{\epsilon_2,\epsilon_3}$ taking
value
\begin{align}
\begin{split}
v_2 &= \epsilon_2\cos_2 - 2 x_2 \epsilon_3 \sin_3,
\\
v_3 &= x_2(\epsilon_2\cos_2\cos_3+\epsilon_3\sin_2\sin_3) - 3x_3\epsilon_3\sin_3 
\end{split}
\end{align}
at $(x_2,x_3,\epsilon_2,\epsilon_3)$, where
$\cos_2,\sin_2,\cos_3,\sin_3$ are the functions given earlier
\eqref{eqn:cos2}.  We remark that this gives a well-defined vector
field on $\R_\Omega^2$, because the definitions agree, wherever there
might be an ambiguity along boundary curves that are identified to
form $\R_\Omega^2$.

\index[n]{f@$f$, vector field!on $M$}

\begin{lemma}
Consider the vector field $f$ on $M$ given by the Fuller system.  Then
$(v_2,v_3)$ is the scaled image of $f$ in the tangent space of
$\R_\Omega^2$.  That is, $|z_1|T\pi(f)=(v_2,v_3)\in T\R_\Omega^2$,
which is
independent of the point on the fiber over
$(x_2,x_3,\epsilon_2,\epsilon_3)$.
\end{lemma}
The rescaling factor $|z_1|$ only affects the integral curves of the
vector field by a time reparameterization.

\index[n]{r@$r$, real number!$r_j$, subexpression in vector field}

\begin{proof}
Let $z=(z_1,z_2,z_3)$ follow a trajectory of the Fuller system in $M$ 
let 
$\pi(z) = (x_2,x_3,\epsilon_2,\epsilon_3)$ be the image trajectory.
We compute for $j=2,3$:
\begin{align*}
|z_1| \frac{d~}{dt}(|z_j|/|z_1|^j)
&=|z_1|\RR(z_{j-1},\frac{z_j}{|z_j|})|z_1|^{-j}
-|z_1|j|z_j||z_1|^{-j-1}\RR(z_0,\frac{z_1}{|z_1|})
\\
&=x_{j-1}r_j - j x_j r_1,\quad\text{where}
\\
r_j &= \RR\left(\frac{z_{j-1}}{|z_{j-1}|},\frac{z_j}{|z_j|}\right).
\end{align*}
The functions $r_j$ are invariant under the virial group and descend
to $\R_\Omega^2$.
It is enough to show that $(v_2,v_3)=(x_1r_2-2x_2r_1,x_2r_3-3x_3r_1)$.
This is a routine calculation.
\end{proof}

\section{Equilibrium Points}

\index{equilibrium point}
\index{odd function}

In this section, we investigate the qualitative behavior of the vector
field $(v_2,v_3)$.  The vector field $(v_2,v_3)$ is \emph{odd}: the
values of the vector field at $(x_2,x_3,\epsilon_2,\epsilon_3)$ and at
$(x_2,x_3,-\epsilon_2,-\epsilon_3)$ have opposite signs.  This means
that trajectories are the same, except reversed in time at points with
opposite signs.  Earlier, we introduced a time-reversal operation
$\tau$ on trajectories on $M$.  The image under $\pi$ of a time
reversed trajectory in $M$ is the sign reversed
$\epsilon_j\mapsto-\epsilon_j$ trajectory in $\R_\Omega^2$.

Next we analyze the zeros of the vector field.
\begin{lemma}
The vector field $(v_2,v_3)$ is zero if and only if
$(x_2,x_3,\epsilon_2,\epsilon_3)$ is one of the following three
points:
\begin{align*}
(x_2,x_3)&=(2,2)\ \text{(all choices of signs give the same point)},
\\
(x_2^*,x_3^*)&:=(2/\sqrt{10},\sqrt{2}/5)\ 
\text{where }\epsilon_2=\epsilon_3\in\{\pm1\}\ 
\text{(one point for each sign choice)}.
\end{align*}
\index[n]{q@$q$, point on manifold!$q_\pm^*$, image of the log-spiral}
Moreover, the image under $\pi$ of the outward log-spiral
$z^*=(z_1^*,z_2^*,z_3^*)$ constructed in \eqref{eqn:log-spiral} is the
single point $q_+^*:=(x_2^*,x_3^*,+,+)\in\Omega_{++}$, while the image
of the inward log-spiral is the single point
$q_-^*:=(x_2^*,x_3^*,-,-)\in\Omega_{--}$.
\end{lemma}
\begin{proof} 
It is clear that these three points give zeros of the vector field, by
direct substitution into the formulas for $(v_2,v_3)$.  By the
observation that the vector field is odd, we can assume that
$\epsilon_2=+1$.  Then we consider the different regions
$\Omega_{+1,\epsilon_3}^0$ and its boundary, solving the equations
$v_2=v_3=0$ for $x_2$ and $x_3$.

We illustrate the case $\Omega_{++}^0$, leaving the other cases as
exercises.  From the formulas for $v_2,v_3$, we find that a zero
in $\Omega_{++}^0$ satisfies the equations
\begin{align*}
\cos_2 &= 2 x_2 \sin_3,\quad \cos_2 \cos_3 = 2 \sin_2 \sin_3,
\end{align*}
which has $(x_2^*,x_3^*) = (2/\sqrt{10},\sqrt{2}/5)$ as the unique
solution.

Each spiral trajectory is contained in a single orbit of the virial
group and must map to a single point in $\R_\Omega^2$.  Explicit
formulas have been given for the log spirals and for the map $\pi$,
and it is an easy calculation to determine which spiral is mapped to
which zero of the vector field.
\end{proof}

\index{stability at equilibrium point}

Next we analyze stability at the equilibrium points.  Because of a
square root, the vector field $(v_2,v_3)$ is not differentiable at
$(x_2,x_3)=(2,2)$ and we cannot compute a Jacobian.  However, we can
compute eigenvalues of the Jacobian for the other two equilibrium
points.

\index[n]{J@$\op{Jac}$, Jacobian matrix}

\begin{lemma}
Let $\op{Jac}$ be the $2\times2$ Jacobian matrix with entries
$\partial v_j/\partial x_k$, where $j,k\in\{2,3\}$.  The eigenvalues
of $\op{Jac}$ are $-\sqrt{2}\pm i\sqrt{3}$ at $q_+^*$.  In particular,
the eigenvalues have negative real part, and $q_+^*$ is a stable
equilibrium point.
\end{lemma}
\noindent
By symmetry, the equilibrium point $q_-^*$ is unstable.

\begin{proof} This is an elementary calculation.
\end{proof}

\section{Global Behavior}

\begin{remark}
We warn the reader that the square roots appearing in the definition
of $\cos_i,\sin_i$ cause the vector field $(v_2,v_3)$ to be
non-Lipschitz along the boundary curves of
$\Omega_{\epsilon_2,\epsilon_3}$.  Thus, trajectories are not uniquely
determined by the vector field $(v_2,v_3)$.  This is not an idle
warning.  The trajectories truly fail to be unique.  Along these
boundary curves where uniqueness breaks down, we make reference to the
trajectory upstairs in $M$ (where trajectories are uniquely
determined) to determine which path the trajectory downstairs should
follow.  Nevertheless, on each interior part
$\Omega_{\epsilon_2,\epsilon_3}^0$, the trajectories are uniquely
determined by the vector field.
\end{remark}

\mcite{MCA:5387527}
\begin{figure}[ht]
\centering
\includegraphics[scale=0.35]{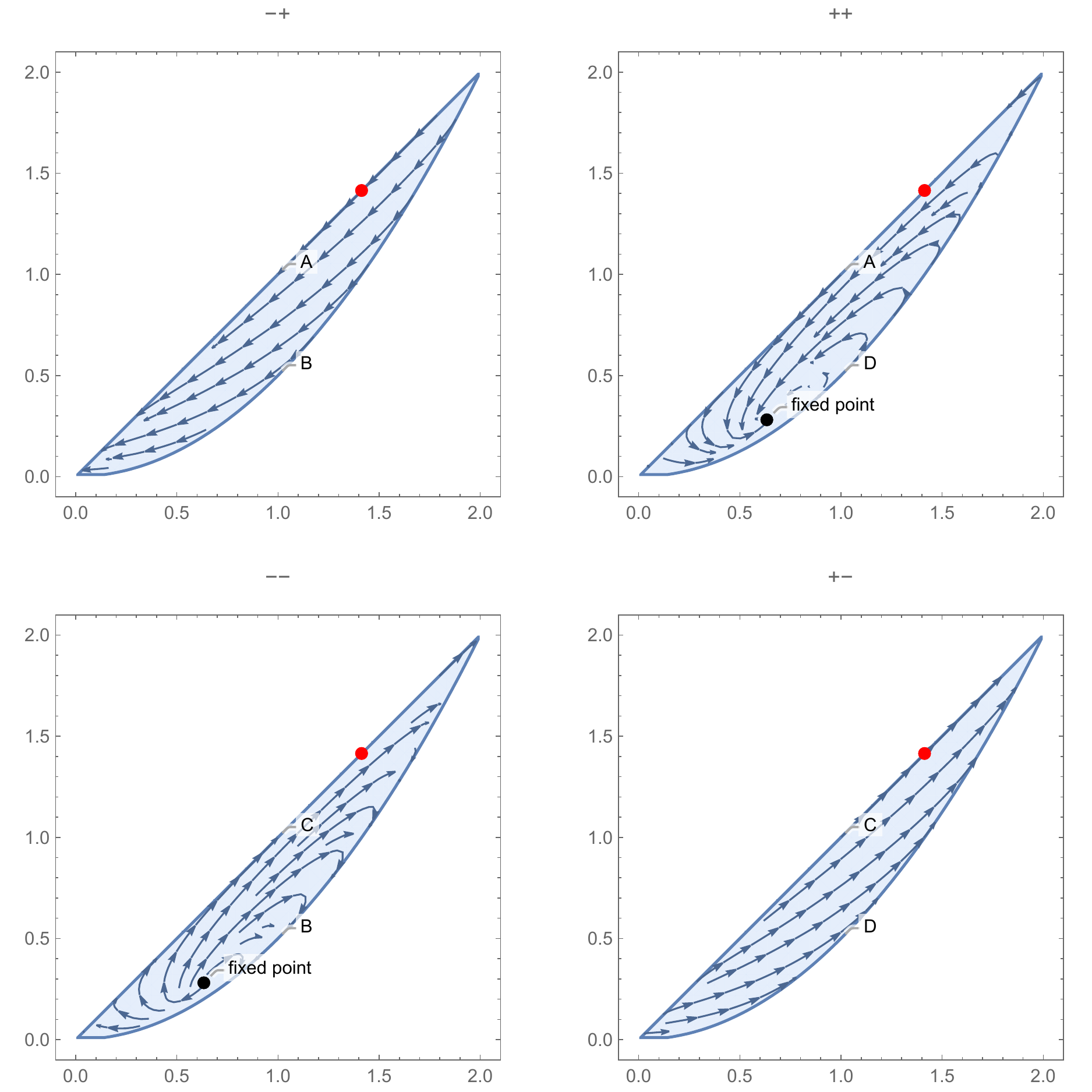}
\caption{The dynamical system on $\R_\Omega^2$.  The unstable $q_{-}^*\in\Omega_{--}$
and stable fixed points $q_{+}^*\in\Omega_{++}$ are shown. In the four frames,
the two edges marked $A$ are to be identified, as are the two edges marked
$B$, the two marked $C$, and the two marked $D$. In this way, the four
frames belong to a single dynamical system in the topological plane $\R_\Omega^2$.
The four red points have coordinates $(\sqrt2,\sqrt2)$.  The direction
of the flow across the upper boundary is reversed at $(\sqrt2,\sqrt2)$.}
\label{fig:omega}
\end{figure}

\FloatBarrier

Figure~\ref{fig:omega} depicts the dynamical system in the plane
$\R_\Omega^2$ and two fixed points. A third fixed point
$(2,2)\in\R_\Omega^2$ lies at the upper corner of the figures.  The
main result of this chapter is the following theorem.  From
Figure~\ref{fig:omega}, we observe that the theorem is geometrically
plausible.

\begin{theorem}\label{thm:global} Let $z(t)$ be any Fuller trajectory in $M$.
Assume that $\pi(z(t))\not\in\{q_{\pm}^*\}$.
Then 
\begin{itemize}
\item The trajectory $z(t)$ is defined for all $t\in\R$.  
\item The trajectory $\pi(z(t))$ remains bounded away from $(0,0)\in\R^2_\Omega$.
\item If $U\subseteq\R_\Omega^2$ is any neighborhood of $q_{+}^*$, the
  trajectory $\pi(z(t))$ eventually enters and remains in $U$.
\item If $U\subseteq\R_\Omega^2$ is any neighborhood of $q_{-}^*$, the
  trajectory $\pi(z(t))$ was in $U$ for all sufficiently negative
  times.
\end{itemize}
\end{theorem}

The proof appears at the end of the chapter in
Section~\ref{sec:global:proof}, after a series of lemmas.  The first
of these lemmas describes the movement of trajectories in
$\R_\Omega^2$ across the boundaries of the regions
$\Omega_{\epsilon_2\epsilon_3}$.

\begin{lemma} 
Along each boundary curve between regions
$\Omega_{\epsilon_2,\epsilon_3}$ the vector field points along the
tangent to the curve.  At the boundary curve $x_3 = x_2^2/2$
(excluding endpoints $(0,0)$ and $(2,2)$), the images $\pi(z(t))$ of
Fuller trajectories pass from $\Omega_{\epsilon_2-}$ into
$\Omega_{\epsilon_2+}$, for $\epsilon_2=\pm1$.  At the boundary curve
$x_3 = x_2$ (excluding endpoints $(0,0)$ and $(2,2)$), the images of
Fuller trajectories pass from $\Omega_{-\epsilon_3}$ into
$\Omega_{+\epsilon_3}$ if $x_2=x_3<\sqrt2$; and they pass in the other
direction from $\Omega_{+\epsilon_3}$ into $\Omega_{-\epsilon_3}$ if
$x_2=x_3>\sqrt2$, for $\epsilon_3=\pm1$.
\end{lemma}

\begin{proof}
Along the boundary $x_2=x_3$ the vector field has the form $v_2=v_3$,
so that the vector field is tangent to the boundary.  Similarly, the
vector field along the boundary $x_3=x_2^2/2$ is also tangent to the
boundary.  However, because of non-uniqueness of trajectories, the
flow does \emph{not} move along the boundaries!

\index[n]{p@$p$, point in bundle!value of section $\psi$}

We obtain a better approximation to the flow near a boundary of
$\Omega_{\epsilon_2,\epsilon_3}$ by taking the section
$p=\psi(x_2,x_3,\epsilon_2,\epsilon_3)\in M$, then expanding the Fuller
trajectory $z(t)$ with initial condition $p$ at $t=0$ in a Taylor
approximation $f$, then taking $\pi(z(t))$.

\index[n]{s@$s$, real parameter!$s\in(0,2)$, local parameter}
\index[n]{u@$\mb{u}$, vector!normal}

Following this procedure at the boundary point
$(x_2,x_3,\epsilon_2,\epsilon_3)=(s,s^2/2,\epsilon_2,\epsilon_3)$, for
$s\in(0,2)$, with normal $\mb{u}=(s,-1)$, we find that the trajectory moves
from $\Omega_{\epsilon_2-}$ to $\Omega_{\epsilon_2+}$.
\[
\mb{u}\cdot \pi(z(t))=  -\frac{(4+7s^2)}{8}t^2+O(t^3)\quad
\epsilon_2 = \epsilon_2,\quad
\epsilon_3 = \text{sign}(t).
\]

Following this procedure at the boundary point
$(x_2,x_3,\epsilon_2,\epsilon_3)=(s,s,\epsilon_2,\epsilon_3)$, for
$s\in(0,2)$, with normal $\mb{u}=(-1,1)$, we find that the trajectory
moves from $\Omega_{-\epsilon_3}$ to $\Omega_{+\epsilon_3}$ if
$s<\sqrt2$, and the direction between regions reverses when
$s>\sqrt2$.
\[
\mb{u}\cdot \pi(z(t))=  - \frac{(s^2-2)^2}{8u} t^2+ O(t^3),\quad
\epsilon_2 = \text{sign}(t(2-s^2)),\quad
\epsilon_3 = \epsilon_3.
\]
\end{proof}

The next lemma analyzes behavior near $(x_2,x_3)=(0,0)$.
\begin{lemma}\label{lem:00}
Let $z$ be a Fuller trajectory in $M$, defined on some open time
interval.  The trajectory $z$ extends to a trajectory in $M$ for all
$t\in\R$.  Moreover, the image $t\mapsto\pi(z(t))$ is bounded away
from $(x_2,x_3)=(0,0)$.
\end{lemma}
\begin{proof}
The vector field $(v_2,v_3)$ is bounded.  The base space $\R_\Omega^2$
fails to be compact because of the omission of the corner point
$(0,0)$ from $\Omega$.  The image $\pi(z(t))$ of a Fuller trajectory
on an open time interval can be extended in $\R_\Omega^2$ to all time,
then lifted to $M$ to extend $z(t)$, provided the trajectory
downstairs remains bounded away from $(0,0)$.  Thus, the lemma will
follow if we prove that trajectories downstairs are bounded away from
$(0,0)$, the common endpoint of all boundary curves.

\index[n]{r@$r$, real number!polar coordinate}
\index[n]{zh@$\theta$, angle}

We use polar coordinates $(x_2,x_3)=(r\cos(\theta),r\sin(\theta))$.
We may assume that $(x_2,x_3)$ is not on a boundary edge of
$\Omega_{\epsilon_2,\epsilon_3}$, because earlier analysis shows that
Fuller trajectories cross the boundary edges at isolated times.  We
analyze several subcases according to small neighborhoods of $(0,0)$
in the following separate pieces. We use a hodgepodge of arguments.

On $\Omega_{+-}$, 
\[
v_3 = 2x_3 + O(r^2),
\]
so that $x_3(t)$ is increasing, moving away from $(0,0)$.

On $\Omega_{++}^0$, we consider two subcases.  In
the first subcase, if $x_3\le x_2/2$ in a small neighborhood of
$(0,0)$, then $\theta$ is decreasing and
\[
r' = \cos\theta\sqrt{1-\tan^2\theta} + O(r),
\]
so that $r'$ is positive and bounded away from $0$, so that the
trajectory moves away from $(0,0)$.  In the other subcase, if $x_3\ge
x_2/2$ in a sufficiently small neighborhood of $(0,0)$, then the sign
of the planar curvature of $(x_2(t),x_3(t))$ is positive and the
tangent to the curve separates the trajectory from $(0,0)$.

Next, consider $\Omega_{--}^0$.  In the subcase near $(0,0)$ where
$x_3\ge x_2/10$, the curvature argument from $\Omega_{++}$ also
applies here.  In the subcase near $(0,0)$ where $x_3\le x_2/10$, then
$v_2<0$ and $v_3>0$.  Along a trajectory, $x_3$ is a function of
$x_2$, and we have
\begin{equation}\label{eqn:dx3}
\frac{dx_3}{dx_2} = \frac{v_3}{v_2} \le \sqrt{2 x_3}.
\end{equation}
Integrating this differential inequality, we obtain
\[
\sqrt{2 x_3(t)} \ge x_2(t) + c,
\]
where $c=\sqrt{2x_3(0)}-x_2(0)$ is positive on the interior of
$\Omega_{--}$.  This inequality bounds the trajectory away from
$(0,0)$.

Finally, consider $\Omega_{-+}^0$. We have $v_2,v_3<0$.  In this case,
inequality \eqref{eqn:dx3} holds, and we proceed as in the previous
case.
\end{proof}

\section{A Special Trajectory}

\index[n]{q@$q$, point on manifold!$q_{2,2}$, equilibrium point}
\index{special Fuller trajectory}
\index[n]{yz@$z_i$, Fuller system!$z_{spec}$, special trajectory}

Modulo the action of the virial group, there is a unique Fuller
trajectory whose image in $\R_\Omega^2$ passes through equilibrium
point $q_{2,2}:=(x_2,x_3,\epsilon_2,\epsilon_3) =(2,2,-1,1)$.  Using
the section $\psi$ of the bundle, the Fuller trajectory is determined by
the initial condition $(z_1(0),z_2(0),z_3(0))=(1,2i,2)=\psi(q_{2,2})$ at
$t=0$.  We call this particular trajectory $z_{spec}$
the \emph{special} Fuller trajectory.  Figure~\ref{fig:fuller-22}
shows the image (in red) of the special Fuller trajectory in
$\Omega_{-+}$ and its subsequent trajectory in $\Omega_{++}$.  The
signs $(\epsilon_2,\epsilon_3)$ are discarded, so that the figure
shows $\Omega_{-+}$ superimposed on $\Omega_{++}$.  At the point where
$\pi(z(t))$ meets the edge $x_2=x_3$, with $x_2<2$, the curve
crosses from $\Omega_{-+}$ to $\Omega_{++}$.
Figure~\ref{fig:fuller-22-deviation} shows that the crossing occurs at
the positive zero of $\pi(z(t))\cdot(1,-1)$ near $t=0.9$.  When $t$
is large, the special trajectory approaches the stable equilibrium point
$q_+^*\in\Omega_{++}$.

At $(x_2(0),x_3(0))=(2,2)$, the image trajectory $(x_2(t),x_3(t))$
moves from $\Omega_{+-}$ to $\Omega_{-+}$ (so that $(2,2)$ is not a
true fixed point, when higher order information from the Fuller
trajectory in $M$ is retained).
\[
x_2(t)= 2-\frac{17}{4}t^2 + O(t^3),\quad
x_3(t)= 2-\frac{9}{2}t^2 + O(t^3),\quad
\epsilon_2 =-\text{sign}(t),\quad
\epsilon_3 =\text{sign}(t).
\]
The special trajectory invariant under time reversal $\tau$.
Its trajectory for negative times is obtained by symmetry.

\begin{figure}[ht]
\centering
\includegraphics[scale=0.25]{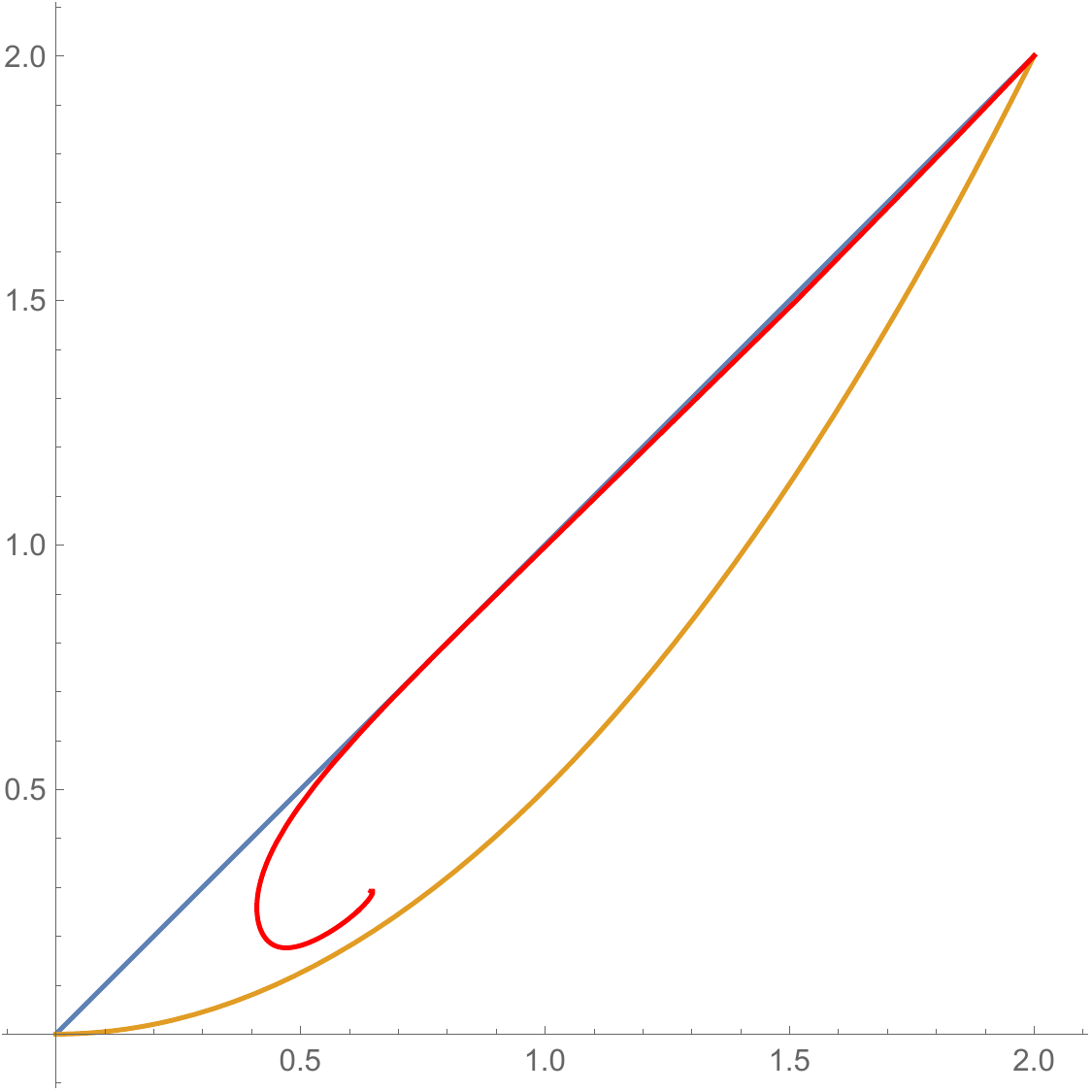} %
\caption{This figure shows the image (in red) in 
$\Omega$ of the special Fuller trajectory through $z=(1,2i,2)$.
(All four regions $\Omega_{\pm\pm}$ are superimposed in this figure.)
The red curve meets the boundary $x_2=x_3$ of $\Omega_{++}$ at two points:
at $(2,2)$ at time zero and at a second point at about time $t=0.9$.
For large values of $t$, the curve approaches the stable equilibrium
point $q_+^*$ in $\Omega_{++}$.
}
\label{fig:fuller-22}
\end{figure}

\begin{figure}[ht]
\centering
\includegraphics[scale=0.25]{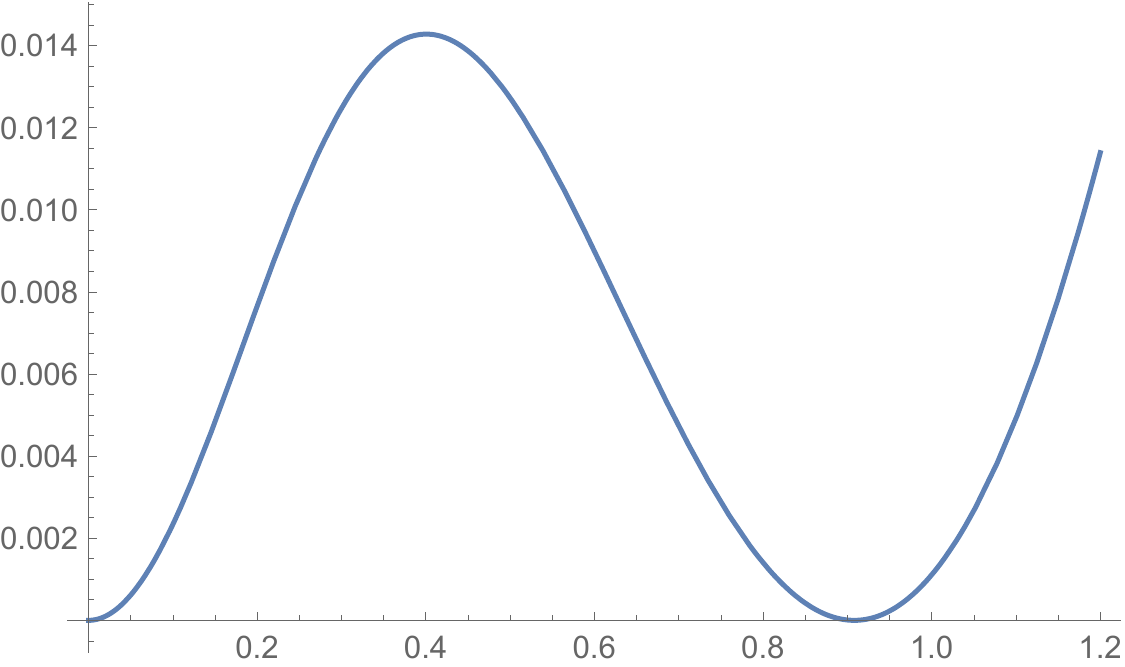} %
\caption{This figure shows $\mb{u}\cdot \pi(z(t))$, where $\mb{u}=(1,-1)$.
The graph gives the deviation of $\pi(z(t))$ from the boundary curve $x_2=x_3$.
Here $z(t)$ is the special Fuller trajectory with initial condition $z=(1,2i,2)$
at $t=0$.  The signs $(\epsilon_2,\epsilon_3)$ are ignored, but
$\pi(z(t))\in\Omega_{-+}$ for $t$ less than the positive zero near $t=0.9$,
then $\pi(z(t))$ passes into $\Omega_{++}$.
}
\label{fig:fuller-22-deviation}
\end{figure}

\index[n]{t@$t\in\R$, time!$t_c$, arrival time at boundary}
\index[n]{q@$q$, point on manifold!$q_c$, arrival point on boundary}
\index[n]{0@$[-,-]$, linear segment between endpoints}
\index[n]{zZ@$\Omega\subset[0,2]^2$!$\Omega_{spec+}\subset\Omega_{-+}$, subregion}

Let $t_{c}\approx0.9$ be the time at which
$q_c:=\pi(z_{spec}(t_c))\in\{x_2=x_3\}$.  Let $\Omega_{spec+}$ be the
narrow region in $\Omega_{-+}$ bounded by $\pi(z_{spec}(t))$, for
$t\in[0,t_c]$, and by the linear segment $[q_c,(2,2)]$ from $q_c$ to
$(2,2)$ along the edge $x_2=x_3$.  By our analysis of boundary
behavior, trajectories in $\Omega_{spec+}$, must enter through $\Omega_{++}$
along the segment $[(\sqrt2,\sqrt2),(2,2)]$ and exit back into
$\Omega_{++}$ along the segment $[q_c,(\sqrt2,\sqrt2)]$.  The
component $v_3$ of the vector field is negative on $\Omega_{spec+}$, so that
the trajectories always progress southward monotonically from entrance
to exit.

\section{Proof}\label{sec:global:proof}


\index{Mathematica}

\begin{proof} We sketch a proof of Theorem~\ref{thm:global}, relying on a few numerical
calculations as needed.  The first two claims of the theorem follow
from Lemma~\ref{lem:00}.

Consider trajectories in $\Omega_{-+}\setminus \Omega_{spec+}$.  The
component $v_3$ of the vector field is negative on $\Omega_{-+}$.
Trajectories must enter from $\Omega_{--}$ along the lower boundary
$x_3=x_2^2/2$ and have a soutward drift until exiting along the edge
$x_2=x_3$ into $\Omega_{++}$ along the open segment between $(0,0)$
and $q_c$.

\index[n]{zZ@$\Omega\subset[0,2]^2$!$\Omega_{spec-}\subset\Omega_{+-}$, sign reversal of $\Omega_{spec+}$}

Let $\Omega_{spec-}\subset \Omega_{+-}$ be the region obtained from
$\Omega_{spec+}$ by reversing signs $\epsilon_2\mapsto-\epsilon_2$ and
$\epsilon_3\mapsto-\epsilon_3$.  Then the flow on
$\Omega_{+-}\setminus \Omega_{spec-}$ is obtained by reversing the
flow on $\Omega_{-+}\setminus \Omega_{spec+}$: the trajectories move
northward, entering from $\Omega_{--}$ and exiting into $\Omega_{++}$
along the lower boundary $x_3=x_2^2/2$.

Similarly, the behavior on $\Omega_{spec-}\cup\Omega_{--}$ will be the time
reversal of $\Omega_{spec+}\cup\Omega_{++}$, which we describe next.

\index{Lyapunov function}
\index[n]{D@$D_r$, small disk}

Finally, we describe the flow on $\Omega_{spec+}\cup\Omega_{++}$.  The
flow on $\Omega_{spec+}$ is described in the previous section.
Consider $\Omega_{++}$.  The point $q_+^*$ is a stable equilibrium
point.  By a constructive procedure using the Lyapunov equation, there
exists an explicit Lyapunov function on a small disk
\[
D_r = \{x=(x_2,x_3)\mid \|x - q_+^*\|=r\}.
\]
around $q_+^*$~\cite{nikravesh2018nonlinear}.  Let $t_1$ be the time at
which $\pi(z_{spec}(t))$ enters the disk $D_r$.  We consider the
curve $\gamma$ from $q_c$ to $q_+^*$ to  given by 
the arc $\pi(z_{spec}(t))$ for
$t\in[t_c,t_1]$ followed by
the linear segment from $\pi(z_{spec}(t_1))$ to $q_+^*$.

We compute $\det(((x_2,x_3)-q_+^*),(v_2,v_3))\ge0$ on $\Omega_{++}$
with equality if and only if $(x_2,x_3)$ is one of the three points $(0,0)$,
$(2,2)$, or $q_+^*$.  This means that the trajectories always
wind monotonically around the fixed point $q_+^*$.  In particular, as it winds,
every trajectory must meet the curve $\gamma$.  By the uniqueness of
trajectories, a trajectory meeting the special trajectory
$\pi(z_{spec}(t))$ must equal the special trajectory.  Every other
trajectory must meet $\gamma$ inside $D_r$.  Thus, every trajectory
(excluding the fixed point at $q_-^*$) must enter the Lyapunov disk
$D_r$, and from there be attracted $q_+^*$.

The point $q_-^*$ is related to $q_+$ by time reversal, so the
final claim of Theorem~\ref{thm:global} follows from what we
have already proved about $q_+^*$.
\end{proof}
\index[n]{zc@$\gamma$, planar curve}



\clearpage
\newpage


\part{A Proof of Mahler's First Conjecture}\label{part:mahler}

\chapter{Fuller System for Triangular Control}

\section{Introduction}

\index{conjecture!Reinhardt}
\index{conjecture!Mahler's First}
\index{Mahler, Kurt}

The Reinhardt conjecture of 1934 asserts that among centrally
symmetric convex disks, the smoothed octagon has the least
greatest packing density.  The smoothed octagon is a modification of
the regular octagon obtained by rounding its corners with hyperbolic
arcs.
In 1947, Kurt Mahler conjectured a weak form of the Reinhardt conjecture,
when he wrote
\begin{quote}
    \textit{It seems highly probable from the convexity condition,
      that the boundary of an extreme convex domain consists of line
      segments and arcs of hyperbolae. So far, however, I have not
      succeeded in proving this assertion.} --Mahler 1947.
\end{quote}
We refer to this assertion as \emph{Mahler's First conjecture}.  The next
year, Mahler rediscovered Reinhardt's conjecture from 1934, which we
call Mahler's Second conjecture.

\index{Fuller dynamical system}
\index{Reinhardt dynamical system}
\index{Poincar\'e map!first recurrence map}
\index{divisor!exceptional}
\index{time reversal, $\tau$}
\index{singular locus}
\index{fixed points!$q_{out},q_{in}$}

In this part of the book, we give a proof of Mahler's First conjecture. 
The basic outline of the proof is as follows.  We make a detailed study of
the Fuller system. By restriction of the dynamical system to
switching times, the Fuller system becomes a discrete dynamical
system with dynamics given by a Poincar\'e first recurrence map.
We blow up the space at the singular locus.  Doing so 
introduces an exceptional divisor, which becomes the focus of
attention.  We find that the discrete Fuller-Poincar\'e map has
exactly two fixed points on the exceptional divisor.
One is stable and the other is unstable.  The fixed points
are exchanged by a time-reversing symmetry.  (These fixed points
can be interpreted as inward and outward self-similar spirals of the
Fuller system.)  We analyze the global dynamics of the Fuller system
on the exceptional divisor and show that the basin of attraction
of the stable fixed point is the entire exceptional divisor (excluding
the other fixed point).  
 
Returning to the Reinhardt dynamical system, we consider its
discrete Poincar\'e map.  We study the stable
and unstable manifolds at the two fixed points (which are also fixed
points of the Reinhardt dynamics).  We show that any trajectory of the
discrete Reinhardt dynamical system that has a cluster point
on the exceptional divisor, must approach the exceptional divisor
along the stable manifolds of the fixed points.
However, we show that the stable and unstable manifolds at the fixed points
do not contain any periodic trajectories, as required by the solution
to the Reinhardt problem.  We conclude that the solution of the Reinhardt
problem is given by a trajectory that does not meet the singular locus.
From this, it follows that the solution is bang-bang with finitely many
switches. It then follows that the solution the Reinhardt problem is a
smoothed polygon.

\section{Fuller system for Triangular Control}

We define the Fuller system for triangular control to be
the following controlled system of ordinary differential equations, taking
values in $\C$.
\begin{equation}
z_3' = z_2,\quad
z_2' = z_1,\quad
z_1' = -i u,\quad
u(t)\in\{1,\zeta,\zeta^2\}=:V_T,
\end{equation}
where $\zeta=\exp(2\pi{i}/3)$ is a primitive cube root of unity.
The control function $u$ is a measurable function of a real variable,
taking values in $V_T$.
We set $z_0:= -i u$ so that $z_1' = z_0$.
When $u\in\C$ is constant, we can
solve the Fuller system ODEs, to obtain
\begin{align}\begin{split}\label{eqn:z1z2z3}
z_0\phantom{(t)} &= -i u,\\
z_1(t) &= -i t u+ z_1^0,\\
z_2(t) &= -i t^2 u/2! + z_1^0 t + z_2^0,\\
z_3(t) &= -i t^3 u/3! + z_1^0 t^2/2! + z_2^0 t + z_3^0,
\end{split}
\end{align}
with initial conditions $(z_1^0,z_2^0,z_3^0)$ at $t=0$.

\index[n]{yz@$z_i$, Fuller system!$z=(z_1,z_2,z_3)\in\C^3$}
\index[n]{i@$i=\sqrt{-1}$}
\index[n]{u@$u$, control!$u\in{}V_T$}
\index[n]{V@$V_T=\{1,\zeta,\zeta^2\}$, vertices of control set}
\index[n]{zf@$\zeta=\exp(2\pi{i}/3)$, cube root of unity}
\index[n]{0@$-^0$ or $-_0$, initial value}
\index{Fuller system!triangular control}
\index{Hamiltonian!Fuller system}

\subsection{Hamiltonian}
The Hamiltonian for the Fuller system is
\begin{align}
\H_F(z,u) &= \RR(z_1,z_2 i) + \RR(u,z_3),\quad z =(z_1,z_2,z_3)
\\
&= \frac{1}{2} \sum_{j=0}^3(-1)^j\RR(z_{3-j},iz_j).
\end{align}
We seek solutions on an interval $[t_1,t_2]$ 
for which the control function $u$ maximizes the Hamiltonian.
\begin{equation}\label{eqn:max}
\H_F(z(t),u(t))\ge \H_F(z(t),\zeta^k),\quad \zeta^k\in\{1,\zeta,\zeta^2\},
\end{equation}
and such that
\begin{equation}\label{eqn:H0}
\H_F(z(t),u(t))=0,\quad\forall t\in[t_1,t_2].
\end{equation}
We define the maximized Hamiltonian to be
\[
\H_F^+(z):= \max_{u\in V_T} \H_F(z,u).
\]

\index[n]{H@$\H$, Hamiltonian!$\H_F$, Fuller system}
\index[n]{H@$\H$, Hamiltonian!$\H^+$, maximized}

It is easy to check that $\H_F$ is constant along every segment with
constant control. Also, if $u$ is chosen according to the maximum
principle, then $\H_F$ is constant along trajectories.

\subsection{Switching Function}

We say that $u\in V_T$ is the \emph{first control}, if it is the
optimal control starting at switching time $t=0$ until the first
positive switching time $t_{sw}>0$, for a given initial condition
$z^0\in\C^3\setminus\mb{0}$.  (It will become clear from the proof of
Lemma~\ref{lem:first-control} that $t=0$ is an isolated zero of the
relevant switching function, and that the notion of first control is
well-defined.)  

Let $z^0=(z_1^0,z_2^0,z_3^0)\in \C^3$.  We write the initial
conditions in polar coordinates: $z_j^0=r_je^{i\theta_j}$.  Let
$z_3(t,u,z^0)$ be the solution \eqref{eqn:z1z2z3} to the Fuller system
with first control $u=\zeta^k$ and initial condition $z^0$.  Set
$r_0=1$, $\theta_0=(2\pi{k}/3)-\pi/2$, and $z_0^0 = -i u = r_0
e^{i\theta_0}$.  The switching function from control $\zeta^i$ to
$\zeta^j$ is
\[
\chi^{u_k}_{u_i-u_j}(t)=\chi^k_{ij}(t)=\chi^{u_k}_{ij}(t):= 
\frac{1}{\sqrt{3}} \RR(u_i-u_j,z_3(t,u_k,z^0)).
\]
When $k=i$, we drop the superscript and write $\chi_{ij}$ for
$\chi^i_{ij}$.  We assume $i\ne j\mod 3$ so that $|u_i-u_j|=\sqrt3$,
and write
\[
u_i - u_j = \sqrt{3} e^{i\theta_{ij}},
\quad i,j\in\Z,\quad (\text{and }i=\sqrt{-1}).
\]
Then $\chi^u_{ij}(t)=\RR(e^{i\theta_{ij}},z_3(t,u,z^0))$.
We have
\begin{align*}
\theta_{i,i+1} &= -\pi/6 + 2\pi i/3\\
\theta_{i+1,i} &= 5\pi/6 + 2\pi i/3,\quad 
\theta_{i,i-1} = \pi/6 + 2\pi i/3,\quad i\in\Z.
\end{align*}
The switching function simplifies to the form
\[
\chi^k_{ij}(t)=
\sum_{m=0}^3 \frac{t^m}{m!} r_{3-m}\cos(\theta_{3-m}-\theta_{ij}).
\]

\index[n]{zx@$\chi_{ij}$, switching function}
\index[n]{zh@$\theta$, angle} 

\subsection{Symmetry}

We consider the symmetries of the system.  We start with positive
scaling.

\begin{lemma}\label{lem:rescale}
If $(z_1,z_2,z_3,u)$ is a solution to the Fuller system on $[t_1,t_2]$
satisfying \eqref{eqn:max} and \eqref{eqn:H0}, then
$(\tilde{z}_1,\tilde{z}_2,\tilde{z}_3,\tilde u)$ is a solution on $[t_1r,t_2r]$
satisfying the same constraints, where $\tilde z_j(t) = r^j z_j(t/r)$,
$\tilde u(t) = u(t/r)$, and where $r$ is real and positive.
\end{lemma}

\begin{proof}
This holds by direction substitution into the Fuller system and into
the Hamiltonian.
\end{proof}

There is a discrete rotational symmetry.

\begin{lemma}\label{lem:discrete}
If $(z_1,z_2,z_3,u)$ is a solution to the Fuller system
satisfying \eqref{eqn:max} and \eqref{eqn:H0}
with initial value $z^0=(z_1^0,z_2^0,z_3^0)$, then $(\zeta z_1,\zeta
z_2,\zeta z_3,\zeta u)$ is also a solution with initial value
$\zeta z^0$ satisfying the same constraints.
\end{lemma}

\begin{proof} Again, this holds by direct substitution.
\end{proof}

We call the group $\G$ generated by discrete rotational symmetry and
rescalings the \emph{virial group}.  (This virial group is analogous
to but not identical to the virial group that was introduced earlier
for circular control.)  We say that $z$ and $\tilde{z}$
are \emph{equivalent} and write $z\equiv \tilde{z}$ if one can be
carried to the other by the virial group, that is, by a combination of
scaling and discrete rotations, as described by
Lemmas~\ref{lem:rescale} and \ref{lem:discrete}.

\index[n]{G@$\G$, virial group}
\index{virial!group}
\index{equivalent (under virial group)}

There is also a time-reversal symmetry.
Let $\bar{\cdot}$ denote complex conjugation.
If $(z_1,z_2,z_3)\in\C^3$, set 
$\tau(z_1,z_2,z_3)=(\bar{z}_1,-\bar{z}_2,\bar{z}_3)$.

\index[n]{0@$\bar{\cdot}$, complex conjugation}
\index[n]{zt@$\tau$, time reversal involution}

\begin{lemma}\label{lem:reversal}
If $(z_1,z_2,z_3,u)$ is a solution to the Fuller system on $[t_1,t_2]$
satisfying \eqref{eqn:max} and \eqref{eqn:H0}, then
$(\bar{z}_1(-t),-\bar{z}_2(-t),\bar{z}_3(-t),\bar{u}(-t))$ is a
solution on $[-t_2,-t_1]$ satisfying the same constraints and
with \emph{terminal} value $\tau(z^0)$.
\end{lemma}

\begin{proof}
Direct substitution.
\end{proof}

The following simple lemma will allow us to draw powerful conclusions
about the discontinuities of Fuller system dynamics.  It relates the
multiplicities of roots of one switching function to the
multiplicities of roots of another switching function at time $t=0$.

\begin{lemma}\label{lem:switch-reversal}\mcite{MCA:7477621}
Let $t_0\in\R$, $z^0\in\C^3$, and $u\in\C$.  Let $z^*=z(t_0,z^0,u)$ be
the value of the Fuller ODE at time $t_0$, with initial condition
$z^0$ at time $t=0$, using constant control $u$ for all $t$.  For any
$v_1-v_2\in\C$, the switching function satisfies
\[
\chi^u_{v_1-v_2}(t_0-t,z^0)=\chi^{\bar{u}}_{\bar{v}_1-\bar{v}_2}(t,\tau(z^*)).
\]
\end{lemma}
\begin{proof}
Direct substitution.
\end{proof}

We also have invariance with respect to multiplication by $\zeta$.
\begin{equation}
\chi^{\zeta u}_{\zeta(v_1-v_2)}(t,\zeta z^0) = \chi^u_{v_1-v_2}(t,z^0).
\end{equation}

We say that $u\in V_T$ is the \emph{most recent control}, if it is the
optimal control for small negative time $t<0$ until switching time
$t=0$, for a given initial condition $z^0\in\C^3\setminus\mb{0}$.

\index{control!most recent}

For each $z^0=(z_1^0,z_2^0,z_3^0) \in\C^3$,
and $u\in V_T$,
let $z(t,z^0,u)$ be the solution to the Fuller ODE with initial condition
$(z_1^0,z_2^0,z_3^0)$ and control $u$.

\begin{lemma}\label{lem:first-negative-control}\mcite{MCA:1407931}
We have for each $t\in\R$,
\[
\tau(z(-t,z^0,u)) = z(t,\tau(z^0),\bar{u}).
\]
Also, the most recent control of $z^0$ is $u$, if $\bar{u}$ is the
first control of $\tau(z^0)$.  Moreover, the most recent 
switching time for initial condition $z^0$ is $-t_{sw}$, where
$t_{sw}$ is the first positive switching time for $\tau(z^0)$.
\end{lemma}

\index[n]{t@$t\in\R$, time!$t_{sw}$, switching time}

\begin{proof} 
The lemma follows 
from the time-reversal symmetry (Lemma~\ref{lem:reversal}).  
\end{proof}

\subsection{Walls}
We consider the Hamiltonian maximization condition \eqref{eqn:max}
in more detail. The part of the Hamiltonian depending on the
control $u$ is $\RR(z_3,u)$. At a switching time, up to equivalence by 
a discrete rotation, the maximization principle takes the form
\[
\RR(z_3,\zeta)=\RR(z_3,\zeta^2)\ge\RR(z_3,1).
\]
This defines the switching between controls $\zeta$ and $\zeta^2$.
The set of solutions in $z_3$ is $\R_{\le0}$.  We call this set
a \emph{wall}.  Now allowing discrete rotations, we define the set of
walls
\[
\W := \R_{\le0}\cup \R_{\le0}\zeta \cup \R_{\le0}\zeta^2 =:
\W_0\cup\W_1\cup\W_2.
\]
Switching of controls can only occur when $z_3\in\W$.

\index{wall}
\index[n]{W@$\W,\W_i$, wall}

\subsection{Switching Times}

Define the following function
\begin{align*}
\mb{v}&:(\C^3\setminus\mb{0})\times{}V_T\to\R^3\\
\mb{v}(z^0,u) &=(\RR(z_3^0,u),\RR(z_2^0,u),\RR(z_1^0,u)).
\end{align*}
(Note the backwards indexing.)
We call $\mb{v}(z^0,u)$ the \emph{control vector} of $u$.

\index{control!first}
\index{control!vector}
\index[n]{v@$\mb{v}$, vector!$\mb{v}(z,u)$, control vector}
\index{lexicographic order}
\index[n]{0=@$\prec$, lexicographic order}

We define the lexicographic order on vectors $\mb{v}=(\mb{v}_1,\ldots)$ by
\begin{align*}
\mb{0}\prec\mb{v}\quad&\Leftrightarrow\quad
0=\mb{v}_1=\mb{v}_2=\cdots=\mb{v}_k,~0<\mb{v}_{k+1}\\
\mb{u}\prec\mb{v}\quad&\Leftrightarrow\quad\mb{0}\prec\mb{v}-\mb{u}\\
\end{align*}

\begin{lemma}\label{lem:first-control}
Suppose $z^0=(z_1^0,z_2^0,z_3^0)\in\C^3\setminus\mb{0}$.
Let $V_{T,\max}\subset{}V_T=\{1,\zeta,\zeta^2\}$ 
be the set of controls that have the maximum lexicographical value among
$\{\mb{v}(z^0,1),\mb{v}(z^0,\zeta),\mb{v}(z^0,\zeta^2)\}$. 
(That is, let $V_{T,\max}=\arg\max_{u\in V_T}\mb{v}(z^0,u)\subseteq{V_T}$.)
Then
\begin{itemize}
\item If $V_{T,\max}=\{u\}$, then $u$ is the first control.
\item If $V_{T,\max}=\{\zeta^i,\zeta^{i+1}\}$ has two elements, then the first control is
$\zeta^i$.
\item $V_{T,\max}\ne{}V_T$.
\end{itemize}
\end{lemma}

\index[n]{@V$V_{T,\max}$}

\begin{proof} 
If two vectors have the same lexicographical order then the two
vectors are equal.  If $V_{T,\max}=V_T$, then all three control
vectors are equal.  This implies that $z^0=\mb{0}$, which is contrary
to hypothesis. Thus, $V_{T,\max}$ is a proper subset of $V_T$.

The first control is determined by the maximum principle for $t$ small
and nonnegative.  By the maximum principle at $t=0$ and the form of
the Hamiltonian, the first control must be among the controls $u\in V_T$
that
maximize 
\(
\RR(z_3^0,u)
\).
If there is more than one maximizer, then we break the tie by passing
to the first-order term in the control function 
\(
z_3(t)=z_3^0 + z_2^0 t+ \cdots
\).  
That means we consider terms $\RR(z_2^0,u)$. If again, there is a tie,
we consider terms $\RR(z_1^0,u)$.  In this way, the first control must
maximizes the lexicographical order.

Finally, if $V_{T,\max}$ contains two controls, we break the tie by
considering the highest order term $-iut^3/3!$ of $z_3(t)$, where now
$u$ is itself the first control.  
By rotational symmetry, assume without loss of generality that
$V_{T,\max}=\{1,\zeta\}$.
Assume for a contradiction that the
first control is $u=\zeta$. By the maximum principle we obtain
the following contradiction for $t$ small and positive,
\[
- \frac{t^3}{4\sqrt{3}} = \RR(-i\zeta t^3/3!,\zeta-1) \ge 0.
\]
This contradiction implies that $u=1$.
\end{proof}

\begin{corollary}\label{cor:10} Let $(z_1,z_2,z_3)\in\C^3\setminus\mb{0}$ satisfy
$|z_i|\le{r^i}$, for $i=1,2,3$.  Let $u$ be the first control.
Then the switching function from $u$ to $u/\zeta$ has a positive root
that is less than $10r$.  
In particular, the first switching time $t_{sw}$ is less than $10r$.
\end{corollary}

\begin{proof} Permuting by $V_T$, we may assume without loss of generality
that the first control mode is $u=\zeta$.  

We claim that $\mb{v}(z,\zeta)$ is greater than $\mb{v}(z,1)$ in the
lexicographic order.  By the choice of first control mode, $\mb{v}(z,\zeta)$
is at least as great as $\mb{v}(z,1)$.  Assume for a
contradiction, that the two vectors are equal.  By the lemma
$V_{T,\max}=\{\zeta,1\}$.  Again, by the lemma, this implies that the
first control is $u=1$, which is contrary to hypothesis.

The switching function from $u=\zeta$ to $u/\zeta=1$ is
\[
\sqrt{3}\chi(t)=\RR(\zeta-1,z_3(t)) 
= -\frac{t^3}{4\sqrt3} + \sum_{k=0}^2 \RR(\zeta-1,z_{3-k})\frac{t^k}{k!}.
\]
The highest order term has negative coefficient, and by our claim
about control vectors, the lowest order nonzero coefficient of the
polynomial is
positive. Hence a positive root exists.  
If $t\ge 10r$, then the bound
\[
t^{-3}\sum_{k=0}^2|\RR(\zeta-1,z_{3-k})|t^k/k!\le\sum\sqrt3|z_{3-k}|t^{k-3}/k!<
\sqrt3\sum 10^{k-3}/k! < 1/(4\sqrt3)
\]
on the lower order 
terms of $\chi(t)$ implies that $\chi(t)<0$. This completes the
proof.
\end{proof}

The switching function has a simple but remarkable symmetry.

\begin{lemma}
Fix $z^0\in\C^3\setminus\{\mb{0}\}$, with $z_3^0\in\R_{\le0}$, the wall between
control modes $\zeta$ and $\zeta^2$.  
Then the switching functions 
with initial conditions $z^0$ at $t=0$ satisfy
\[
\chi_{\zeta,\zeta^2}(t) +\chi_{\zeta^2,\zeta}(t)=0,
\]
for all $t$.  
\end{lemma}

\begin{proof} This is an easy computation.
\end{proof}

As a corollary, as we switch back and forth between control modes
$\zeta$ and $\zeta^2$, the consecutive switching times are given by
the consecutive spacings between roots of a single cubic polynomial.
Upon reaching the largest root of the cubic, the control mode is
forced to switch to $u=1$.

\section{Singular Locus}

We define the \emph{singular locus} to be the origin in $\C^3$.  The
next lemma shows that we cannot reach the singular locus, with a
bang-bang solution with finitely many switches.  

\index{singular locus}

\begin{lemma}\label{lem:nonzero}\mcite{MCA:9042347}
Let $z^0\in \C^3\setminus\mb{0}$ with $z_3^0\in\W$
and let $z$ be the trajectory for $t\ge0$ with initial
condition $z(0)=z^0$ at $t=0$ with constant control given by the maximum
principle.  Then 
$z(t)\ne\mb{0}$, for all $t$.
\end{lemma}

\begin{proof}
Up to equivalence, we can assume that the constant control
is $u=1$. 
Assume for a contradiction that $z(t_0)=\mb{0}$. 
Solving the system of
linear equations \eqref{eqn:z1z2z3} with control $u=1$, we obtain
\[
z_j^0 = (-1)^{j+1} i t_0^j/j!,\quad j=1,2,3.
\]
By assumption $z^0\ne\mb{0}$, so that $t_0\ne0$.
Then $z_3^0=i t_0^3/6 \not\in\W$, which is contrary to
assumption.  Thus, $z(t)\ne0$, for all $t$.
\end{proof}

Let $U_T\subset\C$ be the convex hull of $V_T$.  The Fuller system has
a singular arc given by $z_1(t)=z_2(t)=z_3(t)=0$ and $u(t)=0$ (the
center of $U_T$) for all $t$.  This is an obvious solution to the
Fuller ODE.  We show the nonexistence of singular arcs, other than
this one.  The nonexistence of singular arcs was proved previously for
the Reinhardt system.  It comes as no surprise that it holds for the
Fuller system.

\begin{lemma}\label{lemma:no-singular}
Let $(z_1,z_2,z_3,u)$, $z_j:[t_1,t_2]\to\C^3$ absolutely continuous,
$u:[t_1,t_2]\to{}U_T$ measurable, be a controlled Fuller trajectory
satisfying the maximum principle for the Hamiltonian $\H_F$.  Suppose
that the trajectory is singular in the sense that for all
$t\in[t_1,t_2]$, the set of controls maximizing the Hamiltonian is a
face (and not a vertex).  Then $(z_1,z_2,z_3)$ is identically zero, 
and the control function is zero almost everywhere.
\end{lemma}

\begin{proof}
We may assume $t_1=0$.
Let $u:[0,t_2]\to U_T$ be a measurable control function.  Let $u_3$
be a solution to the initial value problem $u_3'''=u$,
$u(0)=u'(0)=u''(0)=0$.  (More precisely, we assume that $u_3''$ is
absolutely continuous and its derivative is $u$ almost everywhere.)  A
solution to the initial value problem is
\begin{equation}\label{eqn:u3}
u_3(t) = \frac{1}{2}\int_0^t u(s) (t-s)^2\,ds.
\end{equation}
This representation of a solution leads to an estimate
\[
|u_3|\le \frac{1}{2}\int_0^t s^2\,ds = O(t^3).
\]
Then 
\begin{equation}\label{eqn:z3u3}
z_3(t)=-iu_3 + z_1^0t^2/2! + z_2^0t+z_3^0.
\end{equation}
We consider two cases.  In the first case, suppose that over the
interval $[0,t_2]$, the Hamiltonian is independent of the control in
$U_T$, so that the set of maximizers is all of $U_T$.  Since $U_T$
spans $\C$, by the form of the Hamiltonian, this implies that
$z_3(t)=0$ on $[0,t_2]$.  By the form of the solution $z_3$ and the
$O(t^3)$ estimate on $u_3$, we have $z_1^0=z_2^0=z_3^0=0$.  Then
$z_3=-iu_3=0$, identically.  Then also $u(t)=u_3'''(t)=0$, almost
everywhere.  This is the singular arc described above.

In the second case, assume for a contradiction that over the interval
$[0,t_2]$, the Hamiltonian is independent of the control function
$u(t)$ taking values in the edge $[\zeta,\zeta^2]\subset{}U_T$ (say).
If $u(t)\in[\zeta,\zeta^2]$, then $\Re(u(t))=-1/2$ and $\Re(u_3(t)) =
- t^3/12$.  The independence of the Hamiltonian
and \eqref{eqn:z3u3} imply
\[
0=\RR(z_3(t),\zeta-\zeta^2)= \frac{\sqrt3}{2} \frac{t^3}{3!} +
(\text{quadratic in } t)
\]
for all $t$.  This is absurd.
\end{proof}

\section{Blowing up Fuller}
We describe a (weighted) blowing up process at the singular locus.  Set
\[
\phi(z)=\phi(z_1,z_2,z_3) 
:= \left(\sum_{j=1}^3 |z_j|^{6/j}\right)^{1/6},\quad z=(z_1,z_2,z_3)\in\C^3.
\]
Then
\[
\phi(r z_1,r^2 z_2,r^3 z_3) = r\phi(z_1,z_2,z_3),\quad r > 0.
\]
Set
$\Xi:= \{\xi\in\C^3\mid \phi(\xi)=1\}$.
We have a diffeomorphism
\begin{align*}
\C^3\setminus\mb{0}&\leftrightarrow\  (\R_{>0}\times \Xi),\\
(z_1,z_2,z_3)&\mapsto (r,(\xi_1,\xi_2,\xi_3)) = 
(r,(z_1/r,z_2/r^2,z_3/r^3)),\quad r=\phi(z),
\\
(z_1,z_2,z_3)=(r\xi_1,r^2\xi_2,r^3\xi_3)
&\mapsfrom
(r,(\xi_1,\xi_2,\xi_3)).
\end{align*}

\index[n]{znxi@$\xi$, angular component}

We will often move between the two sides of this diffeomorphism
without warning, considering the right-hand side as weighted spherical
coordinates for the left-hand side.  Let $\pi_{rad}:{\R_{>0}}\times
\Xi\to{\R_{>0}}$ and $\pi_{ang}:{\R_{>0}}\times{\Xi}\to\C^3$ be the first
(radial) and second (angular) projections.  We refer to $\pi_{rad}(q)$
as the radial component of $q$ and $\pi_{ang}(q)$ as the angular
component, by analogy with spherical coordinates.  We refer to the
left-hand side as the Cartesian coordinates.

\index[n]{zv@$\phi$, weighted norm}
\index[n]{zp@$\pi$, projection!$\pi_{rad},\pi_{ang}$, radial and angular projections}
\index{angular component}
\index{radial component}
\index{Cartesian coordinates}
\index[n]{znX@$\Xi$}
\index{blowing up}

From a slightly different perspective, $\R_{\ge0}\times{}\Xi\to\C^3$ can
be viewed as an \emph{oriented weighted real blowup} of $\C^3$ at the
origin (the singular locus), where $\Xi$ is a \emph{weighted space},
$\{0\}\times{}\Xi$ is the exceptional divisor over the origin
$\mb{0}\in\C^3$, and each $\R_{\ge0}\times\{\xi\}$ is a real positive
ray through the origin. 

\index{blowup!oriented weighted real}


The diffeomorphism is equivariant with respect to the virial group,
where rescalings act by multiplication on the radial component and the
cyclic group $V_T$ acts by scalar multiplication on the angular
component of $\R_{>0}\times{}\Xi$.

Set 
\[
\Xi_\W=\{(\xi_1,\xi_2,\xi_3)\in{}\Xi
\mid \xi_3\in \W\}.
\] 
We view it as the Poincar\'e section for the Fuller system.
Let $\Xi_\W/V_T$ be the quotient of $\Xi_\W$ by the cyclic group action of
$V_T$ on $\Xi_\W$, acting diagonally.  Under the group action, the three
walls of $\W$ are identified with one another.

\index[n]{znX@$\Xi$!$\Xi_{\W},\Xi_{\W,0}$, Poincar\'e section}
\index{Fuller-Poincar\'e section}

\section{Dynamical System and Equilibrium Points}\label{sec:discrete-time}

We define a discrete-time autonomous dynamical system
$F:(\R_{>0}\times{\Xi})_\W\to(\R_{>0}\times{\Xi})_\W$ as follows.  Let
$q\in(\R_{>0}\times{}\Xi)_{\W_i}$.  The point $q$ has Cartesian
coordinates $z^0\in\C^3\setminus\mb{0}$ with
$z_3^0$ lying in the $i$th wall.  Let $z(t)$ be the solution to the
Fuller differential equations with initial condition $z(0)=z^0$ at time $t=0$
and control defined by the maximum principle.  Let $t_{sw}>0$ be the
first positive switching time.  Then let $F(q)\in(\R_{>0}\times{\Xi})_\W$
equal $z(t_{sw})$, rewritten in spherical coordinates
$\R_{>0}\times{}\Xi$.
By the equivariance of the construction with respect to 
cyclic rotations $V_T$, we find that 
$F$ descends to a well-defined map (denoted by the same symbol):
\[
F:(\R_{>0}\times \Xi)_\W/V_T \to (\R_{>0}\times \Xi)_\W/V_T.
\]

\index[n]{F@$F$, Poincar\'e map!first recurrence}
\index[n]{F@$F$, Poincar\'e map}

We can go further by considering scaling symmetries of the Fuller
system.  By the scaling symmetries, it is clear that if the angular
components $\pi_{ang}(q_1)=\pi_{ang}(q_2)$ are equal, then the angular
components $\pi_{ang}F(q_1)=\pi_{ang}F(q_2)$ of the images are equal.  Thus,
$F$ gives a well defined discrete-time autonomous dynamical system
\[
F_{ang}:\Xi_{\W}\to{\Xi_{\W}},\quad
\text{and by equivariance}\quad 
F_{ang}:\Xi_{\W}/V_T\to{\Xi_{\W}/V_T}.
\]

\index[n]{F@$F$, Poincar\'e map!$F_{ang}$, angular component}

By \emph{Poincar\'e map}, we will always mean
a map that discretizes a bang-bang dynamical system, by passing from one
switching time to the next. In this sense,
we think of $\Xi_\W$ as the Poincar\'e section and $F_{ang}$ as the
Poincar\'e map of the Fuller dynamical system on the angular component
(with the understanding that the virial group symmetries have been
built into $F_{ang}$).  

\begin{remark} The notation $F$ will be used in this chapter
for various versions of the Poincar\'e map in a context-dependent way.
The symbol $F$ can denote either the Poincar\'e map for the Fuller
system or the Poincar\'e map for the Reinhardt system, depending on
the context.  For the Fuller system, various domains are possible:
\[
\C^3\setminus\{\mb{0}\},\quad\Xi_{\W},\quad\Xi_{\W,0},
\quad\Xi_{\W}/V_T,\quad\Xi_{\W,0}/V_T,
\]
or various coordinate charts of these domains.
\end{remark}

\index{Poincar\'e map}
\index{Poincar\'e map!Fuller-Poincar\'e $F$}
\index{Poincar\'e map!Reinhardt-Poincar\'e}

We analyze the equilibrium points of the dynamical system $F_{ang}$.  

\index[n]{q@$q$, point on manifold!$q_{in},q_{out},q_{fix}$, fixed points}
\index[n]{r@$r$, real number!$r_{scale}\approx6.27$, scaling factor}

\begin{lemma} \label{lem:qout}
The dynamical system $F_{ang}$ has exactly two fixed points in
$\Xi_{\W}/V_T$ at which the Hamiltonian vanishes.  They are the
switching points $q_{out},q_{in}$ of outward and inward triangular
spirals. 
They are related by time-reversing symmetry: $\tau(q_{out})=q_{in}$.
After virial rescaling to make the real part of the first
coordinate equal to $-1$, the fixed point $q_{out}$ takes the form
\begin{align}\label{eqn:equil-exact}
\begin{split}
\\
q_{out} &\equiv
(
-1 + i\frac{-1 + r}{\sqrt3 
   (1 + r)}, 
-\frac{-1 + r^3}{
   \sqrt3 (1 + r^3)} + i
\frac{ 1 - 3 r - 2 r^2 - 3 r^3 + 
   r^4}{3 (1 + r + r^3 + 
    r^4)}, 
\\
&\qquad\qquad
\frac{ -2 (1 + r - 4 r^3 - 7 r^4 - 
    9 r^5 - 7 r^6 - 4 r^7 + 
    r^9 + r^{10})}{9 (1 + r)^2 
   (1 - r + r^2) (1 + r^3 + 
    r^6)}
)\mod \G,
\\
\end{split}
\end{align}
where $r=\rs\approx6.27$ is the unique real root greater than $1$ of the
palindromic polynomial
\[
1-5r-7r^2-5r^3-7r^4-5r^5+r^6.
\]  
The first switching time of the initial conditions on the right-hand
side of \eqref{eqn:equil-exact} is
\[
t_{sw} = \frac{2(1+r+r^2)}{\sqrt3 (r+1)}\approx 7.4
\]
\end{lemma}

\begin{proof}
Let $q$ be a lift to $\Xi_\W$ of a fixed point in $\Xi_\W/V_T$.  Let
$u\in{}V_T$ be the control, starting at $t=0$ until the first positive
switching time $t_{sw}$.  Let $z^0\in\C^3\setminus\mb{0}$ be the
Cartesian coordinates of $(1,q)\in{}\R_{>0}\times{}\Xi_\W$.  The
fixed-point conditions are 
\begin{equation}\label{eqn:fixed-point}
z_j(t_{sw}) = r^j \zeta_1 z_j^0,\quad (\text{modulo } \G),
\end{equation} 
for some $r>0$ and some $\zeta_1\in{}V_T$, where $z_j(t_{sw})$
is given by \eqref{eqn:z1z2z3}.  Solving the linear
equations \eqref{eqn:z1z2z3} for $z_3^0$, we find
\[
z_3^0 = -i\frac{t_{sw}^3u(1+2r\zeta_1 + 2 r^2 \zeta_1 + r^3 \zeta_1^2)}
{6 (-1 + r\zeta_1) (-1+r^2\zeta_1) (-1 + r^3\zeta_1)}.
\]
It follows that $z_3^0\ne0$ (because $r>0$, $t_{sw}>0$ and $\zeta_1\in
V_T$).  From the control $u$, at times $t=0,t_{sw}$ we must have by
the maximum principle
\begin{align*}
\RR(z_3^0,u)&\ge \RR(\zeta_2z_3^0),\quad \forall\zeta_2\in V_T
\\
r^3\RR(z_3^0\zeta_1,u)=\RR(z_3(t_{sw}),u)&\ge\RR(\zeta_2z_3(t_{sw})),
\quad\forall\zeta_2\in V_T\\
\RR(z_3^0\bar{u})&=\RR(z_3^0\zeta_1\bar{u})\ge\RR(z_3^0\zeta_1^2),
\\
z_3^0\bar{u}\zeta_1^2&\in\R_{\le0}.
\end{align*}
This final condition implies that the scaling factor in the virial
group is
\[
(r,\zeta_1)\in\{(1,\zeta), (\rs,\zeta^2), (1/\rs,\zeta^2)\}\subset \R_{>0}\times{}V_T,
\]
where the scaling factor $\rs\approx 6.27$ and $1/\rs$ are the only
two real roots of the palindromic polynomial given in the lemma,
and $\zeta=\exp(2\pi i/3)$.  If
$r=1$, then $\H_F(z^0,u)=1/(4\sqrt{3})\ne0$, and the solution is
rejected.  The two other solutions are the outward triangular spiral
with parameters $(r,\zeta_1)=(\rs,\zeta^2)$ and the inward triangular
spiral with parameters $(1/\rs,\zeta^2)$.  The coordinates in the
statement of the lemma have been rotated by $V_T$, choosing the first
control $u=\zeta_1^2$, to make the third coordinate real and negative.
\end{proof}

\begin{remark}
The fixed point $q_{out}$ of $F_{ang}$ is an outward triangular spiral for $F$
in a precise sense.  By
\eqref{eqn:fixed-point}, the iterates of $F$ satisfy
\[
F^k(q_{out})= (r_{scale},\zeta^2)^k\cdot q_{out},
\]
where $\cdot$ is the virial action.  These are discrete points on a
logarithmic spiral.  The points move outward because $r_{scale}>1$.
Similarly, $F^k(q_{in})$ are points on an inward moving logarithmic
spiral.
\end{remark}

\begin{remark} 
The fixed point $q_{fix}$ corresponding to
$(r,\zeta_1)=(1,\zeta)\in\G$ in the proof has the form
\[
q_{fix}=(1,-i/2,-1/2) \mod \G.
\]  
This fixed point has a nice interpretation.
In Section~\ref{sec:6k+2-gons}, we constructed a one-dimensional
family of Pontryagin extremals of the Reinhardt control problem,
indexed by a parameter $y_0\in(1/\sqrt{3},1)$.  This family includes
the smoothed octagon and the smoothed $6k+2$-gons.  Setting $y_0 =
1-r$, we may express this one-parameter family of extremals in
coordinates $(z_1(r),z_2(r),z_3(r))$, following a procedure described
below in Section~\ref{sec:asymptotic}.  Letting $r$ tend to zero, we
have asymptotics
\[
(z_1(r),z_2(r),z_3(r))=q_{fix} + \text{higher order terms} \mod \G.
\]
Thus, in a precise sense, $q_{fix}$ is the fixed point in the Fuller
system coming from the family of extremals in Reinhardt system that
includes the smoothed octagon.  It is particularly noteworthy that the
Fuller-system Hamiltonian is not zero at $q_{fix}$, although it is
constructed as a limit of points in the Reinhardt-system Hamiltonian zero set.
\end{remark}
\mcite{MCA:zfix1}

By the constancy of the Hamiltonian, the map $F$ restricts to the
zero set of the Hamiltonian. 

\begin{lemma}\label{lem:eigenvalues}\mcite{MCA:2776759}
Restrict $F_{ang}$ to the subset $\Xi_{\W,0}/V_T$ of $\Xi_\W/V_T$
on which the Hamiltonian vanishes.  On that
subset, the fixed point $q_{out}\in{}\Xi_{\W,0}/V_T$ 
is an asymptotically stable
equilibrium, and the fixed point $q_{in}$ is unstable.
\end{lemma}

\index{Mathematica}

\begin{proof}
The second assertion follows from the first, because $q_{in}$ is
obtained by time reversal from $q_{out}$.  It suffices to show that
$q_{out}$ is asymptotically stable.  This is a routine stability
calculation.  An open neighborhood of $q_{out}$ in $\Xi_{\W,0}/V_T$
is diffeomorphic to an open subset of $\R^3$.  An explicit calculation
of the eigenvalues of the Jacobian matrix in terms of local
coordinates centered at $q_{out}$ gives the result.  Numerically, the
three eigenvalues have absolute value less than $0.1$.  The
calculations were made in Mathematica.
\end{proof}

\chapter{Stable and Unstable Manifolds at Fixed Points}

In this chapter, we return to the Reinhardt dynamical system.
More specifically, we return to the blowup of the Reinhardt system and
consider the Fuller system as the restriction of the Reinhardt system
to the exceptional divisor of the blowup.  We describe the stable and
unstable manifolds at the fixed points $q_{in}$ and $q_{out}$, now viewed
as fixed points in the blowup of the Reinhardt system.

\begin{figure}
    \centering
    \includegraphics[scale=0.5]{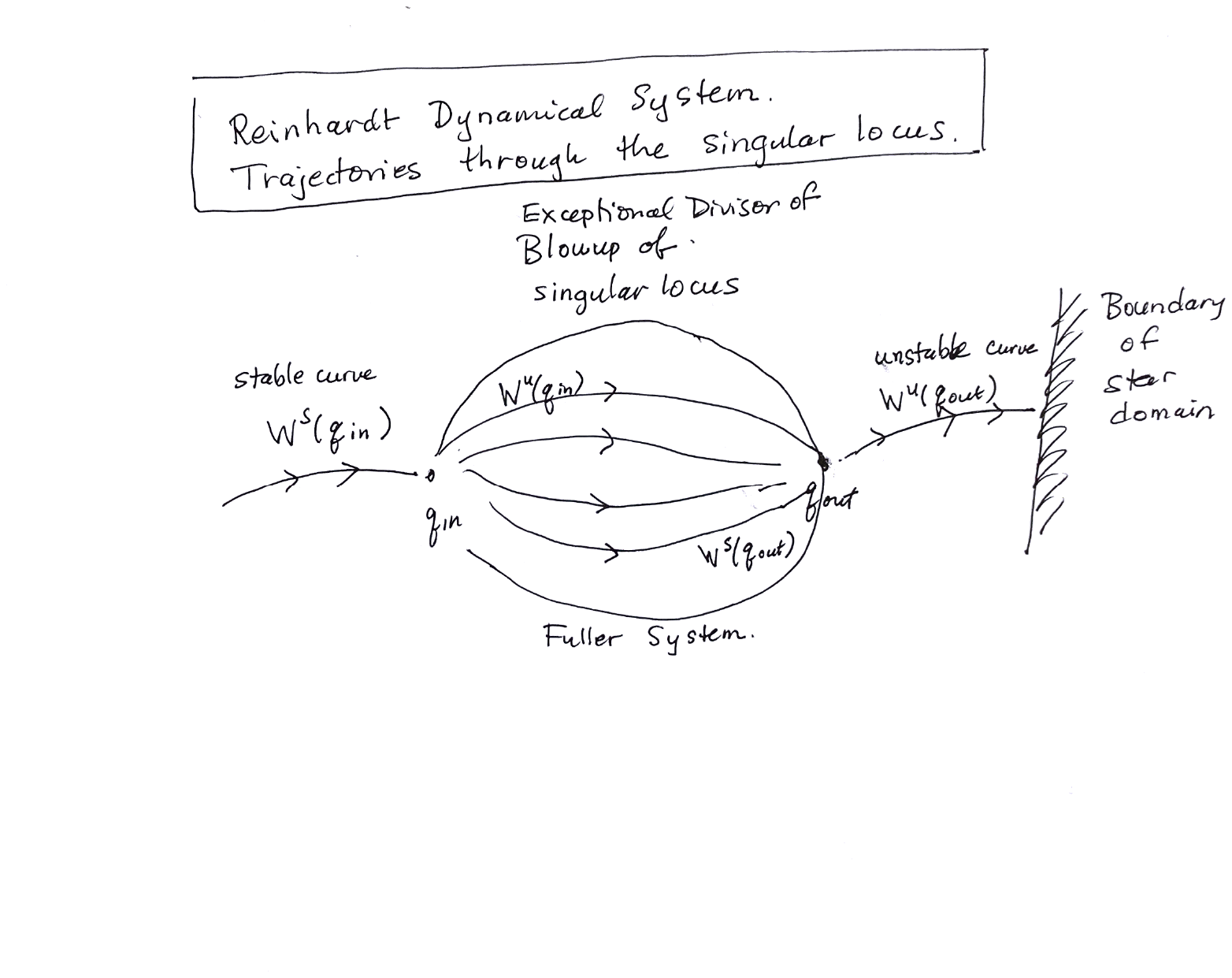}
     \caption{Schematic representation of the Poincar\'e map of the Reinhardt system along
trajectories that meet the singular locus.  The picture is four-dimensional
and the exceptional divisor is three dimensional.  
The picture has a time-reversing
 symmetry that reverses the direction of arrows.}
     \label{fig:scan-schematic}
\end{figure}

\index[n]{d@$d$, determinant!$d_1$,\ $\det(\Lambda_1)=d=\epsilon d_1^2$}
\index[n]{ze@$\epsilon\in\{-1,0,1\}$, sign!$\epsilon\in\{-1,0,1\}$, sign $\det(\Lambda_1)$}
\index[n]{zr@$\rho=\beta/\alpha >0$, control parameter}
\index[n]{zL@$\Lambda$, costate!$\lambda_{cost}$, cost multiplier}
\index[n]{W@$W^s(q),W^u(q)$, stable and unstable manifolds}

We use the parameter values $d_1=3/2$, $\epsilon=1$, $\rho=2$,
$\lambda_{cost}=-1$.  Any trajectory that meets the singular locus
must have these parameter values.  Figure~\ref{fig:scan-schematic}
gives a schematic representation of the Poincar\'e map for Reinhardt
system trajectories meeting the singular locus.  This chapter and the
next one will prove theorems about the qualitative features of this
picture.  In particular, the trajectories approach the singular locus
toward the fixed point $q_{in}$ along a stable curve $W^s(q_{in})$.
The flow along the exceptional divisor is governed by the Fuller
system, and every point except $q_{in}$ lies in the basin of
attraction of $q_{out}$.  The unstable manifold $W^u(q_{in})$ flows
into the fixed point $q_{out}$. The trajectories exit the exceptional
divisor along an unstable curve $W^u(q_{out})$ at $q_{out}$ that meets
the boundary of the star domain.

\section{Lie Algebra Coordinates}

In much of this book, we have worked with hyperboloid coordinates in
both Cartesian $z$ and spherical $(r,\xi)$ coordinate form.  We
now return to Lie algebra $\sl(\R)$ coordinates.  Moving forward, we will
use the explicit solutions to the ODEs with constant control,
expressed in Lie algebra coordinates.  For computer calculations,
$\sl$ coordinates have the slight advantage of avoiding complex
numbers.


In the interest of developing asymptotic formulas near the
singular locus, for any $X,\Lambda_1,\Lambda_R\in\sl$
subject to the usual conditions $\det(X)=1$, $\det(\Lambda_1)=d_1^2=9/4$, we write%
\footnote{Here, $\tilde\Lambda_R$ is unrelated to
an earlier term with the same name.}
\begin{equation}\label{eqn:Xtilde}
X = J + r \tilde X,\quad
\Lambda_1 = (-3/2)J +  r^2 \tilde\Lambda_1,\quad
\Lambda_R = r^3 \tilde\Lambda_R,
\end{equation}
for some $\tilde X$, $\tilde\Lambda_1$, and $\tilde\Lambda_R\in\sl$,
where $r>0$ is a real parameter.
Let $\tilde{X}$, $\tilde{\Lambda}_1$ and $\tilde{\Lambda}_R$ have
matrix entries $\tilde{x}_{ij}$, $\tilde{\ell}_{ij}$ and
$\tilde{\ell}_{Rij}$, respectively.  
\[
\tilde{X} = \begin{pmatrix} \tilde{x}_{11}&\tilde{x}_{12}\\
\tilde{x}_{21}&\tilde{x}_{22}
\end{pmatrix}\quad
\tilde{\Lambda}_1 = \begin{pmatrix} \tilde{\ell}_{11}&\tilde{\ell}_{12}\\
\tilde{\ell}_{21}&\tilde{\ell}_{22}
\end{pmatrix}\quad
\tilde{\Lambda}_R = \begin{pmatrix} \tilde{\ell}_{R11}&\tilde{\ell}_{R12}\\
\tilde{\ell}_{R21}&\tilde{\ell}_{R22}
\end{pmatrix}.
\]
We are particularly interested in
points where $r$ is small and positive, and where $\tilde{x}_{11}>0$.
We make the representation
\eqref{eqn:Xtilde} unique by scaling by $r>0$ so that $\tilde{x}_{11}=1$.


In Lie algebra coordinates, the rotational group action by the cyclic
group of order three is the action by powers of $\Ad_{R}$.  The walls
are determined by the vanishing of the switching functions.  Up to
rotational symmetry, we can assume that the trajectory starts at the
wall $\chi_{23}=0$ between control matrices $Z_{010}$ and $Z_{001}$.

\index[n]{YZ@$Z_u$, control matrix}

To introduce a coordinate system, we restrict the domain by the conditions
\begin{align}\label{eqn:coord-restriction}
0<r,\quad \tilde{x}_{11}=1,\quad \tilde{\ell}_{21}r^2 <3/2.
\end{align}
Suppose $\Phi(z)=J+r\tilde{X}$, $z=x+iy\in \hstar$, with $x>0$.  We
have $\tilde{x}_{21}\le\sqrt3-1/r$ if and only if $x\ge1/\sqrt3$,
which lies outside the star domain $\hstar\subset\h$.  Therefore we
assume $\tilde{x}_{21}>\tilde{x}_{21}^*(r):=\sqrt3-1/r$.  Set
\[
H^* := \{(r,\tilde{x})\in\R^4\mid \tilde{x}_{21}> \tilde{x}^*_{21}(r),
\ \tilde{\ell}_{21}r^2 <3/2\}.
\]
where we write
$\tilde{x}:=(\tilde{x}_{21},\tilde{\ell}_{11},\tilde{\ell}_{21})\in\R^3$.

\index[n]{H@$H$, Euclidean region!$H^*\subset\R^4$, coordinate chart}
\index[n]{x@$\tilde{x}^*_{21}(r)=\sqrt{3}-1/r$, star domain boundary curve}

If $r>0$ and $\tilde{x}_{21}>\tilde{x}^*_{21}$, then
\[
1 + r\tilde{x}_{21} \ge 1+r\tilde{x}^*_{21} = \sqrt3 r >0.
\]
Hence, we may invert $1 + r\tilde{x}_{21}$.  Then $X$ is determined
uniquely by $r$ and $\tilde{x}_{21}$.  Specifically,
\[
\tilde{X}=\mattwo{r}{\frac{-(1+r^2)}{1 + r\tilde{x}_{21}}}{1+r\tilde{x}_{21}}{-r}.
\]
The element $\Lambda_1$ is uniquely determined by $r$,
$\tilde{l}_{11}$ and $\tilde{l}_{21}$ when $\tilde{\ell}_{21}r^2<3/2$ 
via the relations
$\det(\Lambda_1)=9/4$, $\tr(\Lambda_1)=0$.  Furthermore,
$\Lambda_R\in\sl$ is uniquely determined as a function of
$(r,\tilde{x})\in{H^*}$ with $r>0$ by the three linear equations
\begin{equation}\label{eqn:LambdaR-linear}
\bracks{X}{\Lambda_R}=\H(Z_{010},X,\Lambda_1,\Lambda_R)=\chi_{23}(X,\Lambda_R)=0,
\end{equation}
(which is always a full rank system of linear equations for
$\Lambda_R$).  The diagonal entries of $\Lambda_R$ have order $O(r^4)$
and the off-diagonal entries have order $O(r^3)$. The entries of
$\Lambda_R$ are polynomials in $r$, $\tilde{x}_{ij}$,
$\tilde{\ell}_{ij}$, $(1 + r\tilde{x}_{21})^{-1}$, and
$(3/2-\tilde{\ell}_{21}r^2)^{-1}$.  In summary,
$(r,\tilde{x})\in{}H^*$ with $r>0$ is a local coordinate for
$(X,\Lambda_1,\Lambda_R)$.

\index[n]{O@$O$, Landau big oh}

\section{Asymptotics}\label{sec:asymptotic}

\index[n]{znX@$\Xi$!$\Xi_{\W},\Xi_{\W,0}$, Poincar\'e section}
\index{Fuller-Poincar\'e section}

Let 
\[
\tilde{x}=(\tilde{x}_{21},
\tilde{\ell}_{11},\tilde{\ell}_{21})\in{\R^3},\quad \tilde{x}_{11}=1.
\]  
We construct a uniquely determined element
$\xi(\tilde{x})\in{\Xi_{\W,0}}$ as follows.  For each $r>0$
sufficiently small, we have $1+r\tilde{x}_{21}\ne0$,
$\tilde{\ell}_{21}r^2 < 3/2$, and these conditions allow us to form a
triple
$(X(r,\tilde{x}),\Lambda_1(r,\tilde{x}),\Lambda_R(r,\tilde{x}))$, as
just described.  By the Cayley transform to $\SU$, expressed in terms
of hyperboloid coordinates, the triple determines
$(w_r,b_r,c_r)\in\C^3$.  We set
\[
\hat z_1(r,\tilde{x}) = w_r/\rho,\quad \hat z_2(r,\tilde{x}) 
= -i b_r/(2\rho),\quad \hat z_3(r,\tilde{x}) = c_r/(6\rho),\quad \rho=2,
\]
according to the truncation rules of Equation~\eqref{eqn:wbc-to-z}. 
We use
Landau $O$ to describe asymptotic behavior as $r$ tends to $0$.  
A calculation gives
\begin{align*}
\hat z_1(r,\tilde{x}) 
&= \frac{r}{2}(\tilde{x}_{21} + i \tilde{x}_{11}) + O(r^2),
\\
\hat z_2(r,\tilde{x}) 
&= \frac{r^2}{3!}(\tilde{\ell}_{11} - i\tilde{\ell}_{21}) + O(r^3),
\\
\hat z_3(r,\tilde{x}) 
&= \frac{r^3}{4!}(\tilde{\ell}_{R12}+\tilde{\ell}_{R21}) + O(r^4).
\end{align*}
Then we define
$\xi(\tilde{x})\in{\Xi_{\W,0}}$ to be the angular component of
\[
z(\tilde{x}):=\lim_{r\to0} (\hat z_1(r,\tilde{x})/r,
\hat z_2(r,\tilde{x})/r^2,\hat z_3(r,\tilde{x})/r^3).
\]
(This limit exists and is nonzero.)
By developing in a series, we have an asymptotic relation between the Hamiltonians
in the Reinhardt and Fuller systems.
\[
\H(Z_u,X(r,\tilde{x}),\Lambda_1(r,\tilde{x}),\Lambda_R(r,\tilde{x}))
= 24 r^3 \H_F(\tilde{u},z(\tilde{x}))+O(r^4),
\]
where controls $u\in\{\mb{e}_j\}$ for Reinhardt and
$\tilde{u}\in{}V_T$ for Fuller
correspond by $\mb{e}_j\mapsto\zeta^{j-1}$, for $j=1,2,3$.
The equations~\eqref{eqn:LambdaR-linear} imply
\[
\H_F(z(\tilde{x}),\zeta)=\H_F(z(\tilde{x}),\zeta^2)=0.
\]
We are particularly interested in the controls $Z_{010}$ (and
$\zeta\in{}V_T$), because they are the controls at the fixed point
$q_{out}$, when represented according to our conventions.

\begin{remark}
For example, if we take
\[
\tilde{x}_{out} := (\tilde{x}_{21out},\tilde{\ell}_{11out},\tilde{\ell}_{21out}) =
\left(\frac{x_1}{y_1},\frac{3x_2}{2y_1^2},\frac{-3y_2}{2y_1^2}\right) \approx
(-2.39,-4.90,-1.12),
\]
where $z_{out,i}=x_i+iy_i$, then $\xi(\tilde{x}_{out})$ is equal to the
outward fixed point $q_{out}$ of the Fuller system modulo the virial
group.
\end{remark}

We can write the Reinhardt-Poincar\'e map $F$ in local coordinates 
$(r,\tilde{x})\mapsto F(r,\tilde{x})$.
We expect asymptotic expansions in $r$ of the Reinhardt system whose
leading term is given by the Fuller system.  By
Lemma~\ref{lem:rescale}, when taking asymptotics, we should work with
a rescaled time 
\[
s := t/r.
\]  

\index[n]{s@$s$, real parameter!$s=t/r$, rescaled time}

\index{Mathematica}

\begin{remark}
A general strategy is given by Manita and Ronzhina in their inverted
pendulum paper.
We produce essentially equivalent results by working with our explicit
solutions to the constant control ODEs and computing asymptotics using
Mathematica.  Their paper has also inspired our discussion of blowup.
\end{remark}

The next lemma shows that the leading term in the constant control
solution for $(X,\Lambda_1,\Lambda_R)$ is given by the Fuller system.
By cyclic symmetry, we may confine ourselves without loss of
generality to the constant control matrix $Z_{010}$.  For a given
$\tilde{x}$, the first control of
$(X(r,\tilde{x}),\Lambda_1(r,\tilde{x}),\Lambda_R(r,\tilde{x})$ is
$u=\mb{e}_2$ for all sufficiently small $r>0$, provided we assume
\[
\tilde{\ell}_{11}+\tilde{\ell}_{21}\tilde{x}_{21}<0,\quad\text{and}
\quad
\tilde{\ell}_{21}<0.
\]

\begin{lemma}\label{lem:Lie-Fuller}\mcite{MCA:2446603}
Let $(X,\Lambda_1,\Lambda_R)$ be solutions (expressed in local
coordinates as $(r(s),\tilde{x}(s))$ in rescaled time $s=t/r_0$) to
the Reinhardt ODE with constant control $Z_{010}$ and initial
condition $(r_0,\tilde{x}^0)$ (in local coordinates).  Let
$z=z(s)=(z_1(s),z_2(s),z_3(s))$ be solutions to the Fuller ODE with
constant control $u=\zeta\in{}V_T$ and initial condition
$z(\tilde{x}^0)$.  Then for each $s$ such that $r(s)<1$, we have
\begin{align*}
\hat z_1(r(s),\tilde{x}(s)) &= z_1(s)r_0 + O(r_0^2)\\
\hat z_2(r(s),\tilde{x}(s)) &= z_2(s)r_0^2 + O(r_0^3)\\
\hat z_3(r(s),\tilde{x}(s)) &= z_3(s)r_0^3 + O(r_0^4).
\end{align*}
That is, the leading term of the solutions of the Reinhardt and Fuller
systems are in agreement.
\end{lemma}

\begin{proof}
We go from the $\SL$ picture to $\SU$ by means of the Cayley
transform,  then switch to hyperboloid coordinates.  Use the explicit
solutions on both sides and expand as a series in the parameter $r_0$.
\end{proof}

\begin{lemma}\label{lem:chi-asymptotic}
With the same setup and matching initial conditions as in the previous lemma,
we have switching function asymptotics
\begin{align}\label{eqn:cubic-sr3}
\chi_{21}(X(s r_0),\Lambda_1(s r_0),\Lambda_R(s r_0)) 
= 24 r_0^3 \RR(\zeta-1,z_3(s )) + O(r_0^4)
\end{align}
\end{lemma}

Note that the term on the right is the switching
function of the Fuller system for the control mode transition
$\zeta\to1$.  Similar formulas hold for the other switching functions.

\begin{proof}
Expand both sides in an explicit series.  
\end{proof}

\section{Analytic Extension of the Reinhardt system}

The following is the key lemma.  It shows that we have succeeded
in transforming the behavior near the singular locus into something
quite pleasant.

\begin{lemma}
Let $\tilde{x}_{out}\in\R^3$ be the parameter associated with
the outward fixed point of the Fuller system.  The
Reinhardt-Poincar\'e map $F$ (initially defined for $r>0$)
extends to an analytic diffeomorphism in
a neighborhood of the fixed point $(0,\tilde{x}_{out})\in\R^4$
(including non-positive values of $r$), such that $F$ coincides with
the Fuller-Poincar\'e map at $r=0$.
\end{lemma}

\begin{remark} Similar analytic extensions across $r=0$ can be carried out
under more general conditions in a neighborhood of other points $\tilde{x}$.
However, we must be cautious when the least positive root of the cubic
of the Fuller system \eqref{eqn:cubic-sr3} is not simple, when $s=0$
is a root of the cubic, or when the resultant of two Fuller switching
functions is zero.
\end{remark}

\begin{proof}
We begin by establishing analyticity of $F$ on some neighborhood of
the fixed point, expressed in the coordinates $(r_0,\tilde{x}_0)$.
(We add subscript $0$ to suggest that these are initial conditions of
the Reinhardt ODEs.)  We can take these coordinates with values in
either $\R$ or $\C$.  Lie algebra coordinates
$(X_0,\Lambda_{10},\Lambda_{R0})$ are rational functions with nonzero
denominators (and hence analytic) in the variables $(r_0,\tilde{x}_0)$
and $(1+r_0\tilde{x}_{21})^{-1}$.  (The denominators are
nonzero in a neighborhood of the fixed point.)  The
constant control $Z_{010}$ solutions $(X(t),\Lambda_1(t),\Lambda_R(t))$ to
the Reinhardt ODEs are given by matrix exponentials and are hence
analytic in time $t$ and initial conditions
$(X_0,\Lambda_{10},\Lambda_{R0})$.  We make a substitution $t = sr_0$
to give reparameterized time.

By Lemma~\ref{lem:chi-asymptotic}, by division of power series, the
function
\[
\chi_{21}(X(sr_0),\Lambda_1(sr_0),\Lambda_R(sr_0))/r_0^3
\]
extends to an analytic function $\tilde\chi_{21}(s,r_0,\tilde{x}_0)$
in a neighborhood of $r_0=0$.  

Restricting at first to real coordinates and $r_0\ge0$, the switching
time in rescaled coordinates is defined by the least positive zero
$s_{sw}=s_{sw}(r_0,\tilde{x}_0)$ of
$\tilde\chi_{21}(s,r_0,\tilde{x}_0)=0$.  At the fixed point
$(r_0,\tilde{x}_0)=(0,x_{out})$, the least positive zero
$s_{sw,out}\approx 8.84>0$ is a simple zero, and $s=0$ is not a zero.
If we remove the condition $r_0\ge0$, then by the analytic implicit
function theorem, we have a unique analytic extension of the switching
time $s_{sw}(r_0,\tilde{x}_0)$ to a neighborhood of the fixed point
such that
\[
\tilde\chi_{21}(s_{sw}(r_0,\tilde{x}_0),r_0,\tilde{x}_0)=0,
\quad s_{sw}(0,x_{out})=s_{sw,out}.
\] 
\index[n]{Y@$Y\in\mathfrak{g}$, Lie algebra element!in $\sl$}

Evaluating the solutions to the ODEs at the unscaled switching time
$t_{sw}=r_0s_{sw}$, and rotating by cyclic virial symmetries
($\ad_{R^{-1}}$), we obtain analytic functions of $(r_0,\tilde{x}_0)$:
\[
Y_{sw}(r_0,\tilde{x}_0):=
\ad_{R^{-1}}(Y(r_0s_{sw}(r_0,\tilde{x}_0),r_0,\tilde{x}_0)),
\quad\text{where } Y=X,\Lambda_1,\Lambda_R.
\]
We write $X=X(t,r_0,\tilde{x}_0)$, etc. to make the dependence on
initial conditions $(r_0,\tilde{x}_0)$ explicit.  The rotation is
chosen to make the fixed point property hold exactly, and not just up
to rotation.  These functions give the value of the
Reinhardt-Poincar\'e map.  

Finally, we show that we can analytically convert the Lie algebra
coordinates $(X,\Lambda_1,\Lambda_R)$ back to the coordinate system
$(r,\tilde{x})$.  The matrix coefficients of
\[
X(r_0s_{sw},r_0,\tilde{x}_0)-J,
\quad\text{and }
\Lambda_1(r_0s_{sw},r_0,\tilde{x}_0)+(3/2)J
\] 
are divisible by $r_0$, and $r_0^2$, respectively (regardless of the
precise form of $s_{sw}$).  The same is true of $X_{sw}-J$ and
$\Lambda_{1sw}+(3/2)J$.  Because of our convention $\tilde{x}_{11}=1$,
we must take the reciprocal of the $(1,1)$ matrix coefficient of
$(X_{sw}-J)/r_0$.  The value of this matrix coefficient at
$(0,\tilde{x}_{out})$ is
\[
r_{scale} \approx 6.27\ne0.
\]
This is the scaling factor, obtained as a root of the palindromic
polynomial considered above.  Since this $(1,1)$ matrix coefficient is
an analytic function that is nonzero in a neighborhood of the fixed
point, its reciprocal is an analytic function.  This completes the
proof of analytic continuation to a neighborhood of the fixed point.

The asymptotic formulas given above show that the restriction of (the
analytic continuation of) $F$ to $r=0$ is precisely the Fuller system.

Next, we show that the Reinhardt-Poincar\'e map is a diffeomorphism.
In similar way to what we have done, we can show analyticity and
analytic continuation of the Reinhardt-Poincar\'e map $F^{-1}$ that
moves backwards in time.  When the parameter $r_0$ is positive and in
a sufficiently small neighborhood of the fixed point, we have that $F$
and $F^{-1}$ are inverse functions.  By analytic continuation, they
are inverse functions in a neighborhood of the fixed point.  Hence,
$F$ is a diffeomorphism.
\end{proof}

\index{exceptional divisor}
\index{stable!manifold}
\index{manifold!stable, unstable}
\index{hyperbolic!fixed point}
\index{analytic continuation}

We refer to the hypersurface $r=0$ as the exceptional divisor.  We
refer to Irwin for background material about local unstable and stable
manifolds near a hyperbolic fixed point of a
diffeomorphism~\cite{irwin}. A brief summary appears in
Appendix~\ref{sec:stable-manifold}.  We write $F$ for the analytic
diffeomorphism that we have constructed, which extends the
Reinhardt-Poincar\'e map.

\begin{theorem}  
The fixed point $(0,\tilde{x}_{out})$ of the diffeomorphism $F$ is
hyperbolic.  The local unstable manifold is a $C^\infty$ curve.  In a
neighborhood of the fixed point, the local stable manifold coincides
with the three-dimensional exceptional divisor $r=0$.
\end{theorem}

\begin{remark} 
By time reversal, the Reinhardt-Poincar\'e map $F$ has a hyperbolic
fixed point at $(0,x_{in})\in\R^4$, its local stable manifold is
one-dimensional, and its local unstable manifold is the
three-dimensional exceptional divisor.
\end{remark}

\begin{proof}
Near the fixed point, the analytic continuation $F$ of the
Reinhardt-Poincar\'e map agrees with the Fuller-Poincar\'e map, when
$r=0$, when we use coordinates $\tilde{x}\in\R^3$ for points on the
exceptional divisor $\xi(\tilde{x})\in{}\Xi_{\W,0}$ as above.  The
exceptional divisor $r=0$ is a three-dimensional invariant subset of
the diffeomorphism.  We have seen that the linearization of the
Fuller-Poincar\'e map $F$ near the fixed point is a contraction on the
three dimensional exceptional divisor.  Moving away from the
exceptional divisor in the radial direction, the Reinhardt-Poincar\'e
map has scaling factor $r_{scale}>1$.  The Fuller-Poincar\'e map is
therefore hyperbolic, with three eigenvalues $|\lambda|<1$ and one
eigenvalue $|\lambda|>1$.  By general theory, the unstable curve is
$C^\infty$, because the diffeomorphism is $C^\infty$.
\end{proof}

\section{A Computation of the Unstable Manifold}

We extend the local unstable manifold to the global unstable manifold.
The previous theorem shows the existence of a $C^\infty$ curve (the
global unstable manifold around the fixed point $(0,x_{out})$):
\[
t\mapsto (r(t),\tilde{x}(t))\in\R^4,\quad 
(r(0),\tilde{x}(0))=(0,x_{out}),\quad t\ge0,
\]
for some local parameter $t$.  In fact, we use $t=r$ (the first
coordinate of the system \eqref{eqn:Xtilde}) as the local parameter.
Figure~\ref{fig:unstable-curve} shows a numerical
computation of the unstable curve.  Although we have not done so
because it did not seem especially worthwhile, these calculations
might be repeated using more rigorous numerical methods such as
interval arithmetic. The numerical situation is favorable: because of
contraction in the stable directions, any numerical errors in
computing the \emph{unstable} curve will tend to be self-effacing (in
the same way that under mild assumptions, Von Mises iteration of
matrix powers converge to the dominant eigenvalue).  Another numerical
advantage is that three contractive eigenvalues are small (less than
$0.1$ according to the proof of Lemma~\ref{lem:eigenvalues}).

\index{Von Mises iteration}

\mcite{MCA:unstable-curve}
\begin{figure}
    \centering 
    \includegraphics[scale=0.3]{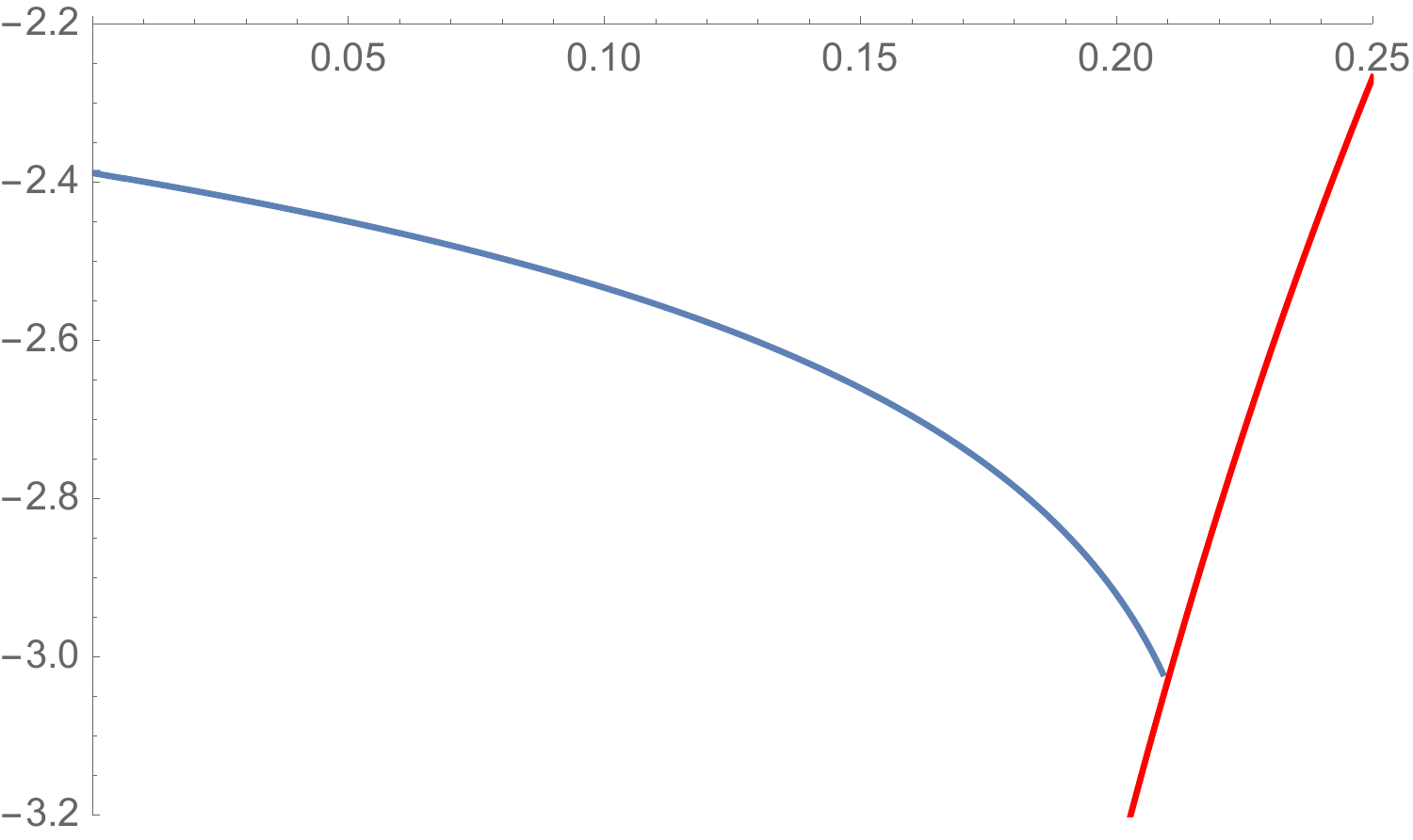} 
    \caption{The unstable curve 
$r\mapsto(r,\tilde{x}_{21}(r))$ (in blue) starts at the outward fixed
point $\tilde{x}_{21out}\approx -2.39$ and continues to the boundary
$\tilde{x}_{21}^*(r)=\sqrt{3}-1/r$ of the star domain (in red).  The
curve has been approximated in Mathematica, using
the \emph{ListLinePlot} command to create a piecewise linear curve
joining $434$ data points.}
\label{fig:unstable-curve}
\end{figure}

\index{Mathematica}

As we move away from the fixed point, computations use the unscaled
switching time $t_{sw}=t_{sw}(r,\tilde{x}(r))$ rather than the scaled
switching time.  We check that the other switching functions
$\chi_{ij}$ remain positive, so that the switching sequence is always
cyclical $3\mapsto2\mapsto1\mapsto3\mapsto\cdots$.  The unstable curve
reaches the boundary of the star domain near
$(r,\tilde{x}_{21})\approx(0.21,-3.03)$.  Once $r$ is at least about
$0.065$, a Reinhardt trajectory starting on the unstable curve reaches
the boundary of the star domain before the next switching time, and
the forward step of the Reinhardt-Poincar\'e map is no longer
defined.


\begin{theorem}\label{thm:no-return} 
A trajectory of the Reinhardt system that emanates
from the fixed point $(r,\tilde{x})=(0,x_{out})$ on the exceptional
divisor does not return to the exceptional divisor.
A trajectory of the Reinhardt system that tends to
the fixed point $(r,\tilde{x})=(0,x_{in})$ on the exceptional
divisor did not emanate at an earlier time from the exceptional divisor.
\end{theorem}

\begin{proof} All switches must be on the unstable curve, which meets
that exceptional divisor at a single point $(0,x_{out})$. Any
trajectory that returns to the singular locus must chatter (that is,
must use infinitely many switches to arrive).  This would require the
unstable curve to contain a sequence of \emph{forward} iterates of
points tending to the exceptional divisor.  This does not happen,
because the unstable curve hits the boundary of the star domain.

The second statement follows from the first by time reversal.
\end{proof}

\mcite{MCA:outward-spirals}
\begin{figure}
    \centering
    \includegraphics[scale=0.3]{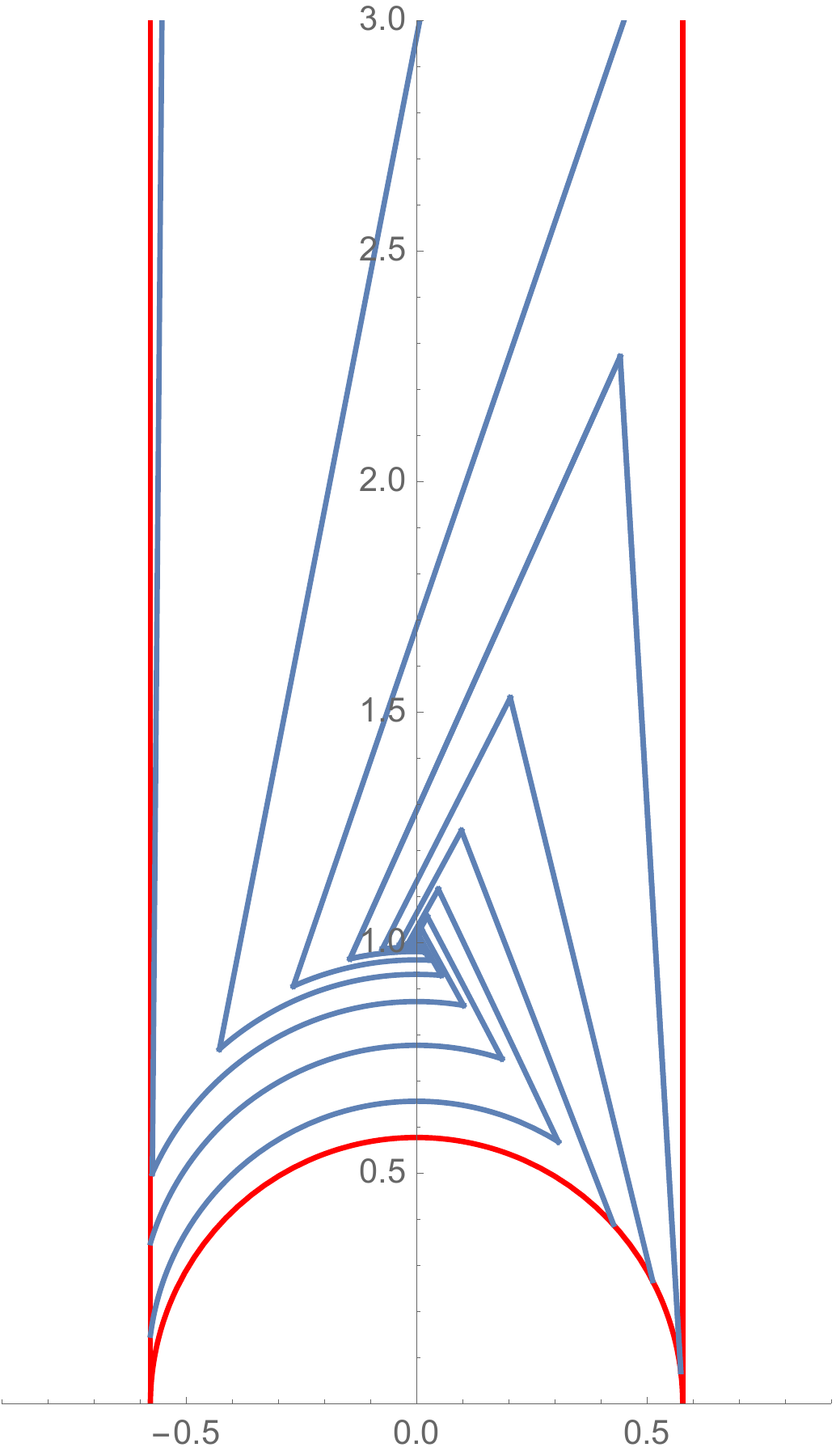}
\caption{Outward triangular spirals of the Reinhardt system that
     start at the singular locus.  The figure shows the image of
    trajectories $z(t)\in\hstar$ in the upper half plane.}
     \label{fig:triangular-spiral}
\end{figure}

Figure~\ref{fig:triangular-spiral} shows the resulting outward spiral
trajectories $z(t)\in\hstar$ (where $\Phi(z(t))=X(t)$). Trajectories
chatter as they exit from the singular locus, and move in a triangular
spiral until they hit the boundary of the star domain in finite time.
The switching points are clearly visible as corners of the triangular
spirals.  The unstable curves, which are related by rotations, are
obtained by joining these switching points by smooth curves.

\chapter{Geometry of the Fuller-Poincar\'e Map}

Throughout this section, dropping the subscript on $F_{ang}$, we let
$F$ denote the Poincar\'e map for the Fuller system on the Poincar\'e
section $\Xi_{\W,0}/V_T$ or on the domain given by the cells covering
$\Xi_{\W,0}/V_T$, as described below.

\section{Three-Cells}

Let $z=(z_1,z_2,z_3)\in\C^3$.  Write $z_j=r_je^{i\theta_j}=x_j+iy_j$,
for $j=1,2,3$.  Assume that $z_3\in\W$, a wall.  By passing to a
$V_T$-equivalent point, we may assume that $z_3=\Re(z_3)=x_3\le0$.

Assume that the Hamiltonian vanishes at $z=(z_1,z_2,z_3)$ for some
control $u\in V_T$.  The vanishing Hamiltonian condition implies
\begin{equation}\label{eqn:x3}
x_3 = 2 r_1 r_2 \sin(\theta_1 - \theta_2) = -2 r_1 r_2 \sin\psi,
\end{equation}
where $\psi:=\theta_2-\theta_1$.  If $x_3<0$, this implies
$r_1r_2\ne0$ and $\psi\in(0,\pi)$ (modulo integer multiples of
$2\pi$).

\index[n]{zy@$\psi=\theta_2-\theta_1$, phase difference}

\begin{lemma} For each $z=(z_1,z_2,z_3)\ne\mb{0}$ satisfying \eqref{eqn:x3}, 
there is a unique rescaling by a positive scalar in the virial group
so that $r_1+r_2=1$ (and $r_1,r_2\ge0$).
\end{lemma}

\begin{proof} 
If $r_1=r_2=0$, then $z=\mb{0}$ by \eqref{eqn:x3}, which is contrary
to assumption.  Hence, we may assume $r_1+r_2>0$.  Solve the following
quadratic equation for its unique positive root $s>0$,
\[
r_1/s + r_2/s^2 = 1.
\]
Then scale $r_i\mapsto r_i/s^i$.
\end{proof}
Thus, we may write
$r_1 = 1-r_2$, with $r_2\in[0,1]$.

The time reversal symmetry $\tau$, when expressed in terms of
coordinates $(r_2,\psi,\theta_2)$ takes the form
\[
\tau(r_2,\psi,\theta_2)=(r_2,\pi-\psi,\pm\pi-\theta_2).
\]
If $\theta_2\in[-\pi,\pi]$, then
the sign $\pm\pi$ in the third coordinate $\pm\pi-\theta_2$ is
chosen to give a value again in $[-\pi,\pi]$.

We now enumerate the cells partitioning the domain $\Xi_{\W,0}/V_T$ starting with the
two three-dimensional cells.  We use notation
$\CC_k(u,m_A,m_B)$ for cells.  The subscript $k$ denotes the
dimension of the cell; $u$ is the first control; $m_A$ (resp. $m_B$)
is the multiplicity of $t=0$ in the switching polynomial 
$\chi_A=\chi^u_{u-1}$ (resp. $\chi_B=\chi^u_{u-\bar{u}}$).
Let $\chi_{A,m_A}:=\chi_A/t^{m_A}$ and
$\chi_{B,m_B}:= \chi_B/t^{m_B}$.  Let $\Delta_{A,m_A},\Delta_{B,m_B}$
be the discriminants of $\chi_{A,m_A}$ and $\chi_{B,m_B}$.
We sometimes also affix a
superscript $\CC_k^A$ or $\CC_k^B$ to indicate whether the active
switching function is $\chi_{A,m_A}$ or $\chi_{B,m_B}$.

\index[n]{m@$m$, integer dimension!multiplicity of root}
\index[n]{zx@$\chi_{ij}$, switching function!$\chi_{A,m_A},\chi_{B,m_B}$, reduced switching}
\index[n]{zD@$\Delta$, discriminant}
\index[n]{C@$\CC_k(u,m_A,m_B)$ cell of dimension $k$}

\begin{definition}\label{def:big:open}
The cell $\CC_3(\zeta)^0=\CC_3(u,m_A,m_B)^0=\CC_3(\zeta,0,1)^0$ is
defined by conditions $x_3\ne0$ and $y_2>0$.  The cell
$\CC_3(\zeta^2)^0=\CC_3(u,m_A,m_B)^0=\CC_3(\zeta^2,0,1)^0$ is defined
by conditions $x_3\ne0$ and $y_2<0$.  We call $\CC_3(u)^0$ the \emph{big
open cells}.
\end{definition}

\index{big cell}
\index[n]{0@$-^0$, interior}

(We will construct compactifications $\CC_3(u,0,1)$ of the open cells
below.)
$\CC_3(\zeta)^0$ is a three-dimensional open rectangle in $\R^3$ with
coordinates $r_2\in(0,1)$, $\psi\in(0,\pi)$, and $\theta_2\in(0,\pi)$.
The first control is $\zeta$.  Also, $\CC_3(\zeta^2)^0$ is a
three-dimensional open rectangle in $\R^3$ with coordinates
$r_2\in(0,1)$, $\psi\in(0,\pi)$, and $\theta_2\in(-\pi,0)$.  The first
control is $\zeta^2$.  The complement of
$\CC_3(\zeta)^0\cup\CC_3(\zeta^2)^0$ is a union of strata of dimension
at most two.  Thus, these two three-dimensional cells cover most of
the domain.  We also refer to $\CC_3(\zeta)^0$ and $\CC_3(\zeta^2)^0$
as the \emph{first and second big cells}, respectively.  The
involution $\tau$ is given by
$\tau(r_2,\psi,\theta_2)=(r_2,\pi-\psi,\pi-\theta_2)$ on the first big
cell and by $\tau(r_2,\psi,\theta_2)=(r_2,\pi-\psi,-\pi-\theta_2)$ on
the second big cell.

\index{cell!first and second big}

Note that $\chi_{A,0}(0)\ne0$ and $\chi_{B,1}(0)\ne0$ on the big open
cells, and the roots are nonzero. As above, let $\Delta_{A,0}$ and
$\Delta_{B,1}$ be their discriminants, and let $\op{res}_{AB}$ be the
resultant of $\chi_{A,0}$ and $\chi_{B,1}$.  The first switching time
$t_{sw}$ is a discontinuous function on the cell.  The discontinuities
can only appear along the loci $\Delta_{A,0}=0$, $\Delta_{B,1}=0$ and
$\op{res}_{AB}=0$.  However, the loci do not always force a
discontinuity in $t_{sw}$.  For example, $\op{res}_{AB}=0$ does not
give a discontinuity when it represents the equality of negative roots
of $\chi_{A,0}$ and $\chi_{B,1}$.

We study the boundaries of the first and second big cells, with the
aim of extending the dynamical system continuously to the boundaries
(with noted exceptions).

We identify points along the face $r_2=0$, if they have the same image
under the mapping $f:[0,\pi]\times[0,\pi]$,
$f(\psi,\theta_2)=\theta_2-\psi=\theta_1\in[-\pi,\pi]$ on the first
cell, and the mapping $f:[-\pi,0]\times[0,\pi]$,
$f(\psi,\theta_2)=\theta_2-\psi=\theta_1\in[0,2\pi]$ on the second
cell.  (When $r_2=0$, the coordinate $z_2=r_2e^{i\theta_2}$ does not
depend on $\theta_2$.)  On each cell separately, we identify points
along the face $r_2=1$, if they have the same image under the
projection $f(\psi,\theta_2)=\theta_2$, for similar reasons: the
coordinates $z_1=r_1e^{i\theta_1}$ does not depend on $\theta_1$.  We
do \emph{not} identify points $\theta_2=0$ on the bottom face of the
first big cell with points $\theta_2=0$ on the top face of the second
big cell, because they have different first controls $u$ and behave
differently in the Fuller dynamical system.  For the same reason, we
do \emph{not} identify points $\theta_2=\pi$ on the top face of the
first big cell with points $\theta_2=-\pi$ on the bottom face of the
second big cell.

Figure~\ref{fig:1556775} shows shaded in red those points $p_0$ on the
boundary of the two cells where the first switching time satisfies
$\lim_{p\to{}p_0}t_{sw}(p)=0$, where the limit is taken over interior
points of the cells.  Although the switching time is zero, the
dynamical system is best treated as nontrivial (by refraining from
identifying $V_T$-equivalent points on the boundary of the cells).

By taking points in the big open cells near the boundary, we can
determine that the Fuller-Poincar\'e map acts in the following way on
the (red-shaded regions of the) boundary (by continuous extension of
the map on the interior of the cells).  The bottom face $\theta_2=0^+$
of the first big cell maps to the top face $\theta_2=0^-$ of the
second big cell (by the identity map $(r_2,\psi)\mapsto(r_2,\psi)$.
The bottom face $\theta_2=(-\pi)^+$ of the second big cell maps to the
top face $\theta_2=\pi^-$ of the first big cell by the map
$(r_2,\psi,\theta_2)\mapsto(r_2,\psi,2\pi+\theta_2)$.  The boundary
region $\psi\in\{0,\pi\},\theta_2\in[0,\pi/3]$ on the first big cell
is shifted to $V_T$-equivalent points
$(\psi,\theta_2)\mapsto(\psi,\theta_2+2\pi/3)$ on the same faces.
Also, the boundary region $\psi\in\{0,\pi\},\theta_2\in[-\pi/3,0]$ on
the first big cell is shifted to $V_T$-equivalent points
$(\psi,\theta_2)\mapsto(\psi,\theta_2-2\pi/3)$ on the same faces.
Finally, the right and left faces $r_2=0$ and $r_2=1$ of the cells
have been collapsed to edges along the front and back faces
$\psi\in\{0,\pi\}$, and their behavior is dictated by the behavior on
the other faces.  (As stated above, all these boundary behaviors are
obtained by studying the behavior of the dynamical system on the
interior of the cells and taking limits to the boundary.)
\mcite{MCA:3962148}

\begin{figure}
    \centering
    \includegraphics[scale=0.4]{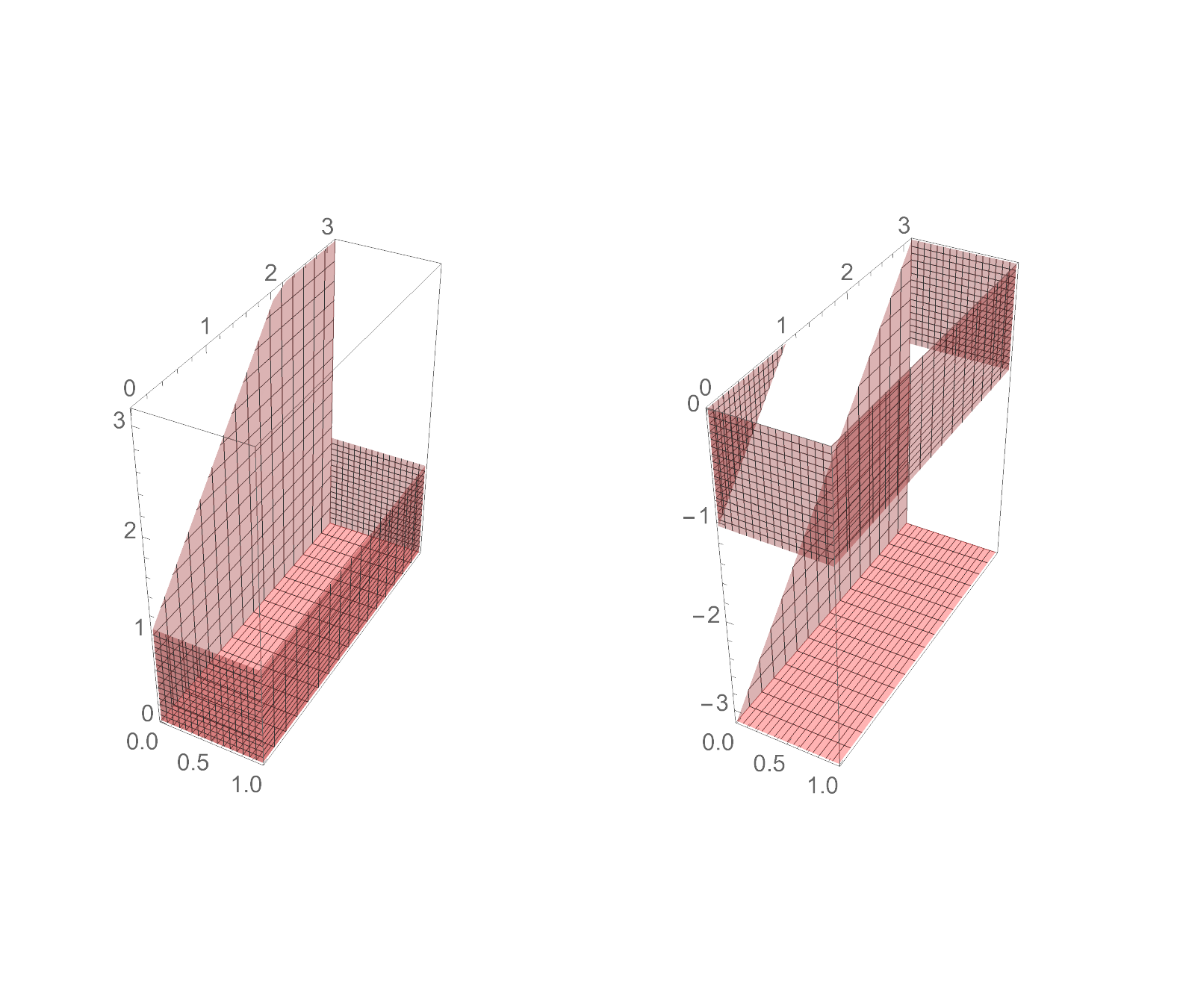}
     \caption{Shaded regions are points on the boundary of the first
and second big cells where the first switching time $t_{sw}$ has
limiting value $0$. Smaller cells $y_2=0$ will be attached to the top faces
of both $3$-cells to cap the top faces, and further smaller cells $z_3=0$ will
be attached to the unshaded regions of the front and back faces to
fill the unshaded regions of the faces.} 
     \label{fig:1556775}
\end{figure}

\section{Smaller cells}

In this section, we partition the complement of the union
$\CC_3(\zeta^{\pm1})^0$ into cells of lower dimension. We find that
the cells of lower dimension can be attached to the faces of the first
and second big cells in a way that preserves continuity.  We continue
to use polar coordinates $z_j=r_je^{i\theta_j}=x_j+iy_j$.  We may
assume $r_2 = 1-r_1\in[0,1]$ and $x_3=-2r_1r_2\sin\psi$,
$\psi=\theta_2-\theta_1\in[0,\pi]$ by \eqref{eqn:x3}.

We begin with the cases such that $x_3\ne0$.  To be in the complement
of the big open cells, we must have $\sin\theta_2=0$ and $y_2=0$.  The
two-dimensional cells of this form have parameters
$u\in\{\zeta^{\pm1}\}$, $(m_A,m_B)=(0,2)$.  That is, $\chi_B$ has a
double root at $t=0$.  We denote these cells $\CC_2(u,m_A,m_B)$,
according to the their parameters.

The cell $\CC_2(u,m_A,m_B)=\CC_2(\zeta,0,2)$ 
is defined by $x_3\ne0$ and $\theta_2=\pi$.
Coordinates are $r_2\in(0,1)$,
$\psi\in(0,\pi)$.  The first control is $\zeta$.  The switching
function $\chi_{B,2}$ never has a positive root.  Thus, the switching
function is always $\chi_{A,0}$.  By Descartes's rule of signs,
$\chi_{A,0}$ always has a unique positive root.  The first switching
time is then a continuous function on the cell.  
If we identify this cell with the top face $\theta_2=\pi$ of the first
three-cell $\CC_3(\zeta)^0$, then the dynamical system extends continuously
from the interior of the first big cell to its top face.

\index{Descartes's rule of signs}

The cell $\CC_2(u,m_A,m_B)=\CC_2(\zeta^2,0,2)$ is defined by
$x_3\ne0$, and $\theta_2=0$.  Coordinates are $r_2\in(0,1)$,
$\psi\in(0,\pi)$.  The first control is $\zeta^2$. If we identify this
$2$-cell with the top face $\theta_2=0$ of the second three-cell
$\CC_3(\zeta^2)^0$, then the dynamical system extends continuously from
the interior of the second big cell to its top face (with exceptional
discontinuities, where the dynamical system is already discontinuous
on $\CC_2(\zeta^2,0,2)$, as noted below) .

Note that we cannot have parameter values $(m_A,m_B)=(0,3)$.  In fact,
if $m_B=3$, the $y_1=y_2=0$. This implies that $r_1=0$, $r_2=0$ or
$\theta_1,\theta_2\in\pi\Z$.  Then $x_3 = -2r_1r_2\sin\psi=0$
by \eqref{eqn:x3}, and $m_A>0$.

In the remaining region, $z_3=0$ (and $m_A>0$).  By \eqref{eqn:x3}, we
have $r_2\in\{0,1\}$, or $\sin\psi=0$.  Each angle $\theta_2$ is
$V_T$-equivalent to a unique angle
$\theta_2\in\leftopen\pi/3,\pi\rightclosed$.  We define
a \emph{closed} $2$-cell $\CC_2(\zeta)=\CC_2(\zeta,1,1)$ with first
control $u=\zeta$ as follows.  The coordinates are
$\tilde{r}_1\in[-1,1]$ and $\theta_2\in[\pi/3,\pi]$.  If we define the
first switching time $t_{sw}$ over the entire closed $2$-cell by
continuous extension of the first switching time on the interior of
the closed $2$-cell, then Figure~\ref{fig:1556775} shows the parts of
the boundary where the continuous extension gives $t_{sw}=0$.  Setting
$r_1 = \tilde{r}_1e^{i\psi}$, and $r_2=1-r_1$, with $\psi\in\{0,\pi\}$
and $r_1\in(0,1)$, we recover the coordinates $(r_2,\psi,\theta_2)$.
(Exceptionally, on the segment $\tilde{r}_1=0$, $r_2=1$, we disregard
$\psi$ and only use $\theta_2$.)  If we attach each point
$(r_2,\psi,\theta_2)$ of this cell to the point on the front and back
boundary faces $\psi\in\{0,\pi\}$ of the first $3$-cell with the same
coordinates $(r_2,\psi,\theta_2)$, then the dynamical system on the
interior of the $3$-cell extends continuously to the front and back
faces, in agreement with the dynamics on the $2$-cell $\CC_2(\zeta)$.
Note that the segment $\tilde{r}_1=0$ maps to the right face of the
$3$-cell, which has been collapse to a segment, and with this
collapsed right face, the map from the $2$-cell to the front and back
faces of the $3$-cell is continuous.

In a similar way, by $V_T$-equivalence, the angle $\theta_2$ can be
brought into the interval
$\theta_2\in\leftopen-\pi,-\pi/3\rightclosed$ by the $V_T$-action.  In
this case, the first control is $u=\zeta^2$ and a closed cell
$\CC_2(\zeta^2,1,1)$ can be attached in a similar way to the front and
back faces of the second $3$-cell $\CC_3(\zeta^2)^0$ in a way that
agrees with the dynamics.  

\begin{remark}
A symmetry of $V_T$ carries the closed cell $\CC_2(\zeta,1,1)$ to
$\CC_2(\zeta^2,1,1)$, but \emph{we refrain from identifying} these two
closed cells with each other.  Instead, we consider the two closed big
cells as disjoint from each other.  We will see that
$F^{-1}(\CC_2(\zeta))$ and $F^{-1}(\CC_2(\zeta^2))$ are the \emph{two
sides} of a hypersurface $\bd_{res}$ in the first big cell, along
which the Fuller-Poincar\'e map is discontinuous.  Because of this
discontinuity, it is best to keep the two two-cells separate.
\end{remark}

We summarize our results in the following lemma.

\begin{proposition} Every point in the domain of $\Xi_{\W,0}/V_T$ is equivalent
to a point in the union of the closures $\CC_3(\zeta^{\pm1})$ of the
first and second big cells (with identifications on the boundaries of
each cell as given above).  The first control $u$ is $\zeta$ on the
first big cell and $\zeta^2$ on the second big cell.  The dynamics on
the faces of the cells is given as the continuous extension from the
dynamics on the interior of the cells.  Every point on every face of
the cells that is not identified with a point in the domain
$\Xi_{\W,0}/V_T$ is a point with vanishing (limiting) first switching
time $t_{sw}=0$.
\end{proposition}

\begin{proof} See the discussion leading up to the statement of the proposition. 
\end{proof}

\section{Involution}

\index[n]{zi@$\itf=\tau\circ{F}$, involution}
\index{involution}
\index[n]{0@$\partial$, boundary!$\partial$, boundary locus}
\index[n]{0@$\partial$, boundary!$\partial_{A},\partial_B$, discriminant locus}
\index[n]{0@$\partial$, boundary!$\partial_{res}$, resultant locus}

Let $F$ be the Fuller-Poincar\'e map, and let $\tau$ be the time
reversing symmetry.  Both have domain given by the union of two closed
$3$-cells.  Since $F^{-1}=\tau\circ{}F\circ{}\tau$ and
$\tau=\tau^{-1}$, it follows that $\itf:=\tau\circ{}F$ is an
involution: $\itf = \itf^{-1}$.  In this section, we use
properties of this involution to describe discontinuities of
the Fuller-Poincar\'e map.

We say that two subset $\DD,\tilde{\DD}$ of the domain are \emph{in
involution} if $\itf(\DD)=\tilde{\DD}$ (and $\itf(\tilde{\DD})=\DD$).
Let $\CC_2(\zeta,1,1)$ be the closed $2$-cell defined above, viewed as
a subset of the front and back faces of the first big cell.  Define
the resultant locus as $\bd_{res}=\itf(\CC_2(\zeta))$. By
Lemma~\ref{lem:switch-reversal}, the resultant of the switching
functions $\chi_{A0}$ and $\chi_{B1}$ is zero along $\bd_{res}$.  By
construction $\bd_{res}$ and $\CC_2(\zeta,1,1)$ are in involution.  Define
$\bd_A=\itf(\CC_2(\zeta,0,2))$, where $\CC_2(\zeta,0,2)$ is viewed as
the top face $\theta_2=\pi$ of the first big cell.  By
Lemma~\ref{lem:switch-reversal}, the discriminant of the switching
function $\chi_{A0}$ is zero along $\bd_A$ and $F$ is discontinuous
along $\bd_A$.  By construction $\bd_A$ and $\CC_2(\zeta)$ are in
involution.  The locus $\bd_{res}\cup{}\bd_A$ lies in the first big
cell and geometrically partitions the first big cell into two
parts. (Here, by a \emph{geometric partition} of a set, we mean a
collection of \emph{regular closed subsets} covering the set whose
interiors are disjoint. A closed set is \emph{regular}, if it is the
closure of its interior.)


\index{geometric partition}
\index{regular closed set}
\index[n]{D@$\DD_j$, parts of a geometric partition}

Let $\DD_1$ be the part of the geometric partition containing the face
$\CC_2(\zeta,0,2)$ (that is, the face $\theta_2=\pi$). There are no
further discontinuities in $\DD_1$; that is, $F$ on the interior of
$\DD_1$ extends continuously to a function $F_1$ with domain $\DD_1$,
with first control $u=\zeta$ and switching function $\chi_{A0}$.  The
domain $\DD_1$ is in self involution.  The two fixed points
$q_{in},q_{out}\in \DD_1$ are in involution.  See
Figure~\ref{fig:resultant}.

\begin{figure}
    \centering
    \includegraphics[scale=0.55]{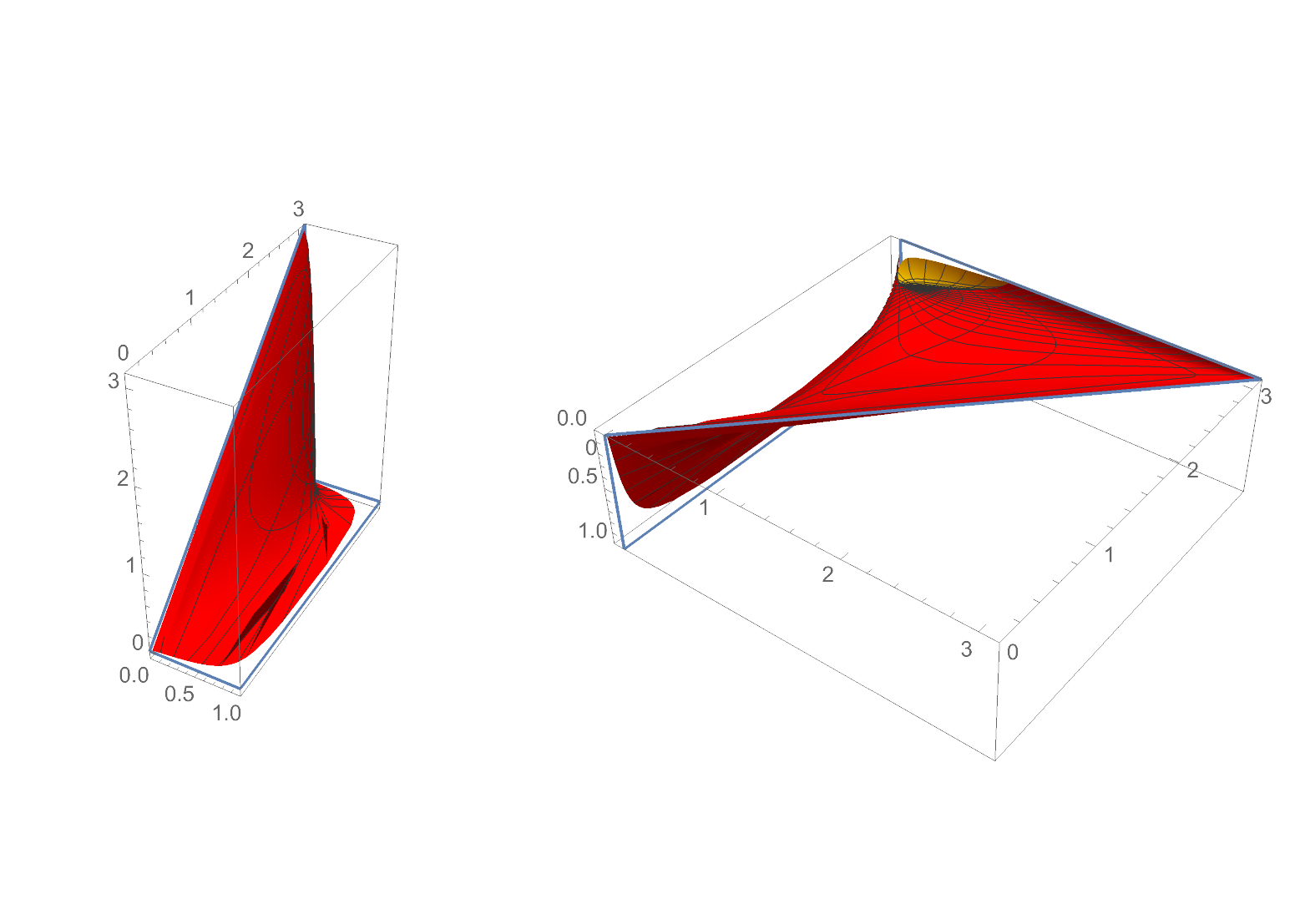}
     \caption{Geometric partition of the first big cell into parts
       $\DD_0,\DD_1,\DD_4$.  The two panels show two different views
       of the boundary separating $\DD_1$ from $\DD_0\cup\DD_4$ in the
       first big cell. On the red shaded part $\bd_{res}$ of the
       boundary, the resultant vanishes. On the yellow shaded part
       $\bd_A$ of the boundary, the discriminant $\Delta_{A,0}$
       vanishes.  These red and yellow boundaries extend to the blue
       perimeter, even if the displayed graphics stop short due to
       imperfect rendering.  The blue perimeter is in involution with
       the perimeter of the red region in the first frame of
       Figure~\ref{fig:1556775}.  The part $\DD_1$ lies above and to
       the right of the boundary in the first panel and to the lower
       right of the boundary in the second panel.  The parts $\DD_0$
       and $\DD_4$ lie below the boundary in the first panel and to
       the upper left in the second panel.  The boundary $\bd_B$
       between $\DD_0$ and $\DD_4$ is not shown.  The part $\DD_0$ is
       a very small bubble, which is attached to the yellow part
       $\bd_A$ of the boundary.  }
     \label{fig:resultant}
\end{figure}

\begin{figure}
\centering 
\includegraphics[scale=0.5]{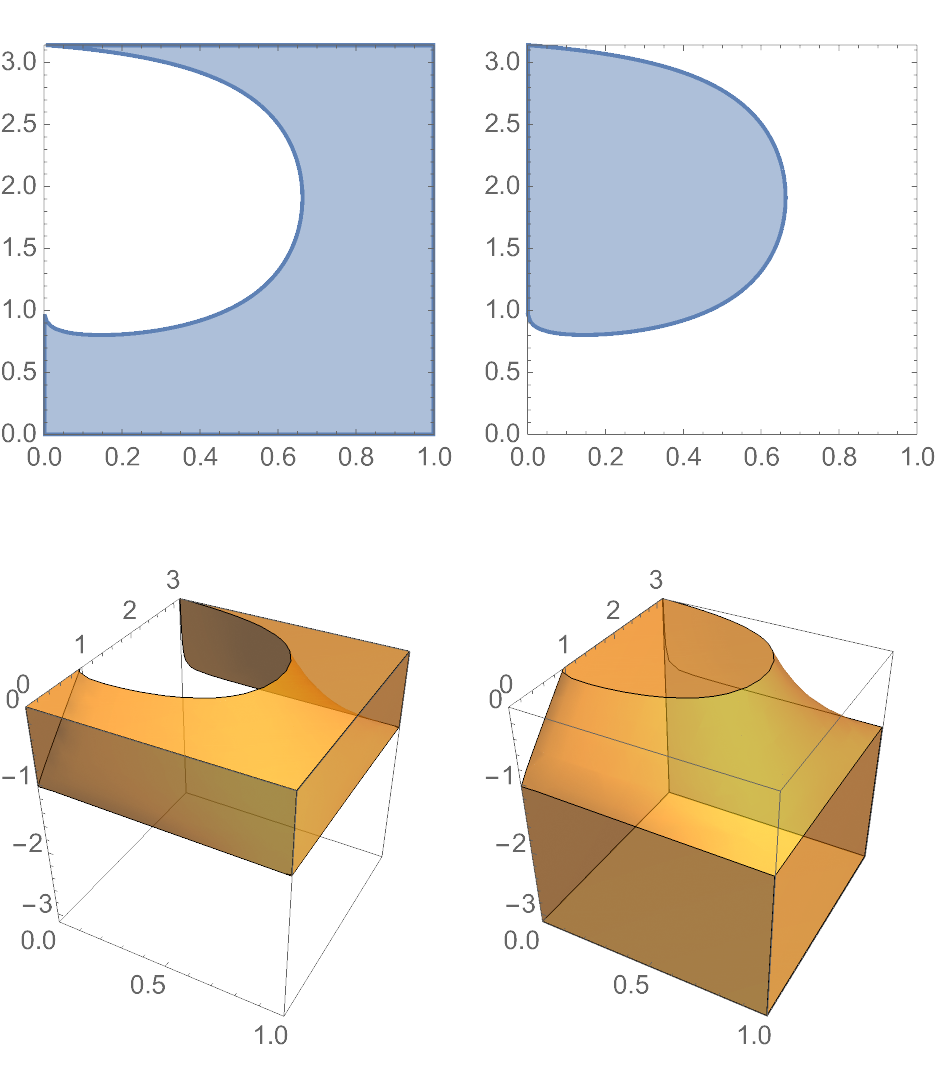} 
\caption{Geometric
    partition of the second big cell into parts $\DD_2$ (lower left panel) and $\DD_3$
    (lower right panel).  The
    top left panel $\CC_2^A$ and top right panel $\CC_2^B$ are given by
    $\Delta_{A,0}\ge0$ and $\Delta_{A,0}<0$, respectively in the top
    face $\theta_2=0$ of the second big cell.  The lower left panel $\DD_2$
    and lower right panel $\DD_3$ are separated by the discriminant locus
    $\Delta_{A,0}=0$.  The part $\DD_2$ is in involution with itself,
    and the involution exchanges its top face $\CC_2^A$ with the
    locus separating $\DD_2$ from $\DD_3$.  The involution exchanges
    $\DD_3$ and $\DD_4$, sending the top face $\CC_2^B$ of $\DD_3$ to
    the boundary $\bd_B$ separating $\DD_4$ from $\DD_0$.}
\label{fig:6828854}
\end{figure}

We geometrically partition the $2$-cell $\CC_2(\zeta^2,0,2)$.
according to the active switching function (Figure~\ref{fig:6828854}).
On one part $\CC_2^A$, we have $\Delta_{A,0}\ge0$. By Descartes's rule
of signs, $\chi_{A,0}$ has two positive roots (counted with algebraic
multiplicity). The first switching time is a root of $\chi_{A,0}$ and
is continuous on $\CC_2^A$.  On the interior of the part $\CC_2^B$, we
have $\Delta_{A,0}<0$.  In this case, $\chi_{A,0}$ has no positive
root, and the first switching time is the unique root of the linear
polynomial $\chi_{B,2}$.  The map $F$ is discontinuous along the
discriminant locus $\Delta_{A,0}=0$.  However, the first switching
time and the Fuller-Poincar\'e map $F$ extend continuously to give
$F_A$ with domain $\CC_2^A$ and $F_B$ with domain $\CC_2^B$.

Set $\bd_{A,\zeta^2}:=\itf(\CC_2^A)$. It is a subset of the second big
cell, and it geometrically partitions the big cell into two parts.
The discriminant $\Delta_{A,0}$ vanishes along $\bd_A$, by
Lemma~\ref{lem:switch-reversal}.  Let $\DD_2$ be the part that
contains $\CC_2^A$.  The active switching function on $\DD_2$ is
$\chi_{A,0}$.  The Fuller-Poincar\'e map $F$ extends continuously from
the interior of $\DD_2$ to a continuous function $F_2$ with domain
$\DD_2$, and the restriction of $F_2$ to $\CC_2^A$ is $F_A$.  The part
$\DD_2$ is in self involution.

Let $\DD_3$ be the other part of the partition of the second big
cell. The active switching function on $\DD_3$ is $\chi_{B,1}$, and
$F$ extends continuously to a function $F_3$ with domain $\DD_3$.

Let $\DD_4:= \itf(\DD_3)$. It is a subset of the first big cell, that
shares the boundary $\bd_{res}$ (and part of the boundary $\bd_{A}$
with $\DD_1$.  Finally, let $\DD_0$ be the closure of the complement
in the first big cell of the union of $\DD_1$ and $\DD_4$.  The
boundary of $\DD_0$ consists of $\bd_B:=\itf(C_2^B)$ and a subset of $\bd_A$.
Along $\bd_B$, the discriminant vanishes: $\Delta_{B,1}=0$.

\begin{remark}
The fixed point $q_{in}\in\DD_1$ is given by coordinates
\[
(r_{2in},\psi_{in},\theta_{2in})\approx(0.267949,0.1705935,2.91574).
\]
The fixed point is remarkably close to the triple juncture of
$\DD_0,\DD_1$, and $\DD_4$. In fact, the line segment
\[
r_2\mapsto (r_2,\psi_{in},\theta_{2in}),\quad r_2\in[0,1]
\]
meets the boundary $\bd_B$ (separating $D_4$ from $D_0$) at
$r_2\approx0.2677$, then meets the boundary $\bd_A$ (separating $D_0$
from $D_1$) at $r_2\approx0.267905$, then reaches the fixed point at
$r_2\approx0.267949$. The fixed point $q_{out}\in\DD_1$ with
coordinates $(r_2,\pi-\psi_{in},\pi-\theta_{2in})\in\DD_1$ is far from
the other parts $\DD_0,\DD_4$.
\end{remark}

\mcite{MCA:boundaryMapping}

In summary, we have the following proposition.

\begin{proposition}
Let $\DD_i$, $i=0,1,2,3,4$ be the geometric partition of the two big
cells defined as above.  The Fuller-Poincar\'e map extends
continuously from the interiors of $\DD_i$ to functions $F_i$ on the
closures $\DD_i$, for $i=0,1,2,3,4$.  The involution acts on the parts
by $\DD_3\leftrightarrow \DD_4$. Moreover, $\DD_1,\DD_2,\DD_0$ are
each in self involution.
\end{proposition}

\begin{proof} 
The proof is contained in the discussion leading up to the
proposition.  To briefly summarize the argument, any discontinuity in
the interior of a big cell must appear along one of the loci
$\op{res}_{AB}=0$, $\Delta_{A0}=0$, or $\Delta_{B1}=0$.  By
Lemma~\ref{lem:switch-reversal}, the involution maps the interior
discontinuities to the boundary faces of the big cells.

The Fuller-Poincar\'e map is continuous on the boundary cells
$\CC_2(\zeta^{\pm1},0,2)$ and $\CC_2(\zeta,1,1)$, again with
exceptions where a resultant or discriminant vanishes.  Analyzing
cases on two-cells, the only discontinuity is given by $\Delta_{A0}=0$
on $\CC_2(\zeta^2,0,2)$.  Using the involution to map these parts of
the faces back into big cells, we obtain a complete description of the
discontinuities.
\end{proof}

\chapter{Global Basin of Attraction and Mahler's First}

Throughout this chapter, dropping the subscript on $F_{ang}$,
we let  $F$ denote the Poincar\'e map for the Fuller
system on the two big cells. Also, $F_i$ denotes the
continuous extension of $F$ to $\DD_i$.

\index[n]{F@$F$, Poincar\'e map!$F_i$ extension to $\DD_i$}

\section{Main result on Basin of Attraction}

\begin{theorem}[Global Basin]\label{thm:basin}
Let $q\ne{}q_{in}$ be a point in $\Xi_{\W,0}/V_T$.  
Then the iterates $F^kq$ under the Fuller-Poincar\'e map
tend to the fixed point $q_{out}$.
\end{theorem}

\begin{proof} 
The proof follows the strategy of \emph{containment functions} from
interval arithmetic.  However, the proof is simple enough that it is
not necessary to adopt the entire infrastructure of interval
arithmetic.

\index{containment function}

We work with the representation of $\Xi_{\W,0}/V_T$ as the union of two
closed big cells $\CC_3(\zeta^{\pm1})$.  Let $\DD_0,\ldots,\DD_4$ be the
geometric partition of the cells, and let $F_i$ be the continuous
extension of $F$ from the interior of $\DD_i$ to $\DD_i$.

\index[n]{D@$\DD_{1,ij}$, geometric partition of $\DD_1$}

We further partition $\DD_1$ into $9$ rectangles.
Set 
\begin{align*}
\DD_{ij}&:=\{(r_2,\psi,\theta_2)\in \CC_3(\zeta)\mid r_2\in[0,1],\ 
\psi\in[a_i,a_{i+1}],\ \theta_2\in[b_{i+1},b_{i}]\},
\\
\DD_{1,ij}&:=\DD_1\cap\DD_{ij},\quad i=0,1,2,\quad j=0,1,2,
\end{align*}
where $(a_0,a_1,a_2,a_3)=(0,\pi/3,2\pi/3,\pi)$ and
$(b_0,b_1,b_2,b_3)=(\pi,\pi-1.1,1.1,0)$.  We have
$q_{in}\in{}\DD_{in}:=\DD_{1,22}$ and
$q_{out}\in{}\DD_{out}:=\DD_{1,00}=\DD_{00}$.

\index[n]{a@$a_i,b_i$, real numbers}

We claim that we have the following domain and range restrictions.
Let $\DD,\DD^*$ run over the sets in Table~\ref{tab:diag}.
\[
\DD,\DD^*\in\{\DD_{in},\DD_{out},\DD_4,\DD_2,
\DD_0,\DD_3,\DD_1\setminus \DD_{in}\}.
\]
Assume $\DD\subseteq\DD_i$. We claim that
$F_i(\DD)\subset{}\cup_{(\DD,\DD^*)\sim*}\DD^*$, where $\DD^*$ is
included in the union whenever the row-column entry $(\DD,\DD^*)$ of
Table~\ref{tab:diag} is marked with an asterisk.  (The dots $\cdot$
are placeholders along the diagonal of the table and do not indicate
inclusion in the union.)

\bgroup
\def\arraystretch{1.5}%
\begin{table}
\begin{tabular}{l|l|l|l|l|l|l|l}
&$\DD_{in}$ & $\DD_{out}$ & $\DD_4$ & $\DD_2$ 
& $\DD_0$ & $\DD_3$ & $\DD_1\setminus \DD_{in}$\\
\hline
$\DD_{in}$ &*& * & * & * & * & * & *\\
$\DD_{out}$ && * &&&&&\\
$\DD_4$ &&&$\cdot$ & * && *  & \\
$\DD_2$ &&&&$\cdot$   && *  & \\
$\DD_0$ &&&&&$\cdot$ & & * \\
$\DD_3$ &&&&&&$\cdot$ & * \\
$\DD_1\setminus \DD_{in}$ && & & & & &*
\end{tabular}
\smallskip
\caption{Upper triangular structure of the Fuller-Poincar\'e map on the two big cells.
The dots $\cdot$ are placeholders along the diagonal.  The nontrivial
diagonal entries appear in the first, second, and last rows.}
\label{tab:diag}
\end{table}
\egroup

We justify the claim and the associated table as follows.  Whenever
$\DD\subset{}\DD_i$ is topologically a closed ball with boundary
$\partial{\DD}$, then in order to show that $F_i(\DD)\subset{}\tilde{\DD}$, where
$\tilde{\DD}$ is a closed convex subset of $\R^3$, it is enough to show
$F_i(\partial{\DD})\subset{}\tilde{\DD}$.  In practice, $\partial{\DD}$ consists of
a small number of analytic surfaces (such as the six faces of a cube),
and the proof of the containment $F_i(\DD)\subset{}\tilde{\DD}$ reduces to the
containment of the images of the faces, which we compute numerically
without difficulty in Mathematica.  We call this the \emph{boundary
method}.  In fact, by using the involution $\itf$, we can mostly avoid
direct use of the Fuller-Poincar\'e map.

\index{Mathematica}
\index{boundary method}

We start with the second row.  We compute the interval containment
$F_1(\DD_{out})\subset{}\DD_{out}$ by the boundary method.  The image
$F_1(\DD_{out})$ and the fixed point $q_{out}$ are shown in
Figure~\ref{fig:2082214}.  The forward iterates
$F_1^k(\DD_{out})$ quickly shrink toward $q_{out}$.  The Jacobian
calculation at the fixed point (appearing earlier) shows that
$q_{out}$ is an asymptotically stable fixed point.

\begin{figure}
\centering 
\includegraphics[scale=0.4]{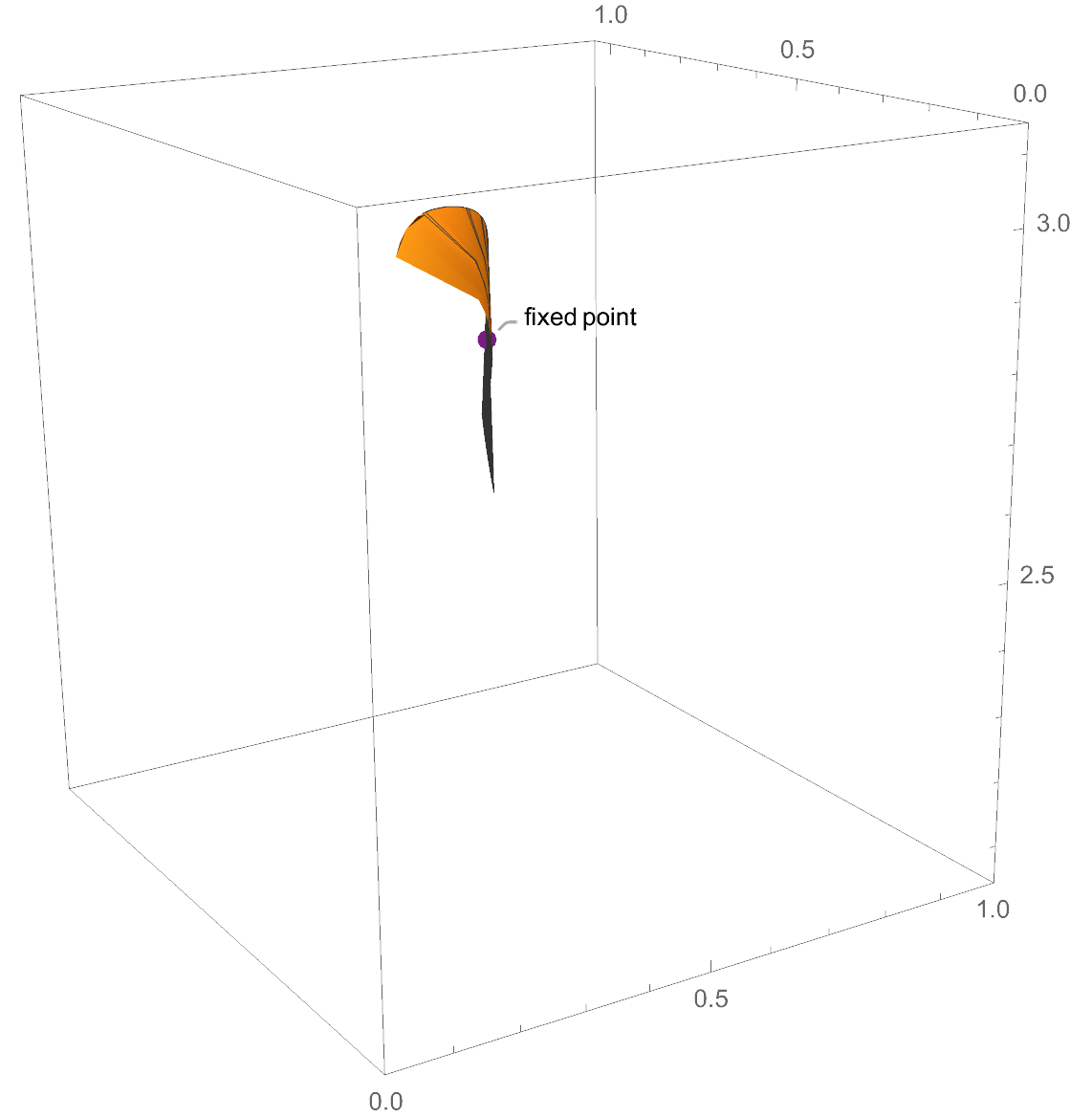} 
\caption{The bounding box is the set $\DD_{out}\subset{}\DD_1$ and
the shaded region is the image $F_1(\DD_{out})$, showing the contraction
of $\DD_{out}$ toward the fixed point $q_{out}\in\DD_{out}$.}
\label{fig:2082214}
\end{figure}


The first row asserts that the range of $\DD_{in}$ can be anything,
and there is nothing to prove in this case.  For the $\DD_4$ row, the
image satisfies $F_4(\DD_4)=\tau(\itf(\DD_4))=\tau(\DD_3)$, which is
contained in the second big cell, covered by $\DD_2\cup{}\DD_3$.  For
the next row, $F_2(\DD_2)=\tau(\itf(\DD_2))=\tau(\DD_2)\subset \DD_3$.
(Note that $\theta_2\in[-\pi/3,0]$ holds on $\DD_2$, so that
$\theta_2\in[-\pi,-2\pi/3]$ holds on $\tau(\DD_2)$, to see the
containment in $\DD_3$.)

Turning to the row for $\DD_0$, we note that
$F_0(\DD_0)=\tau(\itf(\DD_0))=\tau(\DD_0)$.  The boundary of
$\tau(\DD_0)$ consists of $\tau\itf(C_2^B)$ and a subset of
$\tau(\bd_A)$. Both of these boundary components lie in $\DD_{out}$,
which is a subset of $\DD_1\setminus{}\DD_{in}$.  The containment
$F_0(\DD_0)\subset{}\DD_1\setminus{}\DD_{in}$ follows by the boundary
method.

Next consider the row $\DD_3$.  We have
$F_3(\DD_3)=\tau(\itf(\DD_3))=\tau(\DD_4)$, which is a subset of the
first cell.  Recall that the first big cell is covered by the union of
$\DD_0$, $\DD_1$, and $\DD_4$.  The inequality
$\theta_1=\theta_2-\psi\le0$ holds on $\DD_4$, $\DD_0$, and
$\DD_{in}$, but $\theta_1=\theta_2-\psi\ge0$ holds on
$\tau(\DD_4)$. Hence by exclusion,
$\tau(\DD_4)\subset{}\DD_1\setminus{}\DD_{in}$. (In the boundary case
$\theta_2=\psi$, the edge given by equations $r_2=0$, $\theta_2=\psi$
belongs to both $\DD_1$ and $\DD_4$, but we still have
$\tau(\DD_4)\subset{}\DD_1\setminus{}\DD_{in}$.)

We have that $F^2\DD_3\subset\DD_{00}$, as shown by the
calculation in Figure~\ref{fig:1350181}.  Applying the function
$F\circ\tau$ to both sides of this inclusion, we obtain
\begin{equation}
\DD_4\subset F(\DD_{22}),
\end{equation}
because
\[
\DD_4=\itf(\DD_3)=F\tau{}F^2\DD_3
\subset{}F\tau\DD_{00}=F(\DD_{22}).
\]

\begin{figure}
\centering 
\includegraphics[scale=0.4]{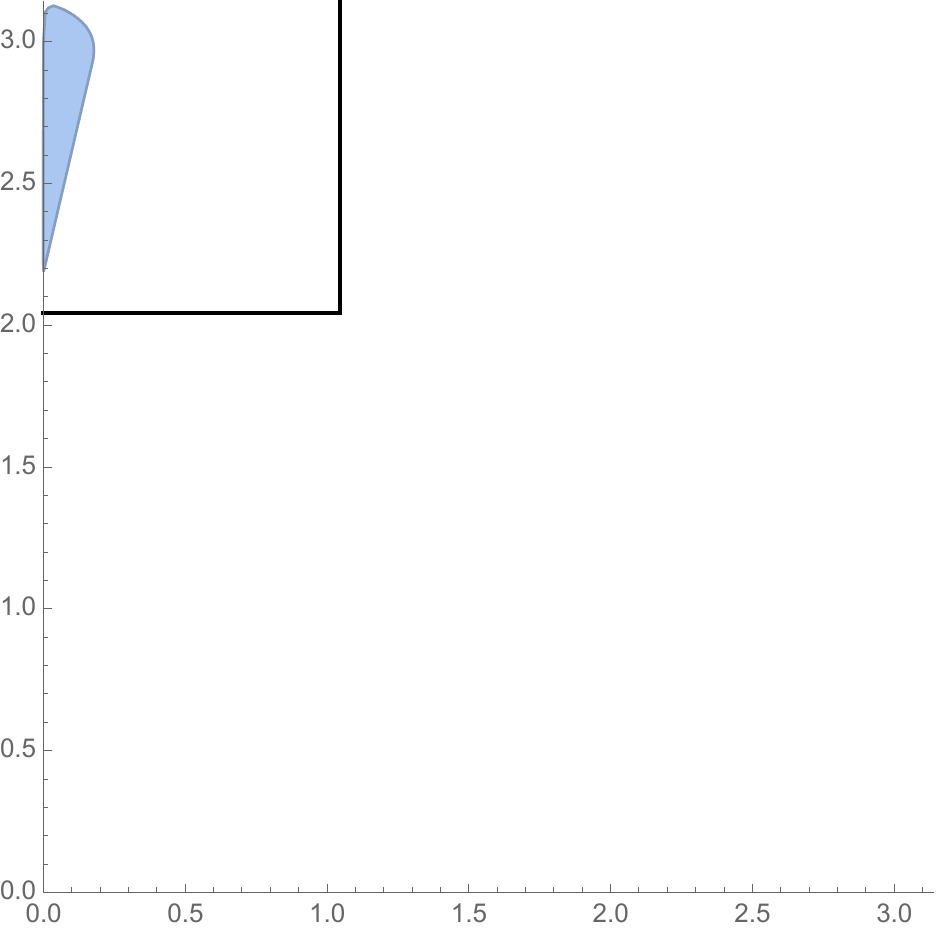} 
\caption{The shaded region is the convex hull 
of the projection to the $(\psi,\theta_2)$-plane
of the iterated image $F^2\DD_3$.  The black lines show the
boundary of the projection of $\DD_{out}$. The conclusion is
that $F^2\DD_3\subset\DD_{out}$.}
\label{fig:1350181}
\end{figure}

The last row of the table is justified below by taking an interval
refinement.  We claim that
$F_1(\DD_1\setminus{}\DD_{in})\subset \CC_3(\zeta)\setminus{}\DD_{in})$.
(Note that if $q_1\in\DD_1\setminus{}\DD_{in}$ and
$F_1(q_1)\in\CC_3(\zeta)\setminus{}\DD_{in}$, then using facts
$\CC_3(\zeta)=\DD_0\cup\DD_1\cup\DD_4$ and $\DD_0\subset\DD_{in}$ and
$\DD_4\subset{}F\DD_{22}\not\ni{}F(q_1)$, we get
$F_1(q_1)\in\DD_1\setminus\DD_{in}$ to complete the justification of
the row.)

\index[n]{0=@$\prec$, lexicographic order}

The cases $(i,j)=(0,0),(2,2)$ are the cases $\DD_{in},\DD_{out}$,
which are treated elsewhere. For $(i,j)\ne(0,0),(2,2)$, we break the
claim into a series of subclaims.  The domain
$\DD_1\setminus{}\DD_{in}$ is covered by the sets $\DD_{1,ij}$, for
$(i,j)\ne(2,2)$.  The subclaims are the following domain and range
restrictions for $(i,j)\ne(0,0),(2,2)$; \emph{subclaim}\relax$_{ij}$:
we have $F_1(\DD_{1,ij})\subset{}\cup_{k,\ell}\DD_{k\ell}$, where the
union runs over $(k,\ell)\prec(i,j)$.  Here we use the lexicographic
total order $(\prec)$ on ordered pairs given by
\[
(k,\ell)\prec(i,j)\quad\Leftrightarrow\quad
(k<i)\text{~or~}(k=i\text{~and~}\ell<j).
\]
These subclaims are established by the boundary method through direct
computation, explained above.  The convex hulls of planar projections
of the images are shown in Figure~\ref{fig:2381734}. The domain and
range restrictions follow by observing that the blue regions are
subsets of the yellow regions.  The three-dimensional images
$F_1(D_{1,21})$ $F_1(D_{1,01})$ are shown in
Figure~\ref{fig:5049419}. Their projections appears in panels $(2,1)$
and $(0,1)$ of Figure~\ref{fig:2381734}.

\begin{figure}
\centering 
\includegraphics[scale=0.4]{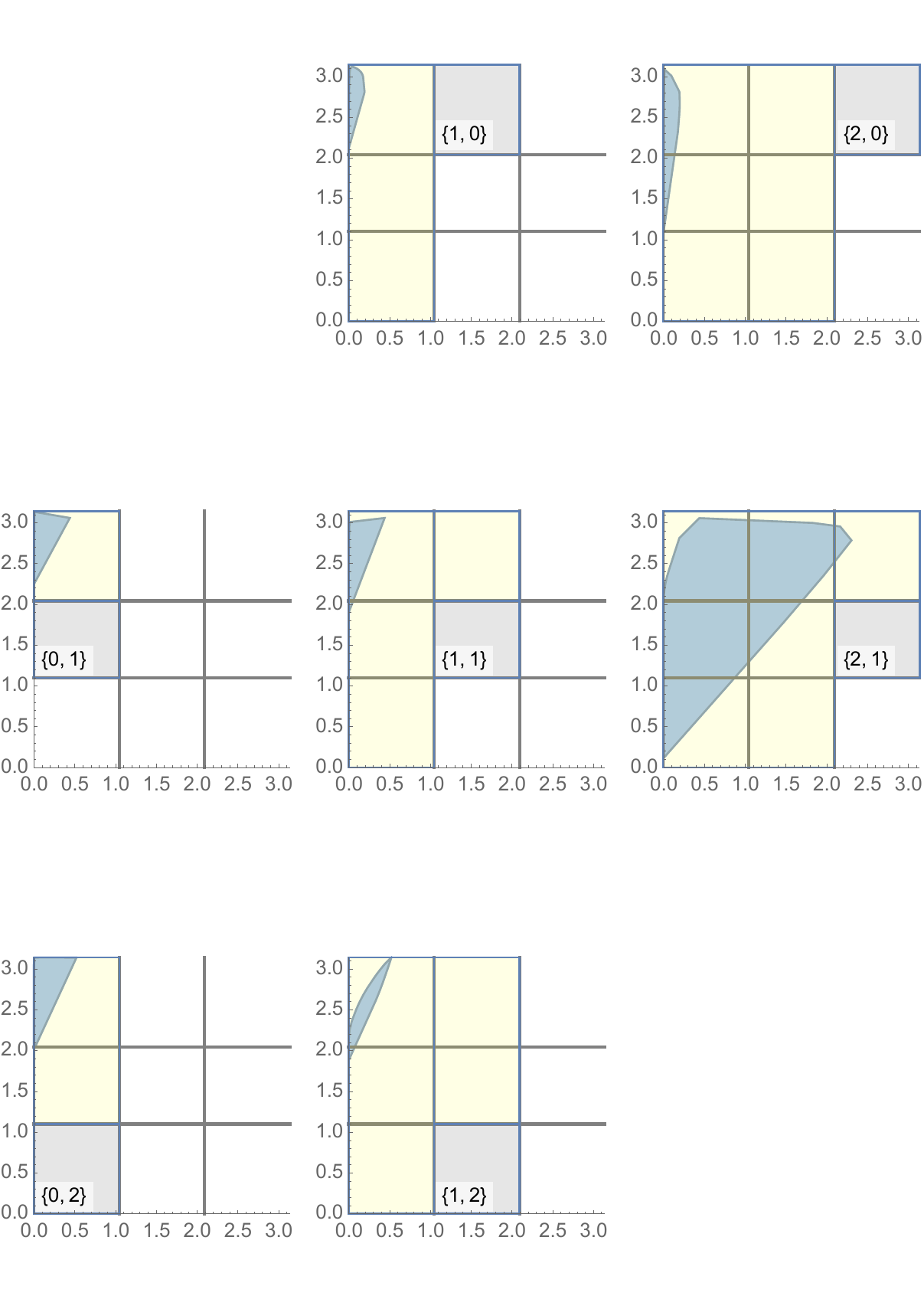} 
\caption{Projections of images $F_1(\DD_{1,ij})$ in panel $(i,j)$. 
The coordinates are $(\psi,\theta_2)$ and projection map is
$(r_2,\psi,\theta_2)\mapsto(\psi,\theta_2)$. The gray square labeled
$\{i,j\}$ is the projection of $\DD_{ij}$, containing the domain.  
The blue region is the
convex hull of the projection of $F_1(\DD_{1,ij})$. The yellow
squares are the projections of $\DD_{kl}$ such that $(k,l)\prec(i,j)$.
The subclaims follow from the observation that the each blue region is
contained in the corresponding yellow region.  }
\label{fig:2381734}
\end{figure}

\begin{figure}
\centering 
\includegraphics[scale=0.35]{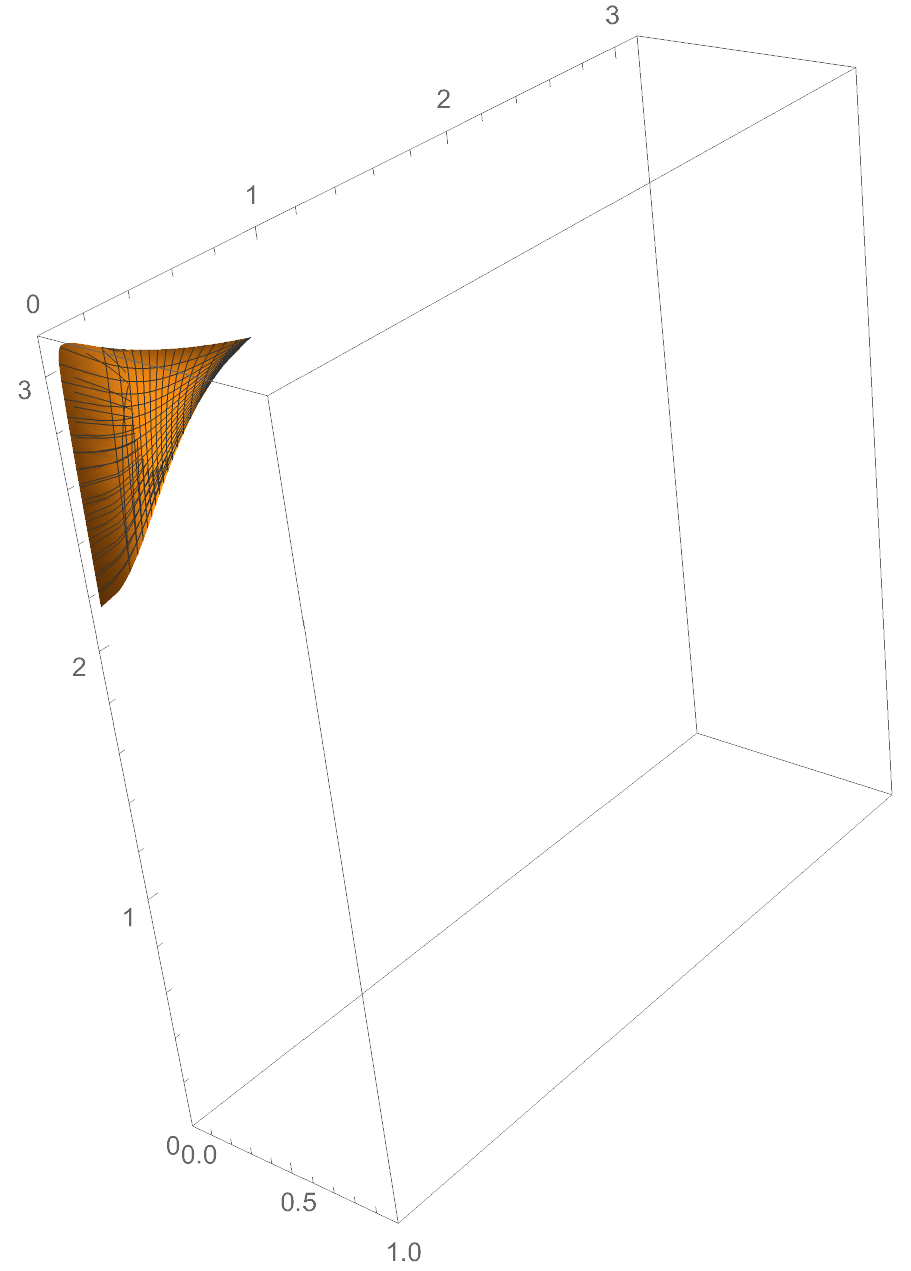} 
\includegraphics[scale=0.35]{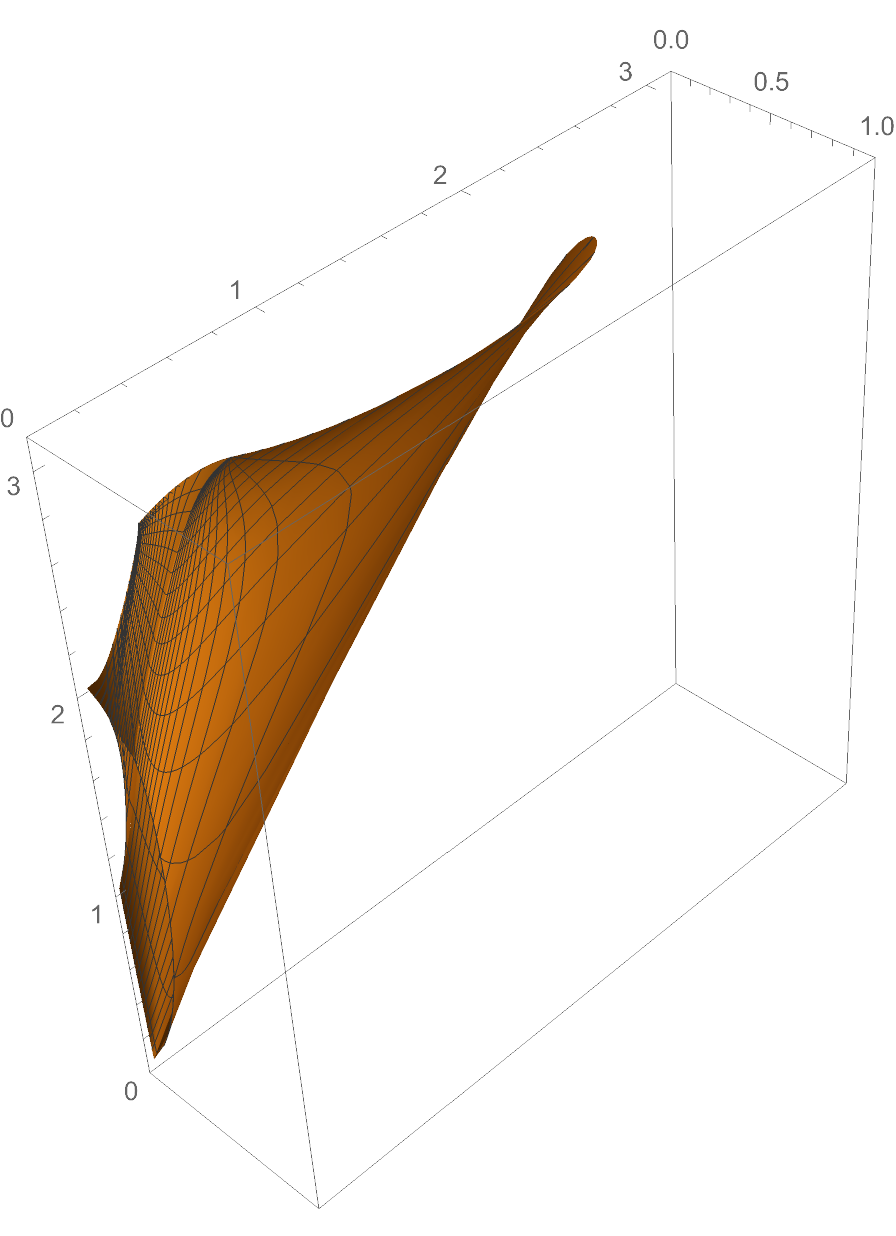} 
\caption{The right panel shows the image $F_1(\DD_{1,21})$ in the first big cell. The projection
of this region appears in panel $(2,1)$ of Figure~\ref{fig:2381734}.
The left panel shows $F_1(\DD_{1,01})$, corresponding to panel $(0,1)$
of Figure~\ref{fig:2381734}.}
\label{fig:5049419}
\end{figure}

We are ready to prove that $q_{out}$ has a global basin.  Let $q\ne
q_{in}$.  Assume first that $q\in \DD_{in}$.  By time reversal
symmetry, for every $q\in{}\DD_{in}$, with $q\ne{}q_{in}$, the
iterates $F^kq$ must eventually exit $\DD_{in}$ (because the iterates
$F^{-k}(\tau{q})$ exit $\DD_{out}$ for all sufficiently large $k$).
Note the upper triangular structure of the first table, with diagonal
entries only for $\DD_{in}$ and $\DD_1\setminus{}\DD_{in}$.  Hence, if
$q\ne{}q_{in}$, the iterates $F^kq$ lie in $\DD_1\setminus{}\DD_{in}$
for all sufficiently large $k$.

Note that the total order $(\prec)$ gives a strict triangular
structure with respect to domain and range interval containments.
Hence, if $q\in{}\DD_1\setminus{}\DD_{in}$, the foward iterates $F^kq$
must then eventually all lie in $\DD_{1,00}=\DD_{out}$.  As already
noted, once in the small rectangle $\DD_{out}$ containing $q_{out}$,
the iterates rapidly converge to the fixed point $q_{out}$.
\end{proof}

\section{Classification of Outward Fuller Trajectories}

In this section we return to the Fuller-Poincar\'e map
$F:(\C^3\setminus\mb{0})\to(\C^3\setminus\mb{0})$ (including the
radial component) and restore the subscript $F_{ang}$ when referring
to the Fuller-Poincar\'e map on the angular component.

\begin{theorem}\label{thm:outward-fuller} 
Consider the Fuller-Poincar\'e dynamical system $F$ on
$(\C^3\setminus\mb{0})/V_T$.  Every outward trajectory that emanates
from the singular locus has all its switching points in the set
$\R_{>0}\times\{q_{out}\}$ modulo $V_T$.  Every inward trajectory to
the singular locus has all its switching points in the set
$\R_{>0}\times\{q_{in}\}$ modulo $V_T$.
\end{theorem}

\begin{proof}
By time reversal symmetry, it is enough to prove the second statement
of the lemma.  Let $q\in{}\Xi_\W/V_T$ be any switching point of any
trajectory.  If $q=q_{in}$, then it is the inward spiral.

Otherwise, $q\ne{}q_{in}$.  In this case, by the Fuller basin theorem
and by the stability calculation near $q_{in}$, for large $j$, the
iterates $F_{ang}^j(q)$ approach the outward spiral.  For every
$r>0$, this implies that the angular component of the Fuller
trajectory $z(t)$ with initial conditions $(r,q)$ approaches the
outward spiral.  The outward spiral moves away from the singular locus
(because of the scaling factor $r_{scale} > 1$), and the Fuller
trajectory must then also move away from the singular locus.  In
particular, $(r,q)$ is not the initial condition of a forward
trajectory that converges to the singular locus.
\end{proof}

\section{Mahler's First: Bang-bang with Finitely Many Switches}

In this section, we return to the Reinhardt dynamical system,
and $F$ now denotes the Reinhardt-Poincar\'e map.

We prove Mahler's First conjecture from 1947.
Theorem~\ref{thm:mahler} is our main result.

\begin{theorem}[Mahler's First conjecture]\label{thm:mahler}
The global minimizer of the Reinhardt optimal control problem is a
bang-bang solution with finitely many switches.  In particular, the
minimizer $K_{\min}$ of the Reinhardt problem is a finite-sided
smoothed polygon with rounded hyperbolic arcs at each corner of the
sort described by Reinhardt and Mahler.
\end{theorem}

\index{Mahler's First conjecture}
\index[n]{K@$K$, convex disk!$K_{\min}$, minimax optimizer}

By \emph{the sort described by Reinhardt and Mahler}, we mean more
precisely that the minimizer $K_{min}$ has no corners and the boundary
alternates between straight edge segments and hyperbolic arcs, whose
asymptotes are lines extending the straight edge segments of the
boundary, as in Figure~\ref{fig:smoothed}.

\begin{proof} 
A bang-bang trajectory with finitely many switches is a polygon with
rounded corners of the sort described by Reinhardt, so the second
statement of the theorem follows from the first.

By Proposition~\ref{prop:edge-extremal}, any globally minimizing
trajectory that avoids the singular locus is is an extremal for the
Reinhardt control problem and is also edge-extremal.  By
Theorem~\ref{thm:finite-bang-bang},  such a trajectory is a
bang-bang trajectory with finitely many switches.  A trajectory cannot
remain on the singular locus for any positive interval of time by
Theorem~\ref{thm:no-singular-arcs}.  A trajectory cannot reach the
singular locus with finitely many switches by
Lemma~\ref{lem:constant-not-singular}.

The proof then reduces to the consideration of a trajectory such that
the infinite sequence of switching points has a subsequence tending to
the singular locus.  Passing to the blow-up, which has a compact
exceptional divisor, the sequence of switching points has a
subsequence tending to a limit on the exceptional divisor in finite
time.  By Theorem~\ref{thm:no-chatter}, which appears below, the limit
is $q_{in}$ and the sequence approaches $q_{in}$ along the stable
curve $W^s(q_{in})$.  By Theorem~\ref{thm:no-return} and its time
reversal, the stable curve $W^s(q_{in})$ did not come from the
exceptional divisor at an earlier time.  Thus, the trajectory is not
periodic. This is contrary to the boundary conditions of the Reinhardt
conjecture.
\end{proof}

\section{Cluster Point Theorem}\label{sec:proof}

The proof of Mahler's First relies on the following theorem.

\begin{theorem}\label{thm:no-chatter}
If the sequence of switching points of a Pontryagin extremal Reinhardt
trajectory has a cluster point on the
exceptional divisor of the blow-up, reached in finite time, and if the
switching points themselves are not on the exceptional divisor, then
that cluster point is the fixed point $q_{in}$ and the switching
points lie on the stable curve $W^s(q_{in})$.
\end{theorem}

The proof of Theorem~\ref{thm:no-chatter} will be presented 
after some preparations.  

\subsection{Coordinates}

We return to the hyperboloid coordinate system $(w,b,c)\in\C^3$ for
the Reinhardt dynamical system.  We use parameters
\[
\rho=2,\quad d_1=3/2,\quad \lambda_{cost}=-1,\quad \epsilon=1,
\]

We recall that we have rescaled variables
\[
(z_1,z_2,z_3)=(w/\rho,-ib/(2\rho),c/6\rho),\quad
(w,b,c)=(\rho z_1,2\rho i z_2,6\rho z_3),\quad \rho=2,
\]
that were introduced in Equation~\eqref{eqn:wbc-to-z}.
We assume that $(z_1,z_2,z_3)\ne\mb{0}$.
Formulas involving hyperboloid variables from previous chapters can
be rewritten in terms of $z$.  We do so without further comment.

Returning to earlier notation, we let $r$ be the radial variable in
the angular decomposition of $\C^3\setminus\{\mb{0}\}$.  As usual, we
call $r$ and $\xi$ the radial and angular components of $z$.  We have
coordinates
\[
z_k = r^k \xi_k,\quad \xi=(\xi_1,\xi_2,\xi_3)\in\C^3,
\quad \phi(\xi)=1,\quad r>0.
\]
The set 
\[
\{\xi\in\C^3\mid \phi(\xi)=1\}
\]
is a compact manifold. This is the 
\emph{angular component} in this context.  For now, we do
not impose the vanishing of the Hamiltonian.  That will be reimposed
later.

\subsection{Reinhardt Switching functions as cubic polynomials}
In Section~\ref{sec:asymptotic}, the asymptotic formulas in $r_0$ for
the Reinhardt switching functions are functions of rescaled time $s$.
These switching functions (and their derivatives with respect
to $s$) are approximated as $r_0\to0$ by the Fuller switching
functions, as functions of $s$.  The Fuller switching functions are
cubic polynomials, whose leading coefficients are nonzero constants.
For $r_0$ sufficiently small, the third derivative of the Reinhardt
switching functions are also nonzero.  This implies that the Reinhardt
switching functions behave qualitatively as cubic polynomials: the
third derivative has fixed sign, the second derivative is monotonic
with at most one zero, the first derivative is convex with at most two
zeros, and the function itself has at most one inflection point, has
at most two local minima, and has at most three zeros.  In summary,
the Reinhardt switching function behaves qualitatively like a monic
cubic polynomial.

\index{Weierstrass!preparation}
\index{Weierstrass!polynomial}
\index[n]{zx@$\chi_{ij}$, switching function!$\chi_{W}$, Weierstrass polynomial of}

We can make this polynomial behavior precise using Weierstrass
preparation.  Let $z(s)$ be the solution to the Reinhardt system in
hyperboloid coordinates with initial condition $(r,\xi^0)$, rescaled time
$t = s r$, and first control $u=\zeta^k$.  We view the Reinhardt
switching functions $\chi^k_{ij}(s)/r^3$ from $\zeta^i$ to $\zeta^j$
as analytic functions of the variables $s,r,\RR(\xi^0_k),\Im(\xi^0_k)$.  By
the earlier asymptotic formulas (adapted to this system of variables), the
switching functions extend analytically to a neighborhood of $r=0$.
When $r=0$, the Reinhardt switching function agrees with the
Fuller switching function and is a cubic polynomial in $s$.

For a given initial condition $\xi^0$ on the angular component, the
restriction of the switching function
$\chi(s,r,\xi^0)=\chi^k_{ij}(s,r,\xi^0)$ to $r=0$ is a cubic
polynomial with roots $s_1,s_2,\ldots$ having algebraic multiplicities
$m_i$, where $\sum_i m_i = 3$.  Applying the Weierstrass preparation
theorem centered at the point $(s,r,\xi)=(s_i,0,\xi^0)$, we obtain a
(monic) Weierstrass polynomial of degree $m_i$.
\[
\chi_{W,i}(s,r,\xi)=\sum_{m=0}^{m_i} (s-s_i)^m b_m(r,\xi),
\]
where $b_{m_i}=1$. For $m<m_i$, the coefficients $b_m$ are analytic
functions of $(r,\xi)$ near $(0,\xi^0)$ such that $b_m(0,\xi^0)=0$.
If $m_i=1$, then $s_i - b_0(r,\xi)$ is simply the implicitly defined
root of $\chi$ near $(s_i,0,\xi^0)$.)  Set $\chi_W(s,r,\xi)
= \Pi_i \chi_{W,i}(s,r,\xi)$, which is defined for all $s$ and for all
$(r,\xi)$ in some open neighborhood of $(0,\xi^0)$.  Since $\chi$ is
a \emph{real} analytic function, the nonreal roots $s_i$ come in
complex conjugate pairs, and the corresponding Weierstrass polynomials
come in pairs.  By the uniqueness of the Weierstrass polynomials,
$\chi_W(s,r,\xi)$ takes real values on real inputs.  By
Weierstrass division, and the continuous dependence of roots on their
coefficients, the polynomial $\chi_W$ captures all real roots of
$\chi$ for all $(r,\xi)$ in some neighborhood of $(0,\xi^0)$.  Thus
$\chi_W$ can be used as the switching function.

Now we drop the subscript $W$, and take the switching function $\chi$
to be a cubic polynomial in $s$, whose coefficients are analytic in
$(r,\xi)$.  A monic polynomial switching function
$\chi(s)=s^3+b_2s^2+b_1s+b_0$, determines complex analytic
varieties for each $\ell\le 3$ by the equations $b_i(r,\xi)=0$ for
$i<\ell$.  We have truncations
\[
\chi_\ell(s) = \sum_{m=\ell}^3 b_m s^{m-\ell}.
\]
On this complex analytic variety, we have
$\chi_\ell(s)=\chi(s)/s^\ell$.

We have defined cells $\CC_k$ for the Fuller system by conditions on
the first control, multiplicities $m_A,m_B$ of the zero at $s=0$ of
the switching functions $\chi_{A}$, $\chi_B$.  We have defined a
further geometric partition according to inequalities on the
discriminants $\Delta_{A,m_A}$, $\Delta_{B,m_B}$, and the resultant
$\op{res}(\Delta_{A,m_A},\Delta_{B,m_B})$, and according to the active
switching function.  All of these defining conditions can be carried
over to the Reinhardt dynamical system in terms of the switching
functions of the Reinhardt system.  We use the superscript $R$ to
designate an extension from the exceptional divisor to a neighborhood,
according to Reinhardt dynamics.  In this way, we extend the
definitions of cells $\CC_k$ to a neighborhood $\CC_k^R$ of the
exceptional division using the Reinhardt system dynamics.  Upon
restriction to $r=0$, the cells agree with the cells defined for the
Fuller system.  There is a shift in dimension of each cell by one,
because the Fuller system exceptional divisor has codimension one (if
we restrict to the vanishing set of the Hamiltonian in both cases).

The rule for the first control for the Fuller system extends to give
first control $u$ on the two big cells $\CC_3^R(u)$.  We obtain a
geometric partition of the two big cells into five regular closed sets
$\DD_i^R$, $i=0,1,2,3,4$ and a continuous extension $F_i$ of the
Reinhardt-Poincar\'e map $F$ from the interior of $\DD_i^R$ to all of
$\DD_i^R$.  The restriction of each $\DD_i^R$ to the exceptional
divisor $r=0$ is the previously defined Fuller-system part $\DD_i$.

Similarly, where we have refined the partition into smaller parts
(such as $\DD_{1,ij}$), we choose a corresponding refinement of the
parts, such as $\DD_1^R$ into $\DD_{1,ij}^R$.  The precise definitions
of these subparts will not matter as long as they agree with
previously established subparts $\DD_{1,ij}$ on the exceptional
divisor $r=0$.

\subsection{Proof}

\begin{proof}
We now turn to the proof of Theorem~\ref{thm:no-chatter}.  Consider a
sequence of switching points of a Pontryagin extremal trajectory that
has a cluster point on the exceptional divisor.  Assume that the
switching points themselves are not on the exceptional divisor.

Note that a periodic set of switching points does not approach the
exceptional divisor.  Hence the sequence of switching points is
injective, and the set of switching points is countably infinite.

We consider the various possibilities for the cluster points on the
exceptional divisor.  If $q_{in}$ is a cluster point, then by the
definition of the stable manifold, the switching points lie on the
stable manifold $W^s(q_{in})$.  This case appears as a possibility in
the statement of the theorem.  In this case, $q_{in}$ is the only
cluster point of the trajectory.  (If the sequence of switching points
has $q_{in}$ as a cluster point, then the finite time hypothesis
implies that the limit of the sequence exists and equals $q_{in}$,
because of the time required to travel from outside an $\epsilon$-ball
with center $q_{in}$ to a point inside an $\epsilon/r_{scale}$-ball
with center $q_{in}$; once entering an $\epsilon/r_{scale}$-ball, a
finite time sequence must eventually remain inside the
$\epsilon$-ball. Here time is measured with respect to the unscaled
time parameter $t$.)

We now assume that $q_{in}$ is not a cluster point of the trajectory.
We assume for a contradiction that the trajectory has a cluster point
other than $q_{in}$.

We claim that $q_{out}$ is not a cluster point.  Otherwise, for a
contradiction, we find that the switching points lie on the stable
manifold $W^s(q_{out})$. (Again, the finite time hypothesis is used to
convert a cluster point to a limit.)  However, this stable manifold is
a subset of the exceptional divisor, which is contrary to the
assumption that the switching points are not on the exceptional
divisor.

Consider the set of cluster points on the exceptional divisor,
viewed as a union of the two big cells $\CC_3(u)$.  If a cluster point
$q$ lies in two or more parts $\DD_i$ (where boundaries meet), then we
can assign $q$ to $\DD_i$, if a convergent subsequence $(q_{n_k})$ of
$(q_n)$ has limit $q$, with $q_{n_k}\in{}\DD_i^R$.  Each cluster point
$q$ can be assigned to at least one part $\DD_i$ in this way.

We claim that with respect to the order on parts imposed by the upper
triangular structure of containment relations in Table \ref{tab:diag},
if $q$ is a cluster point in $\DD$, then there is also a cluster point in
some $\tilde{\DD}$, which is smaller with respect this order.  (Here
$\DD,\tilde{\DD}$ are the parts $\DD_i$ or subparts $\DD_{1,ij}$,
etc. of the geometric partitions that appear in the proof of
Theorem~\ref{thm:basin}.)  In fact, if $q_{n_k}\to q$, with
$q_{n_k}\in{}\DD_i^R$, then $F_i(q_{n_k})$ lies in a finite union of
lesser parts $\tilde{\DD}^R$ (extending $\tilde{\DD}$ to
$\tilde{\DD}^R$).  Passing again to a subsequence, we may assume that
$F_i(q_{n_k})\in \tilde{\DD}^R$ converges to a limit in some
lesser $\tilde{\DD}$.  Repeating the argument of
Theorem~\ref{thm:basin}, eventually we obtain a cluster point
$q\in \DD_{out}^R$.

However, a cluster point in the local stable manifold $\DD_{out}^R$ 
contradicts the local structure of the stable and
unstable manifolds at $q_{out}$.  (The dynamics are analytic on
$\DD_{out}^R$, and the set of cluster points being closed, we would find
that $q_{out}$ itself would be a cluster point, which has already been
ruled out.)  Thus, the cluster point $q$ cannot exist.
\end{proof}


\newpage
\part{Appendices}

\appendix
\chapter{Background Material}


\section{Gronwall inequality}\label{sec:gronwall}

We give two versions of Gronwall's inequality.

\index[n]{zy@$\psi$, local auxiliary function or integral}

\begin{lemma}[Gronwall inequality] 
Let $I\subset\R$ be an interval, $t_0\in I$, and
let $\psi_1,\psi_2,x$ be continuous nonnegative functions on $I$.
If
\[
x(t)\le \psi_1(t) + \left|\int_{t_0}^t \psi_2(s)x(s) ds\right|,\quad
\text{for all } t\in I,
\]
then for all $t\in I$,
\[
x(t) \le \psi_1(t)+ \left|\int_{t_0}^t \psi_1(s)\psi_2(s)
\exp{|\int_s^t \psi_2(\tau) d\tau|} ds\right|.
\]
\end{lemma}
\begin{proof}
See ~\cite[p.90]{amann2011ordinary}.
\end{proof}


Here is the second version.
\begin{lemma}
Let $x:[t_0,t_1]\to\R^n$ be absolutely continuous and satisfy
\[
\|x'(t)\|\le \psi_2(t)\|x(t)\| + \psi_1(t),\quad t\in[t_0,t_1]\ \text{a.e.},
\]
where $\psi_1,\psi_2\in L^1(t_0,t_1)$, with $\psi_2$ nonnegative.  Then, for all
$t\in[t_0,t_1]$, we have
\[
\|x(t)-x(t_0)\|\le 
\int_{t_0}^t \exp(\int_s^t \psi_2(t) d\tau)
(\psi_2(s)\|x(t_0)\| + \psi_1(s)) ds
\]
\end{lemma}
\begin{proof}
\cite[Th.6.41]{clarke2013functional}.
\end{proof}

\index[n]{x@$x$, function of time}
\begin{corollary}\label{thm:Gronwall} 
Let $x:[0,t_1]\to\R$ be a nonnegative continuous function, let $n$
be positive integer, and let $C,C_1$ nonnegative real numbers.  Assume
\[
x(t)\le C t^n  + C_1 \int_0^t x(t) dt,\quad\text{for all } t\in[0,t_1].
\]
Then $x(t) =O(t^n)$ for $t$ nonnegative and sufficiently close to
$t=0$.
\end{corollary}


\section{Functional derivative}

\index[n]{F@$F,G$, smooth functions}
\index{functional derivative}
\index[n]{zd@$\delta/\delta X$, functional derivative} 
\index[n]{0@$\bracks{-}{-}_*$, canonical pairing}
\index[n]{v@$\mb{v}$, vector!$\mb{v},\mb{w}$, vectors}
\index[n]{0@$-'$, derivative!directional}
\index{directional derivative}
\index[n]{V@$V$, vector space}

\begin{definition}[functional derviative]
Consider a real finite-dimensional vector space $V$, and its linear
dual $V^*$.  Let $F:V\to\R$ be smooth.  We define the \emph{functional
derivative} $\delta{F}/\delta\mb{v}\in V^*$ in terms of
the \emph{directional derivative} $F'$ of $F$ at $\mb{v}\in V$ by
\begin{equation}\label{eqn:functional-derivative}
F'(\mb{v};\mb{w}):=
\lim_{t \to 0}\frac{1}{t}(F(\mb{v} + t \mb{w}) - F(\mb{v}))
=: \left\langle \mb{w}, \frac{\delta F}{\delta\mb{v}} \right\rangle_{*}
\end{equation}
for all directions $\mb{w} \in V$, where
$\langle\cdot,\cdot\rangle_{*}$ is the natural pairing between a
vector space and its dual.
\end{definition}

\section{Stable and Unstable Manifolds}\label{sec:stable-manifold}

\index{fixed point!hyperbolic}
\index{stable manifold}
\index{generalized eigenspace}
\index{hyperbolic fixed point}

We review basic facts about stable and unstable manifolds at a
hyperbolic fixed point~\cite[Chapter~6]{irwin}.  Let $M$ be a manifold
and let $f:M\to{}M$ be a diffeomorphism with fixed point $p=f(p)$.  The
\emph{global stable set} at $p$ is the set of all $q\in{}M$ such that
$\lim_{n\to\infty}f^n(q)=p$.  If $U\subseteq{}M$ is open, the
\emph{local stable set} at $p$ is the set of all $q\in{}U$ such that
$\lim_{n\to\infty}f^n(q)=p$.

A fixed point $p$ is \emph{hyperbolic} if the tangent map
$T_pf:T_pM\to{}T_pM$ has no eigenvalues of absolute value $1$.  At a
hyperbolic fixed point, the tangent space $T_pM$ is a direct sum of
two summands, according to the factorization of the characteristic
polynomial of $T_pf$ into two factors: the \emph{stable} factor with
eigenvalues $|\lambda|<1$ and the \emph{unstable} factor with
eigenvalues $|\lambda|>1$.

\begin{theorem}[Irwin~\cite{irwin}\relax]  Let $p$ be a hyperbolic
fixed point of a $C^r$ diffeomorphism ($r\ge1$) of $M$.  Then, for
some open neighborhood $U$ of $p$, the local stable set $W^s(p)$ of
$f|_U$ at $p$ is a $C^r$ embedded submanifold of $M$, tangent at $p$
to the stable summand of $T_pf$.  The global stable set at $p$ is a
$C^r$ immersed submanifold of $M$, tangent at $p$ to the stable
summand of $T_pf$.
\end{theorem}

There are corresponding statements for local and global unstable
manifolds $W^u(p)$.  Unstable for $f$ means stable for $f^{-1}$.

\section{Classical Lie Groups and Lie Algebras}\label{sec:Lie}

\index[n]{GL@$\mathrm{GL}_n$, general linear group}
\index{general linear group}
\index[n]{SL@$\mathrm{SL}_n$, special linear group}
\index{special linear group}
\index[n]{G@$G$, Lie group}

Let $\mathrm{GL}_n(\C)$ be the general linear group, consisting of all
invertible linear transformations $\C^n\to\C^n$. Let
$\mathrm{SL}_n(\C)$ be the special linear subgroup, consisting of all
linear transformations of determinant $1$.  
Let $\mathrm{GL}_n(\R)$ and $\mathrm{SL}_n(\R)$ be the general linear and
special linear groups of linear transformations $\R^n\to\R^n$.
All of the groups $\mathrm{GL}_n(\C)$, $\mathrm{SL}_n(\C)$,
$\mathrm{GL}_n(\R)$, and $\mathrm{SL}_n(\R)$ are Lie groups.

\index[n]{GL@$\mathrm{GL}_n$, general linear group!$\mathfrak{gl}_n$, Lie algebra}
\index[n]{SL@$\mathrm{SL}_n$, special linear group!$\mathfrak{sl}_n$, Lie algebra}
\index[n]{G@$G$, Lie group!$\mathfrak{g}$, Lie algebra}

The Lie algebras of these groups are $\mathfrak{gl}_n(\C)$ ($n\times
n$ matrices with complex entries), $\mathfrak{sl}_n(\C)$ (complex
entries and trace zero), $\mathfrak{gl}_n(\R)$ (real entries), and
$\mathfrak{sl}_n(\R)$ (real entries and trace zero).

\index[n]{SU@$\mathrm{SU}(1,1)$, special unitary group!$\mathrm{U}(1,1)$, full unitary group}
\index[n]{SU@$\mathrm{SU}(1,1)$, special unitary group}
\index{unitary group}
\index{special unitary group}
\index[n]{J@$\Jsu=\mathrm{diag}(-i,i)$, Cayley transform of $J$}
\index[n]{SU@$\mathrm{SU}(1,1)$, special unitary group!$\su$, Lie algebra}
\index[n]{t@$t$, real number!matrix entry}
\index[n]{yz@$z\in\C$}

The unitary group $U(1,1)$ of signature $(1,1)$ is
\[
U(1,1) = \{g\in\mathrm{GL}_2(\C)\mid
\bar{g}^{tr} \Jsu g = \Jsu\},
\]
where $\Jsu = \mathrm{diag}(-i,i)$.  
The special unitary group $\SU$ is the determinant $1$ subgroup of $U(1,1)$.
The Lie algebra $\su$ of $\SU$ is given
by 
\[
\{X\in\sl(\C)\mid \bar{X}^{tr} \Jsu + \Jsu X = 0\} = 
\left\{ \begin{pmatrix} i t & z \\ \bar{z} & -i t\end{pmatrix}\mid t\in\R, z\in\C\right\}.
\]

\index[n]{0@$-^{tr}$, transpose}
\index[n]{SO@$\mathrm{SO}$, special orthogonal group!$\mathfrak{so}$, Lie algebra}
\index{special orthogonal group}

The special orthogonal group $\mathrm{SO}(m,n)$ is the subgroup of
$\mathrm{SL}_{m+n}(\R)$ preserving a symmetric matrix $A$ of signature
$(m,n)$:
\[
\mathrm{SO}(m,n) = \{g\in\mathrm{SL}_{m+n}(\R)\mid g^{tr} A g = A\},
\]
where $-^{tr}$ is the transpose. Different choices of matrices $A$ of
the same $(m,n)$ or reversed $(n,m)$ signature give isomorphic Lie
groups.  The Lie algebra is
\[
\mathfrak{so}(m,n) = \{X\in\mathfrak{sl}_{m+n}(\R)\mid
X^{tr} A + A X = 0\}.
\]

\index{adjoint representation!Lie group}
\index{adjoint representation!Lie algebra}
\index[n]{Ad@$\Ad$, adjoint representation of Lie group}
\index[n]{ad@$\ad$, adjoint representation of the Lie algebra}

For a general Lie group $G$, the Lie algebra $\mathfrak{g}$ can be defined
as the tangent space $T_eG$ at the neutral element $e\in G$.
The group $G$ acts as inner automorphisms (conjugation) on itself.
Passing to the tangent spaces, the differential of inner automorphism 
affords a representation $\op{Ad}:G\to\mathrm{GL}(\mathfrak{g})$
on the Lie algebra $\mathfrak{g}$, called the adjoint representation.
Again by taking derivatives, this in turn affords a representation
of the Lie algebra
$\mathrm{ad}:\mathfrak{g}\to\mathfrak{gl}(\mathfrak{g})$,
called the adjoint representation of the Lie algebra.

\index{coadjoint representation}
\index[n]{ad@$\ad^*$, coadjoint representation of the Lie algebra}
\index[n]{Ad@$\Ad^*$, coadjoint representation of the Lie group}

Let $\mathfrak{g}^*$ be the linear dual of the Lie algebra $\mathfrak{g}$.
The coadjoint representation $\op{Ad}^*:G\to\mathrm{GL}(\mathfrak{g}^*)$
of $G$
is defined by
\[
\bracks{\op{Ad}_g^* Y}{X}_* = \bracks{Y}{\op{Ad}_{g^{-1}} X}_*,
\]
for all $Y\in\mathfrak{g}^*$ and $X\in\mathfrak{g}$.
The coadjoint representation $\op{ad}^*:\mathfrak{g}\to\mathfrak{gl}(\mathfrak{g}^*)$
of the Lie algebra $\mathfrak{g}$ is defined by
\[
\bracks{\op{ad}_Z^* Y}{X}_* = \bracks{Y}{-\op{ad}_{Z} X}_*,
\]
for all $Y\in\mathfrak{g}^*$ and $X\in\mathfrak{g}$.

\index{Cayley transform}
\index[n]{Cayley@$\mathrm{Cayley}$, Cayley transform}
\index[n]{A@$A$, matrix or linear map!Cayley transform matrix}

The \emph{Cayley transform} of a $2\times2$ matrix $X$ is defined
as 
\[
\mathrm{Cayley}(X):= A^{-1} X A,\quad\text{where}\ A= \frac{1}{\sqrt{2}}\mattwo 1 i i 1\in\SL(\C).
\]

\section{Exceptional Isomorphisms in Rank One}
\label{sec:sl2-exceptional}

\index[n]{0=@$\cong$, isomorphism}

\begin{lemma}
\begin{itemize}
    \item There is an isomorphism of Lie groups:
    \begin{equation}\label{eqn:SL2-SU11}
        \SL(\R) \cong \SU
    \end{equation}
\item There are isomorphisms of Lie algebras:
\begin{equation}\label{eqn:sl2-isoms}
\sl(\R) \cong \su  \cong \sotwo
\end{equation}
\end{itemize}
\end{lemma}

\begin{proof}
The isomorphism between the special linear and special unitary group
is provided by the Cayley transform.
We have
\begin{align}
   \mathrm{Cayley}(\SL(\R))= A^{-1} \, \SL(\R) \, A &= \SU \\
   \mathrm{Cayley}(\sl(\R))= A^{-1} \, \sl(\R) \, A &= \su.
\end{align}

\index{trace}
\index[n]{0@$\bracks{-}{-}$, bilinear form!trace form on Lie algebra}

To establish the isomorphism with the special orthogonal Lie algebra
consider the adjoint representation of $\mathfrak{g}=\sl(\R)$.
\[
\mathrm{ad}:\sl(\R)\to\mathfrak{gl}(\mathfrak{g}).
\]
We set $\bracks{X}{Y}=\mathrm{trace}(XY)$, for $X,Y\in\sl(\R)$.
This is a quadratic form of signature $(2,1)$ on $\sl(\R)$.
The linear transformation $\mathrm{ad}_X$ preserves the quadratic
form in the sense that 
\[
\bracks{\mathrm{ad}_X Y}{Z} + \bracks{Y}{\mathrm{ad}_X Z} = 0.
\]
This implies that the image of the adjoint representation is contained
in a special orthogonal Lie subalgebra of
$\mathfrak{gl}(\mathfrak{g})$ of signature $(2,1)$. This is an
isomorphism.
\end{proof}

\section{Matrix Identities}
\index[n]{X@$X$, Lie algebra element!$X,Y,Z,W\in\sl$}

We collect the following properties of matrices in $\sl(\C)$.

\begin{proposition}\label{prop:brack-brack}
For matrices $X,Y \in \sl(\C)$ we have
\begin{align*}
\ad_X^2\,Y &=
[[Y,X],X] = -2 \det(X) Y - 2 X Y X
\\
&= 2 \bracks{X}{X} Y - 2 \bracks{X}{Y} X.
\end{align*}
\end{proposition}

\begin{proposition}\label{prop:trace-quotient-sub}
For matrices $X,Y,Z,W \in \sl(\C)$, we have
\[ 
\langle X,Z\rangle\bracks{Y}{W} - \langle Y,Z \rangle \bracks{X}{W} 
= -\frac{1}{2}\bracks{[X,Y]}{[Z,W]}
\]
\end{proposition}
\begin{proof}
Compute.
\end{proof}
\begin{proposition}\label{prop:brack-brack-inner-prod}
For matrices $X,Y,Z, W \in \sl(\C)$ such that $\bracks Y Z = 0$, we have
\[
\bracks{[X,Y]}{[Z,W]} = -2 \bracks{X}{Z} \bracks{Y}{W}
\]
\end{proposition}
\begin{proof}
This is immediate from the previous proposition.
\end{proof}

\index[n]{exp@$\exp$, exponential and matrix exponential}
\index{exponential!matrix}

The matrix exponential is defined by the power series,
which converges for all $n\times n$ matrices $X$:
\[
\exp(X) = \sum_{k=0}^\infty \frac{X^k}{k!}.
\]

\index[n]{sinh@$\sinh$, hyperbolic sine}
\index[n]{cosh@$\cosh$, hyperbolic cosine}

\begin{lemma}\label{lem:exp-of-sl2}
If $X \in \sl(\R)$ and $d=\det(X)$, then $\exp(X) \in \SL(\R)$,
and
\begin{align}\label{eqn:matrix-exp-sl2}
\begin{split}
\exp(t X) &= \cosh{(t\sqrt{-d})} I_2 
+ \frac{\sinh{(t\sqrt{-d})}}{\sqrt{-d}}X,\quad (d<0)\\
&=\cos{(t\sqrt{d})} I_2 
+ \frac{\sin{(t\sqrt{d})}}{\sqrt{d}}X,\quad (d>0)\\
&=I_2 + t X,\quad (d=0).
\end{split}
\end{align}
\end{lemma}

\index{Cayley-Hamilton theorem}
\index{Rodrigues formula for rotations}

\begin{proof}
By the Cayley-Hamilton theorem, for $X\in\sl(\R)$, the matrix
exponential $\exp(t X)$ is a linear combination of $I_2$ and $X$.  The
lemma makes this linear combination explicit.  The lemma is a variant
of the classical Rodrigues formula, which holds for rotation matrices.
The two sides of the identity are equal, both being the unique solution
of the initial value problem
\[
F'(t) = X F(t),\quad F(0)=I_2.
\]
The
determinant of $\exp(X)$ is given by the formula
\[
\det(\exp (X)) = \exp(\tr (X)) = 1
\]
since $\tr(X) = 0$. 
\end{proof}

\section{Symplectic Geometry}\label{sec:symplectic}

\index[n]{V@$V$, vector space}



For any finite
dimensional vector space $V$ with dual $V^*$, we have a nondegenerate
pairing between the exterior
power $\Lambda^k (V^*)$ and  $\Lambda^k V$ that sends
\[
\mb{v}_1^*\wedge\mb{v}_2^*\cdots\wedge\mb{v}_k^*\in\Lambda^k(V^*),\quad
\mb{w}_1\wedge\mb{w}_2\cdots\wedge \mb{w}_k\in\Lambda^kV
\]
to $\det(\bracks{\mb{v}_i^*}{\mb{w}_j}_*)$.
We can regard an element $\omega$ of the exterior power $\Lambda^k(V^*)$
as an alternating multilinear map on $V^k$ by using this pairing.
\[
\omega(\mb{w}_1,\ldots,\mb{w}_k) = 
\bracks{\omega}{\mb{w}_1\wedge\mb{w}_2\cdots\wedge\mb{w}_k}_*.
\]

\index[n]{v@$\mb{v}$, vector!in $T(T^*M)$}
\index[n]{v@$\mb{v}$, vector!$\mb{v}^*\in T^*M$}
\index[n]{v@$\mb{v}$, vector!$\mb{v}_*\in TM$}
\index[n]{zh@$\theta$, differential one-form}
\index[n]{zz@$\omega$, two-form!on symplectic manifold}
\index[n]{F@$F,G$, smooth functions}
\index[n]{0@$-\vec{\phantom{-}}$, vector field of function}
%

Recall that the cotangent bundle $T^*M$ of a smooth manifold carries
the tautological one-form $\theta$.  Each element $\mb{v}$ of
$T(T^*M)$ defines both an element $\mb{v}^*\in T^*M$ (by projection)
and an element $\mb{v}_*\in TM$ (by the tangent map of $T^*M\to M$).
The tautological one-form is defined by the relation
$\bracks{\theta}{\mb{v}}_*=\bracks{\mb{v}^*}{\mb{v}_*}_*$.  The
exterior derivative $\omega=d\theta$ defines a canonical two-form on
$T^*M$, giving the cotangent bundle the structure of a symplectic
manifold.

\index{tautological one-form}
\index{canonical two-form}
\index{symplectic vector field $\vec{F}$}
\index{Poisson bracket}
\index{Hamilton's equation}

Each differentiable function $F$ on a symplectic manifold $(M,\omega)$ defines a
vector field $\vec{F}$ by
\[
\omega_q(\vec{F},\mb{v}_*)=\bracks{dF}{\mb{v}_*}_*, 
\]
for $\mb{v}_*\in{T_qM}$.
The Poisson bracket is defined by
\(
\{F,G\}=\omega(\vec{F},\vec{G})
\).
Hamilton's equation corresponding to a Hamiltonian $\H$ is the ODE
\[
p'=\{p,\H\}.
\]

\section{Lie-Poisson Dynamics on the Lie Algebra}\label{sec:X-lie-poisson}

\index{Lie-Poisson!dynamics}
\index{Lie-Poisson!bracket}
\index[n]{0@$\{-,-\}$, Poisson bracket!Lie-Poisson}

The dual vector space $\mathfrak{g}^*$ can be equipped with a Poisson
bracket\index{Poisson!bracket} called the $\pm$ \textit{Lie-Poisson
  bracket}: if $F,G$ are two smooth functions on $\mathfrak{g}^*$,
then the bracket is given by
\[
\index[n]{0@$\{-,-\}$, Poisson bracket}
\{F,G\}(X^*) = \pm \left\langle X^*, 
\left[ \frac{\delta F}{\delta X^*}, 
\frac{\delta G}{\delta X^*} \right] \right\rangle_{*} \quad 
X^* \in \mathfrak{g}^*
\]
Here we identify $\mathfrak{g}\cong\mathfrak{g}^{**}$.

\index{Lie-Poisson!equations}

Hamilton's equations with respect to this bracket are called
\textit{Lie-Poisson equations} and take the following form (Marsden and
Ratiu~\cite[Proposition~10.7.1]{marsden2013introduction}).

\index[n]{H@$\H$, Hamiltonian!Lie-Poisson}
\index[n]{X@$X$, Lie algebra element!$X^*\in\mathfrak{g}^*$}

\begin{proposition}[Lie-Poisson equations]
Let $G$ be a Lie group. The equations of motion for a smooth
Hamiltonian $\H : \mathfrak{g}^* \to \R$ with respect to the $\pm$
Lie-Poisson brackets on $\mathfrak{g}^*$ are
\begin{equation}\label{eqn:lie-poisson-gen}
\frac{dX^*}{dt} = \mp \ad^*_{\delta \H/ \delta X^*}X^* \quad X^* \in \mathfrak{g}^*
\end{equation}
\end{proposition}

\index{semisimple Lie algebra}
\index[n]{0@$\bracks{-}{-}$, bilinear form!nondegenerate}
\index[n]{X@$X$, Lie algebra element!$X,Y,Z\in\mathfrak{g}$}

Assume further that our Lie algebra $\mathfrak{g}$ is semisimple: it
can be equipped with a nondegenerate bilinear form, which we denote by
$\bracks \cdot \cdot$.  This bilinear form satisfies the following
relation:
\begin{equation}\label{eqn:nondegen-inner-prod}
\bracks {X} {[Y,Z]} = \bracks {[X,Y]} Z, \quad X,Y,Z \in \mathfrak{g}
\end{equation}

\index[n]{0@$-^*$, dual!isomorphism}
\index[n]{X@$X$, Lie algebra element!$X,Y,Z\in\mathfrak{g}$}
\index[n]{Y@$Y\in\mathfrak{g}$, Lie algebra element}

Using this bilinear form, we can identify $\mathfrak{g}^*$ with
$\mathfrak{g}$ as follows:
\begin{equation}\label{eqn:semisimple-lg-ident}
X^*(Y) = \bracks X Y \qquad X, Y\in \mathfrak{g}, \ X^* \in \mathfrak{g}^*,
\end{equation}
where $X^*$ maps to $X$, under this isomorphism.


This isomorphism maps the operator $\ad$ to $\ad^*$ and so equation
\eqref{eqn:lie-poisson-gen} becomes
\[ 
\frac{dX}{dt} = \mp \ad_{\delta \H/\delta X} X 
=  \mp \left[\frac{\delta \H}{\delta X}, X \right],\quad
X\in\mathfrak{g}.
\]

\index[n]{0@$\bracks{-}{-}$, bilinear form!trace form on Lie algebra}

Armed with this background material, in this section we recast the
dynamics for $X$ in our system, as given in Lemma
\ref{lem:X-dynamics}, as the Lie-Poisson equation of a
control-dependent Hamiltonian on the vector space $\sl(\R)^*$.  To do
this we shall need to exhibit a Hamiltonian function. Recall that we
have defined $\bracks{X}{Y} = \tr(XY)$ for matrices $X,Y \in \sl(\R)$.

\index[n]{H@$\H$, Hamiltonian!Lie-Poisson}
\index[n]{YZ@$Z_u$, control matrix!$Z_0$, constant}
\index[n]{ln@$\ln$, natural log}

\begin{proposition}\label{lem:X-ham-frac-deriv}
If 
\(
\H(X)= -\frac{\langle X,X \rangle}{2}\ln 
\frac{\langle X,X \rangle}{\langle X, Z_0 \rangle}
\) 
then 
\[
X' = -\ad_{\delta \H/\delta X}X 
= -\frac{\langle X,X \rangle}{2\langle X,Z_0 \rangle}\left[Z_0, X\right]
\]
\end{proposition}
\begin{proof}
The function $\H$ is well-defined since $\langle X, X \rangle = -2 $
and on the star-domain, by Lemma \ref{lem:sl2-star-condition}, we
have that $\langle X,Z_0 \rangle < 0$.
Now we have that
\begin{align*}
    \ad_{\delta \H/\delta X}X &= \left[ \frac{\delta \H}{\delta X}, X \right] 
\\
    &= \left[\frac{\delta}{\delta X}\left(-\frac{\langle X,X \rangle}{2} 
\ln \frac{\langle X,X \rangle}{\langle X, Z_0 \rangle}\right) ,X\right]
\\
    &= \left[ -X \ln \frac{\langle X,X \rangle}{\langle X, Z_0 \rangle} 
- \frac{\bracks X X}{2} \left( \frac{\bracks X {Z_0}}{\bracks X X} 
\frac{2X \bracks X {Z_0} - \bracks X X Z_0}{{\bracks X {Z_0}}^2}\right), X\right] 
\\
    &= \frac{\bracks X X}{2 \bracks X {Z_0}}\left[ Z_0, X \right] = -X'
\end{align*}
Thus, we see that the dynamics for $X$ is Lie-Poisson with respect
to the Hamiltonian $\H(X) = -\frac{\langle X,X \rangle}{2} \ln
\frac{\langle X,X \rangle}{\langle X, Z_0 \rangle}$.
\end{proof}
\begin{remark}\normalfont
Note that, with the parameterization of Section
\ref{sec:lie-algebra-dynamics}, the Hamiltonian becomes $\H(X) =
\ln(-\bracks{X}{Z_0})$.
\end{remark}

\section{Poisson Reduction of the Extended State Space}\label{sec:poisson-bracket}

The Poisson manifold $T^*T\SL(\R)$ can be Poisson-reduced by
left-translation symmetries arising from the left-multiplication
action of $\SL(\R)$. This reduction results in a Poisson bracket on
the reduced Poisson manifold, which we call the \textit{extended space
  Poisson bracket}.  For preliminaries on Poisson reduction, we refer
to Chapter 10 of Marsden~and~Ratiu~\cite{marsden2013introduction}.

This reduction procedure also reduces a Hamiltonian system on
$T^*T\SL(\R)$ to a system on the quotient 
\begin{equation}\label{eqn:cotangent-invariant}
T^*T\SL(\R)/\SL(\R)\cong\sl(\R)^* 
\times \sl(\R) \times \sl(\R)^*
\cong\sl(\R) 
\times \sl(\R) \times \sl(\R),
\end{equation}
by means of the invariant bilinear form on $\sl(\R)$.  Thus, for
example, the Hamiltonian system arising from the Pontryagin Maximum
Principle gets reduced this way. We have already seen expressions for
integral curves of the reduced Hamiltonian vector field in
Section~\ref{sec:costate-variables}.

These ODEs for $X,\Lambda_1,\Lambda_R$ on the quotient Poisson
manifold can be written in Poisson bracket form with respect to the
extended space Poisson bracket. We have the following expression for
this bracket, which appears in multiple sources. See
Jurdjevic~\cite{jurdjevic2016optimal},
Gay-Balmaz~et~al.~\cite[pp.~34]{gay2012invariant} and
Esen~et~al.~\cite[pp.~13]{esen2021tulczyjew}.

\index[n]{F@$F,G$, smooth functions}
\index{Poisson!extended bracket}

\begin{theorem}\label{thm:extended-space-poisson-bracket}
If $F$ and $G$ are left-invariant smooth functions on $T^*(T\SL(\R))$,
then we can identify them with functions on the quotient
\eqref{eqn:cotangent-invariant}, which is isomorphic to $\sl(\R)^3$.
By using 
coordinates $(X,\Lambda_1,\Lambda_2)$ introduced in Section~\ref{sec:PMP},
their \emph{extended space Poisson bracket} on the quotient Poisson manifold is given
by
\[
\index[n]{0@$\{-,-\}$, Poisson bracket!$-_{ex}$, extended}
\{F,G\}_{ex} := \left\langle \Lambda_1, 
\left[ \frac{\delta F}{\delta\Lambda_1}, 
\frac{\delta G}{\delta\Lambda_1} \right] \right\rangle
 + \bracks{\frac{\delta F}{\delta X}}
{\frac{\delta G}{\delta \Lambda_2}} 
-  \bracks{\frac{\delta F}{\delta \Lambda_2}}{\frac{\delta G}{\delta X}}
\]
which is the sum of the Lie-Poisson bracket on $\sl(\R)^*$ and the
canonical Poisson bracket on $T^*(\sl(\R))$ (where $\sl(\R)$ is
identified with the dual $\sl(\R)^*$ as needed). Here $\delta/\delta
X$ denotes the functional derivative with respect to $X$.
\end{theorem}

Using this bracket, we can deduce the following theorem. 

\begin{theorem} \label{thm:reinhardt-poisson}
  The Reinhardt system defined in
  problem~\ref{pbm:state-costate-reinhardt} can be written in Poisson
  bracket form as follows.
  \begin{align*}
    X' &= \{X,\mathcal{H}\}_{ex}, \\ 
    \Lambda_1' &= \{\Lambda_1,\mathcal{H}\}_{ex}, \\ 
    \Lambda_R' &= \{\Lambda_R,\mathcal{H}\}_{ex},
  \end{align*}
  where $\H(\Lambda_1,\Lambda_R,X;Z_u) = 
  \left\langle \Lambda_1 - \frac{3}{2}\lambda_{cost} J, X \right\rangle  
  - \frac{\langle \Lambda_R,Z_u\rangle }{\langle X, Z_u \rangle}$ and $\Lambda_R = [\Lambda_2,X]$ as usual.
\end{theorem}
\noindent In the theorem, the bracket is applied to each matrix entry, and
  we identify $\sl(\R)^* \cong \sl(\R)$ via the nondegenerate trace form.

\begin{proof} 
  This is a routine calculation. We show the derivation for $X$ and
  omit the others.  For an arbitrary constant $Y \in \sl(\R)$, we have
  \begin{align*}
    \{\langle X,Y\rangle,\mathcal{H}\}_{ex} &= \bracks{\frac{\delta}{\delta X}\bracks{X}{Y}}
    {\frac{\delta \mathcal{H}}{\delta \Lambda_2}} \\ 
    &= \bracks{Y}{\frac{\delta}{\delta\Lambda_2}\frac{\bracks{\Lambda_2}{[Z_u,X]}}{\bracks{Z_u}{X}}} \\ 
    &= \bracks{Y}{\frac{[Z_u,X]}{\bracks{Z_u}{X}}} \\ 
    &= \bracks{X'}{Y}.
  \end{align*}
which proves the first equation. 
\end{proof}



\section{Symplectic Structure of Coadjoint Orbits}\label{sec:kirillov}
\index{Poisson!structure, direct sum}
\index{coadjoint orbit}
\index[n]{O@$\O_-$, adjoint or coadjoint orbit}
\index[n]{zz@$\omega$, two-form!$\omega^K$, Kirillov}
\index[n]{W@$W^*,Z^*\in\mathfrak{g}^*$}

On a Lie group $G$ with Lie algebra $\mathfrak{g}$,
Kirillov~\cite{kirillov2004lectures} has defined a symplectic
structure on the coadjoint orbit $\O_{Z^*}:=\{\Ad^*_{g^{-1}}Z^*~|~g
\in G\}$ through $ Z^* \in \mathfrak{g}^*$ (the linear dual of the Lie
algebra $\mathfrak{g}$). This two-form $\omega^K$ on $\O_{Z^*}$ is given
by
\[
\omega_{Z^*}^K(\ad_X^*W^*,\ad_Y^*W^*) := \bracks{W^*}{[X,Y]}_*, \quad 
W^* \in \O_{Z^*}, \quad 
X,Y \in \mathfrak{g}
\]
where $\ad^*_XW^*,\ad^*_YW^* \in T_{W^*}\O_{{Z^*}}$.  We specialize this
general construction to our setting with $G = \SL(\R)$.

Since the Lie algebra $\sl(\R)$ carries with it the nondegenerate
trace form: $\bracks{X}{Y} = \mathrm{trace}(XY)$, this sets up a
linear isomorphism $\sl(\R)^* \cong \sl(\R)$, which we use to
transport the symplectic structure from coadjoint orbits to adjoint
orbits.

In this section, we prove that the Kirillov symplectic structures on
the adjoint orbit $\O_X \subset \sl(\R)$ and the symplectic structure
on the Poincar\'{e} upper half-plane $\h$ are
(anti)-equivalent. Recall that we have the following map.
\begin{align*}
     \Phi : \mathfrak{h} &\to \mathcal{O}_J 
\\
     z = x+iy &\mapsto \mattwo {x/y} {-(x^2 + y^2)/y} {1/y} {-x/y} =: \Phi(z),
\end{align*} 
from the upper half-plane to adjoint orbit $\OX = \O_J$ in $\sl(\R)$.
\begin{lemma}
The map $\Phi$ (defined in Lemma \ref{lem:def-phi}) is an
anti-symplectomorphism.
\end{lemma}

\index[n]{zz@$\omega$, two-form!on $\h$}
\index[n]{zz@$\omega$, two-form!$\omega^K$, Kirillov}
\index[n]{0@$\wedge$, wedge product of differential forms}

\begin{proof}
Let $\omega$ be the symplectic form of the upper half-plane.
\[
\omega = \frac{dx \wedge dy}{y^2}     
\]
\index[n]{v@$\mb{v}$, vector!$\mb{v},\mb{w}\in T_z\h$, tangent}
\index[n]{v@$\mb{v}$, vector!$v_i,w_i\in\R$, components of $\mb{v},\mb{w}$}
and let $\omega^K$ be the Kirillov two-form on the coadjoint orbit
$\OX$.  We have to show $\omega^K$ pulls back to the two-form $-\omega$
on the upper half-plane by $\Phi : \h \to \OX$. So, at a point $z =
x+iy \in \h$ and tangent vectors $\mb{v},\mb{w} \in T_z\h$:
\begin{align*}
     \Phi^*\omega^K_z(\mb{v},\mb{w}) &= \omega^K_{\Phi(z)}(T_z\Phi(\mb{v}),T_z\Phi(\mb{w})) 
\\
     &=\left\langle \Phi(z),\left[ \mattwo {\frac{v_2}{2y}} 
{\frac{v_1y - v_2 x}{y})} {0} 
{\frac{v_2}{2y}},\mattwo {\frac{w_2}{2y}} {\frac{w_1y - w_2 x}{y})} {0} 
{\frac{w_2}{2y}} \right]\right\rangle 
\\
     &=-\frac{v_1 w_2 - v_2 w_1}{y^2} = -\omega_z(\mb{v},\mb{w})
\end{align*}
This proves that $\Phi^* \omega^K = -\omega$. 
\end{proof}

\section{Riemannian Metric on Coadjoint Orbits}

\index{regular!semisimple element of Lie algebra}
\index{Cartan decomposition}
\index[n]{p@$\mathfrak{p}_X=X^\perp$, component of Cartan decomposition}

Let $X \in \O_J$.  Then $X$ is regular semisimple, and $\RX$ is a rank
one Cartan subalgebra of $\sl(\R)$, where $\RX$ is the span of $X$.
There is a Cartan decomposition $\sl(\R) = \RX \oplus \mathfrak{p}_X$
decomposition adapted to $\sl(\R)$, where $\mathfrak{p}_X=X^\perp$ is
the two-dimensional orthogonal complement of $\RX$ with respect to the
trace form on $\sl(\R)$.

Explicitly,
\[
\mathfrak{p}_X = \{ [X,Y] \mid Y \in \sl(\R)\}.
\]
From Lemma~\ref{lem:half-plane-lie-algebra-iso},
$\mathfrak{p}_X$ is the tangent space $T_X\O_J$,
which is also identified with $T_{z}\h$, where $X = \Phi(z)$.

By transport of structure, the trace form on $\sl(\R)$ restricts to
$\mathfrak{p}_X$ and defines a symmetric bilinear form on
$T_{z}\h$. By general theory, this quadratic form is positive definite
on $\mathfrak{p}_X$.

\index{Riemannian metric!on $\h$}

\begin{lemma}\label{lem:riemannian}
The symmetric bilinear form on $T_{z}\h$ determined by the trace form
on $\mathfrak{p}_X$ is twice the usual invariant Riemannian metric on
$\h$:
\[
  2\frac{dx^2 + dy^2}{y^2}.
\]
\end{lemma}

\index[n]{e@$\mb{e}_i$, standard basis!$\mb{\tilde e}_i$, image of standard basis.}
\index[n]{e@$\mb{e}_i$, standard basis! of $\R^2$}
\index[n]{zd@$\delta_{ij}$, Kronecker delta}

\begin{proof}
Set
  \[
  \mb{\tilde e}_1 = \begin{pmatrix} 0 & 1 \\ 0 & 0 \end{pmatrix},\quad
  \mb{\tilde e}_2 = \begin{pmatrix} 1/(2y) & -x/y \\
    0 & -1/(2 y)\end{pmatrix}.
    \]
Under the map $T\Phi:T_z\h\to T_X\O_J$ of
Lemma~\ref{lem:tangent-maps}, the preimage of $\mb{\tilde{e}}_1$ and
$\mb{\tilde e}_2$ is the basis $\mb{e}_1=\partial/\partial x$,
$\mb{e}_2=\partial/\partial y$ of $\R^2=T_z\h$.  The isomorphism
$\sl(\R)/\RX\to \mathfrak{p}_X =X^\perp$ is $Y \mapsto [Y,X]$.  Thus,
it is enough to check that
\[
  \langle [\mb{\tilde e}_i,\Phi(z)],[\mb{\tilde e}_j,\Phi(z)]\rangle
  = \frac{2\delta_{ij}}{y^2},
\]
  where $\delta_{ij}$ is the Kronecker delta.
  This is easily computed.
\end{proof}

\chapter{Extensions of the Theory}

\section{Hypotrochoids}\label{sec:hypotrochoids}

\index{hypotrochoid}
\index{roulette curve}

\begin{figure}[htbp]
    \centering
    \includegraphics[scale=0.5]{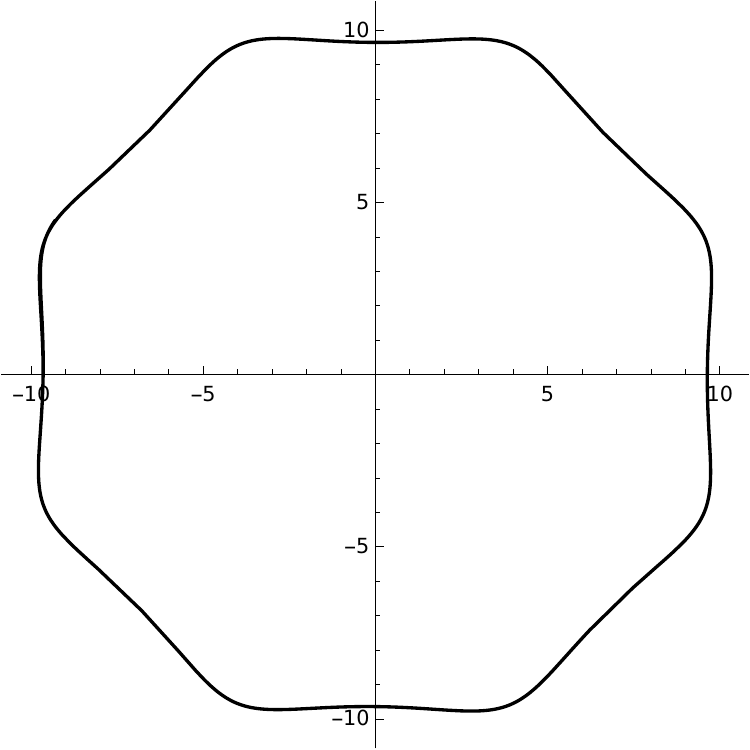}
    \caption{A hypotrochoid resembling the smoothed octagon.}
    \label{fig:smoothed-octagon-hypo}
\end{figure}

\index[n]{r@$r$, real number!$r_i$, hypotrochoid parameter}

A hypotrochoid is a roulette curve which is traced by a point which is
at a distance $r_0$ from the center of a circle of radius $r_1$ as it
rolls without slipping on the inside of a circle of a fixed circle of
radius $r_2$.  The parametric equation of a hypotrochoid in 
the complex plane $\C$ is given by
\begin{align*}
z(t)
= (r_2-r_1)\exp(i t) 
+r_0 \exp\left(-\frac{(r_2-r_1)}{r_1}i t \right).
\end{align*}

This section was motivated by the striking figure in Figure
\ref{fig:smoothed-octagon-hypo}, which depicts a hypotrochoid with
parameters $r_1 = 2.498$, $(r_2-r_1)/r_1= 1/7$, and $r_0=-10$. As we
can see, the Figure
\ref{fig:smoothed} resembles the smoothed octagon.

\index[n]{zf@$\zeta=\exp(2\pi{i}/3)$, cube root of unity}
\index[n]{zs@$\sigma_i$, multi-curve}


If $\zeta$ is a primitive cube root of unity, and $n,j$ are integers,
define
\begin{equation}\label{eqn:hypo}
\sigma_{2j}(t) := 
r \exp(i t)\zeta^j +r_0 \exp(-i  t/n)\zeta^{-j},
\end{equation}
which is a closed curve of period
$2\pi n$. We recover a hypotrochoid from $\sigma_0$ by
setting $1/n = (r_2-r_1)/r_1$ and $r = (r_2-r_1)$.

The smoothed octagon is given by a bang-bang control and hence is not
a real analytic curve. But the following proposition shows that the
hypotrochoid is a multi-curve, realized by a curve in $\SL(\R)$.

\begin{proposition}
If $|r_0| \ne |r|$ and if $n\equiv1\mod3$, then there exists a curve
in $\SL(\R)$ which realizes the hypotrochoid $\sigma_0$.
\end{proposition}
\begin{proof}
We will prove the following identities of the curves $\sigma_{2j}(t)$:
\begin{align*}
\sigma_0(t) + \sigma_2(t) + \sigma_4(t) &= 0, \\
\Re(i{\sigma}_0(t),\sigma_2(t)) &= \mathrm{constant}, \\
\sigma_{2j}(t + \frac{2\pi n}{3}) &= \sigma_{2j+2}(t).
\end{align*}
The first identity is a result of $1+\zeta+\zeta^2=0$. The second
identity follows from
\[
\bar{\sigma}_0(t)\sigma_2(t) 
= r_0^2 \zeta^2 + r^2 \zeta + 2 r_0r\,\Re(\zeta \exp(it + i t/n))
\]
The third follows from the definition of $\sigma_{2j}$.  Identifying
$\mathbb{C}$ with $\R^2$ we get that 
\[
\Re(i{\sigma}_0(t),\sigma_2(t))
= \det(\sigma_0(t), \sigma_2(t)) = \mathrm{constant}.
\]
This means that there is a constant $s > 0$ such that the rescaled
curves $s\,\sigma_{2j}$ (and their negations $-s\,\sigma_{2j}$) form a
multi-curve as in Definition \ref{def:multi-pt-multi-curve}. We can go
through the same construction now as in Section
\ref{sec:lie-group-dynamics} to construct a curve
${g}:[0,t_f]\to\SL(\R)$ so that $s\, \sigma_{2j}(t) =
    {g}(t)\mb{e}_{2j}^*$.
\end{proof}


\begin{remark}
This hypotrochoid result might allow us a further speculation. We can
compute the curvatures $\kappa_j(t)$ of the curves $\sigma_{2j}(t)$
defined in that section and compute their normalization and label them
as \emph{controls}. This then shows that a hypotrochoid
determines a control function in the control vector space
$\{(u_0,u_1,u_2)~|~u_0+u_1+u_2=1\}$.

We might then ask for an optimal control problem which has a
particular hypotrochoid as a global optimizer and investigate how it
might relate to the smoothed octagon.
\end{remark}

\section{Chaos in Numerical Experiments}

This appendix describes some numerical experiments for the Reinhardt
control problem with circular control set for various choices of
parameters. Here we use the hyperboloid coordinates $w,b,c$ introduced
in Section~\ref{sec:coordinate} with fixed angular momentum
$\A_0$.

A numerical experiment suggests that for some values of the
parameters, the trajectories might be chaotic. See
Figure~\ref{fig:chaos}.  However, for other parameter values, the
trajectories appear to be periodic. See Figure~\ref{fig:nonchaos}.
The only difference in parameter values for these two figures is
$\A_0=3$ in the first figure and $\A_0=2.5$ in the second.  We cannot
guarantee the accuracy of these numerical solutions.  



Much further numerical exploration of the solutions would
be desirable, both for circular control sets and for triangular
control sets.

\begin{figure}[ht]
\centering
\includegraphics[scale=0.5]{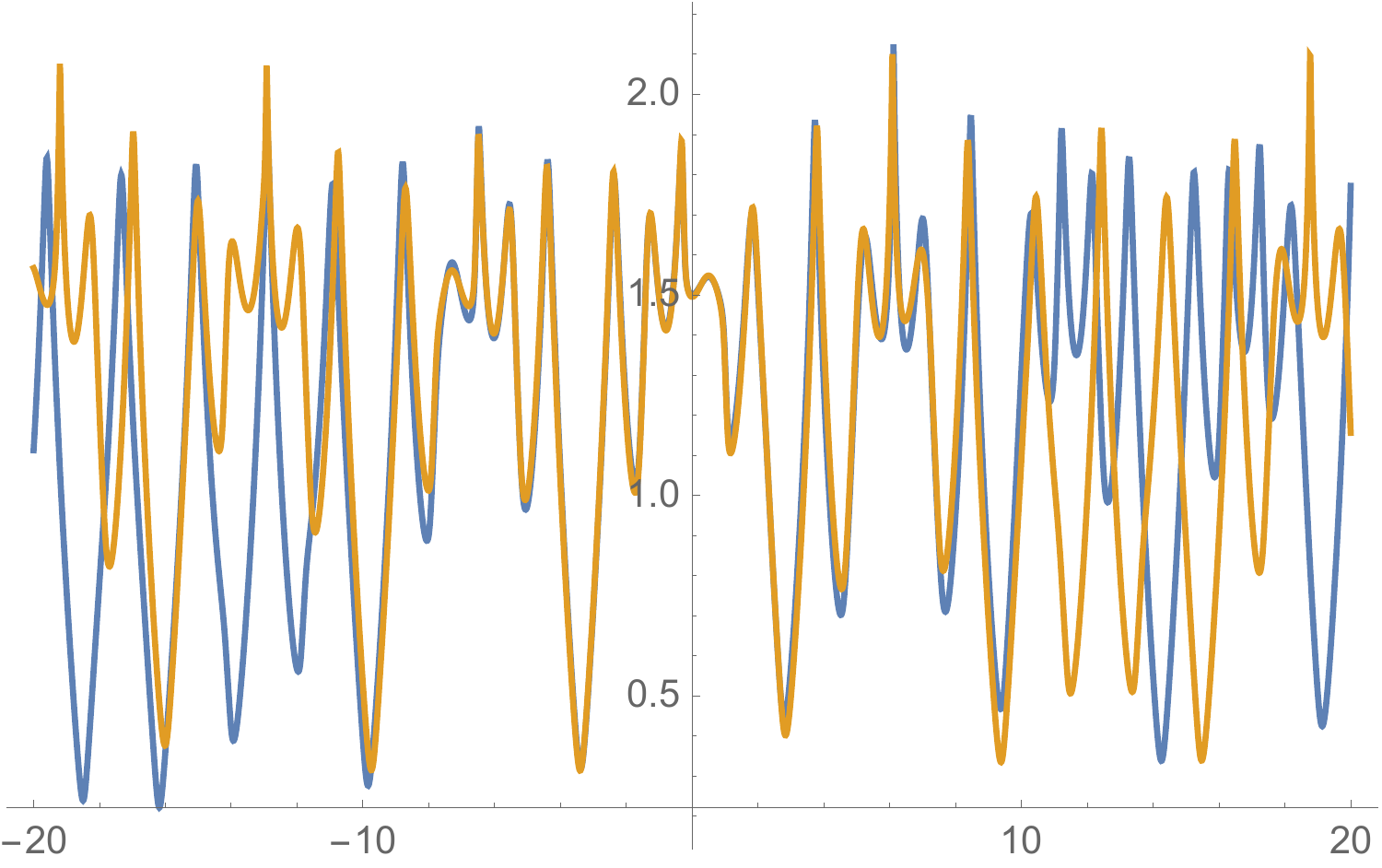}
\caption{The graph of $|w|$ as a function of time.  For nearby initial
  conditions, the trajectories of $|w|$ drift apart in a way that
  suggests the onset of chaos.  The parameter values are $\rho=1.1$,
  $d_1=3/2$, $\epsilon=1$, $\A_0=3$, and $\epsilon_b=1$. The two
  solutions (in orange and blue) have initial values
  $(w_0,c_0)=(1.5,0.5)$ and $(w_0,c_0)=(1.495,0.5)$, respectively.
  The figure was produced using NDSolve, Mathematica's numerical ODE solver.
  }
\label{fig:chaos}
\end{figure}

\index{Mathematica!NDSolve}

\begin{figure}[ht]
\centering
\includegraphics[scale=0.5]{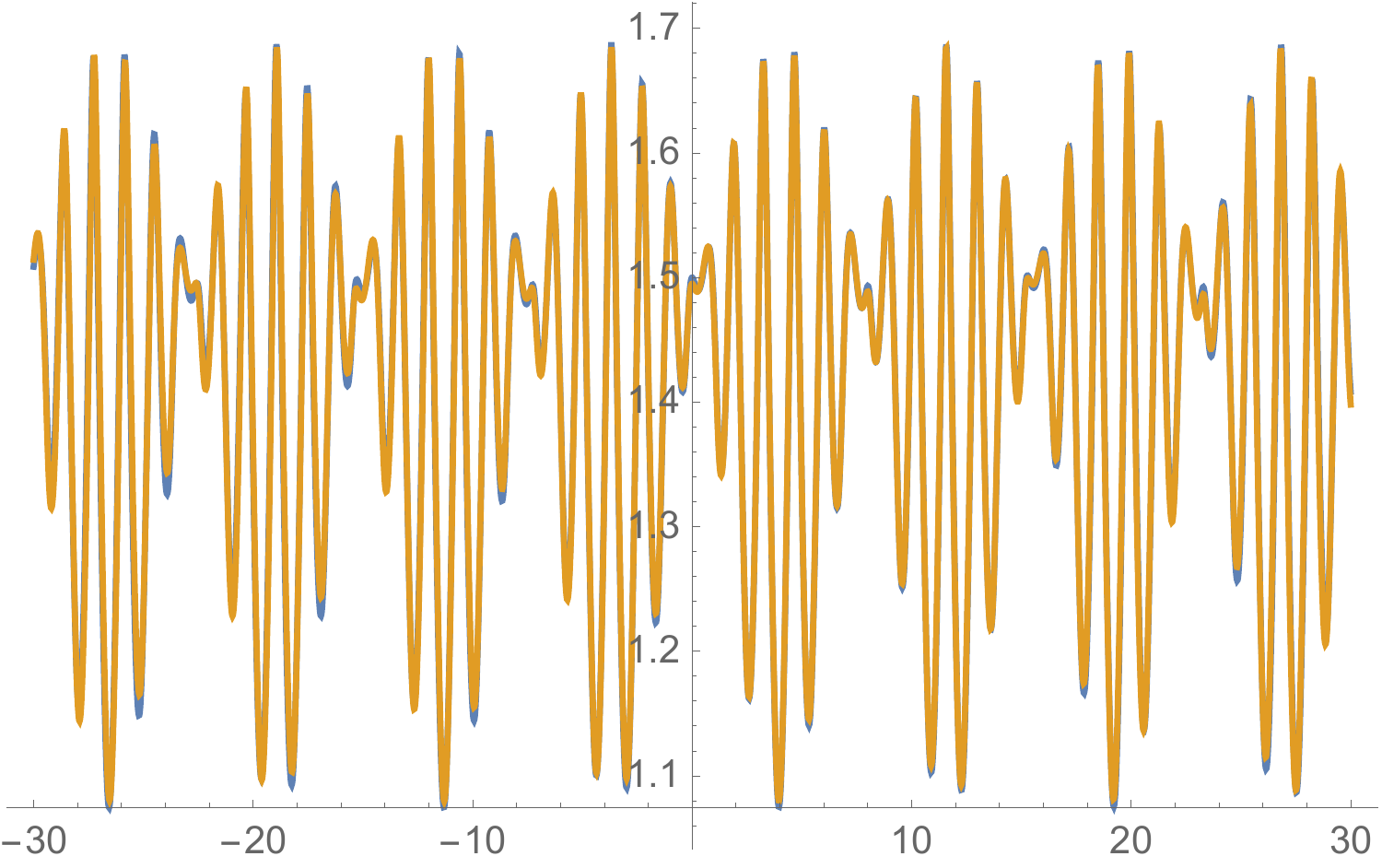}
\caption{The graph of $|w|$ as a function of time.  For other nearby initial conditions, the trajectories of $|w|$ remain
  close to each other.  The trajectories appear to be periodic.
  The parameter values are
  $\rho=1.1$, $d_1=3/2$, $\epsilon=1$, $\A_0=2.5$, and
  $\epsilon_b=1$. The two solutions (in orange and blue)
  have initial values $(w_0,c_0)=(1.5,0.5)$
  and $(w_0,c_0)=(1.495,0.5)$, respectively.  The graphic
  was produced by the numerical solver NDSolve.
}
\label{fig:nonchaos}
\end{figure}

\section{Kuperberg's Area Formula}

\index{Iwasawa decomposition}
\index{Kuperberg, Greg}

Greg Kuperberg has given an area formula for
centrally symmetric disks satisfying the minimality conditions of
Reinhardt~\cite{kuperberg}.  
Write $g\in \SL(\R)$ as
\[
g =
 \begin{pmatrix} 1 &  x\\ 0 & 1 \end{pmatrix}
\begin{pmatrix} y^{1/2} & 0\\ 0 & y^{-1/2} \end{pmatrix}
\begin{pmatrix} \cos\theta & -\sin\theta \\ \sin\theta & \cos\theta
  \end{pmatrix},\quad y>0.
\]
Then $g\cdot{i}=x+i y$.

Let $K$ be a balanced disk in the Euclidean plane, given
by a path $g:[0,t_f]\to\SL(\R)$.  Let $\tilde g:[0,t_f]\to
\h$ be the image $\tilde g(t) = g(t)\cdot{i}$ of the path in the upper-half
plane.  Note that with the usual boundary conditions on $g$, we have
$g(t_f) = g(0)R$, and $\tilde g(t_f) = \tilde g(0)$, so that the curve in
the upper-half plane is
closed.

\index{Cartan-Maurer one-form}

Based on the formula \eqref{eqn:cost}, 
there is a cost one-form, which expressed in the coordinates $(x,y,\theta)$ 
gives
\begin{equation}
  -\frac{3}{2} \tr(J g^{-1} dg) =
   3 d\theta - \frac{3 dx}{2 y},
\end{equation}
where $g^{-1}dg$ is the Cartan-Maurer one-form on $\SL(\R)$,
and $J$ 
is the infinitesimal generator of the rotation group $\SO$.
On any disk with hexagonal symmetry, $\theta(t_f) = \pi/3$.
Also,
\[
d ( dx/ y) = \frac{ dx \wedge dy}{ y^2},
\]
which is the invariant two-form $\omega$ on $\h$.
Thus, the area is
\begin{equation}
\op{area}(K) = \pi - \frac{3}{2}\int\omega
\end{equation}
where the integral is the signed hyperbolic area of the region in the
upper-half plane
enclosed by the path $\tilde g$.
The Reinhardt problem is asking for a maximization of the 
signed area given by the integral.

\index[n]{zz@$\omega$, two-form!on $\h$}

Note that the upper-half plane occurs in two contexts now: as the
codomain of the path $\tilde g:[0,t_f]\to \h$ and as the parameter space for
the tangent $X = \Phi(z)$, $z \in \h$.  Roughly speaking, the
derivative of the first $\h$ is the second $\h$, according to the
relation $g' = g X$.

\begin{example} If $K$ is the circle, then $g(t) = i$ is constant,
  and the signed area $\int\omega=0$.  The area formula reads
  $\op{area}(K) = \pi$.
\end{example}

\begin{example}
  We can imagine the proof of the local optimality of the
  smoothed octagon in Figure~\ref{fig:kuperberg} by the way that the smoothed
octagon is approximately an area maximizing circle.
Also shown are the smoothed $10$-gon, $20$-gon, and $62$-gon.
In general the smoothed $6k-2$-gon
  will have turning number $-k$ (and negative area), and the smoothed $6k+2$
will have turning number $+k$ (and positive area).
  \end{example}
\begin{figure}[ht]
\centering
 \includegraphics[height=2in]{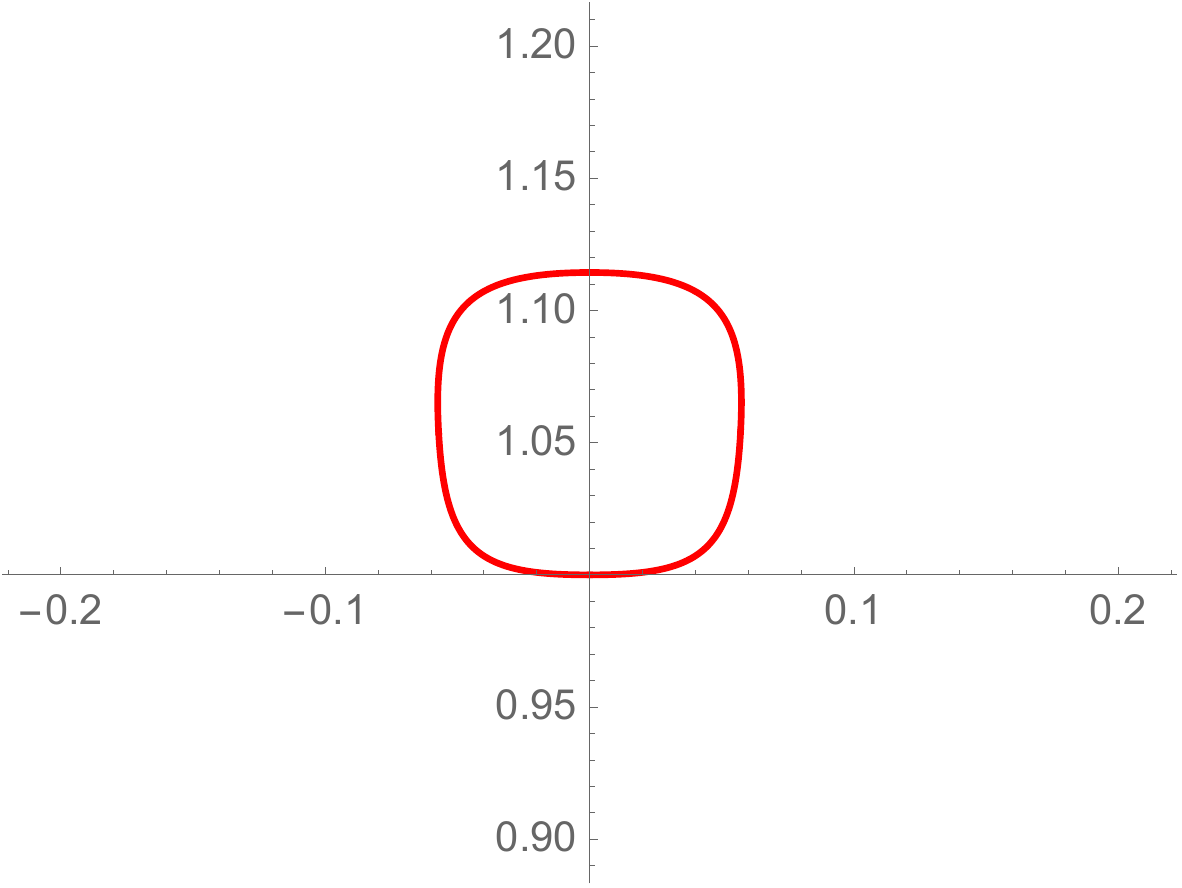}
   \includegraphics[height=2in]{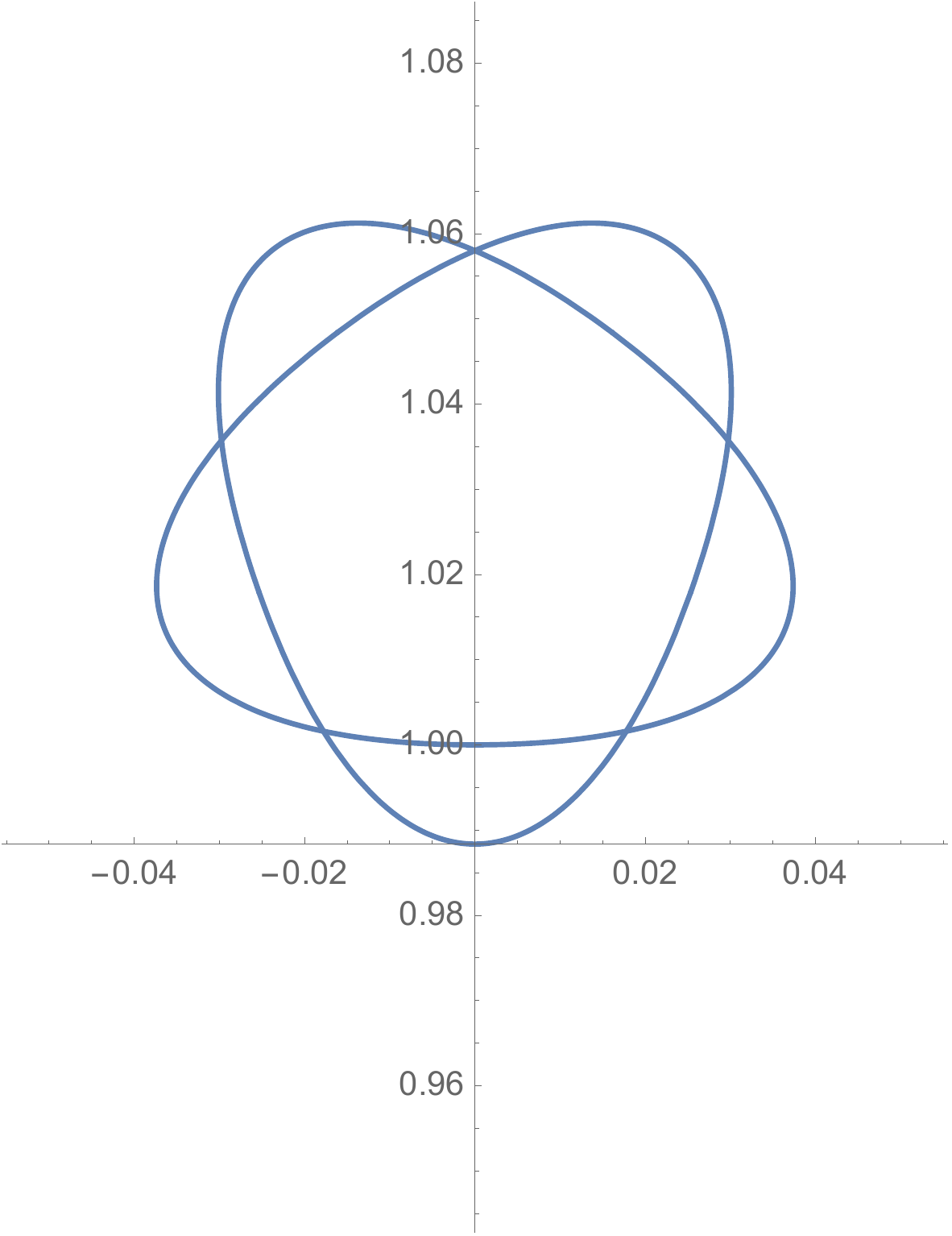}
 \includegraphics[height=2in]{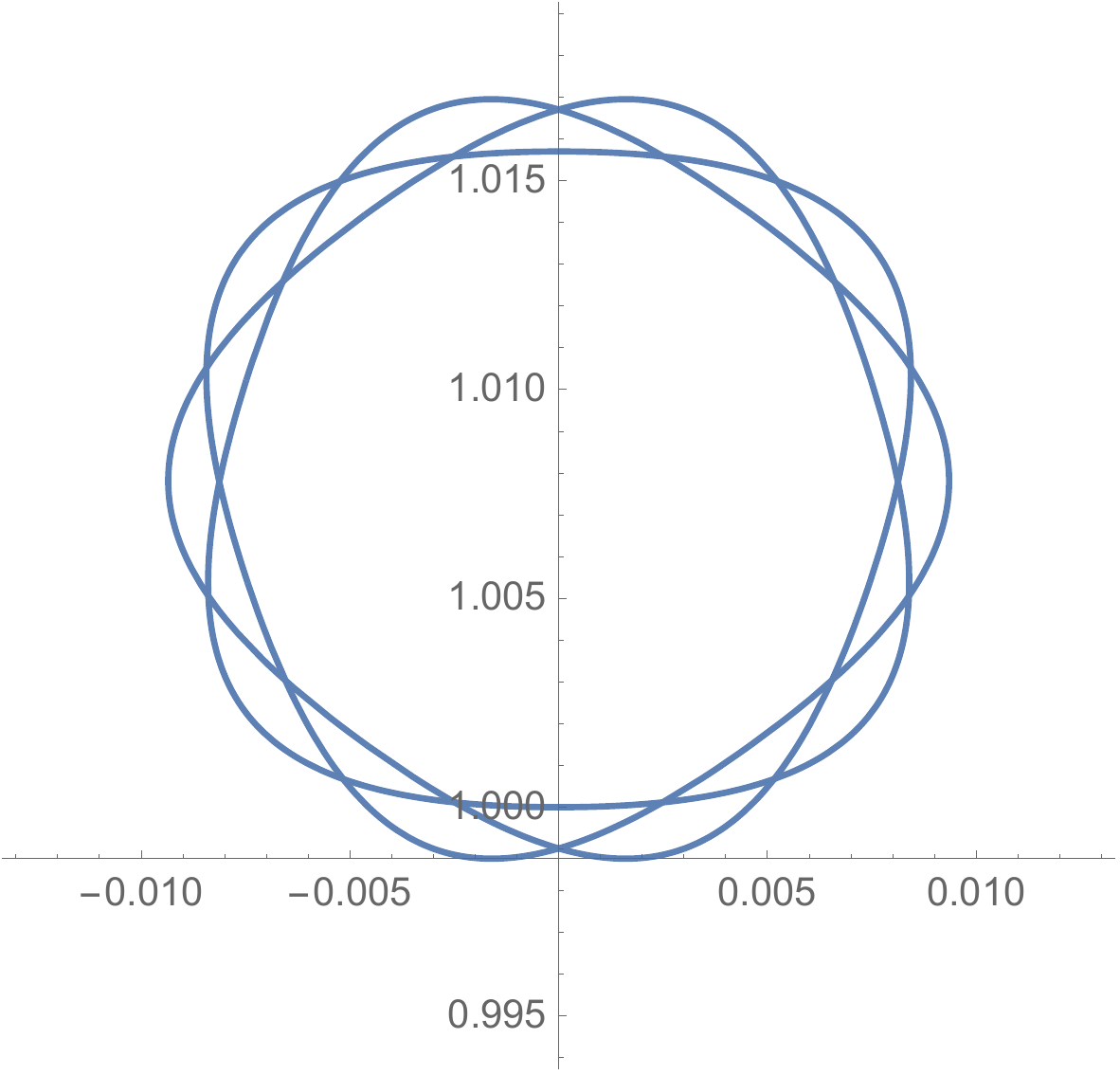}
 \includegraphics[height=2in]{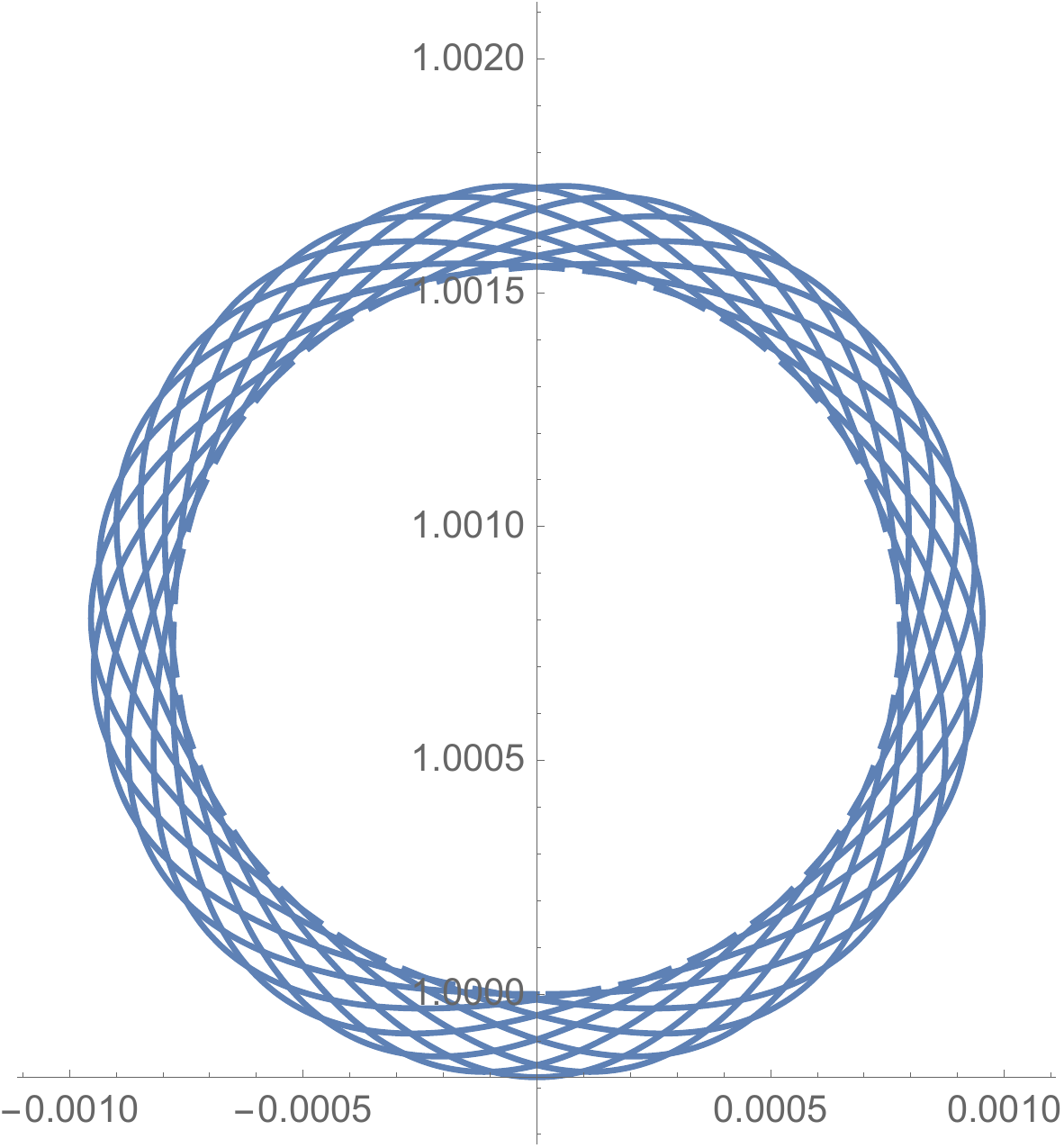}        
\caption{Kuperberg representation of
the smoothed octagon, $10$-gon, $20$-gon, $62$-gon}
\label{fig:kuperberg}
\end{figure}

\section{Research Problems}\label{sec:problems}

We begin with a few questions related to circular control sets.

\begin{enumerate}
\item 
In the context of a circular control set, determine numerically, the
parameter regions that give chaos.
\item Give a comprehensive description of the global dynamics for
circular control, building on the description of global dynamics for
the Fuller system.
\end{enumerate}

Here are some research questions related to triangular control $U_T$.

\begin{enumerate}
\item We do not give an upper bound on the number of edges in the
  smoothed polygon.  It seems to us that an extension of the methods
  presented here might lead to an explicit upper bound on the number
  of edges.  (In fact, we expect that our analysis of the singular
  locus completes the most difficult stage of the proof of the full
  Reinhardt conjecture.)  To obtain a bound on the number of edges, it
  would be useful to extend our analysis from trajectories that meet
  the singular locus to include trajectories that come within a small
  neighborhood of the singular locus.  It might then be possible to
  obtain an upper bound on the number of control mode switches for
  trajectories that avoid a given small neighborhood of the singular
  locus.
\item We have a family of dynamical systems parameterized by
  $d=\det(\Lambda_1)$.  For each $d$, the Poincar\'e section is a
  four-dimensional space. Short of giving the complete solution to the
  Reinhardt problem, the Reinhardt problem might be solved for
  particular $d$ (such as $d=0$).
\item In the particular case $d=\det(\Lambda_1)=d_1^2=9/4$, we might ask
  whether our analysis of the behavior of the dynamical system around
  the singular locus (the stable and unstable manifolds at the fixed
  points) gives the comprehensive picture.  That is, do trajectories
  generally start at the boundary of the star domain, start to move
  inward toward the fixed point $q_{in}$, only to veer toward the other
  fixed point $q_{out}$, and finally move back out to the boundary of
  the star domain?
\item We have described the global behavior of trajectories that meet
  the singular locus, including the global behavior of the Fuller
  system.  To what extent do the methods introduced there (such as the
  involution $\itf$, the geometric partition of the domain, and a
  block triangular structure of the Poincar\'e map) generalize to the
  Reinhardt system?  Are these ideas sufficient to give a full proof
  of the Reinhardt conjecture?
\end{enumerate}

\clearpage







    



\singlespacing

\bibliographystyle{plain}
\bibliography{Bibliography.bib}


\printindex  
\printindex[n] 


\end{document}